\documentclass[10pt,reqno]{amsart}

\usepackage[utf8x]{inputenc}
\usepackage[english]{babel}
\usepackage{amsmath,amsfonts,amssymb,amsthm,shuffle}
\usepackage[T1]{fontenc}
\usepackage[math]{anttor}

\usepackage[top=3.4cm,bottom=3.4cm,left=3.3cm,right=3.3cm]{geometry}

\usepackage{xcolor}

\definecolor{ColBlack}{RGB}{0,0,0} % Black.
\definecolor{ColWhite}{RGB}{255,255,255} % White.
\definecolor{ColAA}{HTML}{520db1} % Violet bleu
\definecolor{ColAB}{HTML}{1a34c0} % Bleu 
\definecolor{ColAC}{HTML}{3851db} % Bleu ocean
\definecolor{ColBA}{HTML}{a80b3a} % Magenta
\definecolor{ColBB}{HTML}{a80b27} % Soleil couchant
\definecolor{ColBC}{HTML}{b10d0d} % Rouge

\usepackage[hyperindex=true,frenchlinks=true,colorlinks=true,citecolor=ColBB,
linkcolor=ColBA,urlcolor=ColBC,linktocpage,pagebackref=true]{hyperref}

\usepackage{tikz}
\usetikzlibrary{shapes}
\usetikzlibrary{fit}
\usetikzlibrary{decorations.pathmorphing}
\usetikzlibrary{arrows.meta}

\usepackage{mathtools}
\usepackage{dsfont}
\usepackage{wasysym}
\usepackage{stmaryrd}
\usepackage{mleftright}
\usepackage{cite}
\usepackage{subfig}
\usepackage{multirow}
\usepackage{enumitem}
\usepackage{multicol}
\usepackage{cancel}

\renewcommand{\arraystretch}{1.4}

\allowdisplaybreaks

\numberwithin{equation}{subsection}

\makeatletter
\def\l@section{\@tocline{1}{3pt}{1pc}{5pc}{}}
\def\l@subsection{\@tocline{2}{2pt}{2pc}{5pc}{}}
\makeatother

\newtheorem{Theorem}{Theorem}[subsection]
\newtheorem{Proposition}[Theorem]{Proposition}
\newtheorem{Lemma}[Theorem]{Lemma}
\newtheorem{Conjecture}[Theorem]{Conjecture}

\renewcommand{\leq}{\leqslant}
\renewcommand{\geq}{\geqslant}

\newcommand{\ColAA}[1]{\textcolor{ColAA}{#1}}
\newcommand{\ColAB}[1]{\textcolor{ColAB}{#1}}

\newcommand{\Hide}[1]{\ColAA{\tt HIDEN}}
\newcommand{\Def}[1]{\ColAB{\em #1}}
\newcommand{\Par}[1]{\mleft(#1\mright)}
\newcommand{\Bra}[1]{\mleft\{#1\mright\}}
\newcommand{\Han}[1]{\mleft[#1\mright]}
\newcommand{\HanL}[1]{\mleft\llbracket#1\mright]}
\newcommand{\Brr}[1]{\mleft|#1\mright|}
\newcommand{\OEIS}[1]{\href{http://oeis.org/#1}{{\bf #1}}}

\tikzstyle{Centering}=[{baseline={([yshift=-0.5ex]current bounding box.center)}}]

\tikzstyle{MarkAA}=[draw=ColAA!80,fill=ColAA!8]
\tikzstyle{MarkAB}=[draw=ColAB!80,fill=ColAB!8]
\tikzstyle{MarkAC}=[draw=ColAC!80,fill=ColAC!8]
\tikzstyle{MarkBA}=[draw=ColBA!80,fill=ColBA!8]
\tikzstyle{MarkBB}=[draw=ColBB!80,fill=ColBB!8]
\tikzstyle{MarkBC}=[draw=ColBC!80,fill=ColBC!8]

\tikzstyle{Node}=[circle,MarkAA,inner sep=1pt,minimum size=2mm,thick,font=\scriptsize]
\tikzstyle{Edge}=[draw=ColBB!80,cap=round,thick,rounded corners=2.5pt]
\tikzstyle{Leaf}=[rectangle,MarkBC,inner sep=0pt,minimum size=1mm,thick]
\tikzstyle{NodeST}=[font=\scriptsize]

\tikzstyle{NodeGraph}=[circle,MarkAB,inner sep=1pt,minimum size=1.5mm,thick]
\tikzstyle{NodeLabelGraph}=[font=\tiny,node distance=3mm]
\tikzstyle{MarkedNodeGraph}=[NodeGraph,rectangle,draw=ColAB!90,fill=ColAB!50,
minimum size=1.5mm]
\tikzstyle{Marked2NodeGraph}=[NodeGraph,regular polygon,regular polygon sides=6,
draw=ColBA!90,fill=ColBA!50,minimum size=1.75mm]
\tikzstyle{EdgeGraph}=[ColBB!70,cap=round,thick]
\tikzstyle{MarkedEdgeGraph}=[EdgeGraph,ColAA!90,very thick]
\tikzstyle{Marked2EdgeGraph}=[EdgeGraph,ColBA!90,very thick]
\tikzstyle{EdgeLabel}=[midway,inner sep=1pt,fill=ColWhite!0,font=\tiny]
\tikzstyle{MapGraph}=[ColAA!100,draw,dashed,-{>[scale=1.0,length=6,width=5]}]
\tikzstyle{FaceXY}=[fill=ColAA,opacity=.1]
\tikzstyle{FaceXZ}=[fill=ColBA,opacity=.2]
\tikzstyle{FaceYZ}=[fill=ColBC,opacity=.2]
\tikzstyle{Face}=[FaceYZ]

\tikzstyle{PathNode}=[NodeGraph]
\tikzstyle{PathStep}=[EdgeGraph]
\tikzstyle{PathDiag}=[EdgeGraph,ColBA!30,dotted]

\tikzstyle{Map}=[ColBlack!100,draw,-{>[scale=1.5,length=4,width=5]}]
\tikzstyle{Injection}=[ColBlack!100,draw,{>[scale=1.5,length=4,width=5]}-{>[scale=1.5,
length=4,width=5]}]
\tikzstyle{MapEmbedding}=[Injection]
\tikzstyle{MapIsomorphism}=[Map,double]

\tikzstyle{LineGrid}=[very thin,dashed,draw=ColAC!40]
\tikzstyle{Grid}=[LineGrid]

\newcommand{\N}{\mathbb{N}}
\newcommand{\Z}{\mathbb{Z}}

\newcommand{\R}{\mathbb{R}}

\newcommand{\K}{\mathbb{K}}

\newcommand{\SetCliff}{\mathsf{Cl}}
\newcommand{\SetAvalanche}{\mathsf{Av}}
\newcommand{\SetHill}{\mathsf{Hi}}
\newcommand{\SetCanyon}{\mathsf{Ca}}
\newcommand{\SetPermutations}{\mathfrak{S}}
\newcommand{\SetDyckPaths}{\mathrm{Dy}}
\newcommand{\SetPrime}{\mathcal{P}}
\newcommand{\AlphabetVar}{\mathbb{A}}
\newcommand{\SubFamilly}{\mathcal{S}}

\newcommand{\GeneratingSeries}{\mathcal{G}}
\newcommand{\HilbertSeries}{\mathcal{H}}
\newcommand{\FussCatalan}{\mathrm{cat}}
\newcommand{\TwistedFussCatalan}{\FussCatalan'}
\newcommand{\Unit}{\mathds{1}}
\newcommand{\IndicatorFunction}{\upharpoonleft}

\newcommand{\Leq}{\preccurlyeq}

\newcommand{\LatticeL}{\mathcal{L}}
\newcommand{\PosetP}{\mathcal{P}}
\newcommand{\Hypercube}{\mathcal{H}}
\newcommand{\JoinIrreducibles}{\mathbf{J}}
\newcommand{\MeetIrreducibles}{\mathbf{M}}

\newcommand{\LeastElement}{\bar{0}}
\newcommand{\GreatestElement}{\bar{1}}

\DeclareMathOperator{\JJoin}{\vee}
\DeclareMathOperator{\Meet}{\wedge}

\newcommand{\Angle}[1]{\left\langle#1\right\rangle}
\DeclareMathOperator{\Covered}{\lessdot}
\newcommand{\DiffIndices}{\mathrm{D}}
\newcommand{\Weight}{\omega}
\newcommand{\TreeT}{\mathfrak{t}}

\newcommand{\LeqSuffix}{\leq_{\mathrm{s}}}
\newcommand{\DimensionDelta}{\mathrm{dim}}
\newcommand{\MapM}{\mathbf{m}}
\newcommand{\MapOne}{\mathbf{1}}
\newcommand{\MapTwo}{\mathbf{2}}

\newcommand{\Next}{\mathrm{F}}

\newcommand{\OutputWings}{\mathcal{O}}
\newcommand{\InputWings}{\mathcal{I}}
\newcommand{\Butterflies}{\mathcal{B}}

\newcommand{\ElevationMap}{\mathbf{e}}
\newcommand{\ElevationImage}{\mathcal{E}}
\newcommand{\TreeMap}{\mathrm{tree}}
\newcommand{\LehmerMap}{\mathrm{leh}}
\newcommand{\Complementary}{\mathrm{c}}
\newcommand{\Reduction}{\mathrm{r}}
\newcommand{\IncreaseLetter}{{\uparrow\!}}
\newcommand{\DecreaseLetter}{{\downarrow\!}}
\newcommand{\IncrMap}{{\Uparrow\!}}
\newcommand{\DecrMap}{{\Downarrow\!}}
\newcommand{\DominantCanyon}{\mathrm{d}}
\newcommand{\CharacteristicCliff}{\chi}
\newcommand{\Id}{\mathrm{I}}

\newcommand{\CubicReal}{\mathfrak{C}}
\newcommand{\Volume}{\mathrm{vol}}

\newcommand{\MaxLastLetter}{\mathrm{m}}

\newcommand{\DerivationOnSet}{\mathcal{D}}
\newcommand{\IntervalTwo}{\MapTwo}

\newcommand{\BasisE}{\mathsf{E}}
\newcommand{\BasisF}{\mathsf{F}}
\newcommand{\BasisG}{\mathsf{G}}
\newcommand{\BasisH}{\mathsf{H}}

\newcommand{\SpaceCliff}{\mathbf{Cl}}
\newcommand{\SpaceV}{\mathcal{V}}
\newcommand{\RelationSpace}{\mathcal{R}}
\newcommand{\SpaceHill}{\mathbf{Hi}}
\newcommand{\SpaceCanyon}{\mathbf{Ca}}
\newcommand{\FQSym}{\mathsf{FQSym}}
\newcommand{\PBT}{\mathsf{PBT}}
\newcommand{\Sym}{\mathsf{Sym}}

\DeclareMathOperator{\OverG}{\dashv}
\DeclareMathOperator{\UnderG}{\vdash}
\DeclareMathOperator{\Product}{\cdot}
\DeclareMathOperator{\Coproduct}{\Delta}

\DeclareMathOperator{\Over}{
\begin{tikzpicture}[Centering,scale=.16]
    \draw(0,0)--(1,1.25);
\end{tikzpicture}}

\DeclareMathOperator{\Under}{
\begin{tikzpicture}[Centering,scale=.16]
    \draw(0,0)--(1,-1.25);
\end{tikzpicture}}

\newcommand{\DrawGridSpace}[3]{
    \foreach \x in {0,...,#1} {\foreach \y in {0,...,#2} {\foreach \z in {0,...,#3} {
        \draw[LineGrid](0,\y,\z)--(\x,\y,\z);
        \draw[LineGrid](\x,0,\z)--(\x,\y,\z);
        \draw[LineGrid](\x,\y,0)--(\x,\y,\z);}}}}

\title[Fuss-Catalan posets and algebras]
    {Three Fuss-Catalan posets in interaction \\ and their associative algebras}
\date{\today}
\author{Camille Combe}
\address{\scriptsize Institut de Recherche Mathématique Avancée UMR 7501,
Université de Strasbourg et CNRS,
7 rue René Descartes
67000 Strasbourg, France.}
\email{combe@math.unistra.fr}
\author{Samuele Giraudo}
\address{\scriptsize LIGM, Université Gustave Eiffel, CNRS, ESIEE Paris, F-$77454$
Marne-la-Vallée, France.}
\email{samuele.giraudo@u-pem.fr}

\begin{document}

%%%%%%%%%%%%%%%%%%%%%%%%%%%%%%%%%%%%%%%%%%%%%%%%%%%%%%%%%%%%%%%%%%%%%%%%%%%%%%%%%%%%%%%%%%%%
%%%%%%%%%%%%%%%%%%%%%%%%%%%%%%%%%%%%%%%%%%%%%%%%%%%%%%%%%%%%%%%%%%%%%%%%%%%%%%%%%%%%%%%%%%%%
%%%%%%%%%%%%%%%%%%%%%%%%%%%%%%%%%%%%%%%%%%%%%%%%%%%%%%%%%%%%%%%%%%%%%%%%%%%%%%%%%%%%%%%%%%%%
\begin{abstract}
We introduce $\delta$-cliffs, a generalization of permutations and increasing trees
depending on a range map $\delta$. We define a first lattice structure on these objects and
we establish general results about its subposets. Among them, we describe sufficient
conditions to have EL-shellable posets, lattices with algorithms to compute the meet and the
join of two elements, and lattices constructible by interval doubling. Some of these
subposets admit natural geometric realizations. Then, we introduce three families of
subposets which, for some maps $\delta$, have underlying sets enumerated by the Fuss-Catalan
numbers. Among these, one is a generalization of Stanley lattices and another one is a
generalization of Tamari lattices.  These three families of posets fit into a chain for the
order extension relation and they share some properties. Finally, in the same way as the
product of the Malvenuto-Reutenauer algebra forms intervals of the right weak Bruhat order
of permutations, we construct algebras whose products form intervals of the lattices of
$\delta$-cliff. We provide necessary and sufficient conditions on $\delta$ to have
associative, finitely presented, or free algebras. We end this work by using the previous
Fuss-Catalan posets to define quotients of our algebras of $\delta$-cliffs. In particular,
one is a generalization of the Loday-Ronco algebra.
\end{abstract}

\maketitle

\vspace{-2em}

\begin{footnotesize}
\tableofcontents
\end{footnotesize}

%%%%%%%%%%%%%%%%%%%%%%%%%%%%%%%%%%%%%%%%%%%%%%%%%%%%%%%%%%%%%%%%%%%%%%%%%%%%%%%%%%%%%%%%%%%%
%%%%%%%%%%%%%%%%%%%%%%%%%%%%%%%%%%%%%%%%%%%%%%%%%%%%%%%%%%%%%%%%%%%%%%%%%%%%%%%%%%%%%%%%%%%%
%%%%%%%%%%%%%%%%%%%%%%%%%%%%%%%%%%%%%%%%%%%%%%%%%%%%%%%%%%%%%%%%%%%%%%%%%%%%%%%%%%%%%%%%%%%%
\section*{Introduction}
The theory of combinatorial Hopf algebras takes a prominent place in algebraic
combinatorics. The Malvenuto-Reutenauer algebra $\FQSym$~\cite{MR95,DHT02} is a central
object in this theory. This structure is defined on the linear span of all permutations and
the product of two permutations has the notable property to form an interval of the right
weak Bruhat order. Moreover, $\FQSym$ admits a lot of substructures, like the Loday-Ronco
algebra of binary trees $\PBT$~\cite{LR98,HNT05} and the algebra of noncommutative symmetric
functions $\Sym$~\cite{GKLLRT94}. Each of these structures brings out in a beautiful and
somewhat unexpected way the combinatorics of some partial orders, respectively the Tamari
order~\cite{Tam62} and the Boolean lattice, playing the same role as the one played by the
right weak Bruhat order for~$\FQSym$. To be slightly more precise, all these algebraic
structures have, as common point, a product $\Product$ which expresses, on their so-called
fundamental bases $\Bra{\BasisF_x}_x$, as
\begin{equation}
    \BasisF_x \Product \BasisF_y
    = \sum_{x \Over y \Leq z \Leq  x \Under y} \BasisF_z,
\end{equation}
where $\Leq$ is a partial order on basis elements, and $\Over$ and $\Under$ are some binary
operations on basis elements (in most cases, some sorts of concatenation operations).
\medbreak

The point of departure of this work consists in considering a different partial order
relation on permutations and ask to what extent analogues of $\FQSym$ and a similar
hierarchy of algebras arise in this context. We consider here first a very natural order on
permutations: the componentwise ordering $\Leq$ on Lehmer codes of
permutations~\cite{Leh60}.  A study of these posets $\SetCliff_\MapOne(n)$ appears
in~\cite{Den13}. Each poset $\SetCliff_\MapOne(n)$ is an order extension of the right weak
Bruhat order of order $n$. To give a concrete point of comparison, the Hasse diagrams of the
right weak Bruhat order of order $3$ and of $\SetCliff_\MapOne(3)$ are respectively

\begin{minipage}{.4\textwidth}
\begin{equation}
    \scalebox{1}{
    \begin{tikzpicture}[Centering,xscale=.8,yscale=.7]
        \node[NodeGraph](000)at(0,0){};
        \node[NodeGraph](001)at(-1,-1){};
        \node[NodeGraph](002)at(-1,-2){};
        \node[NodeGraph](010)at(1,-1){};
        \node[NodeGraph](011)at(1,-2){};
        \node[NodeGraph](012)at(0,-3){};
        \node[NodeLabelGraph,above of=000]{$123$};
        \node[NodeLabelGraph,left of=001]{$132$};
        \node[NodeLabelGraph,left of=002]{$312$};
        \node[NodeLabelGraph,right of=010]{$213$};
        \node[NodeLabelGraph,right of=011]{$231$};
        \node[NodeLabelGraph,below of=012]{$321$};
        \draw[EdgeGraph](000)--(001);
        \draw[EdgeGraph](000)--(010);
        \draw[EdgeGraph](001)--(002);
        \draw[EdgeGraph](010)--(011);
        \draw[EdgeGraph](002)--(012);
        \draw[EdgeGraph](011)--(012);
    \end{tikzpicture}}
\end{equation}
\end{minipage}
\qquad and
\begin{minipage}{.4\textwidth}
\begin{equation}
    \scalebox{1}{
    \begin{tikzpicture}[Centering,xscale=.8,yscale=.7]
        \node[NodeGraph](000)at(0,0){};
        \node[NodeGraph](001)at(-1,-1){};
        \node[NodeGraph](002)at(-1,-2){};
        \node[NodeGraph](010)at(1,-1){};
        \node[NodeGraph](011)at(1,-2){};
        \node[NodeGraph](012)at(0,-3){};
        \node[NodeLabelGraph,above of=000]{$000$};
        \node[NodeLabelGraph,left of=001]{$001$};
        \node[NodeLabelGraph,left of=002]{$002$};
        \node[NodeLabelGraph,right of=010]{$010$};
        \node[NodeLabelGraph,right of=011]{$011$};
        \node[NodeLabelGraph,below of=012]{$012$};
        \draw[EdgeGraph](000)--(001);
        \draw[EdgeGraph](000)--(010);
        \draw[EdgeGraph](001)--(002);
        \draw[EdgeGraph](001)--(011);
        \draw[EdgeGraph](010)--(011);
        \draw[EdgeGraph](002)--(012);
        \draw[EdgeGraph](011)--(012);
    \end{tikzpicture}}.
\end{equation}
\end{minipage}

\noindent As we can observe, the right weak Bruhat order relation on permutations of size
$3$ is included into the order relation of~$\SetCliff_\MapOne(3)$.
\medbreak

In this work, we consider a more general version of Lehmer codes, called $\delta$-cliffs,
leading to distributive lattices $\SetCliff_\delta$. Here $\delta$ is a parameter which is a
map $\N \setminus \{0\} \to \N$, called range map, assigning to each position of the words a
maximal allowed value.  The linear spans $\SpaceCliff_\delta$ of these sets are endowed with
a very natural product related to the intervals of $\SetCliff_\delta$. Some properties of
this product are implied by the general shape of $\delta$.  For instance, when $\delta$ is
so-called valley-free, $\SpaceCliff_\delta$ is an associative algebra, and when $\delta$ is
weakly increasing, $\SpaceCliff_\delta$ is free as a unital associative algebra. The
particular algebra $\SpaceCliff_\MapOne$ is in fact isomorphic to $\FQSym$, so that for
any range map $\delta$, $\SpaceCliff_\delta$ is a generalization of this latter. For
instance, when $\delta$ is the map $\MapM$ satisfying $\MapM(i) = (i - 1) m$ with $m \in
\N$, then all $\SpaceCliff_{\MapM}$ are free associative algebras whose bases are indexed by
increasing trees wherein all nodes have $m + 1$ children.
\medbreak

In the same way as the Tamari order can be defined by restricting the right weak Bruhat
order to some permutations, one builds three subposets of $\SetCliff_\delta$ by restricting
$\Leq$ to particular $\delta$-cliffs. This leads to three families $\SetAvalanche_\delta$,
$\SetHill_\delta$, and $\SetCanyon_\delta$ of posets. When $\delta$ is the particular map
$\MapM$ defined above with $m \geq 0$, the underlying sets of all these posets of order $n
\geq 0$ are enumerated by the $n$-th $m$-Fuss-Catalan number~\cite{DM47}
\begin{equation}
    \FussCatalan_m(n) := \frac{1}{mn + 1} \binom{mn + n}{n}.
\end{equation}
These posets have some close interactions: when $\delta$ is an increasing map,
$\SetHill_\delta$ is an order extension of $\SetCanyon_\delta$, which is itself an order
extension of $\SetAvalanche_\delta$. Besides, $\SetHill_\MapM$ (resp.\ $\SetCanyon_\MapM$)
generalizes for any $m \geq 0$ the Stanley lattice~\cite{Sta75,Knu04} (resp.\ Tamari
lattice), which occurs when $\MapM = \MapOne$. Our generalization of Tamari lattices is
different from the classical one introduced in~\cite{BPR12}. Besides, from these posets
$\SetHill_\MapM$ and $\SetCanyon_\MapM$, one defines respectively two quotient algebras
$\SpaceHill_m$ and $\SpaceCanyon_m$ of $\SpaceCliff_\MapM$. Notably, The algebra
$\SpaceCanyon_1$ is isomorphic to $\PBT$, and the other ones $\SpaceCanyon_m$, $m \geq 2$,
are not free as associative algebras.
\medbreak

This paper is organized as follows.
\medbreak

Section~\ref{sec:cliff_posets} is intended to introduce $\delta$-cliffs and to set some
notations and recalls about poset theory. As a by-product in the process of establishing
links between the posets $\SetCliff_\delta(n)$ and the weak Bruhat order, we introduce an
alternative poset $\Par{\SetCliff_\delta(n), \Leq'}$ when $\delta$ satisfies some particular
conditions. We prove that when $\delta = \MapOne$, the obtained poset is the weak Bruhat
order and we conjecture that for all authorized range maps $\delta$, the posets
$\Par{\SetCliff_\delta(n), \Leq'}$ are semi-distributive lattices. Besides, even if the
posets $\SetCliff_\delta(n)$ have a very simple structure, they contain interesting
subposets $\SubFamilly(n)$. To study these substructures, we establish a series of
sufficient conditions on $\SubFamilly(n)$ for the fact that these posets are
EL-shellable~\cite{BW96,BW97}, are lattices (and give algorithms to compute the meet and the
join of two elements), and are constructible by interval doubling~\cite{Day79}. Moreover,
under some precise conditions, each subposet $\SubFamilly(n)$ can be seen as a geometric
object in $\R^n$. We call this the geometric realization of $\SubFamilly(n)$. We introduce
here the notion of cell and expose a way to compute the volume of the geometrical object.
\medbreak

Next, in Section~\ref{sec:fuss_catalan_posets}, we study the posets $\SetAvalanche_\delta$,
$\SetHill_\delta$, and $\SetCanyon_\delta$. For each of these, we provide some general
properties (EL-shellability, lattice property, constructibility by interval doubling), and
describe its input-wings, output-wings, and butterflies elements, that are elements having
respectively a maximal number of covered elements, covering elements, or both properties at
the same time. We observe a surprising phenomenon: some posets $\SetAvalanche_\delta$,
$\SetHill_\delta$, or $\SetCanyon_\delta$ are isomorphic to their subposets restrained on
input-wings, output-wings, or butterflies elements. Moreover, a notable link among other
ones is that the subposet of $\SetCanyon_\MapM(n)$ is isomorphic to the subposet of
$\SetHill_{\mathbf m - 1}(n)$ restrained to its input-wings.  We also study further
interactions between our three families of Fuss-Catalan posets: there are for instance
bijective posets morphisms (but not poset isomorphisms) between $\SetAvalanche_\delta$ and
$\SetCanyon_\delta$, and between $\SetCanyon_\delta$ and $\SetHill_\delta$, when $\delta$ is
increasing.
\medbreak

Finally, Section~\ref{sec:cliff_algebras} presents a study of the algebra
$\SpaceCliff_\delta$.  We start by introducing a natural coproduct on $\SpaceCliff_\delta$
in order to obtain by duality a product, associative in some cases. Three alternative bases
of $\SpaceCliff_\delta$ are introduced, including two that are multiplicative and are
defined from the order on $\delta$-cliffs.  We then rely on these bases to give a
presentation by generators and relations of $\SpaceCliff_\delta$. When $\delta$ is
valley-free and is $1$-dominated (that is a certain condition on range maps),
$\SpaceCliff_\delta$ admits a finite presentation (a finite number of generators and a
finite number of nontrivial relations between the generators). When $\delta$ is weakly
increasing, $\SpaceCliff_\delta$ is free as an associative algebra. We end this work by
constructing, given a subfamilly $\SubFamilly$ of $\SetCliff_\delta$, a quotient space
$\SpaceCliff_\SubFamilly$ of $\SpaceCliff_\delta$ isomorphic to the linear span of
$\SubFamilly$. A sufficient condition on $\SubFamilly$ to have moreover a quotient algebra
of $\SpaceCliff_\delta$ is introduced.  We also describe a sufficient condition on
$\SubFamilly$ for the fact that the product of two basis elements of
$\SpaceCliff_\SubFamilly$ is an interval of a poset $\SubFamilly(n)$.  These results are
applied to construct and study the two quotients $\SpaceHill_m :=
\SpaceCliff_{\SetHill_\MapM}$ and $\SpaceCanyon_m := \SpaceCliff_{\SetCanyon_\MapM}$ of
$\SpaceCliff_\MapM$. The algebra $\SpaceCanyon_1$ is isomorphic to the Loday-Ronco algebra
and the other algebras $\SpaceCanyon_m$, $m \geq 2$, provide generalizations of this later
which are not free. On the other hand, for any $m \geq 1$, all $\SpaceHill_m$ are other
associative algebras whose dimensions are also Fuss-Catalan numbers and are not free.
\medbreak

This paper is an extended version of~\cite{CG20} containing the proofs of the presented
results and presenting new ones as the geometrical aspects of the studied posets.
\medbreak

%%%%%%%%%%%%%%%%%%%%%%%%%%%%%%%%%%%%%%%%%%%%%%%%%%%%%%%%%%%%%%%%%%%%%%%%%%%%%%%%%%%%%%%%%%%%
\subsubsection*{General notations and conventions}
For any integers $i$ and $j$, $[i, j]$ denotes the set $\{i, i + 1, \dots, j\}$. For any
integer $i$, $[i]$ denotes the set $[1, i]$ and $\HanL{i}$ denotes the set $[0, i]$. Graded
sets are sets decomposing as a disjoint union $S = \bigsqcup_{n \geq 0} S(n)$. For any $x
\in S$, the unique $n \geq 0$ such that $x \in S(n)$ is the degree $|x|$ of $x$. A graded
subset of $S$ is a graded set $S'$ such that for all $n \geq 0$, $S'(n) \subseteq S(n)$. The
generating series of $S$ is the series $\GeneratingSeries_S(t) := \sum_{x \in S} t^{|x|}$.
The empty word is denoted by $\epsilon$. If $P$ is a statement, we denote by
$\IndicatorFunction_P$ the indicator function (equals to $1$ if $P$ holds and $0$
otherwise).
\medbreak

%%%%%%%%%%%%%%%%%%%%%%%%%%%%%%%%%%%%%%%%%%%%%%%%%%%%%%%%%%%%%%%%%%%%%%%%%%%%%%%%%%%%%%%%%%%%
%%%%%%%%%%%%%%%%%%%%%%%%%%%%%%%%%%%%%%%%%%%%%%%%%%%%%%%%%%%%%%%%%%%%%%%%%%%%%%%%%%%%%%%%%%%%
%%%%%%%%%%%%%%%%%%%%%%%%%%%%%%%%%%%%%%%%%%%%%%%%%%%%%%%%%%%%%%%%%%%%%%%%%%%%%%%%%%%%%%%%%%%%
\section{$\delta$-cliff posets and general properties} \label{sec:cliff_posets}
This section is devoted to introduce the lattices of $\delta$-cliffs and to set some
notations and definitions about posets and lattices. Then, we will review some properties of
its subposets, like EL-shellability, constructibility by interval doubling, and geometric
realizations.
\medbreak

%%%%%%%%%%%%%%%%%%%%%%%%%%%%%%%%%%%%%%%%%%%%%%%%%%%%%%%%%%%%%%%%%%%%%%%%%%%%%%%%%%%%%%%%%%%%
%%%%%%%%%%%%%%%%%%%%%%%%%%%%%%%%%%%%%%%%%%%%%%%%%%%%%%%%%%%%%%%%%%%%%%%%%%%%%%%%%%%%%%%%%%%%
\subsection{$\delta$-cliffs}
We introduce here $\delta$-cliffs, their links with Lehmer codes, permutations, and
particular increasing trees.
\medbreak

%%%%%%%%%%%%%%%%%%%%%%%%%%%%%%%%%%%%%%%%%%%%%%%%%%%%%%%%%%%%%%%%%%%%%%%%%%%%%%%%%%%%%%%%%%%%
\subsubsection{First definitions} \label{subsubsec:first_definitions_cliffs}
A \Def{range map} is a map $\delta : \N \setminus \{0\} \to \N$.  We shall specify range
maps as infinite words $\delta = \delta(1) \delta(2) \dots$. For this purpose, for any $a
\in \N$, we shall denote by $a^\omega$ the infinite word having all its letters equal to
$a$. We say that $\delta$
\begin{itemize}
    \item is \Def{rooted} if $\delta(1) = 0$;

    \item is \Def{weakly increasing} if for all $i \geq 1$, $\delta(i) \leq \delta(i + 1)$;

    \item is \Def{increasing} if for all $i \geq 1$, $\delta(i) < \delta(i + 1)$);

    \item has an \Def{ascent} if there are $1 \leq i_1 < i_2$ such that $\delta\Par{i_1} <
    \delta\Par{i_2}$;

    \item has an \Def{descent} if there are $1 \leq i_1 < i_2$ such that $\delta\Par{i_1} >
    \delta\Par{i_2}$;

    \item has a \Def{valley} if there are $1 \leq i_1 < i_2 < i_3$ such that
    $\delta\Par{i_1} > \delta\Par{i_2} < \delta\Par{i_3}$;

    \item is \Def{valley-free} (or \Def{unimodal}) if $\delta$ has no valley;

    \item is \Def{$j$-dominated} for a $j \geq 1$ if there is $k \geq 1$ such that for all
    $k' \geq k$, $\delta(j) \geq \delta\Par{k'}$.
\end{itemize}
For any $n \geq 0$, the \Def{$n$-th dimension} of $\delta$ is the integer
$\DimensionDelta_n(\delta) := \# \Bra{i \in [n] : \delta(i) \ne 0}$.
\medbreak

Given a range map $\delta$, a word $u$ of nonnegative integers of length $n$ is a
\Def{$\delta$-cliff} if for any $i \in [n]$, $u_i \in \HanL{\delta(i)}$. The \Def{size}
$|u|$ of a $\delta$-cliff $u$ is its length as a word, and the \Def{weight} $\Weight(u)$ of
$u$ is the sum of its letters. The graded set of all $\delta$-cliffs where the degree of a
$\delta$-cliff is its size, is denoted by $\SetCliff_\delta$.  In the sequel, for any $m
\geq 0$, we shall denote by $\mathbf{m}$ the range map satisfying
$\MapM(i) = (i - 1) m$ for any $i \in \N \setminus \{0\}$. For instance,
\begin{subequations}
\begin{equation}
    \SetCliff_\MapOne(3) = \Bra{000, 001, 002, 010, 011, 012},
\end{equation}
\begin{equation}
    \SetCliff_\MapTwo(3) =
    \Bra{000, 001, 002, 003, 004, 010, 011, 012, 013, 014, 020, 021, 022, 023, 024}.
\end{equation}
\end{subequations}
Let us denote respectively by $\LeastElement_\delta(n)$ and by $\GreatestElement_\delta(n)$
the $\delta$-cliffs $0^n$ and $\delta(1) \dots \delta(n)$.
\medbreak

It follows immediately from the definition of $\delta$-cliffs that the cardinality of
$\SetCliff_\delta(n)$ satisfies
\begin{equation}
    \# \SetCliff_\delta(n) = \prod_{i \in [n]} (\delta(i) + 1).
\end{equation}
The first numbers of $\mathbf{m}$-cliffs are
\begin{subequations}
\begin{equation}
    1, 1, 1, 1, 1, 1, 1, 1, \qquad m = 0,
\end{equation}
\begin{equation}
    1, 1, 2, 6, 24, 120, 720, 5040, \qquad m = 1,
\end{equation}
\begin{equation}
    1, 1, 3, 15, 105, 945, 10395, 135135, \qquad m = 2.
\end{equation}
\end{subequations}
Last sequence is Sequence~\OEIS{A001147} of~\cite{Slo}.
\medbreak

%%%%%%%%%%%%%%%%%%%%%%%%%%%%%%%%%%%%%%%%%%%%%%%%%%%%%%%%%%%%%%%%%%%%%%%%%%%%%%%%%%%%%%%%%%%%
\subsubsection{Lehmer codes and permutations}
There is a classical correspondence between permutations and \Def{Lehmer
codes}~\cite{Leh60}, that are certain words of integers.  Here, we consider a slight
variation of Lehmer codes, establishing a bijection between the set of $\MapOne$-cliffs of
size $n$ and the set of permutations of the same size. Given a permutation $\sigma$ of size
$n$, let the ${\mathbf 1}$-cliff $u$ such that for any $j \in [n]$, $u_j$ is the number of
values $i$ such that $i < j$ while $\sigma^{-1}(i) > \sigma^{-1}(j)$. We denote by
$\LehmerMap(\sigma)$ the ${\mathbf 1}$-cliff thus associated with the permutation $\sigma$.
For instance, $\LehmerMap(436512) = 002323$.
\medbreak

%%%%%%%%%%%%%%%%%%%%%%%%%%%%%%%%%%%%%%%%%%%%%%%%%%%%%%%%%%%%%%%%%%%%%%%%%%%%%%%%%%%%%%%%%%%%
\subsubsection{Weakly increasing range maps and increasing trees}
Given a rooted weakly increasing range map $\delta$, let $\Delta_\delta : \N \setminus \{0\}
\to \N$ be the map defined by $\Delta_\delta(i) := \delta(i + 1) - \delta(i)$. For instance,
for any $a \geq 0$, $\Delta_{a^\omega} = 0^\omega$, and for any $m \geq 0$, $\Delta_{\MapM}
= m^\omega$. A \Def{$\delta$-increasing tree} is a planar rooted tree where internal nodes
are bijectively labeled from $1$ to $n$, any internal node labeled by $i \in [n]$ has arity
$\Delta_\delta(i) + 1$, and every children of any node labeled by $i \in [n]$ are leaves or
are internal nodes labeled by $j \in [n]$ such that $j > i$. The \Def{size} of such a tree
is its number of internal nodes. The leaves of a $\delta$-increasing tree are implicitly
numbered from $1$ to its total number of leaves from left to right.
\medbreak

Observe that, regardless of any particular condition on $\delta$, any $\delta$-cliff $u$ of
size $n \geq 1$ recursively decomposes as $u = u' a$ where $a \in \HanL{\delta(n)}$ and $u'$
is a $\delta$-cliff of size $n - 1$. Relying on this observation, when $\delta$ is rooted
and weakly increasing, let $\TreeMap_\delta$ be the map sending any $\delta$-cliff $u$ of
size $n$ to the $\delta$-increasing tree of size $n$ recursively defined as follows. If $n =
0$, $\TreeMap_\delta(u)$ is the leaf. Otherwise, by using the above decomposition of $u$,
$\TreeMap_\delta(u)$ is the tree obtained by grafting on the $(a + 1)$-st leaf of the tree
$\TreeMap\Par{u'}$ a node of arity $\Delta_\delta(n) + 1$ labeled by $n$. For instance,
\begin{equation}
    \TreeMap_\MapTwo(0230228) \enspace = \enspace
    \scalebox{.8}{
    \begin{tikzpicture}[Centering,xscale=0.18,yscale=0.10]
        \node[Leaf](0)at(0.00,-8.80){};
        \node[Leaf](11)at(7.00,-4.40){};
        \node[Leaf](12)at(8.00,-13.20){};
        \node[Leaf](14)at(9.00,-13.20){};
        \node[Leaf](15)at(10.00,-13.20){};
        \node[Leaf](17)at(11.00,-13.20){};
        \node[Leaf](19)at(12.00,-13.20){};
        \node[Leaf](2)at(1.00,-8.80){};
        \node[Leaf](20)at(13.00,-13.20){};
        \node[Leaf](21)at(14.00,-8.80){};
        \node[Leaf](3)at(2.00,-17.60){};
        \node[Leaf](5)at(3.00,-17.60){};
        \node[Leaf](6)at(4.00,-17.60){};
        \node[Leaf](8)at(5.00,-13.20){};
        \node[Leaf](9)at(6.00,-13.20){};
        \node[Node](1)at(1.00,-4.40){$4$};
        \node[Node](10)at(7.00,0.00){$1$};
        \node[Node](13)at(9.00,-8.80){$7$};
        \node[Node](16)at(12.00,-4.40){$2$};
        \node[Node](18)at(12.00,-8.80){$3$};
        \node[Node](4)at(3.00,-13.20){$6$};
        \node[Node](7)at(5.00,-8.80){$5$};
        \draw[Edge](0)--(1);
        \draw[Edge](1)--(10);
        \draw[Edge](11)--(10);
        \draw[Edge](12)--(13);
        \draw[Edge](13)--(16);
        \draw[Edge](14)--(13);
        \draw[Edge](15)--(13);
        \draw[Edge](16)--(10);
        \draw[Edge](17)--(18);
        \draw[Edge](18)--(16);
        \draw[Edge](19)--(18);
        \draw[Edge](2)--(1);
        \draw[Edge](20)--(18);
        \draw[Edge](21)--(16);
        \draw[Edge](3)--(4);
        \draw[Edge](4)--(7);
        \draw[Edge](5)--(4);
        \draw[Edge](6)--(4);
        \draw[Edge](7)--(1);
        \draw[Edge](8)--(7);
        \draw[Edge](9)--(7);
        \node(r)at(7.00,3.75){};
        \draw[Edge](r)--(10);
    \end{tikzpicture}},
\end{equation}
and for $\delta := 0233579^\omega$, one has $\Delta_\delta = 2102220^\omega$, and
\begin{equation}
    \TreeMap_\delta(021042) \enspace = \enspace
    \scalebox{.8}{
    \begin{tikzpicture}[Centering,xscale=0.19,yscale=0.11]
        \node[Leaf](0)at(0.00,-8.00){};
        \node[Leaf](10)at(6.00,-12.00){};
        \node[Leaf](12)at(7.00,-12.00){};
        \node[Leaf](13)at(8.00,-12.00){};
        \node[Leaf](15)at(10.00,-8.00){};
        \node[Leaf](2)at(1.00,-8.00){};
        \node[Leaf](3)at(2.00,-12.00){};
        \node[Leaf](5)at(3.00,-12.00){};
        \node[Leaf](6)at(4.00,-12.00){};
        \node[Leaf](9)at(5.00,-8.00){};
        \node[Node](1)at(1.00,-4.00){$4$};
        \node[Node](11)at(7.00,-8.00){$5$};
        \node[Node](14)at(9.00,-4.00){$2$};
        \node[Node](4)at(3.00,-8.00){$6$};
        \node[Node](7)at(5.00,0.00){$1$};
        \node[Node](8)at(5.00,-4.00){$3$};
        \draw[Edge](0)--(1);
        \draw[Edge](1)--(7);
        \draw[Edge](10)--(11);
        \draw[Edge](11)--(14);
        \draw[Edge](12)--(11);
        \draw[Edge](13)--(11);
        \draw[Edge](14)--(7);
        \draw[Edge](15)--(14);
        \draw[Edge](2)--(1);
        \draw[Edge](3)--(4);
        \draw[Edge](4)--(1);
        \draw[Edge](5)--(4);
        \draw[Edge](6)--(4);
        \draw[Edge](8)--(7);
        \draw[Edge](9)--(8);
        \node(r)at(5.00,3.50){};
        \draw[Edge](r)--(7);
    \end{tikzpicture}}.
\end{equation}
\medbreak

\begin{Proposition} \label{prop:cliff_increasing_trees}
    For any rooted weakly increasing range map $\delta$, $\TreeMap_\delta$ is a one-to-one
    correspondence from the set of all $\delta$-cliffs of size $n \geq 0$ and the set of all
    $\delta$-increasing trees of size~$n$.
\end{Proposition}
\begin{proof}
    Let us first prove that $\TreeMap_\delta$ is a well-defined map. This can be done by
    induction on $n$ and arises from the fact that, for any $u \in \SetCliff_\delta(n)$, the
    total number of leaves of $\TreeMap_\delta(u)$ is $1 + \delta(n + 1)$. This is a
    consequence of the fact that
    \begin{equation}
        1 + \delta(n) - 1 + \Delta_\delta(n) + 1 = 1 + \delta(n + 1).
    \end{equation}
    Therefore, there is in $\TreeMap_\delta(u)$ a leaf of index $a + 1$ for any value $a \in
    \HanL{\delta(n + 1)}$. Hence, and due to the fact that by construction,
    $\TreeMap_\delta(u)$ is a $\delta$-increasing tree, the map $\TreeMap_\delta$ is
    well-defined.
    \smallbreak

    Now, let $\phi$ be the map from the set of all $\delta$-increasing trees of size $n$ to
    $\SetCliff_\delta(n)$ defined recursively as follows. If $\TreeT$ is the leaf, set
    $\phi(\TreeT) := \epsilon$.  Otherwise, consider the node with the maximal label in
    $\TreeT$.  Since $\TreeT$ is increasing, this node has no children. Set $\TreeT'$ as the
    $\delta$-increasing tree obtained by replacing this node by a leaf in $\TreeT$, and set
    $a$ as the index of the leaf of $\TreeT'$ on which this maximal node of $\TreeT$ is
    attached (this index is $1$ if $\TreeT'$ is the leaf). Then, set $\phi(\TreeT) :=
    \phi\Par{\TreeT'} (a - 1)$. The statement of the proposition follows by showing by
    induction on $n$ that $\phi$ is the inverse of the map~$\TreeMap_\delta$.
\end{proof}
\medbreak

In~\cite{CP19}, $s$-decreasing trees are considered, where $s$ is a sequence of length $n
\geq 0$ of nonnegative integers. These trees are labeled decreasingly and any internal node
labeled by $i \in [n]$ has arity $s_i$.  As a consequence of
Proposition~\ref{prop:cliff_increasing_trees}, any $s$-decreasing tree can be encoded by a
$\delta$-increasing tree where $\delta$ is a rooted weakly increasing range map satisfying
$\delta(i) = \sum_{1 \leq j \leq i - 1} s_{n - j + 1}$ for all $i \in [n + 1]$. The
correspondence between such $s$-decreasing trees and $\delta$-increasing trees consists in
relabeling by $n + 1 - i$ each internal node labeled by $i \in [n]$.  A consequence of all
this is that $\delta$-cliffs can be seen as generalizations of $s$-decreasing trees by
relaxing the considered conditions on~$\delta$.
\medbreak

%%%%%%%%%%%%%%%%%%%%%%%%%%%%%%%%%%%%%%%%%%%%%%%%%%%%%%%%%%%%%%%%%%%%%%%%%%%%%%%%%%%%%%%%%%%%
%%%%%%%%%%%%%%%%%%%%%%%%%%%%%%%%%%%%%%%%%%%%%%%%%%%%%%%%%%%%%%%%%%%%%%%%%%%%%%%%%%%%%%%%%%%%
\subsection{$\delta$-cliff posets}
We endow now the set of all $\delta$-cliffs of a given size with an order relation, give
some recalls about poset and lattice theory, and establish a link between the poset of
$\MapOne$-cliffs and the weak Bruhat order of permutations.
\medbreak

%%%%%%%%%%%%%%%%%%%%%%%%%%%%%%%%%%%%%%%%%%%%%%%%%%%%%%%%%%%%%%%%%%%%%%%%%%%%%%%%%%%%%%%%%%%%
\subsubsection{First definitions} \label{subsubsec:first_definitions_cliff_posets}
Let $\delta$ be a range map and  $\Leq$ be the partial order relation on $\SetCliff_\delta$
defined by $u \Leq v$ for any $u, v \in \SetCliff_\delta$ such that $|u| = |v|$ and $u_i
\leq v_i$ for all $i \in [|u|]$. For any $n \geq 0$, the poset $\Par{\SetCliff_\delta(n),
\Leq}$ is the \Def{$\delta$-cliff poset} of order $n$.
Figure~\ref{fig:examples_cliff_posets} shows the Hasse diagrams of some $\delta$-cliff
posets.
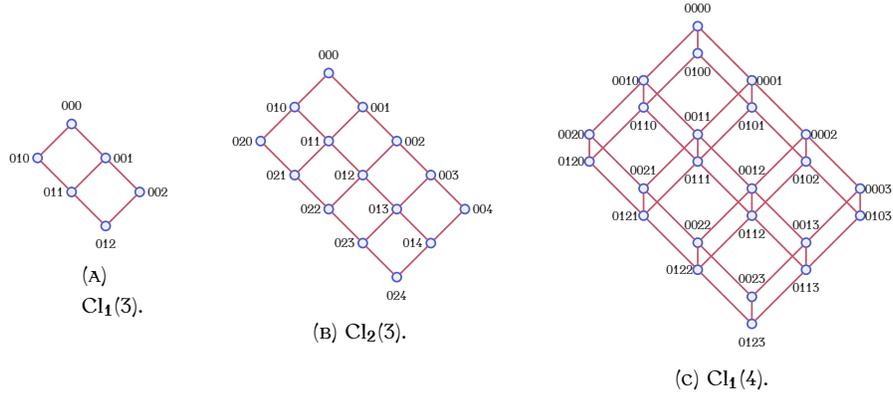
\begin{figure}[ht]
    \centering
    \subfloat[][$\SetCliff_\MapOne(3)$.]{
    \centering
    \scalebox{.8}{
    \begin{tikzpicture}[Centering,xscale=.8,yscale=.8,rotate=-135]
        \draw[Grid](0,0)grid(1,2);
        \node[NodeGraph](000)at(0,0){};
        \node[NodeGraph](001)at(0,1){};
        \node[NodeGraph](002)at(0,2){};
        \node[NodeGraph](010)at(1,0){};
        \node[NodeGraph](011)at(1,1){};
        \node[NodeGraph](012)at(1,2){};
        \node[NodeLabelGraph,above of=000]{$000$};
        \node[NodeLabelGraph,right of=001]{$001$};
        \node[NodeLabelGraph,right of=002]{$002$};
        \node[NodeLabelGraph,left of=010]{$010$};
        \node[NodeLabelGraph,left of=011]{$011$};
        \node[NodeLabelGraph,below of=012]{$012$};
        \draw[EdgeGraph](000)--(001);
        \draw[EdgeGraph](000)--(010);
        \draw[EdgeGraph](001)--(002);
        \draw[EdgeGraph](001)--(011);
        \draw[EdgeGraph](002)--(012);
        \draw[EdgeGraph](010)--(011);
        \draw[EdgeGraph](011)--(012);
    \end{tikzpicture}}
    \label{subfig:cliff_poset_1_3}}
    \quad
    \subfloat[][$\SetCliff_\MapTwo(3)$.]{
    \centering
    \scalebox{.8}{
    \begin{tikzpicture}[Centering,xscale=.8,yscale=.8,rotate=-135]
        \draw[Grid](0,0)grid(2,4);
        \node[NodeGraph](000)at(0,0){};
        \node[NodeGraph](001)at(0,1){};
        \node[NodeGraph](002)at(0,2){};
        \node[NodeGraph](003)at(0,3){};
        \node[NodeGraph](004)at(0,4){};
        \node[NodeGraph](010)at(1,0){};
        \node[NodeGraph](011)at(1,1){};
        \node[NodeGraph](012)at(1,2){};
        \node[NodeGraph](013)at(1,3){};
        \node[NodeGraph](014)at(1,4){};
        \node[NodeGraph](020)at(2,0){};
        \node[NodeGraph](021)at(2,1){};
        \node[NodeGraph](022)at(2,2){};
        \node[NodeGraph](023)at(2,3){};
        \node[NodeGraph](024)at(2,4){};
        \node[NodeLabelGraph,above of=000]{$000$};
        \node[NodeLabelGraph,right of=001]{$001$};
        \node[NodeLabelGraph,right of=002]{$002$};
        \node[NodeLabelGraph,right of=003]{$003$};
        \node[NodeLabelGraph,right of=004]{$004$};
        \node[NodeLabelGraph,left of=010]{$010$};
        \node[NodeLabelGraph,left of=011]{$011$};
        \node[NodeLabelGraph,left of=012]{$012$};
        \node[NodeLabelGraph,left of=013]{$013$};
        \node[NodeLabelGraph,left of=014]{$014$};
        \node[NodeLabelGraph,left of=020]{$020$};
        \node[NodeLabelGraph,left of=021]{$021$};
        \node[NodeLabelGraph,left of=022]{$022$};
        \node[NodeLabelGraph,left of=023]{$023$};
        \node[NodeLabelGraph,below of=024]{$024$};
        \draw[EdgeGraph](000)--(001);
        \draw[EdgeGraph](000)--(010);
        \draw[EdgeGraph](001)--(002);
        \draw[EdgeGraph](001)--(011);
        \draw[EdgeGraph](002)--(003);
        \draw[EdgeGraph](002)--(012);
        \draw[EdgeGraph](003)--(004);
        \draw[EdgeGraph](003)--(013);
        \draw[EdgeGraph](004)--(014);
        \draw[EdgeGraph](010)--(011);
        \draw[EdgeGraph](010)--(020);
        \draw[EdgeGraph](011)--(012);
        \draw[EdgeGraph](011)--(021);
        \draw[EdgeGraph](012)--(013);
        \draw[EdgeGraph](012)--(022);
        \draw[EdgeGraph](013)--(014);
        \draw[EdgeGraph](013)--(023);
        \draw[EdgeGraph](014)--(024);
        \draw[EdgeGraph](020)--(021);
        \draw[EdgeGraph](021)--(022);
        \draw[EdgeGraph](022)--(023);
        \draw[EdgeGraph](023)--(024);
    \end{tikzpicture}}
    \label{subfig:cliff_poset_2_3}}
    \quad
    \subfloat[][$\SetCliff_\MapOne(4)$.]{
    \centering
    \scalebox{.8}{
    \begin{tikzpicture}[Centering,xscale=.9,yscale=.9,
        x={(0,-.5cm)}, y={(-1.0cm,-1.0cm)}, z={(1.0cm,-1.0cm)}]
        \DrawGridSpace{1}{2}{3}
        \node[NodeGraph](0000)at(0,0,0){};
        \node[NodeGraph](0001)at(0,0,1){};
        \node[NodeGraph](0002)at(0,0,2){};
        \node[NodeGraph](0003)at(0,0,3){};
        \node[NodeGraph](0010)at(0,1,0){};
        \node[NodeGraph](0011)at(0,1,1){};
        \node[NodeGraph](0012)at(0,1,2){};
        \node[NodeGraph](0013)at(0,1,3){};
        \node[NodeGraph](0020)at(0,2,0){};
        \node[NodeGraph](0021)at(0,2,1){};
        \node[NodeGraph](0022)at(0,2,2){};
        \node[NodeGraph](0023)at(0,2,3){};
        \node[NodeGraph](0100)at(1,0,0){};
        \node[NodeGraph](0101)at(1,0,1){};
        \node[NodeGraph](0102)at(1,0,2){};
        \node[NodeGraph](0103)at(1,0,3){};
        \node[NodeGraph](0110)at(1,1,0){};
        \node[NodeGraph](0111)at(1,1,1){};
        \node[NodeGraph](0112)at(1,1,2){};
        \node[NodeGraph](0113)at(1,1,3){};
        \node[NodeGraph](0120)at(1,2,0){};
        \node[NodeGraph](0121)at(1,2,1){};
        \node[NodeGraph](0122)at(1,2,2){};
        \node[NodeGraph](0123)at(1,2,3){};
        \node[NodeLabelGraph,above of=0000]{$0000$};
        \node[NodeLabelGraph,right of=0001]{$0001$};
        \node[NodeLabelGraph,right of=0002]{$0002$};
        \node[NodeLabelGraph,right of=0003]{$0003$};
        \node[NodeLabelGraph,left of=0010]{$0010$};
        \node[NodeLabelGraph,above of=0011]{$0011$};
        \node[NodeLabelGraph,above of=0012]{$0012$};
        \node[NodeLabelGraph,above of=0013]{$0013$};
        \node[NodeLabelGraph,left of=0020]{$0020$};
        \node[NodeLabelGraph,above of=0021]{$0021$};
        \node[NodeLabelGraph,above of=0022]{$0022$};
        \node[NodeLabelGraph,above of=0023]{$0023$};
        \node[NodeLabelGraph,below of=0100]{$0100$};
        \node[NodeLabelGraph,below of=0101]{$0101$};
        \node[NodeLabelGraph,below of=0102]{$0102$};
        \node[NodeLabelGraph,right of=0103]{$0103$};
        \node[NodeLabelGraph,below of=0110]{$0110$};
        \node[NodeLabelGraph,below of=0111]{$0111$};
        \node[NodeLabelGraph,below of=0112]{$0112$};
        \node[NodeLabelGraph,below of=0113]{$0113$};
        \node[NodeLabelGraph,left of=0120]{$0120$};
        \node[NodeLabelGraph,left of=0121]{$0121$};
        \node[NodeLabelGraph,left of=0122]{$0122$};
        \node[NodeLabelGraph,below of=0123]{$0123$};
        \draw[EdgeGraph](0000)--(0001);
        \draw[EdgeGraph](0000)--(0010);
        \draw[EdgeGraph](0000)--(0100);
        \draw[EdgeGraph](0001)--(0002);
        \draw[EdgeGraph](0001)--(0011);
        \draw[EdgeGraph](0001)--(0101);
        \draw[EdgeGraph](0002)--(0003);
        \draw[EdgeGraph](0002)--(0012);
        \draw[EdgeGraph](0002)--(0102);
        \draw[EdgeGraph](0003)--(0013);
        \draw[EdgeGraph](0003)--(0103);
        \draw[EdgeGraph](0010)--(0011);
        \draw[EdgeGraph](0010)--(0020);
        \draw[EdgeGraph](0010)--(0110);
        \draw[EdgeGraph](0011)--(0012);
        \draw[EdgeGraph](0011)--(0021);
        \draw[EdgeGraph](0011)--(0111);
        \draw[EdgeGraph](0012)--(0013);
        \draw[EdgeGraph](0012)--(0022);
        \draw[EdgeGraph](0012)--(0112);
        \draw[EdgeGraph](0013)--(0023);
        \draw[EdgeGraph](0013)--(0113);
        \draw[EdgeGraph](0020)--(0021);
        \draw[EdgeGraph](0020)--(0120);
        \draw[EdgeGraph](0021)--(0022);
        \draw[EdgeGraph](0021)--(0121);
        \draw[EdgeGraph](0022)--(0023);
        \draw[EdgeGraph](0022)--(0122);
        \draw[EdgeGraph](0023)--(0123);
        \draw[EdgeGraph](0100)--(0101);
        \draw[EdgeGraph](0100)--(0110);
        \draw[EdgeGraph](0101)--(0102);
        \draw[EdgeGraph](0101)--(0111);
        \draw[EdgeGraph](0102)--(0103);
        \draw[EdgeGraph](0102)--(0112);
        \draw[EdgeGraph](0103)--(0113);
        \draw[EdgeGraph](0110)--(0111);
        \draw[EdgeGraph](0110)--(0120);
        \draw[EdgeGraph](0111)--(0112);
        \draw[EdgeGraph](0111)--(0121);
        \draw[EdgeGraph](0112)--(0113);
        \draw[EdgeGraph](0112)--(0122);
        \draw[EdgeGraph](0113)--(0123);
        \draw[EdgeGraph](0120)--(0121);
        \draw[EdgeGraph](0121)--(0122);
        \draw[EdgeGraph](0122)--(0123);
    \end{tikzpicture}}
    \label{subfig:cliff_poset_1_4}}
    \caption{\footnotesize Hasse diagrams of some $\delta$-cliff posets.}
    \label{fig:examples_cliff_posets}
\end{figure}
\medbreak

Let us introduce some notation about $\delta$-cliffs. For any $u \in \SetCliff_\delta(n)$
and $i \in [n]$, let $\DecreaseLetter_i(u)$ (resp.\ $\IncreaseLetter_i(u)$) be the word on
$\Z$ of length $n$ obtained by decrementing (resp.\ incrementing) by $1$ the $i$-th letter of
$u$. Let also, for any $u, v \in \SetCliff_\delta(n)$,
\begin{math}
    \DiffIndices(u, v) := \Bra{i \in [n] : u_i \ne v_i}
\end{math}
be the set of all indices of different letters between $u$ and $v$. For any $u, v \in
\SetCliff_\delta(n)$, let $u \Meet v$ be the $\delta$-cliff of size $n$ defined for any $i
\in [n]$ by
\begin{math}
    (u \Meet v)_i := \min \Bra{u_i, v_i}.
\end{math}
We also define $u \JJoin v$ similarly by replacing the $\min$ operation by $\max$ in the
previous definition. For any $u, v \in \SetCliff_\delta(n)$, the \Def{difference} between
$v$ and $u$ is the word $v - u$ on $\Z$ of length $n$ defined for any $i \in [n]$ by
\begin{math}
    (v - u)_i := v_i - u_i.
\end{math}
Observe that when $u \Leq v$, $v - u$ is a $\delta$-cliff. The \Def{$\delta$-complementary}
$\Complementary_\delta(u)$ of $u \in \SetCliff_\delta(n)$ is the $\delta$-cliff 
\begin{math}
    \GreatestElement_\delta(n) - u.
\end{math}
For instance, by setting $u := 0010$, if $u$ is seen as a $\MapOne$-cliff, then
$\Complementary_\MapOne(u) = 0113$, and if $u$ is seen as a $\MapTwo$-cliff, then
$\Complementary_\MapTwo(u) = 0236$. This map $\Complementary_\delta$ is an involution.
\medbreak

%%%%%%%%%%%%%%%%%%%%%%%%%%%%%%%%%%%%%%%%%%%%%%%%%%%%%%%%%%%%%%%%%%%%%%%%%%%%%%%%%%%%%%%%%%%%
\subsubsection{First properties and recalls} \label{subsubsec:first_properties_cliffs}
A study of the $\MapOne$-cliff posets appears in~\cite{Den13}.  Our definition stated here
depending on $\delta$ is therefore a generalization of these posets.  The structure of the
$\delta$-cliff posets is very simple since each of these posets of order $n$ is isomorphic
to the Cartesian product $\HanL{\delta(1)} \times \dots \times \HanL{\delta(n)}$, where
$\HanL{k}$ is the total order on $k + 1$ elements. It follows from this observation that
each $\delta$-cliff poset is a lattice admitting respectively $\Meet$ and $\JJoin$ as meet
and join operations.  Recall that a lattice $(\LatticeL, \Meet, \JJoin)$ is
\Def{distributive} if for all $x, y, z \in \LatticeL$,
\begin{math}
    x \Meet (y \JJoin z) = (x \Meet y) \JJoin (x \Meet z).
\end{math}
Recall also that all distributive lattices are modular and graded~\cite{Sta11}. Since total
orders are distributive and the distributivity is preserved by the Cartesian product of
posets, $\SetCliff_\delta(n)$ is a distributive lattice.
\medbreak

Recall that the \Def{covering relation} of a poset $\PosetP$ is the set of all pairs $(x, y)
\in \PosetP^2$ such that the interval $[x, y]$ has cardinality $2$.  It follows immediately
from the definition of $\Leq$ that the covering relation $\Covered$ of $\SetCliff_\delta(n)$
satisfies $u \Covered v$ if and only if there is an index $i \in [n]$ such that $v =
\IncreaseLetter_i(u)$. Moreover, these posets $\SetCliff_\delta(n)$ are graded, and the rank
of a $\delta$-cliff $u$ is $\Weight(u)$. The least element of the poset is
$\LeastElement_\delta(n)$ while the greatest element~$\GreatestElement_\delta(n)$.
\medbreak

Let us make some reminders about poset morphisms used thereafter. If $\Par{\PosetP_1,
\Leq_1}$ and $\Par{\PosetP_2, \Leq_2}$ are two posets, a map $\phi : \PosetP_1 \to
\PosetP_2$ is a \Def{poset morphism} if for any $x, y \in \PosetP_1$, $x \Leq_1 y$ implies
$\phi(x) \Leq_2 \phi(y)$.  We say that $\PosetP_2$ is an \Def{order extension} of a poset
$\PosetP_1$ if there is a map $\phi : \PosetP_1 \to \PosetP_2$ which is both a bijection and
a poset morphism. A map $\phi : \PosetP_1 \to \PosetP_2$ is a \Def{poset embedding} if for
any $x, y \in \PosetP_1$, $x \Leq_1 y$ if and only if $\phi(x) \Leq_2 \phi(y)$. Observe that
a poset embedding is necessarily injective.  A map $\phi : \PosetP_1 \to \PosetP_2$ is a
\Def{poset isomorphism} if $\phi$ is both a bijection and a poset embedding.
\medbreak

%%%%%%%%%%%%%%%%%%%%%%%%%%%%%%%%%%%%%%%%%%%%%%%%%%%%%%%%%%%%%%%%%%%%%%%%%%%%%%%%%%%%%%%%%%%%
\subsubsection{Links with the weak Bruhat order} \label{subsubsec:links_weak_bruhat_order}
Let $\SetPermutations$ be the graded set of all permutations where the degree of a
permutation is its length as a word. A \Def{coinversion} of a permutation $\sigma$ is a pair
$\Par{\sigma_j, \sigma_i}$ such that $\sigma_j < \sigma_i$ and $i < j$. For any $n \geq 0$,
the \Def{weak Bruhat order} of order $n$ is a partial order $\Par{\SetPermutations(n),
\Leq_{\SetPermutations}}$ wherein for any $\sigma, \nu \in \SetPermutations(n)$, $\sigma
\Leq_{\SetPermutations} \nu$ if the set of all coinversions of $\sigma$ is contained in the
set of all coinversions of $\nu$. By denoting by $s_i$, $i \in [n - 1]$, the $i$-th
elementary transposition, the covering relation $\Covered_{\SetPermutations}$ of this poset
satisfies $\sigma \Covered_{\SetPermutations} \sigma s_i$ for any $\sigma \in
\SetPermutations(n)$ and any $i \in [n - 1]$ such that~$\sigma_i < \sigma_{i + 1}$.
\medbreak

When $\delta$ is a rooted weakly increasing range map, let us consider the binary relation
$\Covered'$ on $\SetCliff_\delta(n)$ defined as follows. Let $u, v \in \SetCliff_\delta$ and
$\TreeT := \TreeMap_\delta(u)$. We have $u \Covered' v$ if there is an index $i \in [n]$
such that $v = \IncreaseLetter_i(u)$ and all the children of the node labeled by $i$ of
$\TreeT$ are leaves, except possibly the first of its brotherhood. For instance, for $\delta
:= 0233579^\omega$ and the $\delta$-cliff $u := 021042$, since
\begin{equation}
    \TreeMap_\delta(u) \enspace = \enspace
    \scalebox{.8}{
    \begin{tikzpicture}[Centering,xscale=0.19,yscale=0.11]
        \node[Leaf](0)at(0.00,-8.00){};
        \node[Leaf](10)at(6.00,-12.00){};
        \node[Leaf](12)at(7.00,-12.00){};
        \node[Leaf](13)at(8.00,-12.00){};
        \node[Leaf](15)at(10.00,-8.00){};
        \node[Leaf](2)at(1.00,-8.00){};
        \node[Leaf](3)at(2.00,-12.00){};
        \node[Leaf](5)at(3.00,-12.00){};
        \node[Leaf](6)at(4.00,-12.00){};
        \node[Leaf](9)at(5.00,-8.00){};
        \node[Node](1)at(1.00,-4.00){$4$};
        \node[Node](11)at(7.00,-8.00){$5$};
        \node[Node](14)at(9.00,-4.00){$2$};
        \node[Node](4)at(3.00,-8.00){$6$};
        \node[Node](7)at(5.00,0.00){$1$};
        \node[Node](8)at(5.00,-4.00){$3$};
        \draw[Edge](0)--(1);
        \draw[Edge](1)--(7);
        \draw[Edge](10)--(11);
        \draw[Edge](11)--(14);
        \draw[Edge](12)--(11);
        \draw[Edge](13)--(11);
        \draw[Edge](14)--(7);
        \draw[Edge](15)--(14);
        \draw[Edge](2)--(1);
        \draw[Edge](3)--(4);
        \draw[Edge](4)--(1);
        \draw[Edge](5)--(4);
        \draw[Edge](6)--(4);
        \draw[Edge](8)--(7);
        \draw[Edge](9)--(8);
        \node(r)at(5.00,3.50){};
        \draw[Edge](r)--(7);
    \end{tikzpicture}},
\end{equation}
we observe that all the children of the nodes labeled by $2$, $3$, $5$ and $6$ are leaves,
except possibly the first ones. For this reason, $u$ is covered by $\IncreaseLetter_3(u) =
022042$ and by $\IncreaseLetter_6(u) = 021043$, but not by $\IncreaseLetter_2(u) = 031042$
since this word is not a $\delta$-cliff.
\medbreak

The reflexive and transitive closure $\Leq'$ of this relation is an order relation. By
Proposition~\ref{prop:cliff_increasing_trees}, this endows the set of all
$\delta$-increasing trees with a poset structure. It follows immediately from the
description of the covering relation $\Covered$ of $\SetCliff_\delta(n)$ provided in
Section~\ref{subsubsec:first_properties_cliffs} that $\Covered'$ is a refinement of
$\Covered$. For this reason $\Par{\SetCliff_\delta(n), \Leq}$ is an order extension
of~$\Par{\SetCliff_\delta(n), \Leq'}$.  Figure~\ref{fig:example_generalization_weak_bruhat}
shows an example of a Hasse diagram of such a poset.
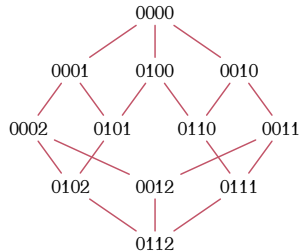
\begin{figure}[ht]
    \centering
    \scalebox{.7}{
    \begin{tikzpicture}[Centering,xscale=.8,yscale=1.1]
        \node[](0000)at(0,0){$0000$};
        \node[](0001)at(-2,-1){$0001$};
        \node[](0100)at(0,-1){$0100$};
        \node[](0010)at(2,-1){$0010$};
        \node[](0002)at(-3,-2){$0002$};
        \node[](0101)at(-1,-2){$0101$};
        \node[](0110)at(1,-2){$0110$};
        \node[](0011)at(3,-2){$0011$};
        \node[](0102)at(-2,-3){$0102$};
        \node[](0012)at(0,-3){$0012$};
        \node[](0111)at(2,-3){$0111$};
        \node[](0112)at(0,-4){$0112$};
        \draw[EdgeGraph](0000)--(0001);
        \draw[EdgeGraph](0000)--(0100);
        \draw[EdgeGraph](0000)--(0010);
        \draw[EdgeGraph](0001)--(0002);
        \draw[EdgeGraph](0001)--(0101);
        \draw[EdgeGraph](0100)--(0101);
        \draw[EdgeGraph](0100)--(0110);
        \draw[EdgeGraph](0010)--(0110);
        \draw[EdgeGraph](0010)--(0011);
        \draw[EdgeGraph](0002)--(0102);
        \draw[EdgeGraph](0002)--(0012);
        \draw[EdgeGraph](0101)--(0102);
        \draw[EdgeGraph](0110)--(0111);
        \draw[EdgeGraph](0011)--(0012);
        \draw[EdgeGraph](0011)--(0111);
        \draw[EdgeGraph](0102)--(0112);
        \draw[EdgeGraph](0012)--(0112);
        \draw[EdgeGraph](0111)--(0112);
    \end{tikzpicture}}
    \caption{\footnotesize The Hasse diagram of the poset $\Par{\SetCliff_{0112^\omega}(4),
    \Leq'}$.}
    \label{fig:example_generalization_weak_bruhat}
\end{figure}

\medbreak

\begin{Proposition} \label{prop:weak_bruhat_order_increasing_trees}
    For any $n \geq 0$, the poset $\Par{\SetCliff_\MapOne(n), \Leq'}$ is isomorphic to
    the weak Bruhat order on permutations of size~$n$.
\end{Proposition}
\begin{proof}
    Let $\phi$ be the map from the set of all words $u$ of size $n$ of integers without
    repeated letters to the set of increasing binary trees of size $n$ where internal
    nodes are bijectively labeled by the letters of $u$, defined recursively as follows. If
    $\sigma$ is the empty word, then $\phi(\sigma)$ is the leaf. Otherwise, $\sigma$
    decomposes as $\sigma = w a w'$ where $a$ is the least letter of $\sigma$, and $w$ and
    $w'$ are words of integers. In this case, $\phi(\sigma)$ is the binary tree consisting
    in a root labeled by $a$ and having as left subtree $\phi\Par{w'}$ and as right subtree
    $\phi(w)$ ---observe the reversal of the order between $w$ and $w'$. Now, by induction
    on $n$, one can prove that for any permutation $\sigma$ of size $n$, the binary trees
    $\phi(\sigma)$ and $\TreeMap_\MapOne(\LehmerMap(\sigma))$ are the same.
    \smallbreak

    Assume that $\sigma$ and $\nu$ are two permutations such that $\sigma
    \Covered_{\SetPermutations} \nu$. Thus, by definition of $\Covered_{\SetPermutations}$,
    $\sigma$ decomposes as $\sigma = w a b w'$ and $\nu$ as $\nu = w b a w'$ where $a$ and
    $b$ are letters such that $a < b$, and $w$ and $w'$ are words of integers.  By
    definition of $\phi$, since $a$ and $b$ are adjacent in $\sigma$, the right subtree of
    the node labeled by $b$ of $\phi(\sigma)$ is empty. Therefore, due to the property
    stated in the first part of the proof, and by definition of the map $\TreeMap_{\mathbf
    1}$ and of the covering relation $\Covered'$, one has $\LehmerMap(\sigma) \Covered'
    \LehmerMap(\nu)$.  Conversely, assume that $u$ and $v$ are two $\MapOne$-cliffs such
    that $u \Covered' v$. Thus, by definition of $\Covered'$, $v$ is obtained by changing a
    letter $u_i$, $i \geq 2$, in $u$ by $u_i + 1$, and in $\TreeMap_\MapOne(u)$, the
    right subtree of the node labeled by $i$ is empty. Let $\sigma := \LehmerMap^{-1}(u)$
    and $\nu := \LehmerMap^{-1}(v)$.  Since $\phi(\sigma)$ and $\TreeMap_\MapOne(u)$ are
    the same increasing binary trees, we have, from the definition of the map $\phi$, that
    $u_{i - 1} < u_i$. Finally, by definition of $\Covered_{\SetPermutations}$, one obtains
    $\sigma \Covered_{\SetPermutations} \nu$.
    \smallbreak

    We have shown that the bijection $\LehmerMap$ between $\SetPermutations(n)$ and
    $\SetCliff_\MapOne(n)$ is such that, for any $\sigma, \nu \in \SetPermutations(n)$,
    $\sigma \Covered_{\SetPermutations} \nu$ if and only if $\LehmerMap(\sigma) \Covered'
    \LehmerMap(\nu)$ . For this reason, $\LehmerMap$ is a poset isomorphism.
\end{proof}
\medbreak

Therefore, Proposition~\ref{prop:weak_bruhat_order_increasing_trees} says in particular that
the $\MapOne$-cliff poset is an extension of the weak Bruhat order. Besides, for all rooted
weakly increasing range maps $\delta$, one can see $\Par{\SetCliff_\delta(n), \Leq'}$ as
generalizations of the weak Bruhat order. Supported by some computer experiments, we state
the following conjecture.
\medbreak

\begin{Conjecture} \label{con:bruhat_order_increasing_trees}
    For any rooted weakly increasing range map $\delta$ and any $n \geq 0$, the poset
    $\Par{\SetCliff_\delta(n), \Leq'}$ is a semi-distributive lattice.
\end{Conjecture}
\medbreak

%%%%%%%%%%%%%%%%%%%%%%%%%%%%%%%%%%%%%%%%%%%%%%%%%%%%%%%%%%%%%%%%%%%%%%%%%%%%%%%%%%%%%%%%%%%%
%%%%%%%%%%%%%%%%%%%%%%%%%%%%%%%%%%%%%%%%%%%%%%%%%%%%%%%%%%%%%%%%%%%%%%%%%%%%%%%%%%%%%%%%%%%%
\subsection{Subposets of $\delta$-cliff posets}
Despite their simplicity, the $\delta$-cliff posets contain subposets having a lot of
combinatorial and algebraic properties. If $\SubFamilly$ is a graded subset of
$\SetCliff_\delta$, each $\SubFamilly(n)$, $n \geq 0$, is a subposet of
$\SetCliff_\delta(n)$ for the order relation $\Leq$.  We denote by $\Covered_\SubFamilly$
the covering relation of each $\SubFamilly(n)$, $n \geq 0$.
\medbreak

We say that $\SubFamilly$ is 
\begin{itemize}
    \item \Def{spread} if for any $n \geq 0$, $\LeastElement_\delta(n) \in \SubFamilly$ and
    $\GreatestElement_\delta(n) \in \SubFamilly$;

    \item \Def{straight} if for any $u, v \in \SubFamilly$ such that $u \Covered_\SubFamilly
    v$, $\# \DiffIndices(u, v) = 1$;

    \item \Def{coated} if for any $n \geq 0$, any $u, v \in \SubFamilly(n)$ such that $u
    \Leq v$, and any $i \in [n - 1]$, $u_1 \dots u_i v_{i + 1} \dots v_n \in \SubFamilly$;

    \item \Def{closed by prefix} if for any $u \in \SubFamilly$, all prefixes of $u$ (that
    are words $u_1 \dots u_k$ where $k \in \HanL{|u|}$) belong to $\SubFamilly$;

    \item \Def{minimally extendable} if $\epsilon \in \SubFamilly$ and for any $u \in
    \SubFamilly$, $u0 \in \SubFamilly$;

    \item \Def{maximally extendable} if $\epsilon \in \SubFamilly$ and for any $u \in
    \SubFamilly$, $u \, \delta(|u| + 1) \in \SubFamilly$.
\end{itemize}
Observe that when $\SubFamilly$ is spread, each poset $\SubFamilly(n)$, $n \geq 0$, is
\Def{bounded}, that is it admits a least and a greatest element. Observe also that if
$\SubFamilly$ is both minimally and maximally extendable, then $\SubFamilly$ is spread.
\medbreak

\begin{Lemma} \label{lem:coated_implies_straight}
    Let $\delta$ be a range map and $\SubFamilly$ be a coated graded subset of
    $\SetCliff_\delta$. Then, $\SubFamilly$ is straight.
\end{Lemma}
\begin{proof}
    Let $n \geq 0$ and $u, v \in \SubFamilly(n)$ such that $u \Leq v$ and $\#
    \DiffIndices(u, v) \geq 2$.  Set $j := \max \DiffIndices(u, v)$ and $w := u_1 \dots u_{j
    - 1} v_j v_{j + 1} \dots v_n$. Since $\SubFamilly$ is coated, $w$ belongs to
    $\SubFamilly$, and moreover, since $j$ is maximal, $w := u_1 \dots u_{j - 1} v_j u_{j
    + 1} \dots u_n$. Therefore, $\# \DiffIndices(u, w) = 1$. This proves that there exists a
    $w' \in \SubFamilly(n)$ such that $u \Covered_\SubFamilly w' \Leq w$ and $\#
    \DiffIndices\Par{u, w'} = 1$.  Thus, $\SubFamilly$ is straight.
\end{proof}
\medbreak

We use here the notion of $n$-th dimension of range maps, defined in
Section~\ref{subsubsec:first_definitions_cliffs}. In the case where $\SubFamilly$ is
straight, we define the graded set of
\begin{itemize}
    \item \Def{input-wings} as the set $\InputWings(\SubFamilly)$ containing any $u \in
    \SubFamilly$ which covers exactly $\DimensionDelta_{|u|}(\delta)$ elements;

    \item \Def{output-wings} as the set $\OutputWings(\SubFamilly)$ containing any $u \in
    \SubFamilly$ which is covered by exactly $\DimensionDelta_{|u|}(\delta)$ elements;

    \item \Def{butterflies} as the set $\Butterflies(\SubFamilly)$ being the intersection
    $\InputWings(\SubFamilly) \cap \OutputWings(\SubFamilly)$.
\end{itemize}
Equivalently, $u \in \SubFamilly$ is an input-wing (resp.\ output-wing) if it is possible to
decrement (resp.\ increment) all values at all positions $i \in [|u|]$ such that $\delta(i)
\ne 0$. Observe also that if there is an $i \geq 1$ such that $\delta(i) = 1$, there are no
butterfly in $\SubFamilly(n)$ for all~$n \geq i$.
\medbreak

We present now general results about subposets $\SubFamilly(n)$, $n \geq 0$, of
$\delta$-cliff posets.
\medbreak

%%%%%%%%%%%%%%%%%%%%%%%%%%%%%%%%%%%%%%%%%%%%%%%%%%%%%%%%%%%%%%%%%%%%%%%%%%%%%%%%%%%%%%%%%%%%
\subsubsection{EL-shellability}
Let $\Par{\PosetP, \Leq_\PosetP}$ and $\Par{\Lambda, \Leq_\Lambda}$ be two posets, and
$\lambda : \Covered_{\PosetP} \to \Lambda$ be a map (here $\Covered_{\PosetP}$ is seen as
the set of all pairs $(u, v)$ such that $v$ covers $u$ in $\PosetP$). For any saturated
chain $\Par{x^{(1)}, \dots, x^{(k)}}$ of $\PosetP$, by a slight abuse of notation, we set
\begin{equation}
    \lambda\Par{x^{(1)}, \dots, x^{(k)}}
    :=
    \Par{\lambda\Par{x^{(1)}, x^{(2)}}, \dots, \lambda\Par{x^{(k - 1)}, x^{(k)}}}.
\end{equation}
We say that a saturated chain of $\PosetP$ is \Def{$\lambda$-increasing} (resp.\
\Def{$\lambda$-weakly decreasing}) if its image by $\lambda$ is an increasing (resp.\ weakly
decreasing) word w.r.t. the partial order relation $\Leq_{\Lambda}$. We say also that a
saturated chain $\Par{x^{(1)}, \dots, x^{(k)}}$ of $\PosetP$ is \Def{$\lambda$-smaller} than
a saturated chain $\Par{y^{(1)}, \dots, y^{(\ell)}}$ of $\PosetP$ if the image by $\lambda$
of $\Par{x^{(1)}, \dots, x^{(k)}}$ is smaller than the image by $\lambda$ of $\Par{y^{(1)},
\dots, y^{(\ell)}}$ for the lexicographic order induced by $\Leq_{\Lambda}$.  The map
$\lambda$ is an \Def{EL-labeling} of $\PosetP$ if there exist such a poset $\Lambda$ and a
map $\lambda$ such that for any $x, y \in \PosetP$ satisfying $x \Leq_\PosetP y$, there is
exactly one $\lambda$-increasing saturated chain from $x$ to~$y$ which is minimal among all
saturated chains from $x$ to $y$ w.r.t. the order on saturated chains just described.  The
poset $\PosetP$ is \Def{EL-shellable}~\cite{BW96,BW97} if $\PosetP$ is bounded and admits an
EL-labeling.
\medbreak

The EL-shellability of a poset $\PosetP$ implies several topological and order theoretical
properties of the associated order complex $\Delta(\PosetP)$ made of all the chains of
$\PosetP$. For instance, one of the consequences for $\PosetP$ for having at most one
$\lambda$-weakly decreasing chain between any pair of its elements is that the Möbius
function of $\PosetP$ takes values in $\{-1, 0, 1\}$. In an equivalent way, the simplicial
complex associated with each open interval of $\PosetP$ is either contractile or has the
homotopy type of a sphere~\cite{BW97}.
\medbreak

For the sequel, we set $\Lambda$ as the poset $\Z^2$ wherein elements are ordered
lexicographically. For any straight graded subset $\SubFamilly$ of $\SetCliff_\delta$, let
us introduce the map
\begin{math}
    \lambda_\SubFamilly : \Covered_\SubFamilly \to \Z^2
\end{math}
defined for any $(u, v) \in \Covered_{\SubFamilly}$ by
\begin{math}
    \lambda_\SubFamilly(u, v) := \Par{-i, u_i}
\end{math}
where $i$ is the unique index $i \in [|u|]$ such that $\DiffIndices(u, v) = \{i\}$.  Observe
that the fact that $\SubFamilly$ is straight ensures that $\lambda_\SubFamilly$ is
well-defined.
\medbreak

\begin{Theorem} \label{thm:subposets_el_shellability}
    Let $\delta$ be a range map and $\SubFamilly$ be a coated graded subset of
    $\SetCliff_\delta$. For any $n \geq 0$, the map $\lambda_\SubFamilly$ is an EL-labeling
    of $\SubFamilly(n)$. Moreover, there is at most one $\lambda_\SubFamilly$-weakly
    decreasing chain between any pair of elements of~$\SubFamilly(n)$.
\end{Theorem}
\begin{proof}
    By Lemma~\ref{lem:coated_implies_straight}, the fact that $\SubFamilly$ is coated
    implies that $\SubFamilly$ is also straight.  Let $u, v \in \SubFamilly(n)$ such that $u
    \Leq v$.  Since $\SubFamilly$ is straight, the image by $\lambda_\SubFamilly$ of any
    saturated chain from $u$ to $v$ is well-defined.
    \smallbreak

    Now, let
    \begin{math}
        \Par{u = w^{(0)}, w^{(1)}, \dots, w^{(k)} = v}
    \end{math}
    be the sequence of elements of $\SubFamilly(n)$ defined in the following way. For any $i
    \in \HanL{k - 1}$, the word $w^{(i + 1)}$ is obtained from $w^{(i)}$ by increasing by
    the minimal possible value $a \geq 1$ the letter $w^{(i)}_j$ such that $j$ is the
    greatest index satisfying $w^{(i)}_j < v_j$. By construction, for any $i \in \HanL{k -
    1}$, each $w^{(i + 1)}$ writes as $w^{(i + 1)} = u_1 \dots u_{j - 1} \Par{u_j + a} v_{j
    + 1} \dots v_n$, where $a$ is some positive integer.  There is at least one value $a$
    such that $w^{(i)}$ belongs to $\SubFamilly(n)$ since by hypothesis, $\SubFamilly$ is
    coated. For this reason, the considered sequence is a well-defined saturated chain in
    $\SubFamilly(n)$. This saturated chain is also $\lambda_\SubFamilly$-increasing by
    construction. Moreover, since $\SubFamilly$ is straight, if one consider another
    saturated chain from $u$ to $v$, this chain passes through a word obtained by
    incrementing a letter which has not a greatest index, and one has to choose later in the
    chain the letter of the smallest index to increment it. For this reason, this saturated
    chain would not be $\lambda_\SubFamilly$-increasing.
    \smallbreak

    Assume now that there exists in $\SubFamilly$ a $\lambda_\SubFamilly$-weakly decreasing
    saturated chain of the form
    \begin{math}
        \Par{u = w^{(0)}, w^{(1)}, \dots, w^{(k)} = v}.
    \end{math}
    By definition of $\lambda_\SubFamilly$ and of the poset $\Lambda$, for any $i \in
    \HanL{k - 1}$, the word $w^{(i + 1)}$ is obtained from $w^{(i)}$ by increasing by the
    minimal possible value the letter $w^{(i)}_j$ such that $j$ is the smallest index
    satisfying $w^{(i)}_j < v_j$. If it exists, this saturated chain is by construction the
    unique $\lambda_\SubFamilly$-weakly decreasing saturated chain from $u$ to~$v$.
\end{proof}
\medbreak

%%%%%%%%%%%%%%%%%%%%%%%%%%%%%%%%%%%%%%%%%%%%%%%%%%%%%%%%%%%%%%%%%%%%%%%%%%%%%%%%%%%%%%%%%%%%
\subsubsection{Meet and join operations, sublattices, and lattices}
\label{subsubsec:incrementation_decrementation_maps}
Here we give some sufficient conditions on $\SubFamilly$ for the fact that each
$\SubFamilly(n)$, $n \geq 0$, is a lattice.
\medbreak

First, when $\SubFamilly$ is spread and, for any $n \geq 0$ and any $u, v \in
\SubFamilly(n)$, $u \Meet v \in \SubFamilly$ (resp.\ $u \JJoin v \in \SubFamilly$),
by~\cite{Sta11}, each $\SubFamilly(n)$ is a lattice. Moreover, when both $u \Meet v \in
\SubFamilly$ and $u \JJoin v \in \SubFamilly$, each $\SubFamilly(n)$ is a sublattice of
$\SetCliff_\delta(n)$. Again by~\cite{Sta11}, $\SubFamilly(n)$ is in this case distributive
and graded.
\medbreak
%%%

Second, assume instead that $\SubFamilly$ is minimally extendable. For any $n \geq 0$, the
\Def{$\SubFamilly$-de\-cre\-ment\-ation map} is the map
\begin{math}
    \DecrMap_\SubFamilly : \SetCliff_\delta(n) \to \SubFamilly(n)
\end{math}
defined recursively by $\DecrMap_\SubFamilly(\epsilon) := \epsilon$ and, for any $u a
\in \SetCliff_\delta(n)$ where $u \in \SetCliff_\delta$ and $a \in \N$, by
\begin{math}
    \DecrMap_\SubFamilly(u a) := \DecrMap_\SubFamilly(u) \, b
\end{math}
where
\begin{equation} \label{equ:letter_decr_map}
    b := \max \Bra{b \leq a : \DecrMap_\SubFamilly(u) \, b \in \SubFamilly}.
\end{equation}
Observe that the fact that $\SubFamilly$ is minimally extendable ensures that
$\DecrMap_\SubFamilly$ is a well-defined map. Let also, for any $n \geq 0$ and $u, v
\in \SubFamilly(n)$,
\begin{math}
    u \Meet_\SubFamilly v := \DecrMap_\SubFamilly(u \Meet v).
\end{math}
\medbreak

When $\SubFamilly$ is maximally extendable, we denote by $\IncrMap_\SubFamilly$ the
\Def{$\SubFamilly$-incrementation map} defined in the same way as the
$\SubFamilly$-decrementation map with the difference that in the previous definitions, the
operation $\max$ is replaced by the operation $\min$ and the relation $\leq$ is replaced by
the relation $\geq$. Here, the fact that $\SubFamilly$ is maximally extendable ensure that
$\IncrMap_\SubFamilly$ is well-defined. We also define the operation $\JJoin_\SubFamilly$ in
the same way as $\Meet_\SubFamilly$ with the difference that in the previous definitions,
the map $\DecrMap_\SubFamilly$ is replaced by $\IncrMap_\SubFamilly$ and the operation
$\Meet$ is replaced by the operation~$\JJoin$.
\medbreak

\begin{Theorem} \label{thm:decr_incr_meet_join}
    Let $\delta$ be a range map and $\SubFamilly$ be a closed by prefix and minimally
    (resp.\ maximally) extendable graded subset of $\SetCliff_\delta$. The operation
    $\Meet_\SubFamilly$ (resp.\ $\JJoin_\SubFamilly$) is, for any $n \geq 0$, the meet
    (resp.\ join) operation of the poset $\SubFamilly(n)$.
\end{Theorem}
\begin{proof}
    Let us show the property of the statement of the theorem in the case where $\SubFamilly$
    is minimally extendable. The other case is symmetric. We proceed by induction on $n \geq
    0$. When $n = 0$, the property is trivially satisfied. Let $n \geq 1$ and $u, v \in
    \SubFamilly(n)$. Since $\SubFamilly$ is closed by prefix, one has $u = u' a$ and $v = v'
    b$ with $u', v' \in \SubFamilly(n - 1)$ and $a, b \in \N$.  Since $\SubFamilly$ is
    minimally extendable,
    \begin{equation}
        u \Meet_{\SubFamilly} v
        = u' a \Meet_{\SubFamilly} v' b
        = \DecrMap_{\SubFamilly} \Par{\Par{u' \Meet v'} \min \{a, b\}}
        = \DecrMap_{\SubFamilly} \Par{u' \Meet v'} \, c
    \end{equation}
    where
    \begin{math}
        c :=
        \max \Bra{c \leq \min \{a, b\} :
        \DecrMap_{\SubFamilly} \Par{u' \Meet v'} \, c \in \SubFamilly}.
    \end{math}
    Now, by induction hypothesis, we obtain
    \begin{equation}
        \DecrMap_{\SubFamilly} \Par{u' \Meet v'} \, c
        = \Par{u' \Meet_{\SubFamilly} v'} \, c
    \end{equation}
    where $\Meet_{\SubFamilly}$ is the meet operation of the poset $\SubFamilly(n - 1)$.
    First, we deduce from the above computation that for any $i \in [n]$, the $i$-th letter
    of $u \Meet_{\SubFamilly} v$ is nongreater than $\min \Bra{u_i, v_i}$, and that $u
    \Meet_{\SubFamilly} v$ belongs to $\SubFamilly(n)$. Therefore, $u \Meet_{\SubFamilly} v$
    is a lower bound of $\Bra{u, v}$. Second, by induction hypothesis, $w' := u'
    \Meet_{\SubFamilly} v'$ is the greatest lower bound of $\Bra{u', v'}$. By construction,
    since $c$ is the greatest letter such that $c \leq a$, $c \leq b$, and $w' \, c \in
    \SubFamilly$ holds, any other lower bound of $\Bra{u, v}$ is smaller than $w' c$. This
    prove that $w' c$ is the greatest lower bound of $\Bra{u, v}$ and implies the statement
    of the theorem.
\end{proof}
\medbreak

%Together with Proposition~\ref{prop:subposets_lattices},
%Theorem~\ref{thm:decr_incr_meet_join} provides the following sufficient conditions on the
%graded subset $\SubFamilly$ of $\SetCliff_\delta$ for the fact that for all $n \geq 0$,
%the posets $\SubFamilly(n)$ are lattices:
%\begin{enumerate}[label={\it (\roman*)}]
%    \item $\SubFamilly$ is spread and each $\SubFamilly(n)$, $n \geq 0$, is a meet
%    semi-sublattice of $\SetCliff_\delta(n)$;
%
%    \item $\SubFamilly$ is spread and each $\SubFamilly(n)$, $n \geq 0$, is a join
%    semi-sublattice of $\SetCliff_\delta(n)$;
%
%    \item $\SubFamilly$ is minimally and maximally extendable, and closed by prefix.
%\end{enumerate}
\medbreak

%%%%%%%%%%%%%%%%%%%%%%%%%%%%%%%%%%%%%%%%%%%%%%%%%%%%%%%%%%%%%%%%%%%%%%%%%%%%%%%%%%%%%%%%%%%%
\subsubsection{Join-irreducible elements}
Recall that an element $x$ of a lattice $\LatticeL$ is \Def{join-irreducible} (resp.\
\Def{meet-irreducible}) if $x$ covers (resp.\ is covered by) exactly one element in
$\LatticeL$. We denote by $\JoinIrreducibles(\LatticeL)$ (resp.\
$\MeetIrreducibles(\LatticeL)$) the set of join-irreducible (resp.\ meet-irreducible)
elements of $\LatticeL$. These notions are usually considered specially for lattices but we
can take the same definitions even when $\LatticeL$ is just a poset.
\medbreak

\begin{Proposition} \label{prop:join_meet_irreducible_elements_subfamilies}
    Let $\delta$ be a range map and $\SubFamilly$ be a straight graded subset of
    $\SetCliff_\delta$.  For any $n \geq 0$, $u \in \SubFamilly(n)$ is a join-irreducible
    (resp.\ meet-irreducible) element of $\SubFamilly(n)$ if and only if there is a $k \geq
    1$ and a unique $i \in [n]$ such that $\DecreaseLetter_i^k(u) \in \SubFamilly(n)$
    (resp.\ $\IncreaseLetter_i^k(u) \in \SubFamilly(n)$).
\end{Proposition}
\begin{proof}
    Assume first that $u$ is a join-irreducible element of $\SubFamilly(n)$. Then, there is
    exactly one element $u'$ of $\SubFamilly$ such that $u' \Covered_{\SubFamilly} u$.
    Since $\SubFamilly$ is straight, $\# \DiffIndices\Par{u', u} = 1$, implying that $u$
    satisfies the stated condition.
    \smallbreak

    Conversely, assume that $u$ satisfies the stated condition. Assume that there are $u',
    u'' \in \SubFamilly(n)$ such that $u' \Covered_{\SubFamilly} u$ and $u''
    \Covered_{\SubFamilly} u$. Since $\SubFamilly$ is straight, there exist $i', i'' \in
    [n]$ and $k', k'' \geq 1$ such that $u' = \DecreaseLetter_{i'}^{k'}(u)$ and $u'' =
    \DecreaseLetter_{i''}^{k''}(u)$. Due to the property satisfied by $u$, $i' = i''$, so
    that $u' = u''$ since $u'$ and $u''$ are both covered by $u$. Therefore, $u$ is
    join-irreducible.
    \smallbreak

    The part of the statement of the proposition concerning the meet-irreducible elements is
    symmetric.
\end{proof}
\medbreak

%%%%%%%%%%%%%%%%%%%%%%%%%%%%%%%%%%%%%%%%%%%%%%%%%%%%%%%%%%%%%%%%%%%%%%%%%%%%%%%%%%%%%%%%%%%%
\subsubsection{Constructibility by interval doubling} \label{subsubsec:interval_doubling}
We denote by $\IntervalTwo$ the poset $\{0, 1\}$ endowed with the natural order relation on
integers. Let $(\PosetP, \Leq)$ be a poset and $I$ one of its intervals. The \Def{interval
doubling} of $I$ in $\PosetP$ is the poset
\begin{math}
    \PosetP[I] := (\PosetP \setminus I) \sqcup (I \times \IntervalTwo),
\end{math}
having $\Leq'$ as order relation, which is defined as follows. For any $x, y \in
\PosetP[I]$, one has $x \Leq' y$ if one of the following assertions is satisfied:
\begin{enumerate}[label={\it (\roman*)}]
    \item $x \in \PosetP \setminus I$, $y \in \PosetP \setminus I$, and $x \Leq y$;

    \item $x \in \PosetP \setminus I$, $y = \Par{y', b} \in I \times \IntervalTwo$, and $x
    \Leq y'$;

    \item $x = \Par{x', a} \in I \times \IntervalTwo$, $y \in \PosetP \setminus I$, and $x'
    \Leq y$;

    \item $x = \Par{x', a} \in I \times \IntervalTwo$, $y = \Par{y', b} \in I \times
    \IntervalTwo$, and $x' \Leq y'$ and $a \leq b$.
\end{enumerate}
This operation has been introduced in~\cite{Day92} as an operation on posets preserving the
property to being a lattice. On the other way round, we say that $\PosetP$ is obtained by an
\Def{interval contraction} from a poset $\PosetP'$ if there is an interval $I$ of $\PosetP$
such that $\PosetP[I]$ is isomorphic as a poset to $\PosetP'$~\cite{CLM04}.
\medbreak

A lattice $\LatticeL$ is \Def{constructible by interval doubling} (spelled as
``\Def{bounded}'' in the original article) if $\LatticeL$ is isomorphic as a poset to a
poset obtained by performing a sequence of interval doubling from the singleton lattice. It
is known from~\cite{Day79} that such lattices are semi-distributive. Recall that a finite
lattice $\LatticeL$ is constructible by interval doubling if and only if it is
\Def{congruence uniform}, and then in particular, the number of join-irreducible elements of
$\LatticeL$ determines the number of interval doubling steps needed to create $\LatticeL$
(see~\cite{Day79} and~\cite{Muh19}).
\medbreak

The aim of this section is to introduce a sufficient condition on a graded subset
$\SubFamilly$ of $\SetCliff_\delta$ for the fact that each $\SubFamilly(n)$, $n \geq 0$, is
constructible by interval doubling. We shall moreover describe explicitly the sequence of
interval doubling operations involved in the construction of $\SubFamilly(n)$ from the
trivial lattice.
\medbreak

Let $\PosetP$ be a nonempty subposet of $\SetCliff_\delta(n)$ for a given fixed size $n \geq
1$. Let us denote by $\MaxLastLetter(\PosetP)$ the letter $\max \Bra{u_n : u \in \PosetP}$.
For any $a, b \in \HanL{\delta(n)}$, let $\PosetP_a := \Bra{u \in \PosetP : u_n = a}$ and
$\PosetP_{a, b} := \Bra{u b : u a \in \PosetP_a}$.  Observe that $\PosetP_a$ is a subposet
of $\PosetP$ while $\PosetP_{a, b}$ may contain $\delta$-cliffs that do not belong to
$\PosetP$.  The \Def{derivation} of $\PosetP$ is the set
\begin{equation}
    \DerivationOnSet(\PosetP)
    :=
    \PosetP_0 \cup \PosetP_1 \cup \dots \cup \PosetP_{\MaxLastLetter(\PosetP) - 1}
    \cup
    \PosetP_{\MaxLastLetter(\PosetP), \MaxLastLetter(\PosetP) - 1}.
\end{equation}
In other words, $\DerivationOnSet(\PosetP)$ is the set of all the cliffs obtained from
$\PosetP$ by decrementing their last letters if they are equal to $\MaxLastLetter(\PosetP)$
or by keeping them as they are otherwise.  Observe that $\DerivationOnSet(\PosetP)$ is not
necessarily a subposet of $\PosetP$.  Nevertheless, $\DerivationOnSet(\PosetP)$ is still a
subposet of $\SetCliff_\delta(n)$.  Observe also that
\begin{math}
    \MaxLastLetter\Par{\DerivationOnSet(\PosetP)} \leq \MaxLastLetter(\PosetP) - 1.
\end{math}
For instance, by considering the subposet
\begin{equation}
    \PosetP := \{0000, 0111, 0002, 0112, 0103, 0104, 0004\}
\end{equation}
of $\SetCliff_\MapTwo(4)$, we have
\begin{math}
    \PosetP_{2, \MaxLastLetter(\PosetP)} = \{0004, 0114\}
\end{math}
and
\begin{math}
    \DerivationOnSet(\PosetP) = \{0000, 0111, 0002, 0112, 0103, 0003\}.
\end{math}
\medbreak

The subposet $\PosetP$ is \Def{nested} if it is nonempty and
\begin{enumerate}[label={(N\arabic*)}]
    \item \label{item:nested_1}
    for any $a \in \HanL{\MaxLastLetter(\PosetP)}$, the $\delta$-cliff $0^{n - 1} a$
    belongs to $\PosetP$;

    \item \label{item:nested_2}
    for any $a \in \HanL{\MaxLastLetter(\PosetP)}$, $\PosetP_{a, \MaxLastLetter(\PosetP)}$
    is both a subset and an interval of~$\PosetP$.
\end{enumerate}
This definition still holds when $\MaxLastLetter(\PosetP) = 0$. Observe that any
$\delta$-cliff $0^{n - 1} a$, $a \geq 1$, of $\PosetP$ covers exactly the single element
$0^{n - 1} \, (a - 1)$ of $\PosetP$. This element exists by~\ref{item:nested_1}. Therefore,
when $\PosetP$ is a lattice, these $\delta$-cliffs are join-irreducible.
\medbreak

\begin{Lemma} \label{lem:derivation_intervals}
    Let $\delta$ be a range map and $\PosetP$ be a nonempty subposet of
    $\SetCliff_\delta(n)$ for an $n \geq 1$. If $\PosetP$ is nested, then for any $a \in
    \HanL{\MaxLastLetter(\PosetP)}$, $\PosetP_a$ is an interval of $\PosetP$.
\end{Lemma}
\begin{proof}
    First, by~\ref{item:nested_1}, $\PosetP_a$ admits $0^{n - 1} a$ as unique least element.
    It remains to prove that $\PosetP_a$ has at most one greatest element. By contradiction,
    assume that there are in $\PosetP_a$ two different greatest elements $u a$ and $v a$,
    where $u, v \in \SetCliff_\delta(n - 1)$. Then, by setting $b :=
    \MaxLastLetter(\PosetP)$, in $\PosetP_{a, b}$ the $\delta$-cliffs $u b$ and $v b$ are
    still incomparable.  Since these two elements are also greatest elements of $\PosetP_{a,
    b}$, this implies that $\PosetP_{a, b}$ is not an interval in $\PosetP$. This
    contradicts~\ref{item:nested_2}.
\end{proof}
\medbreak

\begin{Lemma} \label{lem:last_interval_of_derivation}
    Let $\delta$ be a range map and $\PosetP$ be a nonempty subposet of
    $\SetCliff_\delta(n)$ for an $n \geq 1$. If $\MaxLastLetter(\PosetP) \geq 1$ and
    $\PosetP$ is nested, then
    \begin{math}
        \DerivationOnSet(\PosetP)_{\MaxLastLetter\Par{\DerivationOnSet(\PosetP)}}
        = \PosetP_{\MaxLastLetter(\PosetP), \MaxLastLetter(\PosetP) - 1}.
    \end{math}
\end{Lemma}
\begin{proof}
    Let $b := \MaxLastLetter(\PosetP)$, $\PosetP' := \DerivationOnSet(\PosetP)$, and $b' :=
    \MaxLastLetter\Par{\PosetP'}$.  First, since $\PosetP$ satisfies~\ref{item:nested_1},
    $b' = b - 1$.  Moreover, directly from the definition of the derivation operation
    $\DerivationOnSet$, we have
    \begin{math}
        \PosetP'_{b'} = \PosetP_{b, b'} \cup \PosetP_{b'}.
    \end{math}
    By~\ref{item:nested_2}, $\PosetP_{b', b}$ is a subset of $\PosetP_b$, so that
    $\PosetP_{b'}$ is a subset of $\PosetP_{b, b'}$. Therefore,
    \begin{math}
        \PosetP'_{b'} = \PosetP_{b, b'}.
    \end{math}
\end{proof}
\medbreak

\begin{Lemma} \label{lem:derivation_property_conservation}
    Let $\delta$ be a range map and $\PosetP$ be a nonempty subposet of
    $\SetCliff_\delta(n)$ for an $n \geq 1$. If $\MaxLastLetter(\PosetP) \geq 1$ and
    $\PosetP$ is nested, then $\DerivationOnSet(\PosetP)$ is nested.
\end{Lemma}
\begin{proof}
    Let $b := \MaxLastLetter(\PosetP)$, $\PosetP' := \DerivationOnSet(\PosetP)$, and $b' :=
    \MaxLastLetter\Par{\PosetP'}$.  First, since $\PosetP$ satisfies~\ref{item:nested_1},
    $b' = b - 1$. Moreover, in particular, for any $a \in \HanL{b'}$, $0^{n - 1} a \in
    \PosetP$.  Hence, $0^{n - 1} a \in \PosetP'$, so that $\PosetP'$
    satisfies~\ref{item:nested_1}. Let $a \in \HanL{b' - 1}$.  By~\ref{item:nested_2},
    $\PosetP_{a, b}$ is an interval of $\PosetP_b$. Due to the fact $a \leq b' - 1$, one has
    $\PosetP_a = \PosetP'_a$, so that $\PosetP'_{a, b}$ is an interval of $\PosetP_b$.  This
    is equivalent to the fact that $\PosetP'_{a, b'}$ is an interval of $\PosetP_{b, b'}$.
    By Lemma~\ref{lem:last_interval_of_derivation}, the relation $\PosetP'_{b'} =
    \PosetP_{b, b'}$ holds and leads to the fact that $\PosetP'_{a, b'}$ is an interval of
    $\PosetP'_{b'}$. Therefore, $\PosetP'$ satisfies~\ref{item:nested_2}.
\end{proof}
\medbreak

\begin{Lemma} \label{lem:derivation_interval_doubling}
    Let $\delta$ be a range map and $\PosetP$ be a nonempty subposet of
    $\SetCliff_\delta(n)$ for an $n \geq 1$. If $\MaxLastLetter(\PosetP) \geq 1$ and
    $\PosetP$ is nested, then $\PosetP$ is isomorphic as a poset to
    $\DerivationOnSet(\PosetP)[I]$ where $I$ is the interval
    $\PosetP_{\MaxLastLetter(\PosetP) - 1}$ of $\DerivationOnSet(\PosetP)$.
\end{Lemma}
\begin{proof}
    Let $b := \MaxLastLetter(\PosetP)$, $\PosetP' := \DerivationOnSet(\PosetP)$, and $b' :=
    \MaxLastLetter\Par{\PosetP'}$. By~\ref{item:nested_1}, $b' = b - 1$.  Let us first prove
    that $I = \PosetP_{b'}$ is an interval of $\PosetP'$.  Let $u, v \in \PosetP_{b'}$ such
    that $u \Leq v$. Assume that there exists $w \in \PosetP'_{b'}$ such that $u \Leq w \Leq
    v$. Let us denote by $u'$ (resp.\ $v'$, $w'$) the prefix of size $n - 1$ of $u$ (resp.\
    $v$, $w$). By~\ref{item:nested_2}, $u' b$ and $v' b$ belong to $\PosetP_b$. Moreover, by
    Lemma~\ref{lem:last_interval_of_derivation}, since $\PosetP'_{b'} = \PosetP_{b, b'}$, $w
    \in \PosetP_{b, b'}$. Therefore, $w' b$ belongs to $\PosetP_b$. Again
    by~\ref{item:nested_2}, this leads to the fact that $w \in \PosetP_{b'}$.  This shows
    that the set $\PosetP_{b'}$ is closed by interval in $\PosetP'_{b'}$. Since finally, by
    Lemma~\ref{lem:derivation_intervals}, $\PosetP_{b'}$ is an interval of $\PosetP$,
    $\PosetP_{b'}$ has a unique least and a unique greatest element. This implies that
    $\PosetP_{b'}$ is an interval of~$\PosetP'$.
    \smallbreak

    Since $I$ is an interval of $\PosetP'$, we can now consider the poset $\PosetP'[I]$. By
    definition of the interval doubling operation,
    \begin{math}
        \PosetP'[I] =
        \Par{\PosetP' \setminus \PosetP_{b'}}
        \sqcup
        \Par{\PosetP_{b'} \times \IntervalTwo}.
    \end{math}
    Let $\phi : \PosetP'[I] \to \PosetP$ be the map defined by
    \begin{subequations}
    \begin{equation}
        \phi(u a) := u a,
        \qquad \mbox{ if } u a \in \PosetP' \setminus \PosetP_{b'}
        \mbox{ and } a \ne b',
    \end{equation}
    \begin{equation}
        \phi\Par{u b'} := u b,
        \qquad \mbox{ if } u b ' \in \PosetP' \setminus \PosetP_{b'},
    \end{equation}
    \begin{equation}
        \phi\Par{\Par{u b', 1}} := u b',
        \qquad \mbox{ if } \Par{u b', 1} \in \PosetP_{b'} \times \IntervalTwo,
    \end{equation}
     \begin{equation}
        \phi\Par{\Par{u b', 2}} := u b,
        \qquad \mbox{ if } \Par{u b', 2} \in \PosetP_{b'} \times \IntervalTwo.
    \end{equation}
    \end{subequations}
    This map $\phi$ is well-defined because, respectively, one has $\PosetP'_a = \PosetP_a$
    for any $a \in \HanL{b' - 1}$, Lemma~\ref{lem:last_interval_of_derivation} holds, $I$
    is in particular a subset of $\PosetP$, and $\PosetP$ satisfies~\ref{item:nested_2}. Let
    now $\psi : \PosetP \to \PosetP'[I]$ be the map satisfying
    \begin{subequations}
    \begin{equation}
        \psi(u a) = u a,
        \qquad \mbox{ if } u a \in \PosetP \mbox{ and } a \in \HanL{b' - 1},
    \end{equation}
    \begin{equation}
        \psi(u b) = u b',
        \qquad \mbox{ if } u b' \in \PosetP' \setminus \PosetP_{b'},
    \end{equation}
    \begin{equation}
        \psi(u b) = \Par{u b', 2},
        \qquad \mbox{ if } u b' \in \PosetP_{b'},
    \end{equation}
    \begin{equation}
        \psi\Par{u b'} = \Par{u b', 1},
        \qquad \mbox{ if } u b' \in \PosetP_{b'}.
    \end{equation}
    \end{subequations}
    By similar arguments as before, this map $\psi$ is well-defined. Moreover, by
    construction, $\psi$ is the inverse of $\phi$. Therefore, $\phi$ is a bijection.  The
    fact that $\phi$ is a poset embedding comes by definition of $\phi$ and from the fact
    that, due to the property of $\PosetP$ to be nested, for any $u b' \in \PosetP'
    \setminus \PosetP_{b'}$, all elements greater than $u b'$ in $\PosetP'$ do not belong
    to~$\PosetP_{b'}$. Thus, $\PosetP'[I]$ is isomorphic as a poset to~$\PosetP$.
\end{proof}
\medbreak

By assuming that $\PosetP$ is nested, the \Def{sequence of derivations} from $\PosetP$ is
the sequence
\begin{equation}
    \Par{\PosetP, \DerivationOnSet(\PosetP), \DerivationOnSet^2(\PosetP),
    \dots, \DerivationOnSet^{\MaxLastLetter(\PosetP)}(\PosetP)}
\end{equation}
of subsets of $\SetCliff_\delta(n)$. Observe that due to~\ref{item:nested_1}, for any $k \in
\Han{\MaxLastLetter(\PosetP) - 1}$, $\MaxLastLetter\Par{\DerivationOnSet^k(\PosetP)} \geq
1$, so that $\DerivationOnSet^{k + 1}(\PosetP)$ is well-defined.
\medbreak

Given a graded subset $\SubFamilly$ of $\SetCliff_\delta$, we say by extension that
$\SubFamilly$ is \Def{nested} if for all $n \geq 0$, the posets $\SubFamilly(n)$ are nested.
\medbreak

\begin{Theorem} \label{thm:constructible_by_interval_doubling_subfamilly}
    Let $\delta$ be a rooted range map and $\SubFamilly$ be a nested and closed by prefix
    graded subset of $\SetCliff_\delta$.  For any $n \geq 1$, $\SubFamilly(n)$ is
    constructible by interval doubling. Moreover,
    \begin{equation} \begin{split}
        \label{equ:constructible_by_interval_doubling_subfamilly}
        \SubFamilly(n) & \to \DerivationOnSet(\SubFamilly(n))
        \to \dots \to \DerivationOnSet^{\MaxLastLetter(\SubFamilly(n))}(\SubFamilly(n))
        \simeq \SubFamilly(n - 1)
        \\
        & \to \DerivationOnSet(\SubFamilly(n - 1)) \to \dots
        \to \DerivationOnSet^{\MaxLastLetter(\SubFamilly(n - 1))}(\SubFamilly(n - 1))
        \simeq \SubFamilly(n - 2) \\
        & \to \dots \to \SubFamilly(0) \simeq \{\epsilon\}
    \end{split} \end{equation}
    is a sequence of interval contractions from $\SubFamilly(n)$ to the trivial
    lattice~$\{\epsilon\}$.
\end{Theorem}
\begin{proof}
    We proceed by induction on $n \geq 0$. If $n = 0$, since $\delta$ is rooted, we
    necessarily have $\SubFamilly(0) \simeq \{\epsilon\}$, and this poset is by
    constructible by interval doubling.  Assume now that $n \geq 1$ and set $\PosetP :=
    \SubFamilly(n)$.  Since $\SubFamilly$ is nested, the sequence of reductions from
    $\PosetP$ is well-defined. By Lemmas~\ref{lem:derivation_property_conservation}
    and~\ref{lem:derivation_interval_doubling}, by setting
    \begin{math}
        \PosetP' := \DerivationOnSet^{\MaxLastLetter(\PosetP)}(\PosetP),
    \end{math}
    $\PosetP$ is obtained by performing a sequence of interval doubling from the poset
    $\PosetP'$. Now, due to the definition of the derivation algorithm $\DerivationOnSet$,
    $\PosetP'$ is made of the $\delta$-cliffs of $\PosetP$ wherein the last letters have
    been replaced by $0$. This poset $\PosetP'$ is therefore isomorphic to the poset
    $\PosetP''$ formed by the prefixes of length $n - 1$ of $\PosetP$.  Since $\SubFamilly$
    is closed by prefix, $\PosetP''$ is thus the poset $\SubFamilly(n - 1)$.  By induction
    hypothesis, this last poset is constructible by interval doubling.  Therefore,
    $\SubFamilly(n)$ also is.  All this produces the
    sequence~\eqref{equ:constructible_by_interval_doubling_subfamilly} of interval
    contractions.
\end{proof}
\medbreak

%%%%%%%%%%%%%%%%%%%%%%%%%%%%%%%%%%%%%%%%%%%%%%%%%%%%%%%%%%%%%%%%%%%%%%%%%%%%%%%%%%%%%%%%%%%%
\subsubsection{Elevation maps} \label{subsubsec:elevation_maps}
We introduce here a combinatorial tool intervening in the study of the three Fuss-Catalan
posets introduced in the sequel.
\medbreak

Let $\SubFamilly$ be a closed by prefix graded subset of $\SetCliff_\delta$.  For any $u \in
\SubFamilly$, let
\begin{equation}
    \Next_\SubFamilly(u) := \Bra{a \in \HanL{\delta(|u| + 1)} : u a \in \SubFamilly}.
\end{equation}
By definition, $\Next_\SubFamilly(u)$ is the set of all the letters $a$ that can follow $u$
to form an element of~$\SubFamilly$.  For any $n \geq 0$, the \Def{$\SubFamilly$-elevation
map} is the map
\begin{math}
    \ElevationMap_{\SubFamilly} : \SubFamilly(n) \to \SetCliff_\delta(n)
\end{math}
defined, for any $u \in \SubFamilly(n)$ and $i \in [n]$, by
\begin{equation}
    \ElevationMap_{\SubFamilly}(u)_i
    := \# \Par{\Next_\SubFamilly\Par{u_1 \dots u_{i - 1}} \cap \HanL{u_i - 1}}
\end{equation}
for any $i \in [n]$. From an intuitive point of view, the value of the $i$-th letter of
$\ElevationMap_{\SubFamilly}(u)$ is the number of cliffs of $\SubFamilly$ obtained by
considering the prefix of $u$ ending at the letter $u_i$ and by replacing this letter by a
smaller one. Remark in particular that $\ElevationMap_{\SetCliff_\delta}$ is the identity
map. Besides, we say that any $u \in \SubFamilly$ is an \Def{exuviae} if
$\ElevationMap_\SubFamilly(u) = u$.
\medbreak

Let $\ElevationImage_\SubFamilly$ be the graded set wherein for any $n \geq 0$,
$\ElevationImage_\SubFamilly(n)$ is the image of $\SubFamilly(n)$ by the
$\SubFamilly$-elevation map.  We call this set the \Def{$\SubFamilly$-elevation image}.
Observe that $\ElevationImage_\SubFamilly$ is a graded subset of $\SetCliff_\delta$. Note
also that for any $u \in \SubFamilly$, $\ElevationMap_\SubFamilly(u) \Leq u$.
\medbreak

\begin{Proposition} \label{prop:elevation_map_injectivity}
    Let $\delta$ be a range map and $\SubFamilly$ be a closed by prefix graded subset of
    $\SetCliff_\delta$.  For any $n \geq 0$, the $\SubFamilly$-elevation map is injective on
    the domain $\SubFamilly(n)$.
\end{Proposition}
\begin{proof}
    We proceed by induction on $n$. When $n = 0$, the property is trivially satisfied. Let
    $u, v \in \SubFamilly(n)$ such that $n \geq 1$ and
    \begin{math}
        \ElevationMap_{\SubFamilly}(u) = \ElevationMap_{\SubFamilly}(v).
    \end{math}
    Since $\SubFamilly$ is closed by prefix, we have $u = u' a$ and $v = v' b$ where $u', v'
    \in \SubFamilly(n - 1)$ and $a, b \in \N$.  By definition of
    $\ElevationMap_{\SubFamilly}$, we have
    \begin{math}
        \ElevationMap_{\SubFamilly}\Par{u'a } = \ElevationMap_{\SubFamilly}\Par{u'} c
    \end{math}
    and
    \begin{math}
        \ElevationMap_{\SubFamilly}\Par{v' b} = \ElevationMap_{\SubFamilly}\Par{v'} c
    \end{math}
    where $c \in \N$. Hence,
    \begin{math}
        \ElevationMap_{\SubFamilly}\Par{u'} = \ElevationMap_{\SubFamilly}\Par{v'}
    \end{math}
    which leads, by induction hypothesis, to the fact that $u' = v'$.  Moreover, we deduce
    from this and from the definition of the $\SubFamilly$-elevation map that there are
    exactly $c$ letters $a'$ smaller than $a$ such that $u' \, a' \in \SubFamilly$ and that
    there are exactly $c$ letters $b'$ smaller than $b$ such that $v' \, b' \in
    \SubFamilly$. Therefore, we have $a = b$ and thus $u = v$, establishing the injectivity
    of~$\ElevationMap_{\SubFamilly}$.
\end{proof}
\medbreak

\begin{Lemma} \label{lem:elevation_image_closed_prefix}
    Let $\delta$ be a range map and $\SubFamilly$ be a closed by prefix graded subset of
    $\SetCliff_\delta$. The $\SubFamilly$-elevation image is closed by prefix.
\end{Lemma}
\begin{proof}
    Let $n \geq 0$ and $v \in \ElevationImage_\SubFamilly(n)$. Then, there exists $u \in
    \SubFamilly(n)$ such that $\ElevationMap_\SubFamilly(u) = v$.  Let $v'$ be a prefix of
    $v$. Since $\SubFamilly$ is closed by prefix, the prefix $u'$ of $u$ of length $n' :=
    \left|v'\right|$ belongs to $\SubFamilly\Par{n'}$. Moreover, by definition of
    $\ElevationMap_\SubFamilly$, we have $\ElevationMap_\SubFamilly\Par{u'} = v'$.
    Therefore, $v' \in \ElevationImage_\SubFamilly$, implying the statement of the lemma.
\end{proof}
\medbreak

\begin{Proposition} \label{prop:elevation_map_poset_morphism}
    Let $\delta$ be a range map and $\SubFamilly$ be a closed by prefix graded subset of
    $\SetCliff_\delta$ such that for any $u, v \in \SubFamilly$, $u \Leq v$ implies
    $\Next_\SubFamilly(v) \subseteq \Next_\SubFamilly(u)$.  For any $n \geq 0$, the map
    $\ElevationMap_{\SubFamilly}^{-1}$ is a poset morphism from
    $\ElevationImage_\SubFamilly(n)$ to~$\SubFamilly(n)$.
\end{Proposition}
\begin{proof}
    First, by Proposition~\ref{prop:elevation_map_injectivity}, the map
    $\ElevationMap_{\SubFamilly}^{-1}$ is well-defined. We now proceed by induction on $n$.
    When $n = 0$, the property is trivially satisfied. Let $u$ and $v$ be elements of
    $\ElevationImage_{\SubFamilly}(n)$ such that $n \geq 1$ and $u \Leq v$. By
    Lemma~\ref{lem:elevation_image_closed_prefix}, we have $u = u' a$ and $v = v' b$ where
    $u', v' \in \ElevationImage_{\SubFamilly}(n - 1)$ and $a, b \in \N$. By definition of
    $\ElevationMap_{\SubFamilly}^{-1}$, we have
    \begin{math}
        \ElevationMap_{\SubFamilly}^{-1}\Par{u' a}
        = \ElevationMap_{\SubFamilly}^{-1}\Par{u'} c
    \end{math}
    and
    \begin{math}
        \ElevationMap_{\SubFamilly}^{-1}\Par{v' b}
        = \ElevationMap_{\SubFamilly}^{-1}\Par{v'} d
    \end{math}
    where $c, d \in \N$.  Since $u \Leq v$, one has $u' \Leq v'$ so that, by induction
    hypothesis,
    \begin{math}
        \ElevationMap_{\SubFamilly}^{-1}\Par{u'}
        \Leq \ElevationMap_{\SubFamilly}^{-1}\Par{v'}.
    \end{math}
    Moreover, $u \Leq v$ implies that  $a \leq b$. Due to the fact that
    $\Next_\SubFamilly\Par{v'} \subseteq \Next_\SubFamilly\Par{u'}$, one has by definition
    of $\ElevationMap_{\SubFamilly}^{-1}$ that $c \leq d$. Therefore,
    \begin{math}
        \ElevationMap_{\SubFamilly}^{-1}\Par{u'} c
        \leq \ElevationMap_{\SubFamilly}^{-1}\Par{v'} d,
    \end{math}
    which implies the statement of the proposition.
\end{proof}
\medbreak

Proposition~\ref{prop:elevation_map_poset_morphism} says that when $\SubFamilly$ is closed
by prefix, for any $n \geq 0$, the poset $\SubFamilly(n)$ is an order extension of
$\ElevationImage_\SubFamilly(n)$.
\medbreak

%%%%%%%%%%%%%%%%%%%%%%%%%%%%%%%%%%%%%%%%%%%%%%%%%%%%%%%%%%%%%%%%%%%%%%%%%%%%%%%%%%%%%%%%%%%%
\subsubsection{Geometric cubic realizations}
Let $\SubFamilly$ be a graded subset of $\SetCliff_\delta$. For any $n \geq 0$, the
\Def{realization} of $\SubFamilly(n)$ is the geometric object
$\CubicReal\Par{\SubFamilly(n)}$ defined in the space $\R^n$ and obtained by placing for
each $u \in \SubFamilly(n)$ a vertex of coordinates $\Par{u_1, \dots, u_n}$, and by forming
for each $u, v \in \SubFamilly(n)$ such that $u \Covered_{\SubFamilly} v$ an edge between
$u$ and $v$.  Remark that the posets of Figure~\ref{fig:examples_cliff_posets} represent
actually the realizations of $\delta$-cliff posets. We will follow this drawing
convention for all the next figures of posets in all the sequel.  When $\SubFamilly$ is
straight, every edge of $\CubicReal\Par{\SubFamilly(n)}$ is parallel to a line passing by
the origin and a point of the form $\Par{0, \dots, 0, 1, 0, \dots, 0}$. In this case, we say
that $\CubicReal\Par{\SubFamilly(n)}$ is \Def{cubic}.
\medbreak

As a side remark, we would like to stress that the present notion of geometric realization
of a poset differs from the usual one saying that it consists in the geometric realization
of the simplicial complex of the chain of the poset.
\medbreak

Let us assume from now that $\SubFamilly$ is straight.  Let $u, v \in \SubFamilly(n)$ such
that $u \Leq v$. The word $u$ is \Def{cell-compatible} with $v$ if for any word $w$ of
length $n$ such that for any $i \in [n]$, $w_i \in \Bra{u_i, v_i}$, then $w \in
\SubFamilly$.  In this case, we call \Def{cell} the set of points
\begin{equation}
    \Angle{u, v} := \Bra{x \in \R^n : u_i \leq x_i \leq v_i \mbox{ for all } i \in [n]}.
\end{equation}
By definition, a cell is an orthotope, that is a parallelotope whose edges are all mutually
orthogonal or parallel. A point $x$ of $\R^n$ is \Def{inside} a cell $\Angle{u, v}$ if for
any $i \in [n]$, $u_i \ne v_i$ implies $u_i < x_i < v_i$.  A cell $\Angle {u, v}$ is
\Def{pure} if there is no point of $\SubFamilly(n)$ inside $\Angle{u, v}$. In other terms,
this says that for all $w \in [u, v]$, there exists $i \in [n]$ such that $u_i \ne v_i$ and
$w_i \in \Bra{u_i, v_i}$.  Two cells $\Angle{u, v}$ and $\Angle{u', v'}$ of
$\CubicReal(\SubFamilly(n))$ are \Def{disjoint} if there is no point of $\R^n$ which is both
inside $\Angle{u, v}$ and $\Angle{u', v'}$. The \Def{dimension} $\dim\Angle{u, v}$ of a
cell $\Angle{u, v}$ is its dimension as an orthotope and it satisfies $\dim \Angle{u, v} =
\# \DiffIndices(u, v)$.  The \Def{volume} $\Volume \Angle{u, v}$ of $\Angle{u, v}$ is its
volume as an orthotope and its satisfies
\begin{equation}
    \Volume \Angle{u, v} = \prod_{i \in \DiffIndices(u, v)} v_i - u_i.
\end{equation}
For any $k \geq 0$, the \Def{$k$-volume} $\Volume_k(\CubicReal(\SubFamilly(n)))$ of
$\CubicReal(\SubFamilly(n))$ is the volume obtained by summing the volumes of all its all
its cells of dimension $k$, computed by not counting several times potential intersecting
orthotopes. The \Def{volume} $\Volume(\CubicReal(\SubFamilly(n)))$ of
$\CubicReal(\SubFamilly(n))$ is defined as $\Volume_k(\CubicReal(\SubFamilly(n)))$ where $k$
is the largest integer such that $\CubicReal(\SubFamilly(n))$ has at least one cell of
dimension $k$.
\medbreak

Figure~\ref{fig:examples_cubic_realizations_cells} shows examples of these notions.
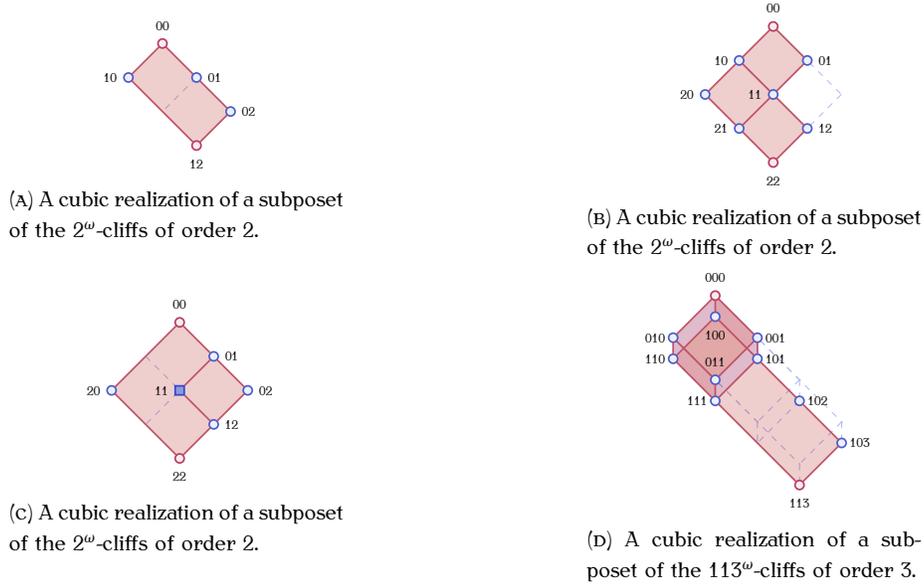
\begin{figure}[ht]
    \centering
    \subfloat[][A cubic realization of a subposet of the $2^\omega$-cliffs of order~$2$.]{
    \begin{minipage}{.45\textwidth}
    \centering
    \scalebox{.8}{
    \begin{tikzpicture}[Centering,xscale=.8,yscale=.8,rotate=-135]
        \draw[Grid](0,0)grid(1,2);
        \draw[Face](0,0)--(0,2)--(1,2)--(1,0);
        \node[NodeGraph,MarkBA](00)at(0,0){};
        \node[NodeGraph](01)at(0,1){};
        \node[NodeGraph](02)at(0,2){};
        \node[NodeGraph](10)at(1,0){};
        \node[NodeGraph,MarkBA](12)at(1,2){};
        \node[NodeLabelGraph,above of=00]{$00$};
        \node[NodeLabelGraph,right of=01]{$01$};
        \node[NodeLabelGraph,right of=02]{$02$};
        \node[NodeLabelGraph,left of=10]{$10$};
        \node[NodeLabelGraph,below of=12]{$12$};
        \draw[EdgeGraph](00)--(01);
        \draw[EdgeGraph](00)--(10);
        \draw[EdgeGraph](01)--(02);
        \draw[EdgeGraph](02)--(12);
        \draw[EdgeGraph](10)--(12);
    \end{tikzpicture}}
    \end{minipage}
    \label{subfig:pure_cell}}
    \qquad
    \subfloat[][A cubic realization of a subposet of the $2^\omega$-cliffs of order~$2$.]{
    \begin{minipage}{.45\textwidth}
    \centering
    \scalebox{.8}{
    \begin{tikzpicture}[Centering,xscale=.8,yscale=.8,rotate=-135]
        \draw[Grid](0,0)grid(2,2);
        \draw[Face](0,0)--(0,1)--(1,1)--(1,0);
        \draw[Face](1,0)--(1,1)--(2,1)--(2,0);
        \draw[Face](1,1)--(1,2)--(2,2)--(2,1);
        \node[NodeGraph,MarkBA](00)at(0,0){};
        \node[NodeGraph](01)at(0,1){};
        \node[NodeGraph](10)at(1,0){};
        \node[NodeGraph](20)at(2,0){};
        \node[NodeGraph](11)at(1,1){};
        \node[NodeGraph](12)at(1,2){};
        \node[NodeGraph](21)at(2,1){};
        \node[NodeGraph,MarkBA](22)at(2,2){};
        \node[NodeLabelGraph,above of=00]{$00$};
        \node[NodeLabelGraph,right of=01]{$01$};
        \node[NodeLabelGraph,left of=10]{$10$};
        \node[NodeLabelGraph,left of=20]{$20$};
        \node[NodeLabelGraph,left of=11]{$11$};
        \node[NodeLabelGraph,right of=12]{$12$};
        \node[NodeLabelGraph,left of=21]{$21$};
        \node[NodeLabelGraph,below of=22]{$22$};
        \draw[EdgeGraph](00)--(01);
        \draw[EdgeGraph](00)--(10);
        \draw[EdgeGraph](01)--(11);
        \draw[EdgeGraph](10)--(11);
        \draw[EdgeGraph](10)--(20);
        \draw[EdgeGraph](20)--(21);
        \draw[EdgeGraph](11)--(12);
        \draw[EdgeGraph](11)--(21);
        \draw[EdgeGraph](12)--(22);
        \draw[EdgeGraph](21)--(22);
    \end{tikzpicture}}
    \end{minipage}
    \label{subfig:no_cell_compatible}}
    \qquad
    \qquad
    \subfloat[][A cubic realization of a subposet of the $2^\omega$-cliffs of order~$2$.]{
    \begin{minipage}{.45\textwidth}
    \centering
    \scalebox{.8}{
    \begin{tikzpicture}[Centering,xscale=.8,yscale=.8,rotate=-135]
        \draw[Grid](0,0)grid(2,2);
        \draw[Face](0,0)--(0,2)--(2,2)--(2,0);
        \node[NodeGraph,MarkBA](00)at(0,0){};
        \node[NodeGraph](01)at(0,1){};
        \node[NodeGraph](02)at(0,2){};
        \node[NodeGraph](20)at(2,0){};
        \node[MarkedNodeGraph](11)at(1,1){};
        \node[NodeGraph](12)at(1,2){};
        \node[NodeGraph,MarkBA](22)at(2,2){};
        \node[NodeLabelGraph,above of=00]{$00$};
        \node[NodeLabelGraph,right of=01]{$01$};
        \node[NodeLabelGraph,right of=02]{$02$};
        \node[NodeLabelGraph,left of=20]{$20$};
        \node[NodeLabelGraph,left of=11]{$11$};
        \node[NodeLabelGraph,right of=12]{$12$};
        \node[NodeLabelGraph,below of=22]{$22$};
        \draw[EdgeGraph](00)--(01);
        \draw[EdgeGraph](00)--(20);
        \draw[EdgeGraph](01)--(02);
        \draw[EdgeGraph](01)--(11);
        \draw[EdgeGraph](11)--(12);
        \draw[EdgeGraph](02)--(12);
        \draw[EdgeGraph](12)--(22);
        \draw[EdgeGraph](20)--(22);
    \end{tikzpicture}}
    \end{minipage}
    \label{subfig:no_pure_cell}}
    \qquad
    \subfloat[][A cubic realization of a subposet of the $113^\omega$-cliffs of order~$3$.]{
    \begin{minipage}{.45\textwidth}
    \centering
    \scalebox{.8}{
    \begin{tikzpicture}[Centering,xscale=.7,yscale=.7,
        x={(0,-.5cm)}, y={(-1.0cm,-1.0cm)}, z={(1.0cm,-1.0cm)}]
        \DrawGridSpace{1}{1}{3}
        \draw[FaceYZ](0,0,0)--(0,0,1)--(0,1,1)--(0,1,0);
        \draw[FaceYZ](1,0,0)--(1,0,1)--(1,1,1)--(1,1,0);
        \draw[FaceYZ](1,0,1)--(1,0,3)--(1,1,3)--(1,1,1);
        \draw[FaceXZ](0,1,0)--(0,1,1)--(1,1,1)--(1,1,0);
        \draw[FaceXZ](0,0,0)--(0,0,1)--(1,0,1)--(1,0,0);
        \draw[FaceXY](0,0,0)--(1,0,0)--(1,1,0)--(0,1,0);
        \draw[FaceXY](0,0,1)--(1,0,1)--(1,1,1)--(0,1,1);
        \node[NodeGraph,MarkBA](000)at(0,0,0){};
        \node[NodeGraph](001)at(0,0,1){};
        \node[NodeGraph](010)at(0,1,0){};
        \node[NodeGraph](011)at(0,1,1){};
        \node[NodeGraph](100)at(1,0,0){};
        \node[NodeGraph](101)at(1,0,1){};
        \node[NodeGraph](102)at(1,0,2){};
        \node[NodeGraph](103)at(1,0,3){};
        \node[NodeGraph](110)at(1,1,0){};
        \node[NodeGraph](111)at(1,1,1){};
        \node[NodeGraph,MarkBA](113)at(1,1,3){};
        \node[NodeLabelGraph,above of=000]{$000$};
        \node[NodeLabelGraph,right of=001]{$001$};
        \node[NodeLabelGraph,left of=010]{$010$};
        \node[NodeLabelGraph,above of=011]{$011$};
        \node[NodeLabelGraph,below of=100]{$100$};
        \node[NodeLabelGraph,right of=101]{$101$};
        \node[NodeLabelGraph,right of=102]{$102$};
        \node[NodeLabelGraph,right of=103]{$103$};
        \node[NodeLabelGraph,left of=110]{$110$};
        \node[NodeLabelGraph,left of=111]{$111$};
        \node[NodeLabelGraph,below of=113]{$113$};
        \draw[EdgeGraph](000)--(001);
        \draw[EdgeGraph](000)--(010);
        \draw[EdgeGraph](000)--(100);
        \draw[EdgeGraph](001)--(011);
        \draw[EdgeGraph](001)--(101);
        \draw[EdgeGraph](010)--(011);
        \draw[EdgeGraph](010)--(110);
        \draw[EdgeGraph](011)--(111);
        \draw[EdgeGraph](100)--(101);
        \draw[EdgeGraph](100)--(110);
        \draw[EdgeGraph](101)--(102);
        \draw[EdgeGraph](101)--(111);
        \draw[EdgeGraph](102)--(103);
        \draw[EdgeGraph](103)--(113);
        \draw[EdgeGraph](110)--(111);
        \draw[EdgeGraph](111)--(113);
    \end{tikzpicture}}
    \end{minipage}
    \label{subfig:cells_3d}}
    \caption{\footnotesize Some cubic realizations of straight subposets of posets of
    $\delta$-cliffs for certain range maps~$\delta$.}
    \label{fig:examples_cubic_realizations_cells}
\end{figure}
Figure~\ref{subfig:pure_cell} shows a cubic realization wherein $00$ is cell-compatible with
$12$. Hence, $\Angle{00, 12}$ is a cell. The point $\Par{\frac{1}{2}, \frac{3}{2}} \in \R^2$
is inside $\Angle{00, 12}$, and since there are no elements of the poset inside the cell,
this cell is pure. Figure~\ref{subfig:no_cell_compatible} shows a cubic realization wherein
$00$ is not cell-compatible with $22$ because $02$ does not belong to the poset.
Nevertheless, $\Angle{00, 11}$, $\Angle{10, 21}$, and $\Angle{11, 22}$ are pure cells of
dimension $2$.  Figure~\ref{subfig:no_pure_cell} shows a cubic realization wherein
$\Angle{00, 22}$ is a non-pure cell. Indeed, the $\delta$-cliff $11$ is an element of the
poset and is inside this cell.  Finally, Figure~\ref{subfig:cells_3d} shows a cubic
realization having $1$ as volume since there is exactly one cell $\Angle{000, 111}$ of
maximal dimension (which is $3$) and of volume $1$. Its $2$-volume is $8$ since this cubic
realization decomposes as the seven disjoint cells $\Angle{000, 011}$, $\Angle{000,
101}$, $\Angle{000, 110}$, $\Angle{001, 111}$, $\Angle{010, 111}$, $\Angle{100, 111}$, and
$\Angle{101, 113}$ of respective volumes $1$, $1$, $1$, $1$, $1$, $1$, and~$2$.
\medbreak

There is a close connection between output-wings (resp.\ input-wings) of $\SubFamilly(n)$,
$n \geq 0$, and the computation of the volume of $\CubicReal(\SubFamilly(n))$: if $\Angle{u,
v}$ is a cell of maximal dimension of $\CubicReal(\SubFamilly(n))$, then due to the fact
that $\SubFamilly$ is straight, $u$ (resp.\ $v$) is an output-wing (resp.\ input-wing) of
$\SubFamilly(n)$. When for any $n \geq 0$,
\begin{enumerate}[label={\it (\roman*)}]
    \item there is a map $\rho : \OutputWings(\SubFamilly)(n) \to
    \InputWings(\SubFamilly)(n)$;
    \item all cells of maximal dimension of $\CubicReal(\SubFamilly(n))$ express as
    $\Angle{u, \rho(u)}$ with $u \in \OutputWings(\SubFamilly)(n)$;
    \item \label{item:volume_condition_3}
    all cells of
    \begin{math}
        \Bra{\Angle{u, \rho(u)} : u \in \OutputWings(\SubFamilly)(n)}
    \end{math}
    are pairwise disjoint;
\end{enumerate}
then the volume of $\CubicReal(\SubFamilly(n))$, $n \geq 0$, writes as
\begin{equation} \label{equ:volume_computation_from_output_wings}
    \Volume\Par{\CubicReal(\SubFamilly(n))} =
    \sum_{u \in \OutputWings(\SubFamilly)(n)} \Volume \Angle{u, \rho(u)}.
\end{equation}
When some cells of
\begin{math}
    \Bra{\Angle{u, \rho(u)} : u \in \OutputWings(\SubFamilly)(n)}
\end{math}
intersect each other, the expression for the volume would not be at as simple
as~\eqref{equ:volume_computation_from_output_wings} and can be written instead as an
inclusion-exclusion formula. Of course, the same property holds when $\rho$ is instead a map
from $\InputWings(\SubFamilly)(n)$ to $\OutputWings(\SubFamilly)(n)$ by changing accordingly
the previous text.
\medbreak

Recall that the \Def{order dimension}~\cite{Tro92} of a poset $\PosetP$ is the smallest
nonnegative integer $k$ such that there exists a poset embedding of $\PosetP$ into
$\Par{\N^k, \Leq}$ where $\Leq$ is the componentwise partial order. Recall that the
\Def{hypercube} of dimension $n \geq 0$ is the poset $\Hypercube_n$ on the set of the
subsets of $[n]$ ordered by set inclusion. It can be shown that the order dimension of
$\Hypercube_n$ is~$n$.
\medbreak

\begin{Proposition} \label{prop:subfamillies_order_dimension}
    Let $\delta$ be a range map and $\SubFamilly$ be a straight graded subset of
    $\SetCliff_\delta$. If, for an $n \geq 0$, $\CubicReal\Par{\SubFamilly(n)}$ has a cell
    of dimension $\DimensionDelta_n(\delta)$, then the order dimension of the poset
    $\SubFamilly(n)$ is~$\DimensionDelta_n(\delta)$.
\end{Proposition}
\begin{proof}
    First, since $\SubFamilly(n)$ is a subposet of $\SetCliff_\delta(n)$, $\SubFamilly(n)$
    is a subposet of the Cartesian product $\N^k$ where $k := \# \Bra{i \in [n] : \delta(i)
    \ne 0}$. This poset has order dimension $\DimensionDelta_n(\delta)$, so that the order
    dimension of $\SubFamilly(n)$ is at most $\DimensionDelta_n(\delta)$. Besides, since
    $\SubFamilly$ is straight, the notion of cell is well-defined in the cubic realization
    of $\SubFamilly(n)$. By hypothesis, $\SubFamilly(n)$ contains a cell $\Angle{u, v}$ of
    dimension $\DimensionDelta_n(\delta)$. Thus, there is a poset embedding of
    $\Hypercube_{\DimensionDelta_n(\delta)}$ into the interval $[u, v]$ of $\SubFamilly(n)$.
    Therefore, the order dimension of $\SubFamilly(n)$ is at
    least~$\DimensionDelta_n(\delta)$.
\end{proof}
\medbreak

As a particular case of Proposition~\ref{prop:subfamillies_order_dimension}, the order
dimension of $\SetCliff_\delta(n)$ is $\DimensionDelta_n(\delta)$. This explains the
terminology of ``$n$-th dimension of $\delta$'' for the notation $\DimensionDelta_n(\delta)$
introduced in Section~\ref{subsubsec:first_definitions_cliffs}.
\medbreak

%%%%%%%%%%%%%%%%%%%%%%%%%%%%%%%%%%%%%%%%%%%%%%%%%%%%%%%%%%%%%%%%%%%%%%%%%%%%%%%%%%%%%%%%%%%%
%%%%%%%%%%%%%%%%%%%%%%%%%%%%%%%%%%%%%%%%%%%%%%%%%%%%%%%%%%%%%%%%%%%%%%%%%%%%%%%%%%%%%%%%%%%%
%%%%%%%%%%%%%%%%%%%%%%%%%%%%%%%%%%%%%%%%%%%%%%%%%%%%%%%%%%%%%%%%%%%%%%%%%%%%%%%%%%%%%%%%%%%%
\section{Some Fuss-Catalan posets} \label{sec:fuss_catalan_posets}
We present here some examples of subposets of $\delta$-cliff posets. We focus in this work
on three posets whose elements are enumerated by $m$-Fuss-Catalan numbers for the case
$\delta = \MapM$, $m \geq 0$.  We provide some combinatorial properties of these posets like
among others, a description of their input-wings, output-wings, and butterflies, a study of
their order theoretic properties, and a study of their cubic realizations. We end this
section by establishing links between these three families of posets in terms of poset
morphisms, poset embeddings, and poset isomorphisms. We shall omit some straightforward
proofs (for instance, in the case of the descriptions of input-wings, output-wings,
butterflies, meet-irreducible and join-irreducible elements of the posets ---the
descriptions of these two families use
Proposition~\ref{prop:join_meet_irreducible_elements_subfamilies}).
\medbreak

We use the following notation conventions. Poset morphisms are denoted through arrows
$\tikz{\draw[Map](0,0)--(.5,0)}$, poset embeddings through arrows
$\tikz{\draw[MapEmbedding](0,0)--(.5,0)}$, and poset isomorphisms through
arrows~$\tikz{\draw[MapIsomorphism](0,0)--(.5,0)}$.
\medbreak

%%%%%%%%%%%%%%%%%%%%%%%%%%%%%%%%%%%%%%%%%%%%%%%%%%%%%%%%%%%%%%%%%%%%%%%%%%%%%%%%%%%%%%%%%%%%
%%%%%%%%%%%%%%%%%%%%%%%%%%%%%%%%%%%%%%%%%%%%%%%%%%%%%%%%%%%%%%%%%%%%%%%%%%%%%%%%%%%%%%%%%%%%
\subsection{$\delta$-avalanche posets}
We begin by introducing a first Fuss-Catalan family of posets. As we shall see, these posets
are not lattices but they form an important tool to study the next two families of
Fuss-Catalan posets.
\medbreak

%%%%%%%%%%%%%%%%%%%%%%%%%%%%%%%%%%%%%%%%%%%%%%%%%%%%%%%%%%%%%%%%%%%%%%%%%%%%%%%%%%%%%%%%%%%%
\subsubsection{Objects} \label{subsubsec:avalanche_objects}
For any range map $\delta$, let $\SetAvalanche_\delta$ be the graded subset of
$\SetCliff_\delta$ containing all $\delta$-cliffs $u$ such that for all nonempty prefixes
$u'$ of $u$, then $\Weight\Par{u'} \leq \delta\Par{\left|u'\right|}$. Any element of
$\SetAvalanche_\delta$ is a \Def{$\delta$-avalanche}. For instance,
\begin{equation}
    \SetAvalanche_\MapTwo(3)
    =
    \Bra{000, 001, 002, 003, 004, 010, 011, 012, 013, 020, 021, 022}.
\end{equation}
\medbreak

\begin{Proposition} \label{prop:properties_avalanche_objects}
    For any weakly increasing range map $\delta$, the graded set $\SetAvalanche_\delta$ is
    \begin{enumerate}[label={\it (\roman*)}]
        \item \label{item:properties_avalanche_objects_1}
        closed by prefix;

        \item \label{item:properties_avalanche_objects_2}
        is minimally extendable;

        \item \label{item:properties_avalanche_objects_3}
        is maximally extendable if and only if $\delta = 0^\omega$.
    \end{enumerate}
\end{Proposition}
\begin{proof}
    Point~\ref{item:properties_avalanche_objects_1} is an immediate consequence of the
    definition of $\delta$-avalanches. Let $n \geq 0$ and $u \in \SetAvalanche_\delta(n)$.
    Since $\delta(n + 1) \geq \delta(n)$, $u0$ is a $\delta$-avalanche. This
    establishes~\ref{item:properties_avalanche_objects_2}. Finally, we have immediately
    that $\SetAvalanche_{0^\omega}$ is maximally extendable. Moreover, when $\delta \ne
    0^\omega$, there is an $n \geq 1$ such that $\delta(n) \geq 1$ and $\delta\Par{n'} = 0$
    for all $1 \leq n' < n$. Therefore, $0^{n - 1} \, \delta(n)$ is a $\delta$-avalanche
    but $0^{n - 1} \, \delta(n) \, \delta(n + 1)$ is not. Therefore,
    \ref{item:properties_avalanche_objects_3} holds.
\end{proof}
\medbreak

\begin{Proposition} \label{prop:enumeration_avalanche}
    For any $m \geq 0$ and $n \geq 0$,
    \begin{math}
        \# \SetAvalanche_\MapM(n) = \FussCatalan_m(n).
    \end{math}
\end{Proposition}
\begin{proof}
    This is a consequence of Proposition~\ref{prop:image_elevation_map_hill} coming next.
    Indeed, by this result, $\SetAvalanche_\MapM(n)$ is the image by the elevation map of a
    graded set of objects enumerated by $m$-Fuss-Catalan numbers. Since this set of objects
    satisfies all the requirements of Proposition~\ref{prop:elevation_map_injectivity}, the
    elevation map is injective, implying that it is a bijection.
\end{proof}
\medbreak

Therefore, by Proposition~\ref{prop:enumeration_avalanche}, the first numbers of
$\MapM$-avalanches by sizes are
\begin{subequations}
\begin{equation}
    1, 1, 1, 1, 1, 1, 1, 1, \qquad m = 0,
\end{equation}
\begin{equation}
    1, 1, 2, 5, 14, 42, 132, 429, \qquad m = 1,
\end{equation}
\begin{equation}
    1, 1, 3, 12, 55, 273, 1428, 7752, \qquad m = 2.
\end{equation}
\end{subequations}
The second and third sequences are respectively Sequences~\OEIS{A000108}
and~\OEIS{A001764} of~\cite{Slo}.
\medbreak

%%%%%%%%%%%%%%%%%%%%%%%%%%%%%%%%%%%%%%%%%%%%%%%%%%%%%%%%%%%%%%%%%%%%%%%%%%%%%%%%%%%%%%%%%%%%
\subsubsection{Posets}
For any $n \geq 0$, the subposet $\SetAvalanche_\delta(n)$ of $\SetCliff_\delta(n)$ is the
\Def{$\delta$-avalanche poset} of order $n$. Figure~\ref{fig:examples_avalanche_posets}
shows the Hasse diagrams of some $\MapM$-avalanche posets.
\begin{figure}[ht]
    \centering
    \subfloat[][$\SetAvalanche_\MapOne(3)$.]{
    \centering
    \scalebox{.7}{
    \begin{tikzpicture}[Centering,xscale=.8,yscale=.8,rotate=-135]
        \draw[Grid](0,0)grid(1,2);
        \node[NodeGraph](000)at(0,0){};
        \node[NodeGraph](001)at(0,1){};
        \node[NodeGraph](002)at(0,2){};
        \node[NodeGraph](010)at(,0){};
        \node[NodeGraph](011)at(1,1){};
        \node[NodeLabelGraph,above of=000]{$000$};
        \node[NodeLabelGraph,right of=001]{$001$};
        \node[NodeLabelGraph,below of=002]{$002$};
        \node[NodeLabelGraph,left of=010]{$010$};
        \node[NodeLabelGraph,below of=011]{$011$};
        \draw[EdgeGraph](000)--(001);
        \draw[EdgeGraph](000)--(010);
        \draw[EdgeGraph](001)--(002);
        \draw[EdgeGraph](001)--(011);
        \draw[EdgeGraph](010)--(011);
    \end{tikzpicture}}
    \label{subfig:avalanche_poset_1_3}}
    \qquad \qquad
    \subfloat[][$\SetAvalanche_\MapTwo(3)$.]{
    \centering
    \scalebox{.7}{
    \begin{tikzpicture}[Centering,xscale=.8,yscale=.8,rotate=-135]
        \draw[Grid](0,0)grid(2,4);
        \node[NodeGraph](000)at(0,0){};
        \node[NodeGraph](001)at(0,1){};
        \node[NodeGraph](002)at(0,2){};
        \node[NodeGraph](003)at(0,3){};
        \node[NodeGraph](004)at(0,4){};
        \node[NodeGraph](010)at(1,0){};
        \node[NodeGraph](011)at(1,1){};
        \node[NodeGraph](012)at(1,2){};
        \node[NodeGraph](013)at(1,3){};
        \node[NodeGraph](020)at(2,0){};
        \node[NodeGraph](021)at(2,1){};
        \node[NodeGraph](022)at(2,2){};
        \node[NodeLabelGraph,above of=000]{$000$};
        \node[NodeLabelGraph,right of=001]{$001$};
        \node[NodeLabelGraph,right of=002]{$002$};
        \node[NodeLabelGraph,right of=003]{$003$};
        \node[NodeLabelGraph,below of=004]{$004$};
        \node[NodeLabelGraph,left of=010]{$010$};
        \node[NodeLabelGraph,below of=011]{$011$};
        \node[NodeLabelGraph,below of=012]{$012$};
        \node[NodeLabelGraph,below of=013]{$013$};
        \node[NodeLabelGraph,left of=020]{$020$};
        \node[NodeLabelGraph,left of=021]{$021$};
        \node[NodeLabelGraph,below of=022]{$022$};
        \draw[EdgeGraph](000)--(001);
        \draw[EdgeGraph](000)--(010);
        \draw[EdgeGraph](001)--(002);
        \draw[EdgeGraph](001)--(011);
        \draw[EdgeGraph](002)--(003);
        \draw[EdgeGraph](002)--(012);
        \draw[EdgeGraph](003)--(004);
        \draw[EdgeGraph](003)--(013);
        \draw[EdgeGraph](010)--(011);
        \draw[EdgeGraph](010)--(020);
        \draw[EdgeGraph](011)--(012);
        \draw[EdgeGraph](011)--(021);
        \draw[EdgeGraph](012)--(013);
        \draw[EdgeGraph](012)--(022);
        \draw[EdgeGraph](020)--(021);
        \draw[EdgeGraph](021)--(022);
        \end{tikzpicture}}
    \label{subfig:avalanche_poset_2_3}}
    \qquad \qquad
    \subfloat[][$\SetAvalanche_\MapOne(4)$.]{
    \centering
    \scalebox{.7}{
    \begin{tikzpicture}[Centering,xscale=.7,yscale=.7,
        x={(0,-.5cm)}, y={(-1.0cm,-1.0cm)}, z={(1.0cm,-1.0cm)}]
        \DrawGridSpace{1}{2}{3}
        \node[NodeGraph](0000)at(0,0,0){};
        \node[NodeGraph](0001)at(0,0,1){};
        \node[NodeGraph](0002)at(0,0,2){};
        \node[NodeGraph](0003)at(0,0,3){};
        \node[NodeGraph](0010)at(0,1,0){};
        \node[NodeGraph](0011)at(0,1,1){};
        \node[NodeGraph](0012)at(0,1,2){};
        \node[NodeGraph](0020)at(0,2,0){};
        \node[NodeGraph](0021)at(0,2,1){};
        \node[NodeGraph](0100)at(1,0,0){};
        \node[NodeGraph](0101)at(1,0,1){};
        \node[NodeGraph](0102)at(1,0,2){};
        \node[NodeGraph](0110)at(1,1,0){};
        \node[NodeGraph](0111)at(1,1,1){};
        \node[NodeLabelGraph,above of=0000]{$0000$};
        \node[NodeLabelGraph,right of=0001]{$0001$};
        \node[NodeLabelGraph,right of=0002]{$0002$};
        \node[NodeLabelGraph,below of=0003]{$0003$};
        \node[NodeLabelGraph,left of=0010]{$0010$};
        \node[NodeLabelGraph,above of=0011]{$0011$};
        \node[NodeLabelGraph,below of=0012]{$0012$};
        \node[NodeLabelGraph,left of=0020]{$0020$};
        \node[NodeLabelGraph,below of=0021]{$0021$};
        \node[NodeLabelGraph,below of=0100]{$0100$};
        \node[NodeLabelGraph,below of=0101]{$0101$};
        \node[NodeLabelGraph,below of=0102]{$0102$};
        \node[NodeLabelGraph,below of=0110]{$0110$};
        \node[NodeLabelGraph,below of=0111]{$0111$};
        \draw[EdgeGraph](0000)--(0001);
        \draw[EdgeGraph](0000)--(0010);
        \draw[EdgeGraph](0000)--(0100);
        \draw[EdgeGraph](0001)--(0002);
        \draw[EdgeGraph](0001)--(0011);
        \draw[EdgeGraph](0001)--(0101);
        \draw[EdgeGraph](0002)--(0003);
        \draw[EdgeGraph](0002)--(0012);
        \draw[EdgeGraph](0002)--(0102);
        \draw[EdgeGraph](0010)--(0011);
        \draw[EdgeGraph](0010)--(0020);
        \draw[EdgeGraph](0010)--(0110);
        \draw[EdgeGraph](0011)--(0012);
        \draw[EdgeGraph](0011)--(0021);
        \draw[EdgeGraph](0011)--(0111);
        \draw[EdgeGraph](0020)--(0021);
        \draw[EdgeGraph](0100)--(0101);
        \draw[EdgeGraph](0100)--(0110);
        \draw[EdgeGraph](0101)--(0102);
        \draw[EdgeGraph](0101)--(0111);
        \draw[EdgeGraph](0110)--(0111);
    \end{tikzpicture}}
    \label{subfig:avalanche_poset_1_4}}
    \caption{\footnotesize Hasse diagrams of some $\delta$-avalanche posets.}
    \label{fig:examples_avalanche_posets}
\end{figure}
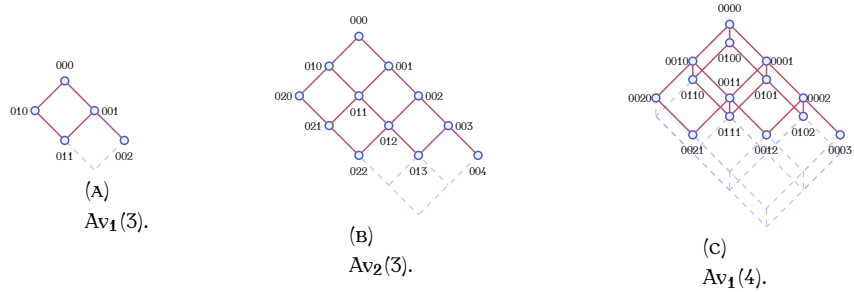
\medbreak

Let $\delta$ be a weakly increasing range map.  Notice that in general,
$\SetAvalanche_\delta(n)$ is not bounded. Since for all $u \in \SetAvalanche_\delta(n)$,
$\Weight\Par{u} \leq \delta(n)$, we have $u \in \max_{\Leq}{\SetAvalanche_\delta}(n)$ if and
only if $\Weight\Par{u} = \delta(n)$. Moreover, due to the fact that any $\delta$-cliff
obtained by decreasing a letter in a $\delta$-avalanche is also a $\delta$-avalanche, the
poset $\SetAvalanche_\delta(n)$ is the order ideal of $\SetCliff_\delta(n)$ generated by
$\max_{\Leq}{\SetAvalanche_\delta}(n)$. Finally, as a particular case, we shall show as a
consequence of upcoming Proposition~\ref{prop:max_avalanche_hill_bijection} that
for any $m \geq 0$ and $n \geq 1$,
$\# \max_{\Leq} \SetAvalanche_\MapM(n) = \FussCatalan_m(n - 1)$.
\medbreak

\begin{Proposition} \label{prop:properties_avalanche_posets}
    For any weakly increasing range map $\delta$ and $n \geq 0$, the poset
    $\SetAvalanche_\delta(n)$
    \begin{enumerate}[label={\it (\roman*)}]
        \item \label{item:properties_avalanche_posets_1}
        is straight, where $u \in \SetAvalanche_\delta(n)$ is covered by $v \in
        \SetAvalanche_\delta(n)$ if and only if there is an $i \in [n]$ such that
        $\IncreaseLetter_i(u) = v$;

        \item \label{item:properties_avalanche_posets_2}
        is coated;

        \item \label{item:properties_avalanche_posets_3}
        is graded, where the rank of an avalanche is its weight;

        \item \label{item:properties_avalanche_posets_4}
        admits an EL-labeling;

        \item \label{item:properties_avalanche_posets_5}
        is a meet semi-sublattice of $\SetCliff_\delta(n)$;

        \item \label{item:properties_avalanche_posets_6}
        is a lattice if and only if $\delta = 0^\omega$.
    \end{enumerate}
\end{Proposition}
\begin{proof}
    Points~\ref{item:properties_avalanche_posets_1},
    \ref{item:properties_avalanche_posets_3}, \ref{item:properties_avalanche_posets_5},
    and~\ref{item:properties_avalanche_posets_6} are immediate. If $u$ and $v$ are two
    $\delta$-avalanches of size $n$ such that $u \Leq v$, then for any $i \in [n - 1]$,
    $\Weight\Par{u_1 \dots u_i} \leq \Weight\Par{v_1 \dots v_i}$. Therefore, the
    $\delta$-cliff $u_1 \dots u_i v_{i + 1} \dots v_n$ is a $\delta$-avalanche. For this
    reason, \ref{item:properties_avalanche_posets_2} checks out.
    Point~\ref{item:properties_avalanche_posets_4} follows
    from~\ref{item:properties_avalanche_posets_2}, and
    Theorem~\ref{thm:subposets_el_shellability}.
\end{proof}
\medbreak

\begin{Proposition} \label{prop:wings_butterflies_avalanche}
    For any $m \geq 1$,
    \begin{enumerate}[label={\it (\roman*)}]
        \item \label{item:wings_butterflies_avalanche_1}
        the graded set $\InputWings\Par{\SetAvalanche_\MapM}$ contains all the
        $\MapM$-avalanches $u$ satisfying $u_i \ne 0$ for all~$i \in [2, |u|]$;

        \item \label{item:wings_butterflies_avalanche_2}
        the graded set $\OutputWings\Par{\SetAvalanche_\MapM}$ contains all the
        $\MapM$-avalanches $u$ satisfying $\Weight\Par{u'} < \MapM\Par{\Brr{u'}}$ for all
        prefixes $u'$ of~$u$ of length $2$ or more;

        \item \label{item:wings_butterflies_avalanche_3}
        the graded set $\Butterflies\Par{\SetAvalanche_\MapM}$ contains all the
        $\MapM$-avalanches $u$ satisfying $u_i \ne 0$ for all $i \in [2, |u|]$, and
        $\Weight\Par{u'} < \MapM\Par{\Brr{u'}}$ for all prefixes $u'$ of~$u$ of length $2$
        or more.
    \end{enumerate}
\end{Proposition}
\medbreak

The four posets of avalanches, input-wings, output-wings, and butterflies are linked in the
following way.
\medbreak

\begin{Theorem} \label{thm:poset_morphisms_avalanche}
    For any $m \geq 1$ and $n \geq 0$,
    \begin{equation}
        \begin{tikzpicture}[Centering,xscale=3.5,yscale=1.6,font=\small]
            \node(Avalanches)at(0,0){$\SetAvalanche_{\mathbf m - 1}(n)$};
            \node(InputWingsAvalanche)at(1,0){$\InputWings\Par{\SetAvalanche_\MapM}(n)$};
            \node(OutputWingsAvalanche)at(2,0){$\OutputWings\Par{\SetAvalanche_\MapM}(n)$};
            \node(ButterfliesAvalanche)at(3,0)
                {$\Butterflies\Par{\SetAvalanche_{\mathbf m + 1}}(n)$};
            \draw[MapIsomorphism](Avalanches)--(InputWingsAvalanche)node[midway,above]
                {$\varphi_1$};
            \draw[MapEmbedding](InputWingsAvalanche)
                --(OutputWingsAvalanche)node[midway,above]
                {$\varphi_2$};
            \draw[MapIsomorphism](OutputWingsAvalanche)
                --(ButterfliesAvalanche)node[midway,above]
                {$\varphi_1$};
        \end{tikzpicture}
    \end{equation}
    is a diagram of poset embeddings or isomorphisms, where the maps $\varphi_1$ and
    $\varphi_2$ are defined, for any $u \in \N^n$ and $i \in [n]$, by
    \begin{math}
        \varphi_1(u)_i := \IndicatorFunction_{i \ne 1} \Par{u_i + 1}
    \end{math}
    and
    \begin{math}
        \varphi_2(u)_i := \IndicatorFunction_{i \ne 1} \Par{u_i - 1}.
    \end{math}
\end{Theorem}
\begin{proof}
    This follows from the descriptions of the input-wings, output-wings, and butterflies of
    $\SetAvalanche_\MapM(n)$ provided by Proposition~\ref{prop:wings_butterflies_avalanche}.
\end{proof}
\medbreak

Figure~\ref{fig:poset_morphisms_avalanche} gives an example of the poset isomorphisms or
embeddings described by the statement of Theorem~\ref{thm:poset_morphisms_avalanche}.
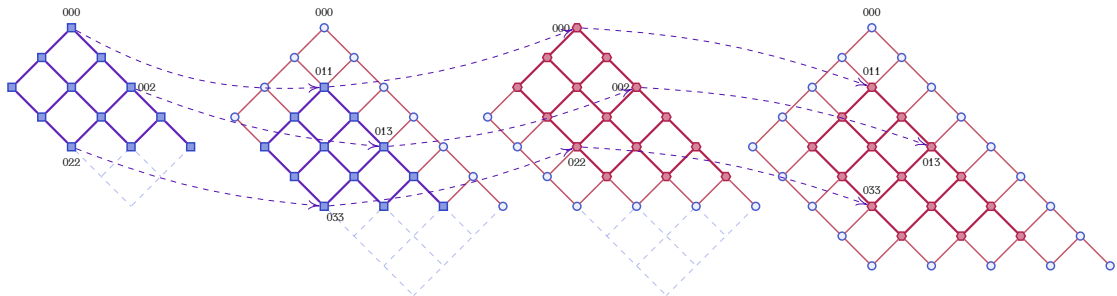
\begin{figure}[ht]
    \centering
    \scalebox{.7}{
    \begin{tikzpicture}[Centering,xscale=.8,yscale=.8]
        % Avalanche, m = 2, n = 3.
        \begin{scope}[xshift=0cm,yshift=0cm,rotate=-135]
        \draw[Grid](0,0)grid(2,4);
        \node[MarkedNodeGraph](A000)at(0,0){};
        \node[MarkedNodeGraph](A001)at(0,1){};
        \node[MarkedNodeGraph](A002)at(0,2){};
        \node[MarkedNodeGraph](A003)at(0,3){};
        \node[MarkedNodeGraph](A004)at(0,4){};
        \node[MarkedNodeGraph](A010)at(1,0){};
        \node[MarkedNodeGraph](A011)at(1,1){};
        \node[MarkedNodeGraph](A012)at(1,2){};
        \node[MarkedNodeGraph](A013)at(1,3){};
        \node[MarkedNodeGraph](A020)at(2,0){};
        \node[MarkedNodeGraph](A021)at(2,1){};
        \node[MarkedNodeGraph](A022)at(2,2){};
        \node[NodeLabelGraph,above of=A000]{$000$};
        \node[NodeLabelGraph,right of=A002]{$002$};
        \node[NodeLabelGraph,below of=A022]{$022$};
        \draw[MarkedEdgeGraph](A000)--(A001);
        \draw[MarkedEdgeGraph](A000)--(A010);
        \draw[MarkedEdgeGraph](A001)--(A002);
        \draw[MarkedEdgeGraph](A001)--(A011);
        \draw[MarkedEdgeGraph](A002)--(A003);
        \draw[MarkedEdgeGraph](A002)--(A012);
        \draw[MarkedEdgeGraph](A003)--(A004);
        \draw[MarkedEdgeGraph](A003)--(A013);
        \draw[MarkedEdgeGraph](A010)--(A011);
        \draw[MarkedEdgeGraph](A010)--(A020);
        \draw[MarkedEdgeGraph](A011)--(A012);
        \draw[MarkedEdgeGraph](A011)--(A021);
        \draw[MarkedEdgeGraph](A012)--(A013);
        \draw[MarkedEdgeGraph](A012)--(A022);
        \draw[MarkedEdgeGraph](A020)--(A021);
        \draw[MarkedEdgeGraph](A021)--(A022);
        \end{scope}
        %
        % Input-wings avalanche, m = 3, n = 3.
        \begin{scope}[xshift=6cm,yshift=0cm,rotate=-135]
        \draw[Grid](0,0)grid(3,6);
        \node[NodeGraph](B000)at(0,0){};
        \node[NodeGraph](B001)at(0,1){};
        \node[NodeGraph](B002)at(0,2){};
        \node[NodeGraph](B003)at(0,3){};
        \node[NodeGraph](B004)at(0,4){};
        \node[NodeGraph](B005)at(0,5){};
        \node[NodeGraph](B006)at(0,6){};
        \node[NodeGraph](B010)at(1,0){};
        \node[MarkedNodeGraph](B011)at(1,1){};
        \node[MarkedNodeGraph](B012)at(1,2){};
        \node[MarkedNodeGraph](B013)at(1,3){};
        \node[MarkedNodeGraph](B014)at(1,4){};
        \node[MarkedNodeGraph](B015)at(1,5){};
        \node[NodeGraph](B020)at(2,0){};
        \node[MarkedNodeGraph](B021)at(2,1){};
        \node[MarkedNodeGraph](B022)at(2,2){};
        \node[MarkedNodeGraph](B023)at(2,3){};
        \node[MarkedNodeGraph](B024)at(2,4){};
        \node[NodeGraph](B030)at(3,0){};
        \node[MarkedNodeGraph](B031)at(3,1){};
        \node[MarkedNodeGraph](B032)at(3,2){};
        \node[MarkedNodeGraph](B033)at(3,3){};
        \node[NodeLabelGraph,above of=B000]{$000$};
        \node[NodeLabelGraph,above of=B011]{$011$};
        \node[NodeLabelGraph,above of=B013]{$013$};
        \node[NodeLabelGraph,below right of=B033]{$033$};
        \draw[EdgeGraph](B000)--(B001);
        \draw[EdgeGraph](B000)--(B010);
        \draw[EdgeGraph](B001)--(B002);
        \draw[EdgeGraph](B001)--(B011);
        \draw[EdgeGraph](B002)--(B003);
        \draw[EdgeGraph](B002)--(B012);
        \draw[EdgeGraph](B003)--(B004);
        \draw[EdgeGraph](B003)--(B013);
        \draw[EdgeGraph](B004)--(B005);
        \draw[EdgeGraph](B004)--(B014);
        \draw[EdgeGraph](B005)--(B006);
        \draw[EdgeGraph](B005)--(B015);
        \draw[EdgeGraph](B010)--(B011);
        \draw[EdgeGraph](B010)--(B020);
        \draw[MarkedEdgeGraph](B011)--(B012);
        \draw[MarkedEdgeGraph](B011)--(B021);
        \draw[MarkedEdgeGraph](B012)--(B013);
        \draw[MarkedEdgeGraph](B012)--(B022);
        \draw[MarkedEdgeGraph](B013)--(B014);
        \draw[MarkedEdgeGraph](B013)--(B023);
        \draw[MarkedEdgeGraph](B014)--(B015);
        \draw[MarkedEdgeGraph](B014)--(B024);
        \draw[EdgeGraph](B020)--(B021);
        \draw[EdgeGraph](B020)--(B030);
        \draw[MarkedEdgeGraph](B021)--(B022);
        \draw[MarkedEdgeGraph](B021)--(B031);
        \draw[MarkedEdgeGraph](B022)--(B023);
        \draw[MarkedEdgeGraph](B022)--(B032);
        \draw[MarkedEdgeGraph](B023)--(B024);
        \draw[MarkedEdgeGraph](B023)--(B033);
        \draw[EdgeGraph](B030)--(B031);
        \draw[MarkedEdgeGraph](B031)--(B032);
        \draw[MarkedEdgeGraph](B032)--(B033);
        \end{scope}
        %
        % Output-wings avalanche, m = 3, n = 3.
        \begin{scope}[xshift=12cm,yshift=0cm,rotate=-135]
        \draw[Grid](0,0)grid(3,6);
        \node[Marked2NodeGraph](C000)at(0,0){};
        \node[Marked2NodeGraph](C001)at(0,1){};
        \node[Marked2NodeGraph](C002)at(0,2){};
        \node[Marked2NodeGraph](C003)at(0,3){};
        \node[Marked2NodeGraph](C004)at(0,4){};
        \node[Marked2NodeGraph](C005)at(0,5){};
        \node[NodeGraph](C006)at(0,6){};
        \node[Marked2NodeGraph](C010)at(1,0){};
        \node[Marked2NodeGraph](C011)at(1,1){};
        \node[Marked2NodeGraph](C012)at(1,2){};
        \node[Marked2NodeGraph](C013)at(1,3){};
        \node[Marked2NodeGraph](C014)at(1,4){};
        \node[NodeGraph](C015)at(1,5){};
        \node[Marked2NodeGraph](C020)at(2,0){};
        \node[Marked2NodeGraph](C021)at(2,1){};
        \node[Marked2NodeGraph](C022)at(2,2){};
        \node[Marked2NodeGraph](C023)at(2,3){};
        \node[NodeGraph](C024)at(2,4){};
        \node[NodeGraph](C030)at(3,0){};
        \node[NodeGraph](C031)at(3,1){};
        \node[NodeGraph](C032)at(3,2){};
        \node[NodeGraph](C033)at(3,3){};
        \node[NodeLabelGraph,left of=C000]{$000$};
        \node[NodeLabelGraph,left of=C002]{$002$};
        \node[NodeLabelGraph,below of=C022]{$022$};
        \draw[Marked2EdgeGraph](C000)--(C001);
        \draw[Marked2EdgeGraph](C000)--(C010);
        \draw[Marked2EdgeGraph](C001)--(C002);
        \draw[Marked2EdgeGraph](C001)--(C011);
        \draw[Marked2EdgeGraph](C002)--(C003);
        \draw[Marked2EdgeGraph](C002)--(C012);
        \draw[Marked2EdgeGraph](C003)--(C004);
        \draw[Marked2EdgeGraph](C003)--(C013);
        \draw[Marked2EdgeGraph](C004)--(C005);
        \draw[Marked2EdgeGraph](C004)--(C014);
        \draw[EdgeGraph](C005)--(C006);
        \draw[EdgeGraph](C005)--(C015);
        \draw[Marked2EdgeGraph](C010)--(C011);
        \draw[Marked2EdgeGraph](C010)--(C020);
        \draw[Marked2EdgeGraph](C011)--(C012);
        \draw[Marked2EdgeGraph](C011)--(C021);
        \draw[Marked2EdgeGraph](C012)--(C013);
        \draw[Marked2EdgeGraph](C012)--(C022);
        \draw[Marked2EdgeGraph](C013)--(C014);
        \draw[Marked2EdgeGraph](C013)--(C023);
        \draw[EdgeGraph](C014)--(C015);
        \draw[EdgeGraph](C014)--(C024);
        \draw[Marked2EdgeGraph](C020)--(C021);
        \draw[EdgeGraph](C020)--(C030);
        \draw[Marked2EdgeGraph](C021)--(C022);
        \draw[EdgeGraph](C021)--(C031);
        \draw[Marked2EdgeGraph](C022)--(C023);
        \draw[EdgeGraph](C022)--(C032);
        \draw[EdgeGraph](C023)--(C024);
        \draw[EdgeGraph](C023)--(C033);
        \draw[EdgeGraph](C030)--(C031);
        \draw[EdgeGraph](C031)--(C032);
        \draw[EdgeGraph](C032)--(C033);
        \end{scope}
        %
        % Output-wings avalanche, m = 4, n = 3.
        \begin{scope}[xshift=19cm,yshift=0cm,rotate=-135]
        %\draw[Grid](0,0)grid(4,8);
        %
        \node[NodeGraph](D000)at(0,0){};
        \node[NodeGraph](D001)at(0,1){};
        \node[NodeGraph](D002)at(0,2){};
        \node[NodeGraph](D003)at(0,3){};
        \node[NodeGraph](D004)at(0,4){};
        \node[NodeGraph](D005)at(0,5){};
        \node[NodeGraph](D006)at(0,6){};
        \node[NodeGraph](D007)at(0,7){};
        \node[NodeGraph](D008)at(0,8){};
        \node[NodeGraph](D010)at(1,0){};
        \node[Marked2NodeGraph](D011)at(1,1){};
        \node[Marked2NodeGraph](D012)at(1,2){};
        \node[Marked2NodeGraph](D013)at(1,3){};
        \node[Marked2NodeGraph](D014)at(1,4){};
        \node[Marked2NodeGraph](D015)at(1,5){};
        \node[Marked2NodeGraph](D016)at(1,6){};
        \node[NodeGraph](D017)at(1,7){};
        \node[NodeGraph](D020)at(2,0){};
        \node[Marked2NodeGraph](D021)at(2,1){};
        \node[Marked2NodeGraph](D022)at(2,2){};
        \node[Marked2NodeGraph](D023)at(2,3){};
        \node[Marked2NodeGraph](D024)at(2,4){};
        \node[Marked2NodeGraph](D025)at(2,5){};
        \node[NodeGraph](D026)at(2,6){};
        \node[NodeGraph](D030)at(3,0){};
        \node[Marked2NodeGraph](D031)at(3,1){};
        \node[Marked2NodeGraph](D032)at(3,2){};
        \node[Marked2NodeGraph](D033)at(3,3){};
        \node[Marked2NodeGraph](D034)at(3,4){};
        \node[NodeGraph](D035)at(3,5){};
        \node[NodeGraph](D040)at(4,0){};
        \node[NodeGraph](D041)at(4,1){};
        \node[NodeGraph](D042)at(4,2){};
        \node[NodeGraph](D043)at(4,3){};
        \node[NodeGraph](D044)at(4,4){};
        \node[NodeLabelGraph,above of=D000]{$000$};
        \node[NodeLabelGraph,above of=D011]{$011$};
        \node[NodeLabelGraph,below of=D013]{$013$};
        \node[NodeLabelGraph,above of=D033]{$033$};
        \draw[EdgeGraph](D000)--(D001);
        \draw[EdgeGraph](D000)--(D010);
        \draw[EdgeGraph](D001)--(D002);
        \draw[EdgeGraph](D001)--(D011);
        \draw[EdgeGraph](D002)--(D003);
        \draw[EdgeGraph](D002)--(D012);
        \draw[EdgeGraph](D003)--(D004);
        \draw[EdgeGraph](D003)--(D013);
        \draw[EdgeGraph](D004)--(D005);
        \draw[EdgeGraph](D004)--(D014);
        \draw[EdgeGraph](D005)--(D006);
        \draw[EdgeGraph](D005)--(D015);
        \draw[EdgeGraph](D006)--(D007);
        \draw[EdgeGraph](D006)--(D016);
        \draw[EdgeGraph](D007)--(D008);
        \draw[EdgeGraph](D007)--(D017);
        \draw[EdgeGraph](D010)--(D011);
        \draw[EdgeGraph](D010)--(D020);
        \draw[Marked2EdgeGraph](D011)--(D012);
        \draw[Marked2EdgeGraph](D011)--(D021);
        \draw[Marked2EdgeGraph](D012)--(D013);
        \draw[Marked2EdgeGraph](D012)--(D022);
        \draw[Marked2EdgeGraph](D013)--(D014);
        \draw[Marked2EdgeGraph](D013)--(D023);
        \draw[Marked2EdgeGraph](D014)--(D015);
        \draw[Marked2EdgeGraph](D014)--(D024);
        \draw[Marked2EdgeGraph](D015)--(D016);
        \draw[Marked2EdgeGraph](D015)--(D025);
        \draw[EdgeGraph](D016)--(D017);
        \draw[EdgeGraph](D016)--(D026);
        \draw[EdgeGraph](D020)--(D021);
        \draw[EdgeGraph](D020)--(D030);
        \draw[Marked2EdgeGraph](D021)--(D022);
        \draw[Marked2EdgeGraph](D021)--(D031);
        \draw[Marked2EdgeGraph](D022)--(D023);
        \draw[Marked2EdgeGraph](D022)--(D032);
        \draw[Marked2EdgeGraph](D023)--(D024);
        \draw[Marked2EdgeGraph](D023)--(D033);
        \draw[Marked2EdgeGraph](D024)--(D025);
        \draw[Marked2EdgeGraph](D024)--(D034);
        \draw[EdgeGraph](D025)--(D026);
        \draw[EdgeGraph](D025)--(D035);
        \draw[EdgeGraph](D030)--(D031);
        \draw[EdgeGraph](D030)--(D040);
        \draw[Marked2EdgeGraph](D031)--(D032);
        \draw[EdgeGraph](D031)--(D041);
        \draw[Marked2EdgeGraph](D032)--(D033);
        \draw[EdgeGraph](D032)--(D042);
        \draw[Marked2EdgeGraph](D033)--(D034);
        \draw[EdgeGraph](D033)--(D043);
        \draw[EdgeGraph](D034)--(D035);
        \draw[EdgeGraph](D034)--(D044);
        \draw[EdgeGraph](D040)--(D041);
        \draw[EdgeGraph](D041)--(D042);
        \draw[EdgeGraph](D042)--(D043);
        \draw[EdgeGraph](D043)--(D044);
        \end{scope}
        %
        % Arrows.
        \draw[MapGraph](A000)edge[bend right=16](B011);
        \draw[MapGraph](A002)edge[bend right=8](B013);
        \draw[MapGraph](A022)edge[bend right=8](B033);
        \draw[MapGraph](B011)edge[bend right=8](C000);
        \draw[MapGraph](B013)edge[bend right=8](C002);
        \draw[MapGraph](B033)edge[bend right=8](C022);
        \draw[MapGraph](C000)edge[bend left=8](D011);
        \draw[MapGraph](C002)edge[bend left=8](D013);
        \draw[MapGraph](C022)edge[bend left=8](D033);
    \end{tikzpicture}}
    \caption{\footnotesize From the top to bottom and left to right, here are the posets
    $\SetAvalanche_\MapTwo(3)$, $\SetAvalanche_{\mathbf 3}(3)$, $\SetAvalanche_{\mathbf
    3}(3)$, and $\SetAvalanche_{\mathbf 4}(3)$. All these posets contain
    $\SetAvalanche_\MapTwo(3)$ as subposet by restricting on input-wings, output-wings,
    or butterflies.}
    \label{fig:poset_morphisms_avalanche}
\end{figure}
As a consequence of Theorem~\ref{thm:poset_morphisms_avalanche}, for
any $m \geq 1$ and $n \geq 0$, the number of input-wings in $\SetAvalanche_\MapM(n)$
is~$\FussCatalan_{m - 1}(n)$.
\medbreak

Let us define for any $m \geq 0$ and $n \geq 1$ the $n$-th \Def{twisted $m$-Fuss-Catalan
number} by
\begin{equation}
    \TwistedFussCatalan_m(n) :=
    \frac{1}{n} \binom{n(m + 1) - 2}{n - 1}.
\end{equation}
\medbreak

\begin{Proposition} \label{prop:enumeration_output_wings_avalanche}
    For any $m \geq 1$ and $n \geq 1$, $\# \OutputWings\Par{\SetAvalanche_\MapM}(0) = 1$
    and
    \begin{math}
        \# \OutputWings\Par{\SetAvalanche_\MapM}(n) = \TwistedFussCatalan_m(n).
    \end{math}
\end{Proposition}
\begin{proof}
    By Proposition~\ref{prop:wings_butterflies_avalanche}, the set
    $\OutputWings\Par{\SetAvalanche_\MapM}(n)$ is in one-to-one correspondence with the set
    of all $\MapM$-cliffs $v$ of size $n$ such that for any $i \in [2, n]$, $v_{i - 1} \leq
    v_i < \MapM(i)$. A possible bijection between these two sets sends any $u \in
    \OutputWings\Par{\SetAvalanche_\MapM}(n)$ to the $\MapM$-cliff $v$ of the same size such
    that for any $i \in [n]$, $v_i := u_1 + \dots + u_i$.  These words are moreover in
    one-to-one correspondence with indecomposable $m$-Dyck paths with $n \geq 1$ up steps,
    that are $m$-Dyck paths which cannot be written as a nontrivial concatenation of two
    $m$-Dyck paths. A possible bijection is the one described in upcoming
    Section~\ref{subsubsec:hill_objects}. Let us denote by $\GeneratingSeries(t)$ (resp.\
    $\GeneratingSeries'(t)$)  the generating series of $m$-Dyck paths (resp.\ indecomposable
    $m$-Dyck paths) enumerated w.r.t. their numbers of up steps. By convention,
    $\GeneratingSeries'(t)$ has no constant term. Since any $m$-Dyck path decomposes in a
    unique way as a concatenation of indecomposable $m$-Dyck paths, one has
    \begin{math}
        \GeneratingSeries(t) = \Par{1 - \GeneratingSeries'(t)}^{-1}.
    \end{math}
    Now, by using the fact that $\GeneratingSeries(t)$ satisfies
    \begin{math}
        \GeneratingSeries(t) = 1 + t \GeneratingSeries(t)^{m + 1},
    \end{math}
    we have
    \begin{equation} \label{equ:enumeration_output_wings_avalanche}
        \GeneratingSeries'(t)
        = \frac{\GeneratingSeries(t) - 1}{\GeneratingSeries(t)}
        = t \GeneratingSeries(t)^m
        = t \Par{\frac{1}{1 - \GeneratingSeries'(t)}}^m
    \end{equation}
    This relation satisfied by $\GeneratingSeries'(t)$ between the first and last members
    of~\eqref{equ:enumeration_output_wings_avalanche} is known to be the one of the
    generating series of twisted $m$-Fuss-Catalan numbers (see~\cite{Slo} for instance).
\end{proof}
\medbreak

By Proposition~\ref{prop:enumeration_output_wings_avalanche}, the first numbers of
output-wings of $\SetAvalanche_\MapM(n)$ by sizes are
\begin{subequations}
\begin{equation}
    1, 1, 1, 1, 1, 1, 1, 1, \qquad m = 0,
\end{equation}
\begin{equation}
    1, 1, 1, 2, 5, 14, 42, 132, \qquad m = 1,
\end{equation}
\begin{equation}
    1, 1, 2, 7, 30, 143, 728, 3876, \qquad m = 2.
\end{equation}
\end{subequations}
The fourth sequence is Sequence~\OEIS{A006013} of~\cite{Slo}. As a side remark, for any $m
\geq 1$, the generating series of the graded set $\OutputWings\Par{\SetAvalanche_\MapM}$ is
$1$ plus the inverse, for the functional composition of series, of the polynomial~$t (1 -
t)^m$.
\medbreak

\begin{Proposition} \label{prop:irreducibles_avalanche}
    For any $m \geq 1$ and $n \geq 1$,
    \begin{enumerate}[label={\it (\roman*)}]
        \item the set $\MeetIrreducibles\Par{\SetAvalanche_\MapM(n)}$ contains all
        $\MapM$-avalanches $u$ such that $u = u' a$ where $u' \in \max_{\Leq}
        \SetAvalanche_\MapM(n - 1)$ and $a \in \HanL{m - 1}$;
        \item the set $\JoinIrreducibles\Par{\SetAvalanche_\MapM(n)}$ contains all
        $\MapM$-avalanches having exactly one letter different from~$0$.
    \end{enumerate}
\end{Proposition}
\medbreak

By Proposition~\ref{prop:irreducibles_avalanche} and by upcoming
Proposition~\ref{prop:max_avalanche_hill_bijection}, the number of meet-irreducible elements
of $\SetAvalanche_\MapM(n)$ satisfies, for any $m \geq 1$ and $n \geq 2$,
\begin{math}
    \# \MeetIrreducibles\Par{\SetAvalanche_\MapM(n)}
    =
    m \FussCatalan_m(n - 2)
\end{math}
and the number of join-irreducibles elements of $\SetAvalanche_\MapM(n)$ satisfies, for any
$m \geq 1$ and $n \geq 1$,
\begin{math}
    \# \JoinIrreducibles\Par{\SetAvalanche_\MapM(n)}
    = m \binom{n}{2}.
\end{math}
\medbreak

%%%%%%%%%%%%%%%%%%%%%%%%%%%%%%%%%%%%%%%%%%%%%%%%%%%%%%%%%%%%%%%%%%%%%%%%%%%%%%%%%%%%%%%%%%%%
\subsubsection{Cubic realization}
The map $\varphi_2$ introduced in Theorem~\ref{thm:poset_morphisms_avalanche} is used here
to describe the cells of maximal dimension of the cubic realization of
$\SetAvalanche_\MapM(n)$, $m \geq 1$, $n \geq 0$.
\medbreak

\begin{Proposition} \label{prop:cells_avalanche}
    For any $m \geq 1$, $n \geq 0$, and $u \in \InputWings\Par{\SetAvalanche_\MapM}(n)$,
    \begin{enumerate}[label={\it (\roman*)}]
        \item \label{item:cells_avalanche_1}
        the $\MapM$-avalanche $\varphi_2(u)$ is cell-compatible with the $\MapM$-avalanche
        $u$;

        \item \label{item:cells_avalanche_2}
        the cell $\Angle{\varphi_2(u), u}$ is pure;

        \item \label{item:cells_avalanche_3}
        all cells of
        \begin{math}
            \Bra{\Angle{\varphi_2(u), u} : u \in \InputWings\Par{\SetAvalanche_\MapM}(n)}
        \end{math}
        are pairwise disjoint.
    \end{enumerate}
\end{Proposition}
\begin{proof}
    Let $v$ be an $\MapM$-cliff of size $n$ satisfying $v_i \in \Bra{\varphi_2(u)_i, u_i}$
    for all $i \in [n]$. By definition of $\varphi_2$, $v_1 = 0$ and $v_i \in \Bra{u_i - 1,
    u_i}$ for all $i \in [2, n]$. Since $u$ is an input-wing of $\SetAvalanche_\MapM$,
    $\varphi_2(u)$ is an $\MapM$-avalanche, and due to the definition of
    $\MapM$-avalanches, any $\MapM$-cliff obtained by decrementing some letters of $u$ is
    still an $\MapM$-avalanche. Thus, $v \in \SetAvalanche_\MapM$
    and~\ref{item:cells_avalanche_1} holds. Points~\ref{item:cells_avalanche_2}
    and~\ref{item:cells_avalanche_3} are consequences of the fact that there is no element
    of $\SetAvalanche_\MapM(n)$ inside a cell $\Angle{\varphi_2(u), u}$. Indeed, since for
    any $i \in [n]$, $\Brr{\varphi_2(u)_i - u_i} \leq 1$, we have $v_i \in
    \Bra{\varphi_2(u)_i, u_i}$ for all $v \in \Angle{\varphi_2(u), u} \cap
    \SetAvalanche_\MapM(n)$.
\end{proof}
\medbreak

As shown by Proposition~\ref{prop:cells_avalanche}, the cells of maximal dimension of the
cubic realization of $\SetAvalanche_\MapM(n)$ are all of the form $\Angle{\varphi_2(u), u}$
where the $u$ are input-wings of~$\SetAvalanche_\MapM(n)$.
\medbreak

\begin{Proposition} \label{prop:volume_avalanche}
    For any $m \geq 1$ and $n \geq 0$,
    \begin{math}
        \Volume\Par{\CubicReal\Par{\SetAvalanche_\MapM(n)}} = \FussCatalan_{m - 1}(n).
    \end{math}
\end{Proposition}
\begin{proof}
    Proposition~\ref{prop:cells_avalanche} describes all the cells of maximal dimension of
    $\CubicReal\Par{\SetAvalanche_\MapM(n)}$ as cells $\Angle{\varphi_2(u), u}$ where $u$ is
    an input-wing of $\SetAvalanche_\MapM(n)$. Since all these cells are pairwise disjoint,
    the volume of $\CubicReal\Par{\SetAvalanche_\MapM(n)}$ expresses
    as~\eqref{equ:volume_computation_from_output_wings}. Moreover, observe that the volume
    of each cell $\Angle{\varphi_2(u), u}$ where $u$ in an input-wing is by definition of
    $\varphi_2$ equal to $1$. Therefore,
    $\Volume\Par{\CubicReal\Par{\SetAvalanche_\MapM(n)}}$ is equal to the number of
    input-wings of $\SetAvalanche_\MapM(n)$. The statement of the proposition follows now
    from Theorem~\ref{thm:poset_morphisms_avalanche}.
\end{proof}
\medbreak

%%%%%%%%%%%%%%%%%%%%%%%%%%%%%%%%%%%%%%%%%%%%%%%%%%%%%%%%%%%%%%%%%%%%%%%%%%%%%%%%%%%%%%%%%%%%
%%%%%%%%%%%%%%%%%%%%%%%%%%%%%%%%%%%%%%%%%%%%%%%%%%%%%%%%%%%%%%%%%%%%%%%%%%%%%%%%%%%%%%%%%%%%
\subsection{$\delta$-hill posets}
We now introduce $\delta$-hills and $\delta$-hill posets as subposets of $\delta$-cliff
posets. As we shall see, some of these posets are sublattices of $\MapM$-cliff lattices.
\medbreak

%%%%%%%%%%%%%%%%%%%%%%%%%%%%%%%%%%%%%%%%%%%%%%%%%%%%%%%%%%%%%%%%%%%%%%%%%%%%%%%%%%%%%%%%%%%%
\subsubsection{Objects} \label{subsubsec:hill_objects}
For any range map $\delta$, let $\SetHill_\delta$ be the graded subset of $\SetCliff_\delta$
containing all $\delta$-cliffs such that that for any $i \in [|u| - 1]$, $u_i \leq u_{i +
1}$. Any element of $\SetHill_\delta$ is a \Def{$\delta$-hill}. For instance,
\begin{equation}
    \SetHill_\MapTwo(3) =
    \Bra{000, 001, 011, 002, 012, 022, 003, 013, 023, 004, 014, 024}.
\end{equation}
\medbreak

\begin{Proposition} \label{prop:properties_hill_objects}
    For any weakly increasing range map $\delta$, the graded set $\SetHill_\delta$ is
    \begin{enumerate}[label={\it (\roman*)}]
        \item \label{item:properties_hill_objects_1}
        closed by prefix;

        \item \label{item:properties_hill_objects_2}
        is minimally extendable if and only if $\delta = 0^\omega$;

        \item \label{item:properties_hill_objects_3}
        is maximally extendable.
    \end{enumerate}
\end{Proposition}
\begin{proof}
    Point~\ref{item:properties_hill_objects_1} is an immediate consequence of the definition
    of $\delta$-hills.  We have immediately that $\SetHill_{0^\omega}$ is minimally
    extendable.  Moreover, when $\delta \ne 0^\omega$, there is an $n \geq 1$ such that
    $\delta(n) \geq 1$.  Therefore, $\GreatestElement_\delta(n)$ is a $\delta$-hill but
    $\GreatestElement_\delta(n) \, 0$ is not.  This
    establishes~\ref{item:properties_hill_objects_2}.  Finally, since for any $n \geq 0$,
    $\delta(n + 1) \geq \delta(n)$, one has $\delta(n + 1) \geq u_n$ for any $u \in
    \SetHill_\delta(n)$. This shows that $u \, \delta(n + 1)$ is a $\delta$-hill. Therefore,
    \ref{item:properties_hill_objects_3} holds.
\end{proof}
\medbreak

For any $m \geq 0$, an \Def{$m$-Dyck path} of size $n$ is a path in from $(0, 0)$ to $((m +
1)n, 0)$ in $\N^2$ staying above the $x$-axis, and consisting only in steps of the form $(1,
-1)$, called \Def{down steps}, or steps of the form $(1, m)$, called \Def{up steps}. We
denote by $\SetDyckPaths_m$ the graded set of all $m$-Dyck paths. There is a one-to-one
correspondence between $\SetHill_\MapM(n)$ and $\SetDyckPaths_m(n)$ wherein an $m$-Dyck path
$w$ of size $n$ is sent to the $\MapM$-hill $u$ of size $n$ such that for any $i \in [n]$,
$u_i$ is the number of down steps to the left of the $i$-th up step of~$w$.  For instance,
the $2$-Dyck path
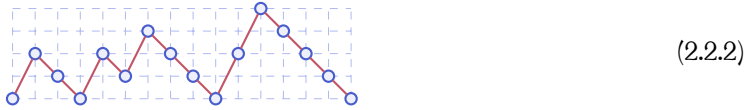
\begin{equation}
    \begin{tikzpicture}[Centering,scale=.3]
        \draw[Grid](0,0)grid(15,4);
        \node[PathNode](0)at(0,0){};
        \node[PathNode](1)at(1,2){};
        \node[PathNode](2)at(2,1){};
        \node[PathNode](3)at(3,0){};
        \node[PathNode](4)at(4,2){};
        \node[PathNode](5)at(5,1){};
        \node[PathNode](6)at(6,3){};
        \node[PathNode](7)at(7,2){};
        \node[PathNode](8)at(8,1){};
        \node[PathNode](9)at(9,0){};
        \node[PathNode](10)at(10,2){};
        \node[PathNode](11)at(11,4){};
        \node[PathNode](12)at(12,3){};
        \node[PathNode](13)at(13,2){};
        \node[PathNode](14)at(14,1){};
        \node[PathNode](15)at(15,0){};
        \draw[PathStep](0)--(1);
        \draw[PathStep](1)--(2);
        \draw[PathStep](2)--(3);
        \draw[PathStep](3)--(4);
        \draw[PathStep](4)--(5);
        \draw[PathStep](5)--(6);
        \draw[PathStep](6)--(7);
        \draw[PathStep](7)--(8);
        \draw[PathStep](8)--(9);
        \draw[PathStep](9)--(10);
        \draw[PathStep](10)--(11);
        \draw[PathStep](11)--(12);
        \draw[PathStep](12)--(13);
        \draw[PathStep](13)--(14);
        \draw[PathStep](14)--(15);
    \end{tikzpicture}
\end{equation}
is sent to the $\MapTwo$-hill $02366$. Since $m$-Dyck paths of size $n$ are known to be
enumerated by $m$-Fuss-Catalan numbers, one has
\begin{math}
    \# \SetHill_\MapM(n) = \FussCatalan_m(n).
\end{math}
\medbreak

\begin{Proposition} \label{prop:image_elevation_map_hill}
    For any range map $\delta$ and any $n \geq 0$,
    \begin{math}
        \ElevationImage_{\SetHill_\delta}(n) = \SetAvalanche_\delta(n).
    \end{math}
\end{Proposition}
\begin{proof}
    First, since $\SetHill_\delta$ is by Proposition~\ref{prop:properties_hill_objects}
    closed by prefix, the $\SetHill_\delta$-elevation map and the
    $\SetHill_\delta$-elevation image are well-defined.  Let $u \in \SetHill_\delta(n)$ and
    $v := \ElevationMap_{\SetHill_\delta}(u)$. By definition of $\delta$-hills and of the
    $\SetHill_\delta$-elevation map, we have $v_1 = u_1$ and, for any $i \in [2, n]$, $v_i =
    u_i - u_{i - 1}$.  Therefore, for any prefix $v' := v_1 \dots v_j$, $j \in [n]$, of $v$,
    we have
    \begin{equation}
        \Weight\Par{v'}
        = u_1 + \Par{u_2 - u_1} + \Par{u_3 - u_2} + \dots + \Par{u_j - u_{j - 1}}
        = u_j.
    \end{equation}
    Since $u$ is in particular a $\delta$-cliff of size $n$, then $u_j \leq \delta(j)$, so
    that $v \in \SetAvalanche_\delta(n)$. This shows that
    $\ElevationImage_{\SetHill_\delta}(n)$ is a subset of $\SetAvalanche_\delta(n)$.
    \smallbreak

    Now, let $u$ be an $\delta$-avalanche of size $n$.  Let us show by induction on $n \geq
    0$ that there exists $v \in \SetHill_\delta(n)$ such that
    $\ElevationMap_{\SetHill_\delta}(v) = u$. When $n = 0$, the property is trivially
    satisfied. When $n \geq 1$, since $\SetAvalanche_\delta$ is, by
    Proposition~\ref{prop:properties_avalanche_objects}, closed by prefix, one has $u = u'
    a$ for a $u' \in \SetAvalanche_\delta(n - 1)$ and an $a \in \N$. By induction
    hypothesis, there exists $v' \in \SetHill_\delta(n - 1)$ such that
    $\ElevationMap_{\SetHill_\delta}\Par{v'} = u'$. Now, let $b := a + v'_{n - 1}$ and set
    $v := v' b$. By using what we have proven in the first paragraph, $\Weight\Par{u'} =
    v'_{n - 1}$. Since $\Weight\Par{u'} + a = \Weight(u) \leq \delta(n)$, we have that $b
    \leq \delta(n)$. Therefore, since moreover $b \geq v'_{n - 1}$, $v$ is a $\delta$-hill
    and it satisfies $\ElevationMap_{\SetHill_\delta}(v) = u$.
\end{proof}
\medbreak

%%%%%%%%%%%%%%%%%%%%%%%%%%%%%%%%%%%%%%%%%%%%%%%%%%%%%%%%%%%%%%%%%%%%%%%%%%%%%%%%%%%%%%%%%%%%
\subsubsection{Posets}
For any $n \geq 0$, the subposet $\SetHill_\delta(n)$ of $\SetCliff_\delta(n)$ is the
\Def{$\delta$-hill poset} of order $n$. Figure~\ref{fig:examples_hill_posets} shows the
Hasse diagrams of some $\MapM$-hill posets.
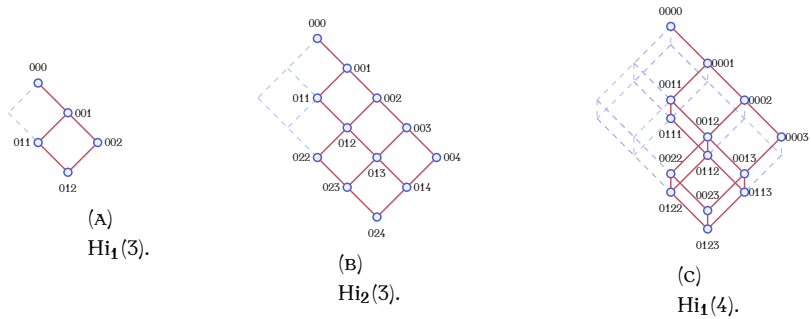
\begin{figure}[ht]
    \centering
    \subfloat[][$\SetHill_\MapOne(3)$.]{
    \centering
    \scalebox{.7}{
    \begin{tikzpicture}[Centering,xscale=.8,yscale=.8,rotate=-135]
        \draw[Grid](0,0)grid(1,2);
        \node[NodeGraph](000)at(0,0){};
        \node[NodeGraph](001)at(0,1){};
        \node[NodeGraph](002)at(0,2){};
        \node[NodeGraph](011)at(1,1){};
        \node[NodeGraph](012)at(1,2){};
        \node[NodeLabelGraph,above of=000]{$000$};
        \node[NodeLabelGraph,right of=001]{$001$};
        \node[NodeLabelGraph,right of=002]{$002$};
        \node[NodeLabelGraph,left of=011]{$011$};
        \node[NodeLabelGraph,below of=012]{$012$};
        \draw[EdgeGraph](000)--(001);
        \draw[EdgeGraph](001)--(002);
        \draw[EdgeGraph](001)--(011);
        \draw[EdgeGraph](002)--(012);
        \draw[EdgeGraph](011)--(012);
    \end{tikzpicture}}
    \label{subfig:hill_poset_1_3}}
    \qquad \qquad
    \subfloat[][$\SetHill_\MapTwo(3)$.]{
    \centering
    \scalebox{.7}{
    \begin{tikzpicture}[Centering,xscale=.8,yscale=.8,rotate=-135]
        \draw[Grid](0,0)grid(2,4);
        \node[NodeGraph](000)at(0,0){};
        \node[NodeGraph](001)at(0,1){};
        \node[NodeGraph](002)at(0,2){};
        \node[NodeGraph](003)at(0,3){};
        \node[NodeGraph](004)at(0,4){};
        \node[NodeGraph](011)at(1,1){};
        \node[NodeGraph](012)at(1,2){};
        \node[NodeGraph](013)at(1,3){};
        \node[NodeGraph](014)at(1,4){};
        \node[NodeGraph](022)at(2,2){};
        \node[NodeGraph](023)at(2,3){};
        \node[NodeGraph](024)at(2,4){};
        \node[NodeLabelGraph,above of=000]{$000$};
        \node[NodeLabelGraph,right of=001]{$001$};
        \node[NodeLabelGraph,right of=002]{$002$};
        \node[NodeLabelGraph,right of=003]{$003$};
        \node[NodeLabelGraph,right of=004]{$004$};
        \node[NodeLabelGraph,left of=011]{$011$};
        \node[NodeLabelGraph,below of=012]{$012$};
        \node[NodeLabelGraph,below of=013]{$013$};
        \node[NodeLabelGraph,right of=014]{$014$};
        \node[NodeLabelGraph,left of=022]{$022$};
        \node[NodeLabelGraph,left of=023]{$023$};
        \node[NodeLabelGraph,below of=024]{$024$};
        \draw[EdgeGraph](000)--(001);
        \draw[EdgeGraph](001)--(002);
        \draw[EdgeGraph](001)--(011);
        \draw[EdgeGraph](002)--(003);
        \draw[EdgeGraph](002)--(012);
        \draw[EdgeGraph](003)--(004);
        \draw[EdgeGraph](003)--(013);
        \draw[EdgeGraph](004)--(014);
        \draw[EdgeGraph](011)--(012);
        \draw[EdgeGraph](012)--(013);
        \draw[EdgeGraph](012)--(022);
        \draw[EdgeGraph](013)--(014);
        \draw[EdgeGraph](013)--(023);
        \draw[EdgeGraph](014)--(024);
        \draw[EdgeGraph](022)--(023);
        \draw[EdgeGraph](023)--(024);
    \end{tikzpicture}}
    \label{subfig:hill_poset_2_3}}
    \qquad \qquad
    \subfloat[][$\SetHill_\MapOne(4)$.]{
    \centering
    \scalebox{.7}{
    \begin{tikzpicture}[Centering,xscale=.7,yscale=.7,
        x={(0,-.5cm)}, y={(-1.0cm,-1.0cm)}, z={(1.0cm,-1.0cm)}]
        \DrawGridSpace{1}{2}{3}
        \node[NodeGraph](0000)at(0,0,0){};
        \node[NodeGraph](0001)at(0,0,1){};
        \node[NodeGraph](0002)at(0,0,2){};
        \node[NodeGraph](0003)at(0,0,3){};
        \node[NodeGraph](0011)at(0,1,1){};
        \node[NodeGraph](0012)at(0,1,2){};
        \node[NodeGraph](0013)at(0,1,3){};
        \node[NodeGraph](0022)at(0,2,2){};
        \node[NodeGraph](0023)at(0,2,3){};
        \node[NodeGraph](0111)at(1,1,1){};
        \node[NodeGraph](0112)at(1,1,2){};
        \node[NodeGraph](0113)at(1,1,3){};
        \node[NodeGraph](0122)at(1,2,2){};
        \node[NodeGraph](0123)at(1,2,3){};
        \node[NodeLabelGraph,above of=0000]{$0000$};
        \node[NodeLabelGraph,right of=0001]{$0001$};
        \node[NodeLabelGraph,right of=0002]{$0002$};
        \node[NodeLabelGraph,right of=0003]{$0003$};
        \node[NodeLabelGraph,above of=0011]{$0011$};
        \node[NodeLabelGraph,above of=0012]{$0012$};
        \node[NodeLabelGraph,above of=0013]{$0013$};
        \node[NodeLabelGraph,above of=0022]{$0022$};
        \node[NodeLabelGraph,above of=0023]{$0023$};
        \node[NodeLabelGraph,below of=0111]{$0111$};
        \node[NodeLabelGraph,below of=0112]{$0112$};
        \node[NodeLabelGraph,right of=0113]{$0113$};
        \node[NodeLabelGraph,below of=0122]{$0122$};
        \node[NodeLabelGraph,below of=0123]{$0123$};
        \draw[EdgeGraph](0000)--(0001);
        \draw[EdgeGraph](0001)--(0002);
        \draw[EdgeGraph](0001)--(0011);
        \draw[EdgeGraph](0002)--(0003);
        \draw[EdgeGraph](0002)--(0012);
        \draw[EdgeGraph](0003)--(0013);
        \draw[EdgeGraph](0011)--(0012);
        \draw[EdgeGraph](0011)--(0111);
        \draw[EdgeGraph](0012)--(0013);
        \draw[EdgeGraph](0012)--(0022);
        \draw[EdgeGraph](0012)--(0112);
        \draw[EdgeGraph](0013)--(0023);
        \draw[EdgeGraph](0013)--(0113);
        \draw[EdgeGraph](0022)--(0023);
        \draw[EdgeGraph](0022)--(0122);
        \draw[EdgeGraph](0023)--(0123);
        \draw[EdgeGraph](0111)--(0112);
        \draw[EdgeGraph](0112)--(0113);
        \draw[EdgeGraph](0112)--(0122);
        \draw[EdgeGraph](0113)--(0123);
        \draw[EdgeGraph](0122)--(0123);
    \end{tikzpicture}}
    \label{subfig:hill_poset_1_4}}
    \caption{\footnotesize Hasse diagrams of some $\delta$-hill posets.}
    \label{fig:examples_hill_posets}
\end{figure}
The $\MapOne$-hill posets are sometimes called Stanley lattices~\cite{Sta75,Knu04}. The
$\delta$-hill posets can be seen as generalizations of these structures.
\medbreak

\begin{Proposition} \label{prop:properties_hill_posets}
    For any weakly increasing range map $\delta$ and $n \geq 0$, the poset
    $\SetHill_\delta(n)$ is
    \begin{enumerate}[label={\it (\roman*)}]
        \item \label{item:properties_hill_posets_1}
        straight, where $u \in \SetHill_\delta(n)$ is covered by $v \in \SetHill_\delta(n)$
        if and only if there is an $i \in [n]$ such that $\IncreaseLetter_i(u) = v$;

        \item \label{item:properties_hill_posets_2}
        coated;

        \item \label{item:properties_hill_posets_3}
        nested;

        \item \label{item:properties_hill_posets_4}
        graded, where the rank of a hill is its weight;

        \item \label{item:properties_hill_posets_5}
        EL-shellable;

        \item \label{item:properties_hill_posets_6}
        a sublattice of $\SetCliff_\delta(n)$;

        \item \label{item:properties_hill_posets_7}
        constructible by interval doubling.
    \end{enumerate}
\end{Proposition}
\begin{proof}
    Points~\ref{item:properties_hill_posets_1}, \ref{item:properties_hill_posets_2},
    \ref{item:properties_hill_posets_3}, \ref{item:properties_hill_posets_4},
    and~\ref{item:properties_hill_posets_6} are immediate.
    Point~\ref{item:properties_hill_posets_5} follows
    from~\ref{item:properties_hill_posets_2} and
    Theorem~\ref{thm:subposets_el_shellability}. Point~\ref{item:properties_hill_posets_7}
    is a consequence of Theorem~\ref{thm:constructible_by_interval_doubling_subfamilly}
    since~\ref{item:properties_hill_posets_3} holds and, from
    Proposition~\ref{prop:properties_hill_objects}, of the fact that $\SetHill_\delta$ is
    closed by prefix.  Alternatively, \ref{item:properties_hill_posets_7} is implied
    by~\ref{item:properties_hill_posets_6} and the fact that any sublattice of a lattice
    constructible by interval doubling is constructible by interval doubling~\cite{Day79},
    which is indeed the case for~$\SetCliff_\delta(n)$.
\end{proof}
\medbreak

\begin{Proposition} \label{prop:wings_butterflies_hill}
    For any $m \geq 0$,
    \begin{enumerate}[label={\it (\roman*)}]
        \item \label{item:wings_butterflies_hill_1}
        the graded set $\InputWings\Par{\SetHill_\MapM}$ contains all the $\MapM$-cliffs
        $u$ satisfying
        \begin{math}
            u_1 < \dots < u_{|u|};
        \end{math}

        \item \label{item:wings_butterflies_hill_2}
        the graded set $\OutputWings\Par{\SetHill_\MapM}$ contains all the $\MapM$-cliffs
        $u$ satisfying
        \begin{math}
            u_1 \leq u_2 < \dots < u_{|u|}
        \end{math}
        and for all $i \in [2, |u|]$, $u_i < \MapM(i)$;

        \item \label{item:wings_butterflies_hill_3}
        the graded set $\Butterflies\Par{\SetHill_\MapM}$ contains all the $\MapM$-cliffs
        $u$ satisfying
        \begin{math}
            u_1 < \dots < u_{|u|}
        \end{math}
        and for all $i \in [2, |u|]$, $u_i < \MapM(i)$.
    \end{enumerate}
\end{Proposition}
\medbreak

The four posets of hills, input-wings, output-wings, and butterflies are linked in the
following way.
\medbreak

\begin{Theorem} \label{thm:poset_morphisms_hill}
    For any $m \geq 1$ and $n \geq 0$,
    \begin{equation}
        \begin{tikzpicture}[Centering,xscale=3.5,yscale=1.6,font=\small]
            \node(Hills)at(0,0){$\SetHill_{\mathbf m - 1}(n)$};
            \node(InputWingsHill)at(1,0){$\InputWings\Par{\SetHill_\MapM}(n)$};
            \node(OutputWingsHill)at(2,0){$\OutputWings\Par{\SetHill_\MapM}(n)$};
            \node(ButterfliesHill)at(1,-1)
                {$\Butterflies\Par{\SetHill_{\mathbf m + 1}}(n)$};
            \draw[MapIsomorphism](Hills)--(InputWingsHill)node[midway,above]
                {$\varphi_3$};
            \draw[MapEmbedding](InputWingsHill)--(ButterfliesHill)node[midway,right]
                {$\Id$};
            \draw[MapIsomorphism](InputWingsHill)
                --(OutputWingsHill)node[midway,above]
                {$\varphi_2$};
        \end{tikzpicture}
    \end{equation}
    is a diagram of poset embeddings or isomorphisms, where the map $\varphi_3$ is defined
    for any $u \in \N^n$ and $i \in [n]$ by
    \begin{math}
        \varphi_3(u)_i := u_i + i - 1,
    \end{math}
    $\Id$ is the identity map, and $\varphi_2$ is the map defined in the statement
    of Theorem~\ref{thm:poset_morphisms_avalanche}.
\end{Theorem}
\begin{proof}
    This follows from the descriptions of the input-wings, output-wings, and butterflies of
    $\SetHill_\MapM(n)$ provided by Proposition~\ref{prop:wings_butterflies_hill}.
\end{proof}
\medbreak

Figure~\ref{fig:poset_morphisms_hill} gives an example of the poset isomorphisms or
embeddings described by the statement of Theorem~\ref{thm:poset_morphisms_hill}.
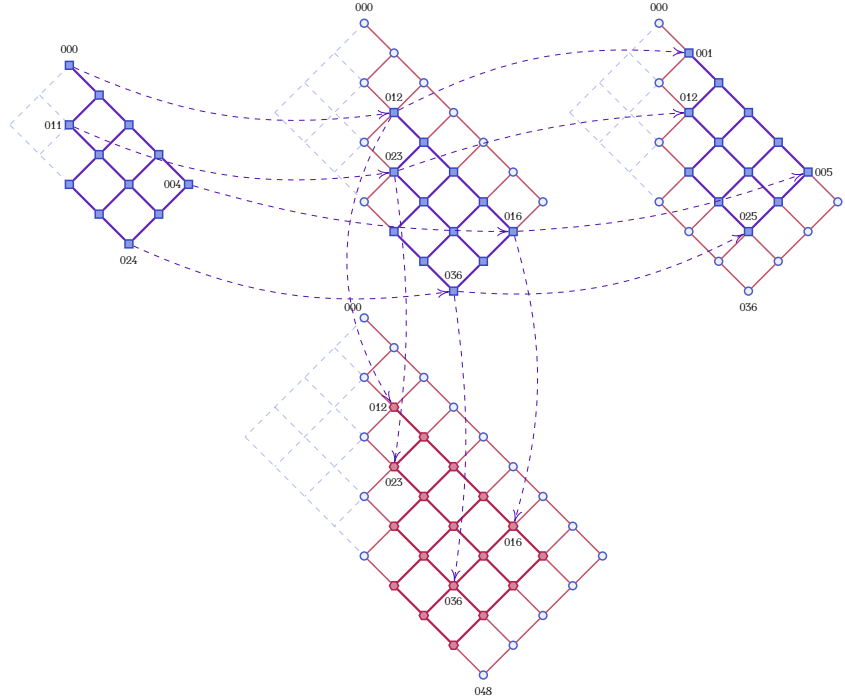
\begin{figure}[ht]
    \centering
    \scalebox{.7}{
    \begin{tikzpicture}[Centering,xscale=.8,yscale=.8]
        % Hill, m = 2, n = 3.
        \begin{scope}[xshift=0cm,yshift=-1cm,rotate=-135]
        \draw[Grid](0,0)grid(2,4);
        \node[MarkedNodeGraph](A000)at(0,0){};
        \node[MarkedNodeGraph](A001)at(0,1){};
        \node[MarkedNodeGraph](A002)at(0,2){};
        \node[MarkedNodeGraph](A003)at(0,3){};
        \node[MarkedNodeGraph](A004)at(0,4){};
        \node[MarkedNodeGraph](A011)at(1,1){};
        \node[MarkedNodeGraph](A012)at(1,2){};
        \node[MarkedNodeGraph](A013)at(1,3){};
        \node[MarkedNodeGraph](A014)at(1,4){};
        \node[MarkedNodeGraph](A022)at(2,2){};
        \node[MarkedNodeGraph](A023)at(2,3){};
        \node[MarkedNodeGraph](A024)at(2,4){};
        \node[NodeLabelGraph,above of=A000]{$000$};
        \node[NodeLabelGraph,left of=A004]{$004$};
        \node[NodeLabelGraph,left of=A011]{$011$};
        \node[NodeLabelGraph,below of=A024]{$024$};
        \draw[MarkedEdgeGraph](A000)--(A001);
        \draw[MarkedEdgeGraph](A001)--(A002);
        \draw[MarkedEdgeGraph](A001)--(A011);
        \draw[MarkedEdgeGraph](A002)--(A003);
        \draw[MarkedEdgeGraph](A002)--(A012);
        \draw[MarkedEdgeGraph](A003)--(A004);
        \draw[MarkedEdgeGraph](A003)--(A013);
        \draw[MarkedEdgeGraph](A004)--(A014);
        \draw[MarkedEdgeGraph](A011)--(A012);
        \draw[MarkedEdgeGraph](A012)--(A013);
        \draw[MarkedEdgeGraph](A012)--(A022);
        \draw[MarkedEdgeGraph](A013)--(A014);
        \draw[MarkedEdgeGraph](A013)--(A023);
        \draw[MarkedEdgeGraph](A014)--(A024);
        \draw[MarkedEdgeGraph](A022)--(A023);
        \draw[MarkedEdgeGraph](A023)--(A024);
        \end{scope}
        %
        % Input-wings hill, m = 3, n = 3.
        \begin{scope}[xshift=7cm,yshift=0cm,rotate=-135]
        \draw[Grid](0,0)grid(3,6);
        \node[NodeGraph](B000)at(0,0){};
        \node[NodeGraph](B001)at(0,1){};
        \node[NodeGraph](B002)at(0,2){};
        \node[NodeGraph](B003)at(0,3){};
        \node[NodeGraph](B004)at(0,4){};
        \node[NodeGraph](B005)at(0,5){};
        \node[NodeGraph](B006)at(0,6){};
        \node[NodeGraph](B011)at(1,1){};
        \node[MarkedNodeGraph](B012)at(1,2){};
        \node[MarkedNodeGraph](B013)at(1,3){};
        \node[MarkedNodeGraph](B014)at(1,4){};
        \node[MarkedNodeGraph](B015)at(1,5){};
        \node[MarkedNodeGraph](B016)at(1,6){};
        \node[NodeGraph](B022)at(2,2){};
        \node[MarkedNodeGraph](B023)at(2,3){};
        \node[MarkedNodeGraph](B024)at(2,4){};
        \node[MarkedNodeGraph](B025)at(2,5){};
        \node[MarkedNodeGraph](B026)at(2,6){};
        \node[NodeGraph](B033)at(3,3){};
        \node[MarkedNodeGraph](B034)at(3,4){};
        \node[MarkedNodeGraph](B035)at(3,5){};
        \node[MarkedNodeGraph](B036)at(3,6){};
        \node[NodeLabelGraph,above of=B000]{$000$};
        \node[NodeLabelGraph,above of=B012]{$012$};
        \node[NodeLabelGraph,above of=B016]{$016$};
        \node[NodeLabelGraph,above of=B023]{$023$};
        \node[NodeLabelGraph,above of=B036]{$036$};
        \draw[EdgeGraph](B000)--(B001);
        \draw[EdgeGraph](B001)--(B002);
        \draw[EdgeGraph](B001)--(B011);
        \draw[EdgeGraph](B002)--(B003);
        \draw[EdgeGraph](B002)--(B012);
        \draw[EdgeGraph](B003)--(B004);
        \draw[EdgeGraph](B003)--(B013);
        \draw[EdgeGraph](B004)--(B005);
        \draw[EdgeGraph](B004)--(B014);
        \draw[EdgeGraph](B005)--(B006);
        \draw[EdgeGraph](B005)--(B015);
        \draw[EdgeGraph](B006)--(B016);
        \draw[EdgeGraph](B011)--(B012);
        \draw[MarkedEdgeGraph](B012)--(B013);
        \draw[EdgeGraph](B012)--(B022);
        \draw[MarkedEdgeGraph](B013)--(B014);
        \draw[MarkedEdgeGraph](B013)--(B023);
        \draw[MarkedEdgeGraph](B014)--(B015);
        \draw[MarkedEdgeGraph](B014)--(B024);
        \draw[MarkedEdgeGraph](B015)--(B016);
        \draw[MarkedEdgeGraph](B015)--(B025);
        \draw[MarkedEdgeGraph](B016)--(B026);
        \draw[EdgeGraph](B022)--(B023);
        \draw[MarkedEdgeGraph](B023)--(B024);
        \draw[EdgeGraph](B023)--(B033);
        \draw[MarkedEdgeGraph](B024)--(B025);
        \draw[MarkedEdgeGraph](B024)--(B034);
        \draw[MarkedEdgeGraph](B025)--(B026);
        \draw[MarkedEdgeGraph](B025)--(B035);
        \draw[MarkedEdgeGraph](B026)--(B036);
        \draw[EdgeGraph](B033)--(B034);
        \draw[MarkedEdgeGraph](B034)--(B035);
        \draw[MarkedEdgeGraph](B035)--(B036);
        \end{scope}
        %
        % Output-wings hill, m = 3, n = 3.
        \begin{scope}[xshift=14cm,yshift=0cm,rotate=-135]
        \draw[Grid](0,0)grid(3,6);
        \node[NodeGraph](C000)at(0,0){};
        \node[MarkedNodeGraph](C001)at(0,1){};
        \node[MarkedNodeGraph](C002)at(0,2){};
        \node[MarkedNodeGraph](C003)at(0,3){};
        \node[MarkedNodeGraph](C004)at(0,4){};
        \node[MarkedNodeGraph](C005)at(0,5){};
        \node[NodeGraph](C006)at(0,6){};
        \node[NodeGraph](C011)at(1,1){};
        \node[MarkedNodeGraph](C012)at(1,2){};
        \node[MarkedNodeGraph](C013)at(1,3){};
        \node[MarkedNodeGraph](C014)at(1,4){};
        \node[MarkedNodeGraph](C015)at(1,5){};
        \node[NodeGraph](C016)at(1,6){};
        \node[NodeGraph](C022)at(2,2){};
        \node[MarkedNodeGraph](C023)at(2,3){};
        \node[MarkedNodeGraph](C024)at(2,4){};
        \node[MarkedNodeGraph](C025)at(2,5){};
        \node[NodeGraph](C026)at(2,6){};
        \node[NodeGraph](C033)at(3,3){};
        \node[NodeGraph](C034)at(3,4){};
        \node[NodeGraph](C035)at(3,5){};
        \node[NodeGraph](C036)at(3,6){};
        \node[NodeLabelGraph,above of=C000]{$000$};
        \node[NodeLabelGraph,right of=C001]{$001$};
        \node[NodeLabelGraph,right of=C005]{$005$};
        \node[NodeLabelGraph,above of=C012]{$012$};
        \node[NodeLabelGraph,above of=C025]{$025$};
        \node[NodeLabelGraph,below of=C036]{$036$};
        \draw[EdgeGraph](C000)--(C001);
        \draw[MarkedEdgeGraph](C001)--(C002);
        \draw[EdgeGraph](C001)--(C011);
        \draw[MarkedEdgeGraph](C002)--(C003);
        \draw[MarkedEdgeGraph](C002)--(C012);
        \draw[MarkedEdgeGraph](C003)--(C004);
        \draw[MarkedEdgeGraph](C003)--(C013);
        \draw[MarkedEdgeGraph](C004)--(C005);
        \draw[MarkedEdgeGraph](C004)--(C014);
        \draw[EdgeGraph](C005)--(C006);
        \draw[MarkedEdgeGraph](C005)--(C015);
        \draw[EdgeGraph](C006)--(C016);
        \draw[EdgeGraph](C011)--(C012);
        \draw[MarkedEdgeGraph](C012)--(C013);
        \draw[EdgeGraph](C012)--(C022);
        \draw[MarkedEdgeGraph](C013)--(C014);
        \draw[MarkedEdgeGraph](C013)--(C023);
        \draw[MarkedEdgeGraph](C014)--(C015);
        \draw[MarkedEdgeGraph](C014)--(C024);
        \draw[EdgeGraph](C015)--(C016);
        \draw[MarkedEdgeGraph](C015)--(C025);
        \draw[EdgeGraph](C016)--(C026);
        \draw[EdgeGraph](C022)--(C023);
        \draw[MarkedEdgeGraph](C023)--(C024);
        \draw[EdgeGraph](C023)--(C033);
        \draw[MarkedEdgeGraph](C024)--(C025);
        \draw[EdgeGraph](C024)--(C034);
        \draw[EdgeGraph](C025)--(C026);
        \draw[EdgeGraph](C025)--(C035);
        \draw[EdgeGraph](C026)--(C036);
        \draw[EdgeGraph](C033)--(C034);
        \draw[EdgeGraph](C034)--(C035);
        \draw[EdgeGraph](C035)--(C036);
        \end{scope}
        %
        % Butterflies hill, m = 4, n = 3.
        \begin{scope}[xshift=7cm,yshift=-7cm,rotate=-135]
        \draw[Grid](0,0)grid(4,8);
        \node[NodeGraph](D000)at(0,0){};
        \node[NodeGraph](D001)at(0,1){};
        \node[NodeGraph](D002)at(0,2){};
        \node[NodeGraph](D003)at(0,3){};
        \node[NodeGraph](D004)at(0,4){};
        \node[NodeGraph](D005)at(0,5){};
        \node[NodeGraph](D006)at(0,6){};
        \node[NodeGraph](D007)at(0,7){};
        \node[NodeGraph](D008)at(0,8){};
        \node[NodeGraph](D011)at(1,1){};
        \node[Marked2NodeGraph](D012)at(1,2){};
        \node[Marked2NodeGraph](D013)at(1,3){};
        \node[Marked2NodeGraph](D014)at(1,4){};
        \node[Marked2NodeGraph](D015)at(1,5){};
        \node[Marked2NodeGraph](D016)at(1,6){};
        \node[Marked2NodeGraph](D017)at(1,7){};
        \node[NodeGraph](D018)at(1,8){};
        \node[NodeGraph](D022)at(2,2){};
        \node[Marked2NodeGraph](D023)at(2,3){};
        \node[Marked2NodeGraph](D024)at(2,4){};
        \node[Marked2NodeGraph](D025)at(2,5){};
        \node[Marked2NodeGraph](D026)at(2,6){};
        \node[Marked2NodeGraph](D027)at(2,7){};
        \node[NodeGraph](D028)at(2,8){};
        \node[NodeGraph](D033)at(3,3){};
        \node[Marked2NodeGraph](D034)at(3,4){};
        \node[Marked2NodeGraph](D035)at(3,5){};
        \node[Marked2NodeGraph](D036)at(3,6){};
        \node[Marked2NodeGraph](D037)at(3,7){};
        \node[NodeGraph](D038)at(3,8){};
        \node[NodeGraph](D044)at(4,4){};
        \node[Marked2NodeGraph](D045)at(4,5){};
        \node[Marked2NodeGraph](D046)at(4,6){};
        \node[Marked2NodeGraph](D047)at(4,7){};
        \node[NodeGraph](D048)at(4,8){};
        \node[NodeLabelGraph,above left of=D000]{$000$};
        \node[NodeLabelGraph,left of=D012]{$012$};
        \node[NodeLabelGraph,below of=D016]{$016$};
        \node[NodeLabelGraph,below of=D023]{$023$};
        \node[NodeLabelGraph,below of=D036]{$036$};
        \node[NodeLabelGraph,below of=D048]{$048$};
        \draw[EdgeGraph](D000)--(D001);
        \draw[EdgeGraph](D001)--(D002);
        \draw[EdgeGraph](D001)--(D011);
        \draw[EdgeGraph](D002)--(D003);
        \draw[EdgeGraph](D002)--(D012);
        \draw[EdgeGraph](D003)--(D004);
        \draw[EdgeGraph](D003)--(D013);
        \draw[EdgeGraph](D004)--(D005);
        \draw[EdgeGraph](D004)--(D014);
        \draw[EdgeGraph](D005)--(D006);
        \draw[EdgeGraph](D005)--(D015);
        \draw[EdgeGraph](D006)--(D007);
        \draw[EdgeGraph](D006)--(D016);
        \draw[EdgeGraph](D007)--(D008);
        \draw[EdgeGraph](D007)--(D017);
        \draw[EdgeGraph](D008)--(D018);
        \draw[EdgeGraph](D011)--(D012);
        \draw[Marked2EdgeGraph](D012)--(D013);
        \draw[EdgeGraph](D012)--(D022);
        \draw[Marked2EdgeGraph](D013)--(D014);
        \draw[Marked2EdgeGraph](D013)--(D023);
        \draw[Marked2EdgeGraph](D014)--(D015);
        \draw[Marked2EdgeGraph](D014)--(D024);
        \draw[Marked2EdgeGraph](D015)--(D016);
        \draw[Marked2EdgeGraph](D015)--(D025);
        \draw[Marked2EdgeGraph](D016)--(D017);
        \draw[Marked2EdgeGraph](D016)--(D026);
        \draw[EdgeGraph](D017)--(D018);
        \draw[Marked2EdgeGraph](D017)--(D027);
        \draw[EdgeGraph](D018)--(D028);
        \draw[EdgeGraph](D022)--(D023);
        \draw[Marked2EdgeGraph](D023)--(D024);
        \draw[EdgeGraph](D023)--(D033);
        \draw[Marked2EdgeGraph](D024)--(D025);
        \draw[Marked2EdgeGraph](D024)--(D034);
        \draw[Marked2EdgeGraph](D025)--(D026);
        \draw[Marked2EdgeGraph](D025)--(D035);
        \draw[Marked2EdgeGraph](D026)--(D027);
        \draw[Marked2EdgeGraph](D026)--(D036);
        \draw[EdgeGraph](D027)--(D028);
        \draw[Marked2EdgeGraph](D027)--(D037);
        \draw[EdgeGraph](D028)--(D038);
        \draw[EdgeGraph](D033)--(D034);
        \draw[Marked2EdgeGraph](D034)--(D035);
        \draw[EdgeGraph](D034)--(D044);
        \draw[Marked2EdgeGraph](D035)--(D036);
        \draw[Marked2EdgeGraph](D035)--(D045);
        \draw[Marked2EdgeGraph](D036)--(D037);
        \draw[Marked2EdgeGraph](D036)--(D046);
        \draw[EdgeGraph](D037)--(D038);
        \draw[Marked2EdgeGraph](D037)--(D047);
        \draw[EdgeGraph](D038)--(D048);
        \draw[EdgeGraph](D044)--(D045);
        \draw[Marked2EdgeGraph](D045)--(D046);
        \draw[Marked2EdgeGraph](D046)--(D047);
        \draw[EdgeGraph](D047)--(D048);
        \end{scope}
        %
        % Arrows.
        \draw[MapGraph](A000)edge[bend right=16](B012);
        \draw[MapGraph](A011)edge[bend right=16](B023);
        \draw[MapGraph](A004)edge[bend right=8](B016);
        \draw[MapGraph](A024)edge[bend right=16](B036);
        \draw[MapGraph](B012)edge[bend left=16](C001);
        \draw[MapGraph](B023)edge[bend left=8](C012);
        \draw[MapGraph](B016)edge[bend right=8](C005);
        \draw[MapGraph](B036)edge[bend right=16](C025);
        \draw[MapGraph](B012)edge[bend right=32](D012);
        \draw[MapGraph](B023)edge[bend left=8](D023);
        \draw[MapGraph](B016)edge[bend left=16](D016);
        \draw[MapGraph](B036)edge[bend left=8](D036);
    \end{tikzpicture}}
    \caption{\footnotesize From the top to bottom and left to right, here are the posets
    $\SetHill_\MapTwo(3)$, $\SetHill_{\mathbf 3}(3)$, $\SetHill_{\mathbf 3}(3)$, and
    $\SetHill_{\mathbf 4}(3)$. All these posets contain $\SetHill_\MapTwo(3)$ as subposet by
    restricting on input-wings, output-wings, or butterflies.}
    \label{fig:poset_morphisms_hill}
\end{figure}
As a consequence of Theorem~\ref{thm:poset_morphisms_hill}, for any $m \geq 1$ and $n \geq
0$, the number of input-wings in $\SetHill_\MapM(n)$ is~$\FussCatalan_{m - 1}(n)$.
\medbreak

\begin{Proposition} \label{prop:enumeration_butterflies_hill}
    For any $m \geq 1$ and $n \geq 1$, $\# \Butterflies\Par{\SetHill_\MapM}(0) = 1$ and
    \begin{math}
        \# \Butterflies\Par{\SetHill_\MapM}(n) = \TwistedFussCatalan_{m - 1}(n).
    \end{math}
\end{Proposition}
\begin{proof}
    By Proposition~\ref{prop:wings_butterflies_hill}, the set
    $\Butterflies\Par{\SetHill_\MapM}(n)$ contains all $\MapM$-cliffs $u$ of size $n$
    satisfying $u_1 < \dots < u_n$ and for any $i \in [2, n]$, $u_i < \MapM(i)$. By setting
    $\MapM' := {\mathbf m - 1}$, this set is in one-to-one correspondence with the set of
    all $\MapM'$-cliffs $v$ of size $n$ satisfying $v_{i - 1} \leq v_i < \MapM'(i)$. A
    possible bijection between these two sets sends any $u \in
    \Butterflies\Par{\SetHill_\MapM}(n)$ to the $\MapM'$-cliff $v$ of the same size such
    that for any $i \in [n]$, $v_i = u_i - i + 1$. We have already seen in the proof of
    Proposition~\ref{prop:enumeration_output_wings_avalanche} that these sets are in
    one-to-one correspondence with $(m - 1)$-Dyck paths which cannot be written as a
    nontrivial concatenation of two $(m - 1)$-Dyck paths.  Therefore, the statement of the
    proposition follows.
\end{proof}
\medbreak

\begin{Proposition} \label{prop:max_avalanche_hill_bijection}
    For any $m \geq 0$ and $n \geq 1$, the map
    \begin{math}
        \rho : \max_{\Leq} \SetAvalanche_\MapM(n) \to \SetHill_\MapM(n - 1)
    \end{math}
    such that any $u \in \max_{\Leq} \SetAvalanche_\MapM(n)$, $\rho(u)$ is the prefix of
    size $n - 1$ of $\ElevationMap_{\SetHill_\MapM}^{-1}(u)$, is a bijection.
\end{Proposition}
\begin{proof}
    First, since $\SetHill_\MapM$ is by Proposition~\ref{prop:properties_hill_objects}
    closed by prefix, by Proposition~\ref{prop:elevation_map_injectivity},
    $\ElevationMap_{\SetHill_\MapM}$ is an injective map. This implies that the map $\rho$,
    defined by considering the inverse of $\ElevationMap_{\SetHill_\MapM}$ is a well-defined
    map.  Let $\rho' : \SetHill_\MapM(n - 1) \to \max_{\Leq} \SetAvalanche_\MapM(n)$ be the
    map defined for any $v \in \SetHill_\MapM(n - 1)$ by $\rho'(v) :=
    \ElevationMap_{\SetHill_\MapM}(v a)$ where $a := m (n - 1)$.  As pointed out before, $u
    \in \max_{\Leq} \SetAvalanche_\MapM(n)$ if and only if $\Weight(u) = m (n - 1)$. This
    implies that $\rho'(v)$ belongs to $\max_{\Leq} \SetAvalanche_\MapM(n)$. Moreover, due
    to the respective definitions of $\rho$ and $\rho'$, both $\rho \circ \rho'$ and $\rho'
    \circ \rho$ are identity maps. Therefore, $\rho$ is a bijection.
\end{proof}
\medbreak

\begin{Proposition} \label{prop:irreducibles_hill}
    For any $m \geq 1$ and $n \geq 1$, the set $\JoinIrreducibles\Par{\SetHill_\MapM(n)}$
    contains all $\MapM$-hills $u$ such that $u = 0^k \, a^{n - k}$ such that $k \in [n -
    1]$ and $a \in [km]$.
\end{Proposition}
\medbreak

\begin{Proposition} \label{prop:join_avalanche_join_hill_bijection}
    For any $m \geq 0$ and $n \geq 0$, the map $\ElevationMap_{\SetHill_\MapM}$ is a
    bijection between $\JoinIrreducibles\Par{\SetHill_\MapM(n)}$
    and~$\JoinIrreducibles\Par{\SetAvalanche_\MapM(n)}$.
\end{Proposition}
\begin{proof}
    This is a straightforward verification using the descriptions of join-irreducible
    elements of $\SetHill_\MapM(n)$ and $\SetAvalanche_\MapM(n)$ brought by
    Propositions~\ref{prop:irreducibles_hill} and~\ref{prop:irreducibles_avalanche}.
\end{proof}
\medbreak

By Proposition~\ref{prop:irreducibles_hill} (or also by
Propositions~\ref{prop:irreducibles_avalanche}
and~\ref{prop:join_avalanche_join_hill_bijection}), the number of join-irreducibles elements
of $\SetHill_\MapM(n)$ satisfies, for any $m \geq 1$ and $n \geq 1$,
\begin{math}
    \# \JoinIrreducibles\Par{\SetHill_\MapM(n)} = m \binom{n}{2}.
\end{math}
Since by Proposition~\ref{prop:properties_hill_posets}, $\SetHill_\MapM(n)$ is constructible
by interval doubling, this is also the number of its meet-irreducible elements~\cite{GW16}.
\medbreak

%%%%%%%%%%%%%%%%%%%%%%%%%%%%%%%%%%%%%%%%%%%%%%%%%%%%%%%%%%%%%%%%%%%%%%%%%%%%%%%%%%%%%%%%%%%%
\subsubsection{Cubic realization}
The map $\varphi_2$ introduced in Theorem~\ref{thm:poset_morphisms_hill} is used here to
describe the cells of maximal dimension of the cubic realization of $\SetHill_\MapM(n)$, $m
\geq 1$, $n \geq 0$.
\medbreak

\begin{Proposition} \label{prop:cells_hill}
    For any $m \geq 1$, $n \geq 0$, and $u \in \InputWings\Par{\SetHill_\MapM}(n)$,
    \begin{enumerate}[label={\it (\roman*)}]
        \item \label{item:cells_hill_1}
        the $\MapM$-hill $\varphi_2(u)$ is cell-compatible with the $\MapM$-hill $u$;

        \item \label{item:cells_hill_2}
        the cell $\Angle{\varphi_2(u), u}$ is pure;

        \item \label{item:cells_hill_3}
        all cells of
        \begin{math}
            \Bra{\Angle{\varphi_2(u), u} : u \in \InputWings\Par{\SetHill_\MapM}(n)}
        \end{math}
        are pairwise disjoint.
    \end{enumerate}
\end{Proposition}
\begin{proof}
    Due to the similarity between the maps $\varphi_2$ and the map $\varphi_1$ introduced in
    the statement of Theorem~\ref{thm:poset_morphisms_avalanche}, the proof here is very
    similar to the one of Proposition~\ref{prop:cells_avalanche}.
\end{proof}
\medbreak

As shown by Proposition~\ref{prop:cells_hill}, the cells of maximal dimension of the cubic
realization of $\SetHill_\MapM(n)$ are all of the form $\Angle{\varphi_2(u), u}$ where the
$u$ are input-wings of~$\SetHill_\MapM(n)$.
\medbreak

\begin{Proposition} \label{prop:volume_hill}
    For any $m \geq 1$ and $n \geq 0$,
    \begin{math}
        \Volume\Par{\CubicReal\Par{\SetHill_\MapM(n)}} = \FussCatalan_{m - 1}(n).
    \end{math}
\end{Proposition}
\begin{proof}
    Proposition~\ref{prop:cells_hill} describes all the cells of maximal dimension of
    $\CubicReal\Par{\SetHill_\MapM(n)}$ as cells $\Angle{\varphi_2(u)}, u$ where $u$ is an
    input-wing of $\SetHill_\MapM(n)$. Since all these cells are pairwise disjoint, the
    volume of $\CubicReal\Par{\SetHill_\MapM(n)}$ expresses
    as~\eqref{equ:volume_computation_from_output_wings}. Moreover, observe that the volume
    of each cell $\Angle{\varphi_2(u), u}$ where $u$ in an input-wing, is by definition of
    $\varphi_2$ equal to $1$. Therefore, $\Volume\Par{\CubicReal\Par{\SetHill_\MapM(n)}}$ is
    equal to the number of input-wings of $\SetHill_\MapM(n)$. The statement of the
    proposition follows now from Theorem~\ref{thm:poset_morphisms_hill}.
\end{proof}
\medbreak

%%%%%%%%%%%%%%%%%%%%%%%%%%%%%%%%%%%%%%%%%%%%%%%%%%%%%%%%%%%%%%%%%%%%%%%%%%%%%%%%%%%%%%%%%%%%
%%%%%%%%%%%%%%%%%%%%%%%%%%%%%%%%%%%%%%%%%%%%%%%%%%%%%%%%%%%%%%%%%%%%%%%%%%%%%%%%%%%%%%%%%%%%
\subsection{$\delta$-canyon posets}
We introduce here our last family of posets. They are defined on particular $\delta$-cliffs
called $\delta$-canyons. As we shall see, under some conditions these posets are lattices
but not sublattices of $\delta$-cliff lattices.
\medbreak

%%%%%%%%%%%%%%%%%%%%%%%%%%%%%%%%%%%%%%%%%%%%%%%%%%%%%%%%%%%%%%%%%%%%%%%%%%%%%%%%%%%%%%%%%%%%
\subsubsection{Objects} \label{subsubsec:canyon_objects}
For any range map $\delta$, let $\SetCanyon_\delta$ be the graded subset of
$\SetCliff_\delta$ containing all $\delta$-cliffs such that $u_{i - j} \leq u_i - j$, for
all $i \in [|u|]$ and $j \in \Han{u_i}$ satisfying $i - j \geq 1$. Any element of
$\SetCanyon_\delta$ is a \Def{$\delta$-canyon}. For instance
\begin{equation}
    \SetCanyon_\MapTwo(3) =
    \Bra{000, 010, 020, 001, 002, 012, 003, 013, 023, 004, 014, 024}.
\end{equation}
As a larger example, the $\MapTwo$-cliff $u := 020100459002301$ is a ${\mathbf
2}$-canyon. Indeed, by picturing an $\MapM$-canyon $u$ by drawing for each position $i \in
[|u|]$ a segment from the point $(i - 1, 0)$ to the point $\Par{i - 1, u_i}$ in the
Cartesian plane, the previous condition says that one can draw lines of slope $1$ passing
through the $x$-axis and the top of each segment without crossing any segment. For instance,
the previous $u$ is drawn as
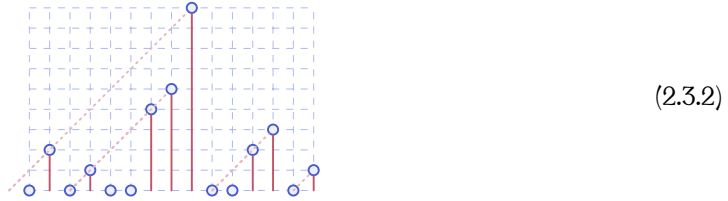
\begin{equation}
    \scalebox{.9}{
    \begin{tikzpicture}[scale=.3,Centering]
        \draw[Grid](1,0)grid(15,9);
        \node[PathNode](1)at(1,0){};
        \node[PathNode](2)at(2,2){};
        \node[PathNode](3)at(3,0){};
        \node[PathNode](4)at(4,1){};
        \node[PathNode](5)at(5,0){};
        \node[PathNode](6)at(6,0){};
        \node[PathNode](7)at(7,4){};
        \node[PathNode](8)at(8,5){};
        \node[PathNode](9)at(9,9){};
        \node[PathNode](10)at(10,0){};
        \node[PathNode](11)at(11,0){};
        \node[PathNode](12)at(12,2){};
        \node[PathNode](13)at(13,3){};
        \node[PathNode](14)at(14,0){};
        \node[PathNode](15)at(15,1){};
        \draw[PathStep](2,0)--(2);
        \draw[PathStep](4,0)--(4);
        \draw[PathStep](7,0)--(7);
        \draw[PathStep](8,0)--(8);
        \draw[PathStep](9,0)--(9);
        \draw[PathStep](12,0)--(12);
        \draw[PathStep](13,0)--(13);
        \draw[PathStep](15,0)--(15);
        \draw[PathDiag](0,0)--(2);
        \draw[PathDiag](3,0)--(4);
        \draw[PathDiag](3,0)--(7);
        \draw[PathDiag](3,0)--(8);
        \draw[PathDiag](0,0)--(9);
        \draw[PathDiag](10,0)--(12);
        \draw[PathDiag](10,0)--(13);
        \draw[PathDiag](14,0)--(15);
    \end{tikzpicture}}
\end{equation}
and one can observe that none of its diagonals, drawn as dotted lines, crosses a segment.
Besides, if $u$ is a $\delta$-cliff of size $n$ and $i, j \in [n]$ are two indices such that 
$i < j$, one has the three following possible configurations depending on the value
$\alpha := u_j - (j - i)$:
\begin{itemize}
    \item If $\alpha < 0$, then we say that $i$ and $j$ are \Def{independent} in $u$
    (graphically, the diagonal of $u_j$ falls under the $x$-axis before reaching the segment
    of $u_i$);
    \item If $\alpha \in \HanL{u_i - 1}$, then we say that $j$ is \Def{hindered} by $i$
    in $u$ (graphically, the diagonal of $u_j$ hits the segment of $u_i$);
    \item If $\alpha \geq u_i$, then we say that $j$ \Def{dominates} $i$ in $u$
    (graphically, the segment of $u_i$ is below or on the diagonal of $u_j$).
\end{itemize}
By definition, a $\delta$-cliff $u$ is a $\delta$-canyon if no index of $u$ is hindered by
another one.
\medbreak

\begin{Proposition} \label{prop:properties_canyon_objects}
    For any range map $\delta$, the graded set $\SetCanyon_\delta$ is
    \begin{enumerate}[label={\it (\roman*)}]
        \item \label{item:properties_canyon_objects_1}
        closed by prefix;

        \item \label{item:properties_canyon_objects_2}
        is minimally extendable;

        \item \label{item:properties_canyon_objects_3}
        is maximally extendable if $\delta$ is increasing.
    \end{enumerate}
\end{Proposition}
\begin{proof}
    Let $u$ be a $\delta$-canyon of size $n \geq 0$. Immediately from the definition of the
    $\delta$-canyons, it follows that $u \, 0$ is a $\delta$-canyon of size $n + 1$, and
    that for any prefix $u'$ of $u$, $u'$ is a $\delta$-canyon. Therefore,
    Points~\ref{item:properties_canyon_objects_1} and~\ref{item:properties_canyon_objects_2}
    check out. Let us now consider the $\delta$-cliff $u' := u \, \delta(n + 1)$. If
    $\delta$ is increasing, for all $j \in [n]$, $u_{n + 1 - j} \leq u_{n + 1} - j$.
    Therefore, $u'$ is a $\delta$-canyon.  Therefore, \ref{item:properties_canyon_objects_3}
    holds.
\end{proof}
\medbreak

Let us now introduce a series of definitions and lemmas in order to show that the sets
$\SetCanyon_\delta(n)$ and $\SetHill_\delta(n)$ are in one-to-one correspondence when
$\delta$ is an increasing range map.
\medbreak

For any $\delta$-canyon $u$ of size $n$, let $\DominantCanyon(u)$ be the $\delta$-canyon
obtained by changing for each index $i \in [n]$ the letter $u_i$ into $0$ if $i$ is
dominated by another index $j \in [i + 1, n]$. For instance, when $\delta = \MapM$ with $m =
2$, $\DominantCanyon(020050012) = 000050002$. Observe that $u \in \SetCanyon_\delta$ is an
exuviae (see Section~\ref{subsubsec:elevation_maps}) if and only if $\DominantCanyon(u) =
u$.
\medbreak

\begin{Lemma} \label{lem:next_exuviae_canyons}
    For any range map $\delta$ and any $\delta$-canyon $u$, $\Next_{\SetCanyon_\delta}(u) =
    \Next_{\SetCanyon_\delta}(\DominantCanyon(u))$.
\end{Lemma}
\begin{proof}
    Assume that $u$ is of size $n$ and set $w := \DominantCanyon(u)$. Assume that $u a$ is a
    $\delta$-canyon for a letter $a \in \N$. Then, the index $n + 1$ is hindered by no other
    index in $u a$. Since $w$ is obtained by changing to $0$ some letters of $u$, the index
    $n + 1$ remains hindered by no other index in $w a$. Therefore, $w a$ is also a
    $\delta$-canyon.  Conversely, assume that $w a$ is a $\delta$-canyon for a letter $a \in
    \N$. Then, the index $n + 1$ is hindered by no other index in $w a$. By contradiction,
    assume that $u a$ is not a $\delta$-canyon. This implies that the index $n + 1$ is
    hindered by an index $i$ in $u a$. Let us take $i$ maximal among all indices satisfying
    this property. Due to the maximality of $i$, $i$ is dominated by no other index in $u$
    so that we have $u_i = w_i$. This implies that $n + 1$ is hindered by $i$ in $w a$,
    which contradicts our hypothesis. Therefore, $u a$ is a $\delta$-canyon.
\end{proof}
\medbreak

\begin{Lemma} \label{lem:next_canyons}
    Let $\delta$ be a range map a $u$ be a $\delta$-canyon of size $n \geq 0$. Then,
    \begin{equation}
        \Next_{\SetCanyon_\delta}(u)
        = \HanL{\delta(n + 1)}
        \setminus \bigsqcup_{\substack{i \in [n] \\ \DominantCanyon(u)_i \ne 0}}
        \Han{n + 1 - i, n +\DominantCanyon(u)_i - i}.
    \end{equation}
\end{Lemma}
\begin{proof}
    Let $w$ be a $\delta$-canyon of size $n$ and let $w := \DominantCanyon(u)$. For any
    letter $a \in \HanL{\delta(n + 1)}$, the $\delta$-cliff $w a$ is a $\delta$-canyon if and
    only if the index $n + 1$ is hindered by no index in $w a$. Now, for any $i \in [n]$
    such that $w_i \ne 0$, the index $i$ hinds the index $n + 1$ in $w a$ if and only if $a
    \in \Han{n + 1 - i, n + w_i - i}$.  By definition of $\DominantCanyon$, all indices of
    $w$ are pairwise independent. Therefore, for any $i, i' \in [n]$ such that $i \ne i'$
    and $w_i \ne 0 \ne w_{i'}$, the sets $\Han{n + 1 - i, n + w_i - i}$ and $\Han{n + 1 -
    i', n + w_{i'} - i'}$ are disjoint.  Lemma~\ref{lem:next_exuviae_canyons} and the fact
    that $\DominantCanyon$ is an idempotent map imply the stated formula.
\end{proof}
\medbreak

\begin{Lemma} \label{lem:weight_exuviae_canyons}
    Let $\delta$ be a range map and $u$ be a $\delta$-canyon. Then,
    \begin{math}
        \Weight\Par{\ElevationMap_{\SetCanyon_\delta}(u)}
        = \Weight\Par{\DominantCanyon(u)}.
    \end{math}
\end{Lemma}
\begin{proof}
    This follows by induction on the size of $u$, by using the relation
    \begin{math}
        \DominantCanyon(u) = \ElevationMap_{\SetCanyon_\delta}(\DominantCanyon(u)),
    \end{math}
    and by using Lemma~\ref{lem:next_exuviae_canyons}.
\end{proof}
\medbreak

\begin{Proposition} \label{prop:image_elevation_map_canyon}
    For any increasing range map $\delta$ and any $n \geq 0$,
    \begin{math}
        \ElevationImage_{\SetCanyon_\delta}(n) = \SetAvalanche_\delta(n).
    \end{math}
\end{Proposition}
\begin{proof}
    First, since $\SetCanyon_\delta$ is by Proposition~\ref{prop:properties_canyon_objects}
    closed by prefix, the $\SetCanyon_\delta$-elevation map and so the
    $\SetCanyon_\delta$-elevation image are well-defined.
    \smallbreak

    By Lemmas~\ref{lem:next_canyons} and~\ref{lem:weight_exuviae_canyons}, and since
    $\delta$ is increasing, for any $\delta$-canyon $u$ of size $n \geq 0$, one has
    \begin{equation} \label{equ:image_elevation_map_canyon}
        \# \Next_{\SetCanyon_\delta}\Par{u} =
        1 + \delta(n + 1) - \Weight\Par{\ElevationMap_{\SetCanyon_\delta}\Par{u}}.
    \end{equation}
    \smallbreak

    Let us proceed by induction on $n$ to prove that for any $u \in \SetCanyon_\delta(n)$,
    $\ElevationMap_{\SetCanyon_\delta}(u)$ is a $\delta$-avalanche. If $n = 0$, the property
    holds immediately. Let $u = u' a$ be a $\delta$-canyon of size $n + 1$ where $u' \in
    \SetCanyon_\delta(n)$ and $a \in \N$. By induction hypothesis,
    $\ElevationMap_{\SetCanyon_\delta}\Par{u'}$ is a $\delta$-avalanche. Therefore, in
    particular, $\Weight\Par{\ElevationMap_{\SetCanyon_\delta}\Par{u'}} \leq \delta(n)$.
    Moreover, by~\eqref{equ:image_elevation_map_canyon}, we have
    \begin{equation} \begin{split}
        \Weight\Par{\ElevationMap_{\SetCanyon_\delta}\Par{u' a}}
        & =
        \Weight\Par{\ElevationMap_{\SetCanyon_\delta}\Par{u'}}
        + \# \Par{\Next_{\SetCanyon_\delta}\Par{u'} \cap \HanL{a - 1}} \\
        & \leq
        \Weight\Par{\ElevationMap_{\SetCanyon_\delta}\Par{u'}}
        + 1 + \delta(n + 1) - \Weight\Par{\ElevationMap_{\SetCanyon_\delta}\Par{u'}} - 1
        =
        \delta(n + 1),
    \end{split} \end{equation}
    showing that $u'a$ is a $\delta$-canyon.
    \smallbreak

    Conversely, let us prove by induction on $n$ that for any $v \in
    \SetAvalanche_\delta(n)$, there exists a $\delta$-canyon $u$ such that
    $\ElevationMap_{\SetCanyon_\delta}(u) = v$.  If $n = 0$, the property holds immediately.
    Let $v = v' b$ be a $\delta$-avalanche of size $n + 1$ where $v' \in
    \SetAvalanche_\delta(n)$ and $b \in \N$. By induction hypothesis, there is $u' \in
    \SetCanyon_\delta(n)$ such that $\ElevationMap_{\SetCanyon_\delta}\Par{u'} = v'$. Since
    $v$ is a $\delta$-avalanche, $b \leq \delta(n + 1) - \Weight\Par{v'}$.  Now,
    by~\eqref{equ:image_elevation_map_canyon}, since there are $1 + \delta(n + 1) -
    \Weight\Par{v'}$ different letters $a$ such that $u' a$ is a $\delta$-canyon, there is
    in particular a $\delta$-canyon $u = u' a$ such that
    $\ElevationMap_{\SetCanyon_\delta}(u) = v$.
\end{proof}
\medbreak

\begin{Proposition} \label{prop:bijection_canyon_hill}
    For any increasing range map $\delta$ and any $n \geq 0$, the map
    \begin{math}
        \ElevationMap_{\SetHill_\delta}^{-1} \circ \ElevationMap_{\SetCanyon_\delta}
    \end{math}
    is a bijection between $\SetCanyon_\delta(n)$ and $\SetHill_\delta(n)$.
\end{Proposition}
\begin{proof}
    First, since $\delta$ is increasing, by Propositions~\ref{prop:properties_hill_objects}
    and~\ref{prop:properties_canyon_objects}, both $\SetHill_\delta$ and $\SetCanyon_\delta$
    are closed by prefix. Therefore, the maps $\ElevationMap_{\SetHill_\delta}$ and
    $\ElevationMap_{\SetCanyon_\delta}$ are well-defined.  By
    Proposition~\ref{prop:elevation_map_injectivity}, the maps
    $\ElevationMap_{\SetCanyon_\MapM}$ and $\ElevationMap_{\SetHill_\MapM}$ are injective,
    and by Propositions~\ref{prop:image_elevation_map_hill}
    and~\ref{prop:image_elevation_map_canyon}, they both share the same image
    $\SetAvalanche_\MapM(n)$. This implies that $\ElevationMap_{\SetCanyon_\MapM}$ is a
    bijection from $\SetCanyon_\MapM$ to $\SetAvalanche_\MapM(n)$, and that
    $\ElevationMap_{\SetHill_\MapM}^{-1}$ is a well-defined map and is a bijection from
    $\SetAvalanche_\MapM(n)$ to $\SetHill_\MapM(n)$. Therefore, the statement of the
    proposition follows.
\end{proof}
\medbreak

As a consequence of Proposition~\ref{prop:bijection_canyon_hill}, for any $m \geq 0$,
$\MapM$-canyons are enumerated by $\MapM$-Fuss-Catalan numbers.
\medbreak

%%%%%%%%%%%%%%%%%%%%%%%%%%%%%%%%%%%%%%%%%%%%%%%%%%%%%%%%%%%%%%%%%%%%%%%%%%%%%%%%%%%%%%%%%%%%
\subsubsection{Posets}
For any $n \geq 0$, the subposet $\SetCanyon_\delta(n)$ is the \Def{$\delta$-canyon poset}
of order $n$. Figure~\ref{fig:examples_canyon_posets} shows the Hasse diagrams of some
$\MapM$-canyon posets.
\begin{figure}[ht]
    \centering
    \subfloat[][$\SetCanyon_\MapOne(3)$.]{
    \centering
    \scalebox{.7}{
    \begin{tikzpicture}[Centering,xscale=.8,yscale=.8,rotate=-135]
        \draw[Grid](0,0)grid(1,2);
        \node[NodeGraph](000)at(0,0){};
        \node[NodeGraph](001)at(0,1){};
        \node[NodeGraph](002)at(0,2){};
        \node[NodeGraph](010)at(1,0){};
        \node[NodeGraph](012)at(1,2){};
        \node[NodeLabelGraph,above of=000]{$000$};
        \node[NodeLabelGraph,right of=001]{$001$};
        \node[NodeLabelGraph,right of=002]{$002$};
        \node[NodeLabelGraph,left of=010]{$010$};
        \node[NodeLabelGraph,below of=012]{$012$};
        \draw[EdgeGraph](000)--(001);
        \draw[EdgeGraph](000)--(010);
        \draw[EdgeGraph](001)--(002);
        \draw[EdgeGraph](002)--(012);
        \draw[EdgeGraph](010)--(012);
    \end{tikzpicture}}
    \label{subfig:canyon_poset_1_3}}
    \qquad \quad
    \subfloat[][$\SetCanyon_\MapTwo(3)$.]{
    \centering
    \scalebox{.7}{
    \begin{tikzpicture}[Centering,xscale=.8,yscale=.8,rotate=-135]
        \draw[Grid](0,0)grid(2,4);
        \node[NodeGraph](000)at(0,0){};
        \node[NodeGraph](001)at(0,1){};
        \node[NodeGraph](002)at(0,2){};
        \node[NodeGraph](003)at(0,3){};
        \node[NodeGraph](004)at(0,4){};
        \node[NodeGraph](010)at(1,0){};
        \node[NodeGraph](012)at(1,2){};
        \node[NodeGraph](013)at(1,3){};
        \node[NodeGraph](014)at(1,4){};
        \node[NodeGraph](020)at(2,0){};
        \node[NodeGraph](023)at(2,3){};
        \node[NodeGraph](024)at(2,4){};
        \node[NodeLabelGraph,above of=000]{$000$};
        \node[NodeLabelGraph,right of=001]{$001$};
        \node[NodeLabelGraph,right of=002]{$002$};
        \node[NodeLabelGraph,right of=003]{$003$};
        \node[NodeLabelGraph,right of=004]{$004$};
        \node[NodeLabelGraph,left of=010]{$010$};
        \node[NodeLabelGraph,below of=012]{$012$};
        \node[NodeLabelGraph,below of=013]{$013$};
        \node[NodeLabelGraph,right of=014]{$014$};
        \node[NodeLabelGraph,left of=020]{$020$};
        \node[NodeLabelGraph,left of=023]{$023$};
        \node[NodeLabelGraph,below of=024]{$024$};
        \draw[EdgeGraph](000)--(001);
        \draw[EdgeGraph](000)--(010);
        \draw[EdgeGraph](001)--(002);
        \draw[EdgeGraph](002)--(003);
        \draw[EdgeGraph](002)--(012);
        \draw[EdgeGraph](003)--(004);
        \draw[EdgeGraph](003)--(013);
        \draw[EdgeGraph](004)--(014);
        \draw[EdgeGraph](010)--(012);
        \draw[EdgeGraph](010)--(020);
        \draw[EdgeGraph](012)--(013);
        \draw[EdgeGraph](013)--(014);
        \draw[EdgeGraph](013)--(023);
        \draw[EdgeGraph](014)--(024);
        \draw[EdgeGraph](020)--(023);
        \draw[EdgeGraph](023)--(024);
    \end{tikzpicture}}
    \label{subfig:canyon_poset_2_3}}
    \qquad \quad
    \subfloat[][$\SetCanyon_\MapOne(4)$.]{
    \centering
    \scalebox{.7}{
    \begin{tikzpicture}[Centering,xscale=.7,yscale=.7,
        x={(0,-.5cm)}, y={(-1.0cm,-1.0cm)}, z={(1.0cm,-1.0cm)}]
        \DrawGridSpace{1}{2}{3}
        \node[NodeGraph](0000)at(0,0,0){};
        \node[NodeGraph](0001)at(0,0,1){};
        \node[NodeGraph](0002)at(0,0,2){};
        \node[NodeGraph](0003)at(0,0,3){};
        \node[NodeGraph](0010)at(0,1,0){};
        \node[NodeGraph](0012)at(0,1,2){};
        \node[NodeGraph](0013)at(0,1,3){};
        \node[NodeGraph](0020)at(0,2,0){};
        \node[NodeGraph](0023)at(0,2,3){};
        \node[NodeGraph](0100)at(1,0,0){};
        \node[NodeGraph](0101)at(1,0,1){};
        \node[NodeGraph](0103)at(1,0,3){};
        \node[NodeGraph](0120)at(1,2,0){};
        \node[NodeGraph](0123)at(1,2,3){};
        \node[NodeLabelGraph,above of=0000]{$0000$};
        \node[NodeLabelGraph,right of=0001]{$0001$};
        \node[NodeLabelGraph,right of=0002]{$0002$};
        \node[NodeLabelGraph,right of=0003]{$0003$};
        \node[NodeLabelGraph,left of=0010]{$0010$};
        \node[NodeLabelGraph,above of=0012]{$0012$};
        \node[NodeLabelGraph,above of=0013]{$0013$};
        \node[NodeLabelGraph,left of=0020]{$0020$};
        \node[NodeLabelGraph,above of=0023]{$0023$};
        \node[NodeLabelGraph,below of=0100]{$0100$};
        \node[NodeLabelGraph,below of=0101]{$0101$};
        \node[NodeLabelGraph,right of=0103]{$0103$};
        \node[NodeLabelGraph,left of=0120]{$0120$};
        \node[NodeLabelGraph,below of=0123]{$0123$};
        \draw[EdgeGraph](0000)--(0001);
        \draw[EdgeGraph](0000)--(0010);
        \draw[EdgeGraph](0000)--(0100);
        \draw[EdgeGraph](0001)--(0002);
        \draw[EdgeGraph](0001)--(0101);
        \draw[EdgeGraph](0002)--(0003);
        \draw[EdgeGraph](0002)--(0012);
        \draw[EdgeGraph](0003)--(0013);
        \draw[EdgeGraph](0003)--(0103);
        \draw[EdgeGraph](0010)--(0012);
        \draw[EdgeGraph](0010)--(0020);
        \draw[EdgeGraph](0012)--(0013);
        \draw[EdgeGraph](0013)--(0023);
        \draw[EdgeGraph](0020)--(0023);
        \draw[EdgeGraph](0020)--(0120);
        \draw[EdgeGraph](0023)--(0123);
        \draw[EdgeGraph](0100)--(0101);
        \draw[EdgeGraph](0100)--(0120);
        \draw[EdgeGraph](0101)--(0103);
        \draw[EdgeGraph](0103)--(0123);
        \draw[EdgeGraph](0120)--(0123);
    \end{tikzpicture}}
    \label{subfig:canyon_poset_1_4}}
    \caption{\footnotesize Hasse diagrams of some $\delta$-canyon posets.}
    \label{fig:examples_canyon_posets}
\end{figure}
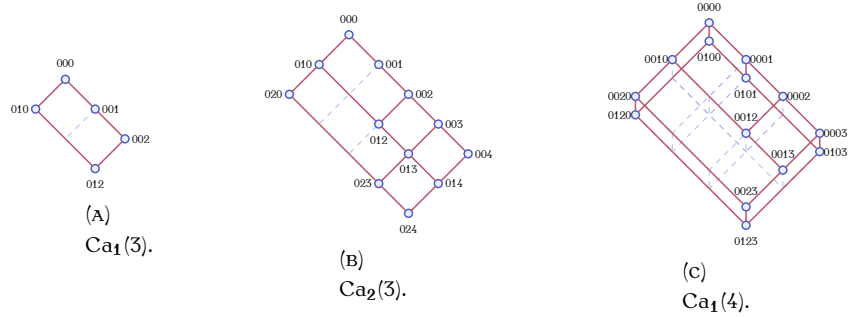
The $\MapOne$-canyons are also known as Tamari diagrams and have been introduced
in~\cite{Pal86}. The set of these objects of size $n$ is in one-to-one correspondence with
the set of binary trees with $n$ internal nodes. It is also known that the componentwise
comparison of Tamari diagrams is the Tamari order~\cite{Pal86}.  Recall that the Tamari
order admits, as covering relation, the right rotation operation in binary trees. It has
also the nice property to endow the set of binary trees of a given size with a lattice
structure~\cite{HT72}. Besides, a study of the posets of the intervals of the Tamari order,
based upon a generalization of Tamari diagrams, has been performed in~\cite{Com19}. The
Tamari posets admit a lot of generalizations, for instance through the so-called $m$-Tamari
posets~\cite{BPR12} where $m \geq 0$, and through the $\nu$-Tamari posets~\cite{PRV17} where
$\nu$ is a binary word.  Our $\delta$-canyon posets can be seen as different generalizations
of Tamari posets.  For any $m \geq 2$, the ${\mathbf m}$-canyon posets are not isomorphic to
the $m$-Tamari posets. Moreover, we shall prove in the sequel that for any increasing map
$\delta$, $\SetCanyon_\delta$ is a lattice.  As already mentioned, Tamari posets have the
nice property to be lattices~\cite{HT72}, are also EL-shellable~\cite{BW97}, and
constructible by interval doubling~\cite{Gey94b}. The same properties hold for $m$-Tamari
lattices, see respectively~\cite{BFP11} and~\cite{Muh15} for the first two ones. The last
one is a consequence of the fact that $m$-Tamari lattices are intervals of $1$-Tamari
lattices~\cite{BFP11} and the fact that the property to be constructible by interval
doubling is preserved for all sublattices of a lattice~\cite{Day79}.  As we shall see here,
the $\delta$-canyon posets have the same three properties.
\medbreak

\begin{Proposition} \label{prop:properties_canyon_posets}
    For any increasing range map $\delta$ and $n \geq 0$, the poset $\SetCanyon_\delta(n)$
    is
    \begin{enumerate}[label={\it (\roman*)}]
        \item \label{item:properties_canyon_posets_1}
        straight;

        \item \label{item:properties_canyon_posets_2}
        coated;

        \item \label{item:properties_canyon_posets_3}
        nested;

        \item \label{item:properties_canyon_posets_4}
        EL-shellable;

        \item \label{item:properties_canyon_posets_5}
        a meet semi-sublattice of $\SetCliff_\delta(n)$;

        \item \label{item:properties_canyon_posets_6}
        a lattice;

        \item \label{item:properties_canyon_posets_7}
        constructible by interval doubling.
    \end{enumerate}
\end{Proposition}
\begin{proof}
    Point~\ref{item:properties_canyon_posets_3} is immediate.  Assume that $u$ and $v$ are
    two $\delta$-canyons of size $n$ such that $u \Leq v$.  Let $k \in [n - 1]$ and consider
    the $\delta$-cliff $w := u_1 \dots u_k v_{k + 1} \dots v_n$. Now, since for any $i \in
    [k]$, $w_i = u_i \leq v_i$, and for any $i \in [k + 1, n]$, $w_i = v_i \geq u_i$, the
    fact that $u$ and $v$ are $\delta$-caynons implies that for any $i \in [n]$ and $j \in
    \Han{w_i}$ such that $i - j \geq 1$, the inequality $w_j \geq w_{i - j} + j$ holds.
    Thus, $w$ is an $\delta$-canyon, so that~\ref{item:properties_canyon_posets_2} holds.
    Now, by Lemma~\ref{lem:coated_implies_straight}, \ref{item:properties_canyon_posets_1}
    checks out, and by Theorem~\ref{thm:subposets_el_shellability},
    \ref{item:properties_canyon_posets_4} also.  Let $u$ and $v$ be two $\delta$-canyons of
    size $n$ and set $w$ as the $\delta$-cliff $u \Meet v$.  For all $j \in \Han{w_i}$ such
    that $i - j \geq 1$, $w_{i - j} \leq w_i - j$.  Indeed, either $w_{i - j} = u_{i - j}$
    or $w_{i- j} = v_{i - j}$, and in the two cases $w_{i - j} \leq (u \Meet v)_i - j$. For
    this reason, $w$ is a $\delta$-canyon. This shows~\ref{item:properties_canyon_posets_5}.
    Besides, due to the fact that by Proposition~\ref{prop:properties_canyon_objects},
    $\SetCanyon_\delta$ is closed by prefix and is maximally extendable,
    Theorem~\ref{thm:decr_incr_meet_join} implies~\ref{item:properties_canyon_posets_6}.
    Point~\ref{item:properties_canyon_posets_7} is a consequence of
    Theorem~\ref{thm:constructible_by_interval_doubling_subfamilly}
    since~\ref{item:properties_canyon_posets_3} holds and $\SetCanyon_\delta$ is closed by
    prefix.
\end{proof}
\medbreak

One can observe that $\SetCanyon_\MapM(n)$ is not a join semi-sublattice of the lattice of
$\delta$-cliffs. Indeed, by setting $u := 0124$ and $v := 0205$, even if $u$ and $v$ are
$\MapTwo$-canyons, $u \JJoin v = 0225$ is not. By
Proposition~\ref{prop:properties_canyon_posets}, the posets $\SetCanyon_\MapM(n)$ are
lattices and Theorem~\ref{thm:decr_incr_meet_join} provides a way to compute the join of two
of their elements. For instance, in $\SetCanyon_\MapOne$, one has
\begin{equation}
    00120 \JJoin_{\SetCanyon_\MapOne} 00201
    = \IncrMap_{\SetCanyon_\MapOne}(00120 \JJoin 00201)
    = \IncrMap_{\SetCanyon_\MapOne}(00221)
    = 00234,
\end{equation}
and, in $\SetCanyon_\MapTwo$, one has
\begin{equation}
    0124 \JJoin_{\SetCanyon_\MapTwo} 0205
    = \IncrMap_{\SetCanyon_\MapTwo}(0124 \JJoin 0205)
    = \IncrMap_{\SetCanyon_\MapTwo}(0225)
    = 0235.
\end{equation}
These computations of the join of two elements are similar to the ones described
in~\cite{Mar92} (see also~\cite{Gey94b}) for Tamari lattices.
\medbreak

Besides, as pointed out by Proposition~\ref{prop:properties_canyon_posets}, when $\delta$ is
an increasing range map, each $\SetCanyon_\delta(n)$ is constructible by interval doubling.
Figure~\ref{fig:contractions_canyon_2_4} shows a sequence of interval contractions performed
from $\SetCanyon_\MapTwo(4)$ in order to obtain~$\SetCanyon_\MapTwo(3)$.
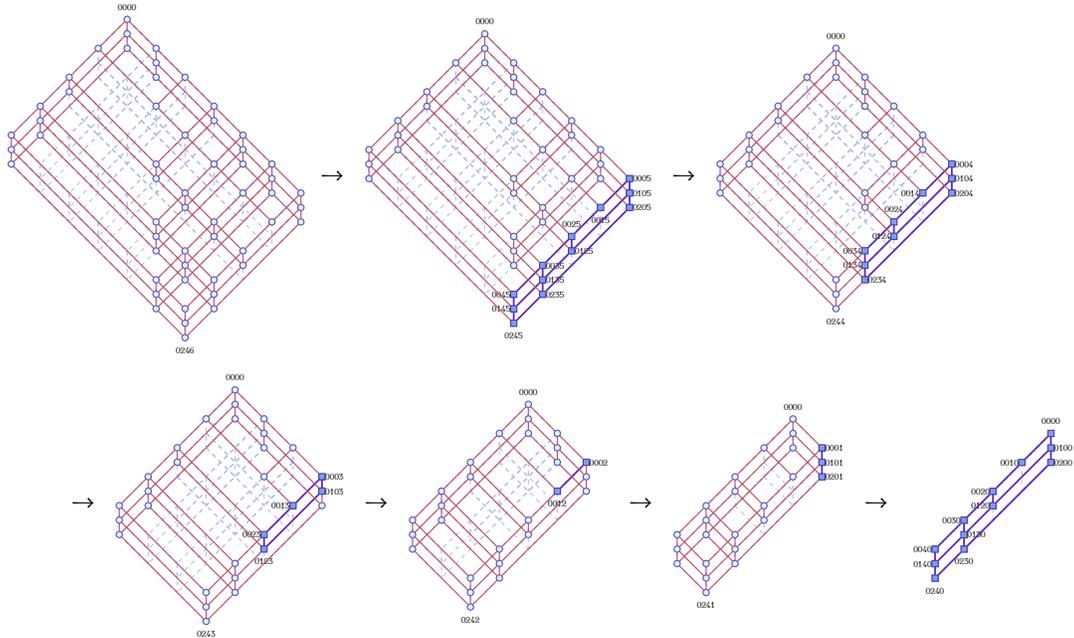
\begin{figure}[ht]
    \centering
    \begin{multline*}
        % Canyon, m = 2, n = 4
        \scalebox{.55}{
        \begin{tikzpicture}[Centering,xscale=.7,yscale=.7,
            x={(0,-.5cm)}, y={(-1.0cm,-1.0cm)}, z={(1.0cm,-1.0cm)}]
            \DrawGridSpace{2}{4}{6}
            \node[NodeGraph](0000)at(0,0,0){};
            \node[NodeGraph](0001)at(0,0,1){};
            \node[NodeGraph](0002)at(0,0,2){};
            \node[NodeGraph](0003)at(0,0,3){};
            \node[NodeGraph](0004)at(0,0,4){};
            \node[NodeGraph](0005)at(0,0,5){};
            \node[NodeGraph](0006)at(0,0,6){};
            \node[NodeGraph](0010)at(0,1,0){};
            \node[NodeGraph](0012)at(0,1,2){};
            \node[NodeGraph](0013)at(0,1,3){};
            \node[NodeGraph](0014)at(0,1,4){};
            \node[NodeGraph](0015)at(0,1,5){};
            \node[NodeGraph](0016)at(0,1,6){};
            \node[NodeGraph](0020)at(0,2,0){};
            \node[NodeGraph](0023)at(0,2,3){};
            \node[NodeGraph](0024)at(0,2,4){};
            \node[NodeGraph](0025)at(0,2,5){};
            \node[NodeGraph](0026)at(0,2,6){};
            \node[NodeGraph](0030)at(0,3,0){};
            \node[NodeGraph](0034)at(0,3,4){};
            \node[NodeGraph](0035)at(0,3,5){};
            \node[NodeGraph](0036)at(0,3,6){};
            \node[NodeGraph](0040)at(0,4,0){};
            \node[NodeGraph](0045)at(0,4,5){};
            \node[NodeGraph](0046)at(0,4,6){};
            \node[NodeGraph](0100)at(1,0,0){};
            \node[NodeGraph](0101)at(1,0,1){};
            \node[NodeGraph](0103)at(1,0,3){};
            \node[NodeGraph](0104)at(1,0,4){};
            \node[NodeGraph](0105)at(1,0,5){};
            \node[NodeGraph](0106)at(1,0,6){};
            \node[NodeGraph](0120)at(1,2,0){};
            \node[NodeGraph](0123)at(1,2,3){};
            \node[NodeGraph](0124)at(1,2,4){};
            \node[NodeGraph](0125)at(1,2,5){};
            \node[NodeGraph](0126)at(1,2,6){};
            \node[NodeGraph](0130)at(1,3,0){};
            \node[NodeGraph](0134)at(1,3,4){};
            \node[NodeGraph](0135)at(1,3,5){};
            \node[NodeGraph](0136)at(1,3,6){};
            \node[NodeGraph](0140)at(1,4,0){};
            \node[NodeGraph](0145)at(1,4,5){};
            \node[NodeGraph](0146)at(1,4,6){};
            \node[NodeGraph](0200)at(2,0,0){};
            \node[NodeGraph](0201)at(2,0,1){};
            \node[NodeGraph](0204)at(2,0,4){};
            \node[NodeGraph](0205)at(2,0,5){};
            \node[NodeGraph](0206)at(2,0,6){};
            \node[NodeGraph](0230)at(2,3,0){};
            \node[NodeGraph](0234)at(2,3,4){};
            \node[NodeGraph](0235)at(2,3,5){};
            \node[NodeGraph](0236)at(2,3,6){};
            \node[NodeGraph](0240)at(2,4,0){};
            \node[NodeGraph](0245)at(2,4,5){};
            \node[NodeGraph](0246)at(2,4,6){};
            \node[NodeLabelGraph,above of=0000]{$0000$};
            \node[NodeLabelGraph,below of=0246]{$0246$};
            \draw[EdgeGraph](0000)--(0001);
            \draw[EdgeGraph](0000)--(0010);
            \draw[EdgeGraph](0000)--(0100);
            \draw[EdgeGraph](0001)--(0002);
            \draw[EdgeGraph](0001)--(0101);
            \draw[EdgeGraph](0002)--(0003);
            \draw[EdgeGraph](0002)--(0012);
            \draw[EdgeGraph](0003)--(0004);
            \draw[EdgeGraph](0003)--(0013);
            \draw[EdgeGraph](0003)--(0103);
            \draw[EdgeGraph](0004)--(0005);
            \draw[EdgeGraph](0004)--(0014);
            \draw[EdgeGraph](0004)--(0104);
            \draw[EdgeGraph](0005)--(0006);
            \draw[EdgeGraph](0005)--(0015);
            \draw[EdgeGraph](0005)--(0105);
            \draw[EdgeGraph](0006)--(0016);
            \draw[EdgeGraph](0006)--(0106);
            \draw[EdgeGraph](0010)--(0012);
            \draw[EdgeGraph](0010)--(0020);
            \draw[EdgeGraph](0012)--(0013);
            \draw[EdgeGraph](0013)--(0014);
            \draw[EdgeGraph](0013)--(0023);
            \draw[EdgeGraph](0014)--(0015);
            \draw[EdgeGraph](0014)--(0024);
            \draw[EdgeGraph](0015)--(0016);
            \draw[EdgeGraph](0015)--(0025);
            \draw[EdgeGraph](0016)--(0026);
            \draw[EdgeGraph](0020)--(0023);
            \draw[EdgeGraph](0020)--(0030);
            \draw[EdgeGraph](0020)--(0120);
            \draw[EdgeGraph](0023)--(0024);
            \draw[EdgeGraph](0023)--(0123);
            \draw[EdgeGraph](0024)--(0025);
            \draw[EdgeGraph](0024)--(0034);
            \draw[EdgeGraph](0024)--(0124);
            \draw[EdgeGraph](0025)--(0026);
            \draw[EdgeGraph](0025)--(0035);
            \draw[EdgeGraph](0025)--(0125);
            \draw[EdgeGraph](0026)--(0036);
            \draw[EdgeGraph](0026)--(0126);
            \draw[EdgeGraph](0030)--(0034);
            \draw[EdgeGraph](0030)--(0040);
            \draw[EdgeGraph](0030)--(0130);
            \draw[EdgeGraph](0034)--(0035);
            \draw[EdgeGraph](0034)--(0134);
            \draw[EdgeGraph](0035)--(0036);
            \draw[EdgeGraph](0035)--(0045);
            \draw[EdgeGraph](0035)--(0135);
            \draw[EdgeGraph](0036)--(0046);
            \draw[EdgeGraph](0036)--(0136);
            \draw[EdgeGraph](0040)--(0045);
            \draw[EdgeGraph](0040)--(0140);
            \draw[EdgeGraph](0045)--(0046);
            \draw[EdgeGraph](0045)--(0145);
            \draw[EdgeGraph](0046)--(0146);
            \draw[EdgeGraph](0100)--(0101);
            \draw[EdgeGraph](0100)--(0120);
            \draw[EdgeGraph](0100)--(0200);
            \draw[EdgeGraph](0101)--(0103);
            \draw[EdgeGraph](0101)--(0201);
            \draw[EdgeGraph](0103)--(0104);
            \draw[EdgeGraph](0103)--(0123);
            \draw[EdgeGraph](0104)--(0105);
            \draw[EdgeGraph](0104)--(0124);
            \draw[EdgeGraph](0104)--(0204);
            \draw[EdgeGraph](0105)--(0106);
            \draw[EdgeGraph](0105)--(0125);
            \draw[EdgeGraph](0105)--(0205);
            \draw[EdgeGraph](0106)--(0126);
            \draw[EdgeGraph](0106)--(0206);
            \draw[EdgeGraph](0120)--(0123);
            \draw[EdgeGraph](0120)--(0130);
            \draw[EdgeGraph](0123)--(0124);
            \draw[EdgeGraph](0124)--(0125);
            \draw[EdgeGraph](0124)--(0134);
            \draw[EdgeGraph](0125)--(0126);
            \draw[EdgeGraph](0125)--(0135);
            \draw[EdgeGraph](0126)--(0136);
            \draw[EdgeGraph](0130)--(0134);
            \draw[EdgeGraph](0130)--(0140);
            \draw[EdgeGraph](0130)--(0230);
            \draw[EdgeGraph](0134)--(0135);
            \draw[EdgeGraph](0134)--(0234);
            \draw[EdgeGraph](0135)--(0136);
            \draw[EdgeGraph](0135)--(0145);
            \draw[EdgeGraph](0135)--(0235);
            \draw[EdgeGraph](0136)--(0146);
            \draw[EdgeGraph](0136)--(0236);
            \draw[EdgeGraph](0140)--(0145);
            \draw[EdgeGraph](0140)--(0240);
            \draw[EdgeGraph](0145)--(0146);
            \draw[EdgeGraph](0145)--(0245);
            \draw[EdgeGraph](0146)--(0246);
            \draw[EdgeGraph](0200)--(0201);
            \draw[EdgeGraph](0200)--(0230);
            \draw[EdgeGraph](0201)--(0204);
            \draw[EdgeGraph](0204)--(0205);
            \draw[EdgeGraph](0204)--(0234);
            \draw[EdgeGraph](0205)--(0206);
            \draw[EdgeGraph](0205)--(0235);
            \draw[EdgeGraph](0206)--(0236);
            \draw[EdgeGraph](0230)--(0234);
            \draw[EdgeGraph](0230)--(0240);
            \draw[EdgeGraph](0234)--(0235);
            \draw[EdgeGraph](0235)--(0236);
            \draw[EdgeGraph](0235)--(0245);
            \draw[EdgeGraph](0236)--(0246);
            \draw[EdgeGraph](0240)--(0245);
            \draw[EdgeGraph](0245)--(0246);
        \end{tikzpicture}}
        \; \to \;
        %
        % D(Canyon), m = 2, n = 4
        \scalebox{.55}{
        \begin{tikzpicture}[Centering,xscale=.7,yscale=.7,
            x={(0,-.5cm)}, y={(-1.0cm,-1.0cm)}, z={(1.0cm,-1.0cm)}]
            \DrawGridSpace{2}{4}{5}
            \node[NodeGraph](0000)at(0,0,0){};
            \node[NodeGraph](0001)at(0,0,1){};
            \node[NodeGraph](0002)at(0,0,2){};
            \node[NodeGraph](0003)at(0,0,3){};
            \node[NodeGraph](0004)at(0,0,4){};
            \node[MarkedNodeGraph](0005)at(0,0,5){};
            \node[NodeGraph](0010)at(0,1,0){};
            \node[NodeGraph](0012)at(0,1,2){};
            \node[NodeGraph](0013)at(0,1,3){};
            \node[NodeGraph](0014)at(0,1,4){};
            \node[MarkedNodeGraph](0015)at(0,1,5){};
            \node[NodeGraph](0020)at(0,2,0){};
            \node[NodeGraph](0023)at(0,2,3){};
            \node[NodeGraph](0024)at(0,2,4){};
            \node[MarkedNodeGraph](0025)at(0,2,5){};
            \node[NodeGraph](0030)at(0,3,0){};
            \node[NodeGraph](0034)at(0,3,4){};
            \node[MarkedNodeGraph](0035)at(0,3,5){};
            \node[NodeGraph](0040)at(0,4,0){};
            \node[MarkedNodeGraph](0045)at(0,4,5){};
            \node[NodeGraph](0100)at(1,0,0){};
            \node[NodeGraph](0101)at(1,0,1){};
            \node[NodeGraph](0103)at(1,0,3){};
            \node[NodeGraph](0104)at(1,0,4){};
            \node[MarkedNodeGraph](0105)at(1,0,5){};
            \node[NodeGraph](0120)at(1,2,0){};
            \node[NodeGraph](0123)at(1,2,3){};
            \node[NodeGraph](0124)at(1,2,4){};
            \node[MarkedNodeGraph](0125)at(1,2,5){};
            \node[NodeGraph](0130)at(1,3,0){};
            \node[NodeGraph](0134)at(1,3,4){};
            \node[MarkedNodeGraph](0135)at(1,3,5){};
            \node[NodeGraph](0140)at(1,4,0){};
            \node[MarkedNodeGraph](0145)at(1,4,5){};
            \node[NodeGraph](0200)at(2,0,0){};
            \node[NodeGraph](0201)at(2,0,1){};
            \node[NodeGraph](0204)at(2,0,4){};
            \node[MarkedNodeGraph](0205)at(2,0,5){};
            \node[NodeGraph](0230)at(2,3,0){};
            \node[NodeGraph](0234)at(2,3,4){};
            \node[MarkedNodeGraph](0235)at(2,3,5){};
            \node[NodeGraph](0240)at(2,4,0){};
            \node[MarkedNodeGraph](0245)at(2,4,5){};
            \node[NodeLabelGraph,above of=0000]{$0000$};
            \node[NodeLabelGraph,right of=0005]{$0005$};
            \node[NodeLabelGraph,below of=0015]{$0015$};
            \node[NodeLabelGraph,above of=0025]{$0025$};
            \node[NodeLabelGraph,right of=0035]{$0035$};
            \node[NodeLabelGraph,left of=0045]{$0045$};
            \node[NodeLabelGraph,right of=0105]{$0105$};
            \node[NodeLabelGraph,right of=0125]{$0125$};
            \node[NodeLabelGraph,right of=0135]{$0135$};
            \node[NodeLabelGraph,left of=0145]{$0145$};
            \node[NodeLabelGraph,right of=0205]{$0205$};
            \node[NodeLabelGraph,right of=0235]{$0235$};
            \node[NodeLabelGraph,below of=0245]{$0245$};
            \draw[EdgeGraph](0000)--(0001);
            \draw[EdgeGraph](0000)--(0010);
            \draw[EdgeGraph](0000)--(0100);
            \draw[EdgeGraph](0001)--(0002);
            \draw[EdgeGraph](0001)--(0101);
            \draw[EdgeGraph](0002)--(0003);
            \draw[EdgeGraph](0002)--(0012);
            \draw[EdgeGraph](0003)--(0004);
            \draw[EdgeGraph](0003)--(0013);
            \draw[EdgeGraph](0003)--(0103);
            \draw[EdgeGraph](0004)--(0005);
            \draw[EdgeGraph](0004)--(0014);
            \draw[EdgeGraph](0004)--(0104);
            \draw[MarkedEdgeGraph](0005)--(0015);
            \draw[MarkedEdgeGraph](0005)--(0105);
            \draw[EdgeGraph](0010)--(0012);
            \draw[EdgeGraph](0010)--(0020);
            \draw[EdgeGraph](0012)--(0013);
            \draw[EdgeGraph](0013)--(0014);
            \draw[EdgeGraph](0013)--(0023);
            \draw[EdgeGraph](0014)--(0015);
            \draw[EdgeGraph](0014)--(0024);
            \draw[MarkedEdgeGraph](0015)--(0025);
            \draw[EdgeGraph](0020)--(0023);
            \draw[EdgeGraph](0020)--(0030);
            \draw[EdgeGraph](0020)--(0120);
            \draw[EdgeGraph](0023)--(0024);
            \draw[EdgeGraph](0023)--(0123);
            \draw[EdgeGraph](0024)--(0025);
            \draw[EdgeGraph](0024)--(0034);
            \draw[EdgeGraph](0024)--(0124);
            \draw[MarkedEdgeGraph](0025)--(0035);
            \draw[MarkedEdgeGraph](0025)--(0125);
            \draw[EdgeGraph](0030)--(0034);
            \draw[EdgeGraph](0030)--(0040);
            \draw[EdgeGraph](0030)--(0130);
            \draw[EdgeGraph](0034)--(0035);
            \draw[EdgeGraph](0034)--(0134);
            \draw[MarkedEdgeGraph](0035)--(0045);
            \draw[MarkedEdgeGraph](0035)--(0135);
            \draw[EdgeGraph](0040)--(0045);
            \draw[EdgeGraph](0040)--(0140);
            \draw[MarkedEdgeGraph](0045)--(0145);
            \draw[EdgeGraph](0100)--(0101);
            \draw[EdgeGraph](0100)--(0120);
            \draw[EdgeGraph](0100)--(0200);
            \draw[EdgeGraph](0101)--(0103);
            \draw[EdgeGraph](0101)--(0201);
            \draw[EdgeGraph](0103)--(0104);
            \draw[EdgeGraph](0103)--(0123);
            \draw[EdgeGraph](0104)--(0105);
            \draw[EdgeGraph](0104)--(0124);
            \draw[EdgeGraph](0104)--(0204);
            \draw[MarkedEdgeGraph](0105)--(0125);
            \draw[MarkedEdgeGraph](0105)--(0205);
            \draw[EdgeGraph](0120)--(0123);
            \draw[EdgeGraph](0120)--(0130);
            \draw[EdgeGraph](0123)--(0124);
            \draw[EdgeGraph](0124)--(0125);
            \draw[EdgeGraph](0124)--(0134);
            \draw[MarkedEdgeGraph](0125)--(0135);
            \draw[EdgeGraph](0130)--(0134);
            \draw[EdgeGraph](0130)--(0140);
            \draw[EdgeGraph](0130)--(0230);
            \draw[EdgeGraph](0134)--(0135);
            \draw[EdgeGraph](0134)--(0234);
            \draw[MarkedEdgeGraph](0135)--(0145);
            \draw[MarkedEdgeGraph](0135)--(0235);
            \draw[EdgeGraph](0140)--(0145);
            \draw[EdgeGraph](0140)--(0240);
            \draw[MarkedEdgeGraph](0145)--(0245);
            \draw[EdgeGraph](0200)--(0201);
            \draw[EdgeGraph](0200)--(0230);
            \draw[EdgeGraph](0201)--(0204);
            \draw[EdgeGraph](0204)--(0205);
            \draw[EdgeGraph](0204)--(0234);
            \draw[MarkedEdgeGraph](0205)--(0235);
            \draw[EdgeGraph](0230)--(0234);
            \draw[EdgeGraph](0230)--(0240);
            \draw[EdgeGraph](0234)--(0235);
            \draw[MarkedEdgeGraph](0235)--(0245);
            \draw[EdgeGraph](0240)--(0245);
        \end{tikzpicture}}
        \; \to \;
        %
        % D^2(Canyon), m = 2, n = 4
        \scalebox{.55}{
        \begin{tikzpicture}[Centering,xscale=.7,yscale=.7,
            x={(0,-.5cm)}, y={(-1.0cm,-1.0cm)}, z={(1.0cm,-1.0cm)}]
            \DrawGridSpace{2}{4}{4}
            \node[NodeGraph](0000)at(0,0,0){};
            \node[NodeGraph](0001)at(0,0,1){};
            \node[NodeGraph](0002)at(0,0,2){};
            \node[NodeGraph](0003)at(0,0,3){};
            \node[MarkedNodeGraph](0004)at(0,0,4){};
            \node[NodeGraph](0010)at(0,1,0){};
            \node[NodeGraph](0012)at(0,1,2){};
            \node[NodeGraph](0013)at(0,1,3){};
            \node[MarkedNodeGraph](0014)at(0,1,4){};
            \node[NodeGraph](0020)at(0,2,0){};
            \node[NodeGraph](0023)at(0,2,3){};
            \node[MarkedNodeGraph](0024)at(0,2,4){};
            \node[NodeGraph](0030)at(0,3,0){};
            \node[MarkedNodeGraph](0034)at(0,3,4){};
            \node[NodeGraph](0040)at(0,4,0){};
            \node[NodeGraph](0044)at(0,4,4){};
            \node[NodeGraph](0100)at(1,0,0){};
            \node[NodeGraph](0101)at(1,0,1){};
            \node[NodeGraph](0103)at(1,0,3){};
            \node[MarkedNodeGraph](0104)at(1,0,4){};
            \node[NodeGraph](0120)at(1,2,0){};
            \node[NodeGraph](0123)at(1,2,3){};
            \node[MarkedNodeGraph](0124)at(1,2,4){};
            \node[NodeGraph](0130)at(1,3,0){};
            \node[MarkedNodeGraph](0134)at(1,3,4){};
            \node[NodeGraph](0140)at(1,4,0){};
            \node[NodeGraph](0144)at(1,4,4){};
            \node[NodeGraph](0200)at(2,0,0){};
            \node[NodeGraph](0201)at(2,0,1){};
            \node[MarkedNodeGraph](0204)at(2,0,4){};
            \node[NodeGraph](0230)at(2,3,0){};
            \node[MarkedNodeGraph](0234)at(2,3,4){};
            \node[NodeGraph](0240)at(2,4,0){};
            \node[NodeGraph](0244)at(2,4,4){};
            \node[NodeLabelGraph,above of=0000]{$0000$};
            \node[NodeLabelGraph,right of=0004]{$0004$};
            \node[NodeLabelGraph,left of=0014]{$0014$};
            \node[NodeLabelGraph,above of=0024]{$0024$};
            \node[NodeLabelGraph,left of=0034]{$0034$};
            \node[NodeLabelGraph,right of=0104]{$0104$};
            \node[NodeLabelGraph,left of=0124]{$0124$};
            \node[NodeLabelGraph,left of=0134]{$0134$};
            \node[NodeLabelGraph,right of=0204]{$0204$};
            \node[NodeLabelGraph,right of=0234]{$0234$};
            \node[NodeLabelGraph,below of=0244]{$0244$};
            \draw[EdgeGraph](0000)--(0001);
            \draw[EdgeGraph](0000)--(0010);
            \draw[EdgeGraph](0000)--(0100);
            \draw[EdgeGraph](0001)--(0002);
            \draw[EdgeGraph](0001)--(0101);
            \draw[EdgeGraph](0002)--(0003);
            \draw[EdgeGraph](0002)--(0012);
            \draw[EdgeGraph](0003)--(0004);
            \draw[EdgeGraph](0003)--(0013);
            \draw[EdgeGraph](0003)--(0103);
            \draw[MarkedEdgeGraph](0004)--(0014);
            \draw[MarkedEdgeGraph](0004)--(0104);
            \draw[EdgeGraph](0010)--(0012);
            \draw[EdgeGraph](0010)--(0020);
            \draw[EdgeGraph](0012)--(0013);
            \draw[EdgeGraph](0013)--(0014);
            \draw[EdgeGraph](0013)--(0023);
            \draw[MarkedEdgeGraph](0014)--(0024);
            \draw[EdgeGraph](0020)--(0023);
            \draw[EdgeGraph](0020)--(0030);
            \draw[EdgeGraph](0020)--(0120);
            \draw[EdgeGraph](0023)--(0024);
            \draw[EdgeGraph](0023)--(0123);
            \draw[MarkedEdgeGraph](0024)--(0034);
            \draw[MarkedEdgeGraph](0024)--(0124);
            \draw[EdgeGraph](0030)--(0034);
            \draw[EdgeGraph](0030)--(0040);
            \draw[EdgeGraph](0030)--(0130);
            \draw[EdgeGraph](0034)--(0044);
            \draw[MarkedEdgeGraph](0034)--(0134);
            \draw[EdgeGraph](0040)--(0044);
            \draw[EdgeGraph](0040)--(0140);
            \draw[EdgeGraph](0044)--(0144);
            \draw[EdgeGraph](0100)--(0101);
            \draw[EdgeGraph](0100)--(0120);
            \draw[EdgeGraph](0100)--(0200);
            \draw[EdgeGraph](0101)--(0103);
            \draw[EdgeGraph](0101)--(0201);
            \draw[EdgeGraph](0103)--(0104);
            \draw[EdgeGraph](0103)--(0123);
            \draw[MarkedEdgeGraph](0104)--(0124);
            \draw[MarkedEdgeGraph](0104)--(0204);
            \draw[EdgeGraph](0120)--(0123);
            \draw[EdgeGraph](0120)--(0130);
            \draw[EdgeGraph](0123)--(0124);
            \draw[MarkedEdgeGraph](0124)--(0134);
            \draw[EdgeGraph](0130)--(0134);
            \draw[EdgeGraph](0130)--(0140);
            \draw[EdgeGraph](0130)--(0230);
            \draw[EdgeGraph](0134)--(0144);
            \draw[MarkedEdgeGraph](0134)--(0234);
            \draw[EdgeGraph](0140)--(0144);
            \draw[EdgeGraph](0140)--(0240);
            \draw[EdgeGraph](0144)--(0244);
            \draw[EdgeGraph](0200)--(0201);
            \draw[EdgeGraph](0200)--(0230);
            \draw[EdgeGraph](0201)--(0204);
            \draw[MarkedEdgeGraph](0204)--(0234);
            \draw[EdgeGraph](0230)--(0234);
            \draw[EdgeGraph](0230)--(0240);
            \draw[EdgeGraph](0234)--(0244);
            \draw[EdgeGraph](0240)--(0244);
        \end{tikzpicture}}
        \\
        \; \to \;
        %
        % D^3(Canyon), m = 2, n = 4
        \scalebox{.55}{
        \begin{tikzpicture}[Centering,xscale=.7,yscale=.7,
            x={(0,-.5cm)}, y={(-1.0cm,-1.0cm)}, z={(1.0cm,-1.0cm)}]
            \DrawGridSpace{2}{4}{3}
            \node[NodeGraph](0000)at(0,0,0){};
            \node[NodeGraph](0001)at(0,0,1){};
            \node[NodeGraph](0002)at(0,0,2){};
            \node[MarkedNodeGraph](0003)at(0,0,3){};
            \node[NodeGraph](0010)at(0,1,0){};
            \node[NodeGraph](0012)at(0,1,2){};
            \node[MarkedNodeGraph](0013)at(0,1,3){};
            \node[NodeGraph](0020)at(0,2,0){};
            \node[MarkedNodeGraph](0023)at(0,2,3){};
            \node[NodeGraph](0030)at(0,3,0){};
            \node[NodeGraph](0033)at(0,3,3){};
            \node[NodeGraph](0040)at(0,4,0){};
            \node[NodeGraph](0043)at(0,4,3){};
            \node[NodeGraph](0100)at(1,0,0){};
            \node[NodeGraph](0101)at(1,0,1){};
            \node[MarkedNodeGraph](0103)at(1,0,3){};
            \node[NodeGraph](0120)at(1,2,0){};
            \node[MarkedNodeGraph](0123)at(1,2,3){};
            \node[NodeGraph](0130)at(1,3,0){};
            \node[NodeGraph](0133)at(1,3,3){};
            \node[NodeGraph](0140)at(1,4,0){};
            \node[NodeGraph](0143)at(1,4,3){};
            \node[NodeGraph](0200)at(2,0,0){};
            \node[NodeGraph](0201)at(2,0,1){};
            \node[NodeGraph](0203)at(2,0,3){};
            \node[NodeGraph](0230)at(2,3,0){};
            \node[NodeGraph](0233)at(2,3,3){};
            \node[NodeGraph](0240)at(2,4,0){};
            \node[NodeGraph](0243)at(2,4,3){};
            \node[NodeLabelGraph,above of=0000]{$0000$};
            \node[NodeLabelGraph,right of=0003]{$0003$};
            \node[NodeLabelGraph,left of=0013]{$0013$};
            \node[NodeLabelGraph,left of=0023]{$0023$};
            \node[NodeLabelGraph,right of=0103]{$0103$};
            \node[NodeLabelGraph,below of=0123]{$0123$};
            \node[NodeLabelGraph,below of=0243]{$0243$};
            \draw[EdgeGraph](0000)--(0001);
            \draw[EdgeGraph](0000)--(0010);
            \draw[EdgeGraph](0000)--(0100);
            \draw[EdgeGraph](0001)--(0002);
            \draw[EdgeGraph](0001)--(0101);
            \draw[EdgeGraph](0002)--(0003);
            \draw[EdgeGraph](0002)--(0012);
            \draw[MarkedEdgeGraph](0003)--(0013);
            \draw[MarkedEdgeGraph](0003)--(0103);
            \draw[EdgeGraph](0010)--(0012);
            \draw[EdgeGraph](0010)--(0020);
            \draw[EdgeGraph](0012)--(0013);
            \draw[MarkedEdgeGraph](0013)--(0023);
            \draw[EdgeGraph](0020)--(0023);
            \draw[EdgeGraph](0020)--(0030);
            \draw[EdgeGraph](0020)--(0120);
            \draw[EdgeGraph](0023)--(0033);
            \draw[MarkedEdgeGraph](0023)--(0123);
            \draw[EdgeGraph](0030)--(0033);
            \draw[EdgeGraph](0030)--(0040);
            \draw[EdgeGraph](0030)--(0130);
            \draw[EdgeGraph](0033)--(0043);
            \draw[EdgeGraph](0033)--(0133);
            \draw[EdgeGraph](0040)--(0043);
            \draw[EdgeGraph](0040)--(0140);
            \draw[EdgeGraph](0043)--(0143);
            \draw[EdgeGraph](0100)--(0101);
            \draw[EdgeGraph](0100)--(0120);
            \draw[EdgeGraph](0100)--(0200);
            \draw[EdgeGraph](0101)--(0103);
            \draw[EdgeGraph](0101)--(0201);
            \draw[MarkedEdgeGraph](0103)--(0123);
            \draw[EdgeGraph](0103)--(0203);
            \draw[EdgeGraph](0120)--(0123);
            \draw[EdgeGraph](0120)--(0130);
            \draw[EdgeGraph](0123)--(0133);
            \draw[EdgeGraph](0130)--(0133);
            \draw[EdgeGraph](0130)--(0140);
            \draw[EdgeGraph](0130)--(0230);
            \draw[EdgeGraph](0133)--(0143);
            \draw[EdgeGraph](0133)--(0233);
            \draw[EdgeGraph](0140)--(0143);
            \draw[EdgeGraph](0140)--(0240);
            \draw[EdgeGraph](0143)--(0243);
            \draw[EdgeGraph](0200)--(0201);
            \draw[EdgeGraph](0200)--(0230);
            \draw[EdgeGraph](0201)--(0203);
            \draw[EdgeGraph](0203)--(0233);
            \draw[EdgeGraph](0230)--(0233);
            \draw[EdgeGraph](0230)--(0240);
            \draw[EdgeGraph](0233)--(0243);
            \draw[EdgeGraph](0240)--(0243);
        \end{tikzpicture}}
        \; \to \;
        %
        % D^4(Canyon), m = 2, n = 4
        \scalebox{.55}{
        \begin{tikzpicture}[Centering,xscale=.7,yscale=.7,
            x={(0,-.5cm)}, y={(-1.0cm,-1.0cm)}, z={(1.0cm,-1.0cm)}]
            \DrawGridSpace{2}{4}{2}
            \node[NodeGraph](0000)at(0,0,0){};
            \node[NodeGraph](0001)at(0,0,1){};
            \node[MarkedNodeGraph](0002)at(0,0,2){};
            \node[NodeGraph](0010)at(0,1,0){};
            \node[MarkedNodeGraph](0012)at(0,1,2){};
            \node[NodeGraph](0020)at(0,2,0){};
            \node[NodeGraph](0022)at(0,2,2){};
            \node[NodeGraph](0030)at(0,3,0){};
            \node[NodeGraph](0032)at(0,3,2){};
            \node[NodeGraph](0040)at(0,4,0){};
            \node[NodeGraph](0042)at(0,4,2){};
            \node[NodeGraph](0100)at(1,0,0){};
            \node[NodeGraph](0101)at(1,0,1){};
            \node[NodeGraph](0102)at(1,0,2){};
            \node[NodeGraph](0120)at(1,2,0){};
            \node[NodeGraph](0122)at(1,2,2){};
            \node[NodeGraph](0130)at(1,3,0){};
            \node[NodeGraph](0132)at(1,3,2){};
            \node[NodeGraph](0140)at(1,4,0){};
            \node[NodeGraph](0142)at(1,4,2){};
            \node[NodeGraph](0200)at(2,0,0){};
            \node[NodeGraph](0201)at(2,0,1){};
            \node[NodeGraph](0202)at(2,0,2){};
            \node[NodeGraph](0230)at(2,3,0){};
            \node[NodeGraph](0232)at(2,3,2){};
            \node[NodeGraph](0240)at(2,4,0){};
            \node[NodeGraph](0242)at(2,4,2){};
            \node[NodeLabelGraph,above of=0000]{$0000$};
            \node[NodeLabelGraph,right of=0002]{$0002$};
            \node[NodeLabelGraph,below of=0012]{$0012$};
            \node[NodeLabelGraph,below of=0242]{$0242$};
            \draw[EdgeGraph](0000)--(0001);
            \draw[EdgeGraph](0000)--(0010);
            \draw[EdgeGraph](0000)--(0100);
            \draw[EdgeGraph](0001)--(0002);
            \draw[EdgeGraph](0001)--(0101);
            \draw[MarkedEdgeGraph](0002)--(0012);
            \draw[EdgeGraph](0002)--(0102);
            \draw[EdgeGraph](0010)--(0012);
            \draw[EdgeGraph](0010)--(0020);
            \draw[EdgeGraph](0012)--(0022);
            \draw[EdgeGraph](0020)--(0022);
            \draw[EdgeGraph](0020)--(0030);
            \draw[EdgeGraph](0020)--(0120);
            \draw[EdgeGraph](0022)--(0032);
            \draw[EdgeGraph](0022)--(0122);
            \draw[EdgeGraph](0030)--(0032);
            \draw[EdgeGraph](0030)--(0040);
            \draw[EdgeGraph](0030)--(0130);
            \draw[EdgeGraph](0032)--(0042);
            \draw[EdgeGraph](0032)--(0132);
            \draw[EdgeGraph](0040)--(0042);
            \draw[EdgeGraph](0040)--(0140);
            \draw[EdgeGraph](0042)--(0142);
            \draw[EdgeGraph](0100)--(0101);
            \draw[EdgeGraph](0100)--(0120);
            \draw[EdgeGraph](0100)--(0200);
            \draw[EdgeGraph](0101)--(0102);
            \draw[EdgeGraph](0101)--(0201);
            \draw[EdgeGraph](0102)--(0122);
            \draw[EdgeGraph](0102)--(0202);
            \draw[EdgeGraph](0120)--(0122);
            \draw[EdgeGraph](0120)--(0130);
            \draw[EdgeGraph](0122)--(0132);
            \draw[EdgeGraph](0130)--(0132);
            \draw[EdgeGraph](0130)--(0140);
            \draw[EdgeGraph](0130)--(0230);
            \draw[EdgeGraph](0132)--(0142);
            \draw[EdgeGraph](0132)--(0232);
            \draw[EdgeGraph](0140)--(0142);
            \draw[EdgeGraph](0140)--(0240);
            \draw[EdgeGraph](0142)--(0242);
            \draw[EdgeGraph](0200)--(0201);
            \draw[EdgeGraph](0200)--(0230);
            \draw[EdgeGraph](0201)--(0202);
            \draw[EdgeGraph](0202)--(0232);
            \draw[EdgeGraph](0230)--(0232);
            \draw[EdgeGraph](0230)--(0240);
            \draw[EdgeGraph](0232)--(0242);
            \draw[EdgeGraph](0240)--(0242);
        \end{tikzpicture}}
        \; \to \;
        %
        % D^5(Canyon), m = 2, n = 4
        \scalebox{.55}{
        \begin{tikzpicture}[Centering,xscale=.7,yscale=.7,
            x={(0,-.5cm)}, y={(-1.0cm,-1.0cm)}, z={(1.0cm,-1.0cm)}]
            \DrawGridSpace{2}{4}{1}
            \node[NodeGraph](0000)at(0,0,0){};
            \node[MarkedNodeGraph](0001)at(0,0,1){};
            \node[NodeGraph](0010)at(0,1,0){};
            \node[NodeGraph](0011)at(0,1,1){};
            \node[NodeGraph](0020)at(0,2,0){};
            \node[NodeGraph](0021)at(0,2,1){};
            \node[NodeGraph](0030)at(0,3,0){};
            \node[NodeGraph](0031)at(0,3,1){};
            \node[NodeGraph](0040)at(0,4,0){};
            \node[NodeGraph](0041)at(0,4,1){};
            \node[NodeGraph](0100)at(1,0,0){};
            \node[MarkedNodeGraph](0101)at(1,0,1){};
            \node[NodeGraph](0120)at(1,2,0){};
            \node[NodeGraph](0121)at(1,2,1){};
            \node[NodeGraph](0130)at(1,3,0){};
            \node[NodeGraph](0131)at(1,3,1){};
            \node[NodeGraph](0140)at(1,4,0){};
            \node[NodeGraph](0141)at(1,4,1){};
            \node[NodeGraph](0200)at(2,0,0){};
            \node[MarkedNodeGraph](0201)at(2,0,1){};
            \node[NodeGraph](0230)at(2,3,0){};
            \node[NodeGraph](0231)at(2,3,1){};
            \node[NodeGraph](0240)at(2,4,0){};
            \node[NodeGraph](0241)at(2,4,1){};
            \node[NodeLabelGraph,above of=0000]{$0000$};
            \node[NodeLabelGraph,right of=0001]{$0001$};
            \node[NodeLabelGraph,right of=0101]{$0101$};
            \node[NodeLabelGraph,right of=0201]{$0201$};
            \node[NodeLabelGraph,below of=0241]{$0241$};
            \draw[EdgeGraph](0000)--(0001);
            \draw[EdgeGraph](0000)--(0010);
            \draw[EdgeGraph](0000)--(0100);
            \draw[EdgeGraph](0001)--(0011);
            \draw[MarkedEdgeGraph](0001)--(0101);
            \draw[EdgeGraph](0010)--(0011);
            \draw[EdgeGraph](0010)--(0020);
            \draw[EdgeGraph](0011)--(0021);
            \draw[EdgeGraph](0020)--(0021);
            \draw[EdgeGraph](0020)--(0030);
            \draw[EdgeGraph](0020)--(0120);
            \draw[EdgeGraph](0021)--(0031);
            \draw[EdgeGraph](0021)--(0121);
            \draw[EdgeGraph](0030)--(0031);
            \draw[EdgeGraph](0030)--(0040);
            \draw[EdgeGraph](0030)--(0130);
            \draw[EdgeGraph](0031)--(0041);
            \draw[EdgeGraph](0031)--(0131);
            \draw[EdgeGraph](0040)--(0041);
            \draw[EdgeGraph](0040)--(0140);
            \draw[EdgeGraph](0041)--(0141);
            \draw[EdgeGraph](0100)--(0101);
            \draw[EdgeGraph](0100)--(0120);
            \draw[EdgeGraph](0100)--(0200);
            \draw[EdgeGraph](0101)--(0121);
            \draw[MarkedEdgeGraph](0101)--(0201);
            \draw[EdgeGraph](0120)--(0121);
            \draw[EdgeGraph](0120)--(0130);
            \draw[EdgeGraph](0121)--(0131);
            \draw[EdgeGraph](0130)--(0131);
            \draw[EdgeGraph](0130)--(0140);
            \draw[EdgeGraph](0130)--(0230);
            \draw[EdgeGraph](0131)--(0141);
            \draw[EdgeGraph](0131)--(0231);
            \draw[EdgeGraph](0140)--(0141);
            \draw[EdgeGraph](0140)--(0240);
            \draw[EdgeGraph](0141)--(0241);
            \draw[EdgeGraph](0200)--(0201);
            \draw[EdgeGraph](0200)--(0230);
            \draw[EdgeGraph](0201)--(0231);
            \draw[EdgeGraph](0230)--(0231);
            \draw[EdgeGraph](0230)--(0240);
            \draw[EdgeGraph](0231)--(0241);
            \draw[EdgeGraph](0240)--(0241);
        \end{tikzpicture}}
        \; \to \;
        %
        % D^6(Canyon), m = 2, n = 4
        \scalebox{.55}{
        \begin{tikzpicture}[Centering,xscale=.7,yscale=.7,
            x={(0,-.5cm)}, y={(-1.0cm,-1.0cm)}, z={(1.0cm,-1.0cm)}]
            \DrawGridSpace{2}{4}{0}
            \node[MarkedNodeGraph](0000)at(0,0,0){};
            \node[MarkedNodeGraph](0010)at(0,1,0){};
            \node[MarkedNodeGraph](0020)at(0,2,0){};
            \node[MarkedNodeGraph](0030)at(0,3,0){};
            \node[MarkedNodeGraph](0040)at(0,4,0){};
            \node[MarkedNodeGraph](0100)at(1,0,0){};
            \node[MarkedNodeGraph](0120)at(1,2,0){};
            \node[MarkedNodeGraph](0130)at(1,3,0){};
            \node[MarkedNodeGraph](0140)at(1,4,0){};
            \node[MarkedNodeGraph](0200)at(2,0,0){};
            \node[MarkedNodeGraph](0230)at(2,3,0){};
            \node[MarkedNodeGraph](0240)at(2,4,0){};
            \node[NodeLabelGraph,above of=0000]{$0000$};
            \node[NodeLabelGraph,left of=0010]{$0010$};
            \node[NodeLabelGraph,left of=0020]{$0020$};
            \node[NodeLabelGraph,left of=0030]{$0030$};
            \node[NodeLabelGraph,left of=0040]{$0040$};
            \node[NodeLabelGraph,right of=0100]{$0100$};
            \node[NodeLabelGraph,left of=0120]{$0120$};
            \node[NodeLabelGraph,right of=0130]{$0130$};
            \node[NodeLabelGraph,left of=0140]{$0140$};
            \node[NodeLabelGraph,right of=0200]{$0200$};
            \node[NodeLabelGraph,below of=0230]{$0230$};
            \node[NodeLabelGraph,below of=0240]{$0240$};
            \draw[MarkedEdgeGraph](0000)--(0010);
            \draw[MarkedEdgeGraph](0000)--(0100);
            \draw[MarkedEdgeGraph](0010)--(0020);
            \draw[MarkedEdgeGraph](0020)--(0030);
            \draw[MarkedEdgeGraph](0020)--(0120);
            \draw[MarkedEdgeGraph](0030)--(0040);
            \draw[MarkedEdgeGraph](0030)--(0130);
            \draw[MarkedEdgeGraph](0040)--(0140);
            \draw[MarkedEdgeGraph](0100)--(0120);
            \draw[MarkedEdgeGraph](0100)--(0200);
            \draw[MarkedEdgeGraph](0120)--(0130);
            \draw[MarkedEdgeGraph](0130)--(0140);
            \draw[MarkedEdgeGraph](0130)--(0230);
            \draw[MarkedEdgeGraph](0140)--(0240);
            \draw[MarkedEdgeGraph](0200)--(0230);
            \draw[MarkedEdgeGraph](0230)--(0240);
        \end{tikzpicture}}
    \end{multline*}
    \caption{\footnotesize A sequence of interval contractions from $\SetCanyon_{\mathbf
    2}(4)$ to a poset isomorphic to $\SetCanyon_\MapTwo(3)$. These interval contractions are
    poset derivations as introduced in Section~\ref{subsubsec:interval_doubling}. The marked
    intervals are the ones involved in the interval doubling operations.}
    \label{fig:contractions_canyon_2_4}
\end{figure}
\medbreak

\begin{Proposition} \label{prop:wings_butterflies_canyon}
    For any $m \geq 0$,
    \begin{enumerate}[label={\it (\roman*)}]
        \item \label{item:wings_butterflies_canyon_1}
        the graded set $\InputWings\Par{\SetCanyon_\MapM}$ contains all the $\MapM$-cliffs
        $u$ satisfying $u_i < u_{i + 1}$ for all $i \in [|u| - 1]$;

        \item \label{item:wings_butterflies_canyon_2}
        the graded set $\OutputWings\Par{\SetCanyon_\MapM}$ contains all the $\MapM$-cliffs
        $u$ satisfying, for all $i \in [2, |u|]$, $u_i < \MapM(i)$, and, for all $i \in
        [|u|]$, if $u_i \ne 0$, then for all $j \in [i - 2]$, $u_{i - j} < u_i - j$;

        \item \label{item:wings_butterflies_canyon_3}
        the graded set $\Butterflies\Par{\SetCanyon_\MapM}$ contains all the $\MapM$-cliffs
        $u$ satisfying $1 \leq u_i < {\mathbf m}(i)$ for all $i \in [2, |u|]$, and $u_i -
        u_{i - 1} \geq 2$ for all $i \in [3, |u|]$.
    \end{enumerate}
\end{Proposition}
\medbreak

Remark that, from the definition of $\MapM$-canyons and the description of
$\InputWings\Par{\SetCanyon_\MapM}$ brought by
Proposition~\ref{prop:wings_butterflies_canyon}, for any $u \in
\InputWings\Par{\SetCanyon_\MapM}$, all $\MapM$-canyons $v$ such that $u \Leq v$ are also
input-wings of $\SetCanyon_\MapM$. For this reason, for any $n \geq 0$,
$\InputWings\Par{\SetCanyon_\MapM}(n)$ is an order filter of $\SetCanyon_\MapM(n)$.
\medbreak

\begin{Proposition} \label{prop:irreducibles_canyons}
    For any $m \geq 1$ and $n \geq 1$, the set $\JoinIrreducibles\Par{\SetCanyon_\MapM(n)}$
    contains all $\MapM$-canyons having exactly one letter different from~$0$.
\end{Proposition}
\medbreak

By Proposition~\ref{prop:irreducibles_canyons}, the number of join-irreducibles elements
of $\SetCanyon_\MapM(n)$ satisfies, for any $m \geq 1$ and $n \geq 1$,
\begin{math}
    \# \JoinIrreducibles\Par{\SetCanyon_\MapM(n)}
    = m \binom{n}{2}.
\end{math}
Since by Proposition~\ref{prop:properties_canyon_posets}, $\SetCanyon_\MapM(n)$ is
constructible by interval doubling, this is also the number of its meet-irreducible
elements~\cite{GW16}.
\medbreak

%%%%%%%%%%%%%%%%%%%%%%%%%%%%%%%%%%%%%%%%%%%%%%%%%%%%%%%%%%%%%%%%%%%%%%%%%%%%%%%%%%%%%%%%%%%%
\subsubsection{Cubic realization}
Let $m \geq 1$ and $n \geq 0$. For any output-wing $u$ of $\SetCanyon_\MapM(n)$, we define
$\rho(u)$ as the $\MapM$-canyon $\IncrMap_{\SetCanyon_m}\Par{u'}$, where $u'$ is the
$\MapM$-cliff obtained by incrementing by $1$ all letters of $u$ except the first one.  For
instance, the output-wing $01007$ of $\SetCanyon_\MapTwo(5)$ is sent by $\rho$ to the
$\MapTwo$-canyon $\IncrMap_{\SetCanyon_2}(02118) = 02348$. We call $\rho(u)$ the
\Def{left-to-right increasing} of~$u$.  This map is not a poset embedding because, for
$\MapM := 2$ and $n := 3$, $\rho(010) = 023 \Leq 013 = \rho(002)$ but $010 \cancel{\Leq}
002$.
\medbreak

\begin{Proposition} \label{prop:cells_canyon}
    For any $m \geq 1$, $n \geq 0$, and $u \in \OutputWings\Par{\SetCanyon_\MapM}(n)$,
    \begin{enumerate}[label={\it (\roman*)}]
        \item \label{item:cells_canyon_1}
        the map $\rho$ is a poset morphism and a bijection between
        $\OutputWings\Par{\SetCanyon_\MapM}(n)$ and $\InputWings\Par{\SetCanyon_\MapM}(n)$;

        \item \label{item:cells_canyon_2}
        the $\MapM$-canyon $u$ is cell-compatible with the $\MapM$-canyon $\rho(u)$;

        \item \label{item:cells_canyon_3}
        the cell $\Angle{u, \rho(u)}$ is pure;

        \item \label{item:cells_canyon_4}
        all cells of
        \begin{math}
            \Bra{\Angle{u, \rho(u)} : u \in \OutputWings\Par{\SetCanyon_\MapM}(n)}
        \end{math}
        are pairwise disjoint.
    \end{enumerate}
\end{Proposition}
\begin{proof}
    Let us first prove that $\rho$ is a well-defined map.  By
    Proposition~\ref{prop:wings_butterflies_canyon}, since for all $i \in [2, n]$, $u_i <
    \MapM(i)$, the word $u'$ obtained by incrementing by $1$ all its letters except the
    first one is an $\MapM$-cliff. Moreover, since by
    Proposition~\ref{prop:properties_canyon_objects}, $\SetCanyon_\MapM$ is maximally
    extendable, $v := \IncrMap_{\SetCanyon_\MapM}\Par{u'}$ is a well-defined $\MapM$-canyon.
    Since by construction, for all $i \in [2, n]$, $v_i \ne 0$, each word obtained by
    replacing by $0$ a letter $v_i$ in $v$ is an $\MapM$-canyon. Therefore, $v$ covers $n -
    1$ elements of $\SetCanyon_\MapM(n)$. These elements are obtained by decreasing $v_i$ by
    some value, due to the fact that by Proposition~\ref{prop:properties_canyon_posets},
    $\SetCanyon_\MapM$ is straight.  For this reason, $v$ is an input-wing, showing that
    $\rho$ is a well-defined map from $\OutputWings\Par{\SetCanyon_\MapM}(n)$ to
    $\InputWings\Par{\SetCanyon_\MapM}(n)$.  Let us now define the map $\rho' :
    \InputWings\Par{\SetCanyon_\MapM}(n) \to \OutputWings\Par{\SetCanyon_\MapM}(n)$ as
    follows. For any $v \in \InputWings\Par{\SetCanyon_\MapM}(n)$, $u := \rho'(v)$ is the
    $\MapM$-cliff satisfying
    \begin{math}
        u_i = \IndicatorFunction_{i \ne 1} \IndicatorFunction_{u_{i - 1} \leq u_i - 2}
        \Par{u_i - 1}
    \end{math}
    for any $i \in [n]$.  It is straightforward to prove that $\rho'$ is a well-defined map.
    Moreover, by induction on $n \geq 0$, one can prove that both $\rho \circ \rho'$ and
    $\rho' \circ \rho$ are identity maps.  This establishes~\ref{item:cells_canyon_1}.
    \smallbreak

    Let $v$ be an $\MapM$-cliff satisfying $v_i \in \Bra{u_i, \rho(u)_i}$ for any $i \in
    [n]$. Since $\rho'$ is the inverse map of $\rho$, this is equivalent to the fact that
    $v_i \in \Bra{\rho'(w)_i, w_i}$ for all $i \in [n]$, where $w$ is the input-wing
    $\rho(u)$ of $\SetCanyon_\MapM(n)$.  Therefore, by definition of $\rho'$, $v_1 = 0$ and
    $v_i \in \HanL{w_i - 1}$ for any $i \in [2, n]$. The fact that $w$ is an input-wing
    implies, by Proposition~\ref{prop:wings_butterflies_canyon}, that $u_i < u_{i + 1}$ for
    all $i \in [n - 1]$. This implies that $v$ is an $\MapM$-canyon, so
    that~\ref{item:cells_canyon_2} checks out.
    \smallbreak

    Point~\ref{item:cells_canyon_3} follows directly from the definition of $\rho$: since
    $\rho(u)$ is obtained by incrementing all the letters of $u$, except the first, in a
    minimal way so that the obtained $\MapM$-cliff is an $\MapM$-canyon, there cannot be any
    $\MapM$-canyon inside the cell $\Angle{u, \rho(u)}$.
    \smallbreak

    Finally, assume that there are two input-wings $v$ and $w$ of $\SetCanyon_\MapM(n)$ such
    that there is a point $x := \Par{x_1, \dots, x_n} \in \R^n$ such that $x$ is inside both
    the cells $\Angle{\rho'(v), v}$ and $\Angle{\rho'(w), w}$. By contradiction, let us
    assume that $v \ne w$ and let us set $i \in [2, n]$ as the smallest position such that
    $v_i \ne w_i$. Therefore, we have in particular
    \begin{equation} \label{equ:cells_canyon_4}
        \rho'(v)_i < x_i < v_i
        \enspace \mbox{ and } \enspace
        \rho'(w) < x_i < w_i.
    \end{equation}
    Without loss of generality, we assume that $v_i < w_i$. Now, if $v_i - 2 \geq v_{i -
    1}$, then $\rho'(v)_i = v_i - 1$ and $\rho'(w)_i = w_i - 1$.  It follows
    from~\eqref{equ:cells_canyon_4} that $v_i = w_i$.  Otherwise, when $v_i - 2 < v_{i -
    1}$, we have $\rho'(v)_i = 0$ and $\rho'(w)_i = w_i - 1$.  It follows again,
    from~\eqref{equ:cells_canyon_4}, that $v_i = w_i$.  This contradicts our hypothesis and
    shows that $v = w$. Therefore, \ref{item:cells_canyon_4} holds.
\end{proof}
\medbreak

This algorithm $\rho$ brought by Proposition~\ref{prop:cells_canyon} describes the cells of
maximal dimension of the cubic realization of $\SetCanyon_\MapM(n)$.  The definition of
$\rho$ is inspired by an analogous algorithm introduced by the first author in~\cite{Com19}
to describe the cells of a geometric realization of the lattices of Tamari intervals.
Figure~\ref{fig:output_to_input_wings_bijection_canyon} shows some examples of images of
output-wings of $\SetCanyon_\MapM(n)$ by~$\rho$.
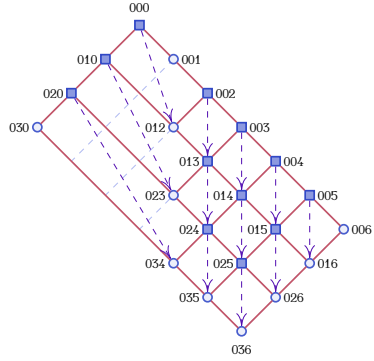
\begin{figure}[ht]
    \centering
    \scalebox{.8}{
    \begin{tikzpicture}[Centering,xscale=.8,yscale=.8,rotate=-135]
        \draw[Grid](0,0)grid(3,6);
        \node[MarkedNodeGraph](000)at(0,0){};
        \node[NodeGraph](001)at(0,1){};
        \node[MarkedNodeGraph](002)at(0,2){};
        \node[MarkedNodeGraph](003)at(0,3){};
        \node[MarkedNodeGraph](004)at(0,4){};
        \node[MarkedNodeGraph](005)at(0,5){};
        \node[NodeGraph](006)at(0,6){};
        \node[MarkedNodeGraph](010)at(1,0){};
        \node[NodeGraph](012)at(1,2){};
        \node[MarkedNodeGraph](013)at(1,3){};
        \node[MarkedNodeGraph](014)at(1,4){};
        \node[MarkedNodeGraph](015)at(1,5){};
        \node[NodeGraph](016)at(1,6){};
        \node[MarkedNodeGraph](020)at(2,0){};
        \node[NodeGraph](023)at(2,3){};
        \node[MarkedNodeGraph](024)at(2,4){};
        \node[MarkedNodeGraph](025)at(2,5){};
        \node[NodeGraph](026)at(2,6){};
        \node[NodeGraph](030)at(3,0){};
        \node[NodeGraph](034)at(3,4){};
        \node[NodeGraph](035)at(3,5){};
        \node[NodeGraph](036)at(3,6){};
        \node[NodeLabelGraph,above of=000]{$000$};
        \node[NodeLabelGraph,right of=001]{$001$};
        \node[NodeLabelGraph,right of=002]{$002$};
        \node[NodeLabelGraph,right of=003]{$003$};
        \node[NodeLabelGraph,right of=004]{$004$};
        \node[NodeLabelGraph,right of=005]{$005$};
        \node[NodeLabelGraph,right of=006]{$006$};
        \node[NodeLabelGraph,left of=010]{$010$};
        \node[NodeLabelGraph,left of=012]{$012$};
        \node[NodeLabelGraph,left of=013]{$013$};
        \node[NodeLabelGraph,left of=014]{$014$};
        \node[NodeLabelGraph,left of=015]{$015$};
        \node[NodeLabelGraph,right of=016]{$016$};
        \node[NodeLabelGraph,left of=020]{$020$};
        \node[NodeLabelGraph,left of=023]{$023$};
        \node[NodeLabelGraph,left of=024]{$024$};
        \node[NodeLabelGraph,left of=025]{$025$};
        \node[NodeLabelGraph,right of=026]{$026$};
        \node[NodeLabelGraph,left of=030]{$030$};
        \node[NodeLabelGraph,left of=034]{$034$};
        \node[NodeLabelGraph,left of=035]{$035$};
        \node[NodeLabelGraph,below of=036]{$036$};
        \draw[EdgeGraph](000)--(001);
        \draw[EdgeGraph](000)--(010);
        \draw[EdgeGraph](001)--(002);
        \draw[EdgeGraph](002)--(003);
        \draw[EdgeGraph](002)--(012);
        \draw[EdgeGraph](003)--(004);
        \draw[EdgeGraph](003)--(013);
        \draw[EdgeGraph](004)--(005);
        \draw[EdgeGraph](004)--(014);
        \draw[EdgeGraph](005)--(006);
        \draw[EdgeGraph](005)--(015);
        \draw[EdgeGraph](006)--(016);
        \draw[EdgeGraph](010)--(012);
        \draw[EdgeGraph](010)--(020);
        \draw[EdgeGraph](012)--(013);
        \draw[EdgeGraph](013)--(014);
        \draw[EdgeGraph](013)--(023);
        \draw[EdgeGraph](014)--(015);
        \draw[EdgeGraph](014)--(024);
        \draw[EdgeGraph](015)--(016);
        \draw[EdgeGraph](015)--(025);
        \draw[EdgeGraph](016)--(026);
        \draw[EdgeGraph](020)--(023);
        \draw[EdgeGraph](020)--(030);
        \draw[EdgeGraph](023)--(024);
        \draw[EdgeGraph](024)--(025);
        \draw[EdgeGraph](024)--(034);
        \draw[EdgeGraph](025)--(026);
        \draw[EdgeGraph](025)--(035);
        \draw[EdgeGraph](026)--(036);
        \draw[EdgeGraph](030)--(034);
        \draw[EdgeGraph](034)--(035);
        \draw[EdgeGraph](035)--(036);
        % Arrows.
        \draw[MapGraph](000)--(012);
        \draw[MapGraph](010)--(023);
        \draw[MapGraph](002)--(013);
        \draw[MapGraph](020)--(034);
        \draw[MapGraph](013)--(024);
        \draw[MapGraph](003)--(014);
        \draw[MapGraph](004)--(015);
        \draw[MapGraph](024)--(035);
        \draw[MapGraph](014)--(025);
        \draw[MapGraph](005)--(016);
        \draw[MapGraph](015)--(026);
        \draw[MapGraph](025)--(036);
    \end{tikzpicture}}
    \caption{\footnotesize The poset $\SetCanyon_{\mathbf 3}(3)$ wherein output-wings are
    marked. The arrows connect these elements to their images by the bijection $\rho$.}
    \label{fig:output_to_input_wings_bijection_canyon}
\end{figure}
\medbreak

The three posets of the input-wings, output-wings, and butterflies of canyon posets are
linked in the following way.
\medbreak

\begin{Theorem} \label{thm:poset_morphisms_canyon}
    For any $m \geq 1$ and $n \geq 0$,
    \begin{equation}
        \begin{tikzpicture}[Centering,xscale=3.5,yscale=1.6,font=\small]
            \node(InputWingsCanyon)at(0,0){$\InputWings\Par{\SetCanyon_\MapM}(n)$};
            \node(OutputWingsCanyon)at(-1,0){$\OutputWings\Par{\SetCanyon_\MapM}(n)$};
            \node(ButterfliesCanyon)at(1,0)
                {$\Butterflies\Par{\SetCanyon_{\mathbf m + 1}}(n)$};
            \draw[MapIsomorphism](InputWingsCanyon)--(ButterfliesCanyon)node[midway,above]
                {$\varphi_4$};
            \draw[Map](OutputWingsCanyon)--(InputWingsCanyon)node[midway,above]{$\rho$};
        \end{tikzpicture}
    \end{equation}
    is a diagram of poset morphisms or isomorphisms, where the map $\varphi_4$ is defined,
    for any $u \in \N^n$ and $i \in [n]$, by
    \begin{math}
        \varphi_4(u)_i := \IndicatorFunction_{i \ne 1} \Par{u_i + i - 2}.
    \end{math}
\end{Theorem}
\begin{proof}
    This follows from Point~\ref{item:cells_canyon_1} of Proposition~\ref{prop:cells_canyon}
    and from the descriptions of the input-wings, output-wings, and butterflies of
    $\SetCanyon_\MapM(n)$ provided by Proposition~\ref{prop:wings_butterflies_canyon}.
\end{proof}
\medbreak

Figure~\ref{fig:isomorphism_input_wings_butterflies_canyons} gives an example of the poset
morphisms or isomorphisms described by the statement of
Theorem~\ref{thm:poset_morphisms_canyon}.
\begin{figure}[ht]
    \centering
    \scalebox{.8}{
    \begin{tikzpicture}[Centering,xscale=.8,yscale=.8]
        % Output-wings canyon, m = 2, n = 3.
        \begin{scope}[xshift=-7cm,yshift=0cm,rotate=-135]
        \draw[Grid](0,0)grid(2,4);
        \node[Marked2NodeGraph](A000)at(0,0){};
        \node[NodeGraph](A001)at(0,1){};
        \node[Marked2NodeGraph](A002)at(0,2){};
        \node[Marked2NodeGraph](A003)at(0,3){};
        \node[NodeGraph](A004)at(0,4){};
        \node[Marked2NodeGraph](A010)at(1,0){};
        \node[NodeGraph](A012)at(1,2){};
        \node[Marked2NodeGraph](A013)at(1,3){};
        \node[NodeGraph](A014)at(1,4){};
        \node[NodeGraph](A020)at(2,0){};
        \node[NodeGraph](A023)at(2,3){};
        \node[NodeGraph](A024)at(2,4){};
        \node[NodeLabelGraph,above of=A000]{$000$};
        \node[NodeLabelGraph,above of=A003]{$003$};
        \node[NodeLabelGraph,above of=A013]{$013$};
        \node[NodeLabelGraph,below right of=A024]{$024$};
        \draw[Marked2EdgeGraph](A000)--(A001);
        \draw[Marked2EdgeGraph](A000)--(A010);
        \draw[Marked2EdgeGraph](A001)--(A002);
        \draw[Marked2EdgeGraph](A002)--(A003);
        \draw[Marked2EdgeGraph](A002)--(A012);
        \draw[EdgeGraph](A003)--(A004);
        \draw[Marked2EdgeGraph](A003)--(A013);
        \draw[EdgeGraph](A004)--(A014);
        \draw[Marked2EdgeGraph](A010)--(A012);
        \draw[EdgeGraph](A010)--(A020);
        \draw[Marked2EdgeGraph](A012)--(A013);
        \draw[EdgeGraph](A013)--(A014);
        \draw[EdgeGraph](A013)--(A023);
        \draw[EdgeGraph](A014)--(A024);
        \draw[EdgeGraph](A020)--(A023);
        \draw[EdgeGraph](A023)--(A024);
        \end{scope}
        %
        % Input-wings canyon, m = 2, n = 3.
        \begin{scope}[xshift=0cm,yshift=0cm,rotate=-135]
        \draw[Grid](0,0)grid(2,4);
        \node[NodeGraph](B000)at(0,0){};
        \node[NodeGraph](B001)at(0,1){};
        \node[NodeGraph](B002)at(0,2){};
        \node[NodeGraph](B003)at(0,3){};
        \node[NodeGraph](B004)at(0,4){};
        \node[NodeGraph](B010)at(1,0){};
        \node[MarkedNodeGraph](B012)at(1,2){};
        \node[MarkedNodeGraph](B013)at(1,3){};
        \node[MarkedNodeGraph](B014)at(1,4){};
        \node[NodeGraph](B020)at(2,0){};
        \node[MarkedNodeGraph](B023)at(2,3){};
        \node[MarkedNodeGraph](B024)at(2,4){};
        \node[NodeLabelGraph,above left of=B000]{$000$};
        \node[NodeLabelGraph,below of=B012]{$012$};
        \node[NodeLabelGraph,below of=B014]{$014$};
        \node[NodeLabelGraph,left of=B023]{$023$};
        \node[NodeLabelGraph,below of=B024]{$024$};
        \draw[EdgeGraph](B000)--(B001);
        \draw[EdgeGraph](B000)--(B010);
        \draw[EdgeGraph](B001)--(B002);
        \draw[EdgeGraph](B002)--(B003);
        \draw[EdgeGraph](B002)--(B012);
        \draw[EdgeGraph](B003)--(B004);
        \draw[EdgeGraph](B003)--(B013);
        \draw[EdgeGraph](B004)--(B014);
        \draw[EdgeGraph](B010)--(B012);
        \draw[EdgeGraph](B010)--(B020);
        \draw[MarkedEdgeGraph](B012)--(B013);
        \draw[MarkedEdgeGraph](B013)--(B014);
        \draw[MarkedEdgeGraph](B013)--(B023);
        \draw[MarkedEdgeGraph](B014)--(B024);
        \draw[EdgeGraph](B020)--(B023);
        \draw[MarkedEdgeGraph](B023)--(B024);
        \end{scope}
        %
        % Butterflies canyon, m = 3, n = 3.
        \begin{scope}[xshift=7cm,yshift=1cm,rotate=-135]
        \draw[Grid](0,0)grid(3,6);
        \node[NodeGraph](C000)at(0,0){};
        \node[NodeGraph](C001)at(0,1){};
        \node[NodeGraph](C002)at(0,2){};
        \node[NodeGraph](C003)at(0,3){};
        \node[NodeGraph](C004)at(0,4){};
        \node[NodeGraph](C005)at(0,5){};
        \node[NodeGraph](C006)at(0,6){};
        \node[NodeGraph](C010)at(1,0){};
        \node[NodeGraph](C012)at(1,2){};
        \node[MarkedNodeGraph](C013)at(1,3){};
        \node[MarkedNodeGraph](C014)at(1,4){};
        \node[MarkedNodeGraph](C015)at(1,5){};
        \node[NodeGraph](C016)at(1,6){};
        \node[NodeGraph](C020)at(2,0){};
        \node[NodeGraph](C023)at(2,3){};
        \node[MarkedNodeGraph](C024)at(2,4){};
        \node[MarkedNodeGraph](C025)at(2,5){};
        \node[NodeGraph](C026)at(2,6){};
        \node[NodeGraph](C030)at(3,0){};
        \node[NodeGraph](C034)at(3,4){};
        \node[NodeGraph](C035)at(3,5){};
        \node[NodeGraph](C036)at(3,6){};
        \node[NodeLabelGraph,above of=C000]{$000$};
        \node[NodeLabelGraph,above of=C013]{$013$};
        \node[NodeLabelGraph,above of=C015]{$015$};
        \node[NodeLabelGraph,left of=C024]{$024$};
        \node[NodeLabelGraph,below of=C025]{$025$};
        \node[NodeLabelGraph,below of=C036]{$036$};
        \draw[EdgeGraph](C000)--(C001);
        \draw[EdgeGraph](C000)--(C010);
        \draw[EdgeGraph](C001)--(C002);
        \draw[EdgeGraph](C002)--(C003);
        \draw[EdgeGraph](C002)--(C012);
        \draw[EdgeGraph](C003)--(C004);
        \draw[EdgeGraph](C003)--(C013);
        \draw[EdgeGraph](C004)--(C005);
        \draw[EdgeGraph](C004)--(C014);
        \draw[EdgeGraph](C005)--(C006);
        \draw[EdgeGraph](C005)--(C015);
        \draw[EdgeGraph](C006)--(C016);
        \draw[EdgeGraph](C010)--(C012);
        \draw[EdgeGraph](C010)--(C020);
        \draw[EdgeGraph](C012)--(C013);
        \draw[MarkedEdgeGraph](C013)--(C014);
        \draw[EdgeGraph](C013)--(C023);
        \draw[MarkedEdgeGraph](C014)--(C015);
        \draw[MarkedEdgeGraph](C014)--(C024);
        \draw[EdgeGraph](C015)--(C016);
        \draw[MarkedEdgeGraph](C015)--(C025);
        \draw[EdgeGraph](C016)--(C026);
        \draw[EdgeGraph](C020)--(C023);
        \draw[EdgeGraph](C020)--(C030);
        \draw[EdgeGraph](C023)--(C024);
        \draw[MarkedEdgeGraph](C024)--(C025);
        \draw[EdgeGraph](C024)--(C034);
        \draw[EdgeGraph](C025)--(C026);
        \draw[EdgeGraph](C025)--(C035);
        \draw[EdgeGraph](C026)--(C036);
        \draw[EdgeGraph](C030)--(C034);
        \draw[EdgeGraph](C034)--(C035);
        \draw[EdgeGraph](C035)--(C036);
        \end{scope}
        %
        % Arrows.
        \draw[MapGraph](A000)edge[bend left=16](B012);
        \draw[MapGraph](A003)edge[bend left=16](B014);
        \draw[MapGraph](A013)edge[bend left=8](B024);
        \draw[MapGraph](B012)edge[bend right=16](C013);
        \draw[MapGraph](B014)edge[bend right=16](C015);
        \draw[MapGraph](B024)edge[bend right=16](C025);
    \end{tikzpicture}}
    \caption{\footnotesize From the top to bottom and left to right, here are the posets
    $\SetCanyon_\MapTwo(3)$, $\SetCanyon_\MapTwo(3)$, and $\SetCanyon_{\mathbf 3}(3)$. The
    two last posets contain $\InputWings\Par{\SetHill_\MapOne}(3)$ as subposets. There is a
    poset morphism between the output-wings of the first one and the input-wings of the
    second one.}
    \label{fig:isomorphism_input_wings_butterflies_canyons}
\end{figure}
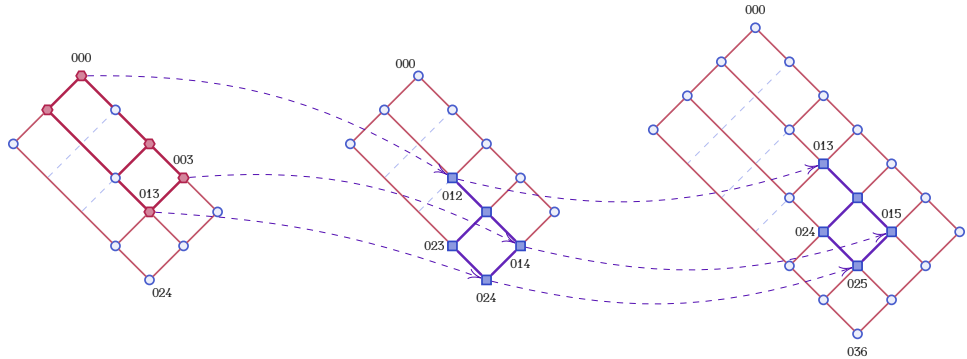
\medbreak

\begin{Proposition} \label{prop:volume_canyon}
    For any $m \geq 1$ and $n \geq 1$,
    \begin{equation}
        \Volume\Par{\CubicReal\Par{\SetCanyon_\MapM(n)}}
        = \Volume\Par{\CubicReal\Par{\SetCliff_\MapM(n)}}
        = m^{n - 1} (n - 1)!.
    \end{equation}
\end{Proposition}
\begin{proof}
    Directly from the definition of $\MapM$-canyons, one has that the $\MapM$-canyon
    $\LeastElement_\MapM(n)$ is cell-compatible with $\GreatestElement_\MapM(n)$.
    Therefore, $\Angle{\LeastElement_\MapM(n), \GreatestElement_\MapM(n)}$ is a cell of
    $\CubicReal\Par{\SetCanyon_\MapM(n)}$. Since all others cells of this cubic realization
    are contained in this one, one obtains that $\CubicReal\Par{\SetCanyon_\MapM(n)}$ is an
    orthotope. This leads to the stated expression for the volume of the cubic realization
    of~$\SetCanyon_\MapM(n)$.
\end{proof}
\medbreak

%%%%%%%%%%%%%%%%%%%%%%%%%%%%%%%%%%%%%%%%%%%%%%%%%%%%%%%%%%%%%%%%%%%%%%%%%%%%%%%%%%%%%%%%%%%%
%%%%%%%%%%%%%%%%%%%%%%%%%%%%%%%%%%%%%%%%%%%%%%%%%%%%%%%%%%%%%%%%%%%%%%%%%%%%%%%%%%%%%%%%%%%%
\subsection{Poset morphisms and other interactions}
The purpose of this part is to state the main links between the three posets
$\SetAvalanche_\delta$, $\SetHill_\delta$, and $\SetCanyon_\delta$ when $\delta$ is an
increasing range map. We shall also consider their subposets formed by their input-wings,
output-wings, and butterflies elements in the particular case where $\delta = \MapM$ for an
$m \geq 0$.
\medbreak

%%%%%%%%%%%%%%%%%%%%%%%%%%%%%%%%%%%%%%%%%%%%%%%%%%%%%%%%%%%%%%%%%%%%%%%%%%%%%%%%%%%%%%%%%%%%
\subsubsection{Order extensions}
Observe that the map $\ElevationMap_{\SetCanyon_\delta}$ is not a poset morphism. Indeed,
for instance in $\SetCanyon_\MapOne$ one has $002 \Leq 012$ but
\begin{math}
    \ElevationMap_{\SetCanyon_\MapOne}(002)
    = 002 \; \cancel{\Leq} \; 011
    = \ElevationMap_{\SetCanyon_\MapOne}(012).
\end{math}
Nevertheless, by composing this map on the left with the inverse of the
$\SetHill_\delta$-elevation map, we obtain a poset morphism, as stated by the next theorem.
\medbreak

\begin{Lemma} \label{lem:weight_elevation_canyon}
    Let $\delta$ be a range map, and $u$ and $v$ be two $\delta$-canyons of size $n$. If $u
    \Leq v$, then
    \begin{math}
        \Weight\Par{\ElevationMap_{\SetCanyon_\delta}(u)}
        \leq
        \Weight\Par{\ElevationMap_{\SetCanyon_\delta}(v)}.
    \end{math}
\end{Lemma}
\begin{proof}
    First, since by Proposition~\ref{prop:properties_canyon_objects}, $\SetCanyon_\delta$ is
    closed by prefix, $\ElevationMap_{\SetCanyon_\delta}$ is well-defined. By considering
    the contrapositive of the statement of the lemma and by
    Lemma~\ref{lem:weight_exuviae_canyons}, we have to show that for any $\delta$-canyons
    $u$ and $v$ of size $n$, $\Weight(\DominantCanyon(u)) > \Weight(\DominantCanyon(v))$
    implies that there exists $i \in [n]$ such that $u_i > v_i$. We proceed by induction on
    $n$. If $n = 0$, the property holds immediately. Assume now that $u = u' a$ and $v = v'
    b$ are two $\delta$-canyons of size $n + 1$ such that
    $\Weight\Par{\DominantCanyon\Par{u' a}} > \Weight\Par{\DominantCanyon\Par{v' b}}$ where
    $u'$ and $v'$ are $\delta$-canyons of size $n$ and $a, b \in \N$. If
    $\Weight\Par{\DominantCanyon\Par{u'}} > \Weight\Par{\DominantCanyon\Par{v'}}$, then by
    induction hypothesis, there is $i \in [n]$ such that $u'_i > v'_i$. Since $u_i = u'_i$
    and $v_i = v'_i$, the property holds. Otherwise, $\Weight\Par{\DominantCanyon\Par{u'}}
    \leq \Weight\Par{\DominantCanyon\Par{v'}}$. Since $\Weight(\DominantCanyon(u)) >
    \Weight(\DominantCanyon(v))$ and by definition of the map $\DominantCanyon$, we
    necessarily have $a > b$. Therefore one has $u_{n + 1}
    > v_{n + 1}$, showing that the property holds.
\end{proof}
\medbreak

\begin{Theorem} \label{thm:morphism_canyon_hill}
    For any increasing range map $\delta$ and any $n \geq 0$, the map
    \begin{math}
        \ElevationMap_{\SetHill_\delta}^{-1} \circ \ElevationMap_{\SetCanyon_\delta}
    \end{math}
    from $\SetCanyon_\delta(n)$ to $\SetHill_\delta(n)$ is a poset morphism.
\end{Theorem}
\begin{proof}
    First of all, by Proposition~\ref{prop:bijection_canyon_hill}, the map
    \begin{math}
        \phi := \ElevationMap_{\SetHill_\delta}^{-1} \circ \ElevationMap_{\SetCanyon_\delta}
    \end{math}
    is well-defined.  By definition of the maps $\ElevationMap_{\SetHill_\delta}$ and
    $\ElevationMap_{\SetCanyon_\delta}$, for any $\delta$-canyon $w$ of size $n$ and any $i
    \in [n]$,
    \begin{math}
        \phi(w)_i =
        \Weight\Par{\ElevationMap_{\SetCanyon_\delta}(w_1 \dots w_i)}.
    \end{math}
    Assume now that $u$ and $v$ are two $\delta$-canyons of size $n$ such that $u \Leq v$.
    Then, for any $i \in [n]$, $u_1 \dots u_i \Leq v_1 \dots v_i$. By
    Lemma~\ref{lem:weight_elevation_canyon}, this implies
    \begin{math}
        \Weight\Par{\ElevationMap_{\SetCanyon_\delta}(u_1 \dots u_i)}
        \leq
        \Weight\Par{\ElevationMap_{\SetCanyon_\delta}(v_1 \dots v_i)}.
    \end{math}
    Moreover, by the above remark, this implies $\phi(u)_i \leq \phi(v)_i$. Therefore, we
    have $\phi(u) \Leq \phi(v)$, establishing the statement of the theorem.
\end{proof}
\medbreak

Even if, by Proposition~\ref{prop:bijection_canyon_hill},
\begin{math}
    \ElevationMap_{\SetHill_\delta}^{-1} \circ \ElevationMap_{\SetCanyon_\delta}
    : \SetCanyon_\delta(n) \to \SetHill_\delta(n)
\end{math}
is a bijection, this map is not a poset isomorphism. This is the case since there does not
exist for instance a poset isomorphism between $\SetCanyon_\MapOne(3)$ and
$\SetHill_\MapOne(3)$ ---their Hasse diagrams are not superimposable.  Moreover, as a
consequence of Theorem~\ref{thm:morphism_canyon_hill}, for any $n \geq 0$,
$\SetHill_\delta(n)$ is an order extension of $\SetCanyon_\delta(n)$.  Furthermore, it is
possible to show by induction on the length of the $\delta$-canyons and by using
Lemma~\ref{lem:next_canyons} that $\SetCanyon_\delta$ satisfies the prerequisites of
Proposition~\ref{prop:elevation_map_poset_morphism}.  Therefore, $\SetCanyon_\delta(n)$ is
an order extension of $\SetAvalanche_\delta(n)$.
\medbreak

To summarize the situation, when $\delta$ is an increasing range map, the three
families of Fuss-Catalan posets fit into the chain
\begin{equation} \label{equ:chain_extensions_posets}
    \begin{tikzpicture}[Centering,xscale=3.5,yscale=2.6,font=\small]
        \node(Avalanche)at(0,0){$\SetAvalanche_\delta(n)$};
        \node(Canyon)at(1,0){$\SetCanyon_\delta(n)$};
        \node(Hill)at(2,0){$\SetHill_\delta(n)$};
        \draw[Map](Avalanche)--(Canyon)node[midway,above]
            {$\ElevationMap_{\SetCanyon_\delta}^{-1}$};
        \draw[Map](Canyon)--(Hill)node[midway,above]
            {$\ElevationMap_{\SetHill_\delta}^{-1}
        \circ
        \ElevationMap_{\SetCanyon_\delta}$};
    \end{tikzpicture}
\end{equation}
of posets for the order extension relation. This phenomenon is analogous to the one stating
that Stanley lattices are order extensions of Tamari lattices, which in turn are order
extension of Kreweras lattices~\cite{Kre72} (see for instance~\cite{BB09}). In the present
case, avalanche posets play the role of the Kreweras lattices, canyon posets play the role
of the Tamari posets, and hill posets play the role of Stanley posets. The difference is
that for $\delta = \MapOne$, even if canyon lattices coincide with Tamari lattices and hill
lattices coincide with Stanley lattices, the avalanche posets are not isomorphic to Kreweras
lattices. Figure~\ref{fig:poset_extensions_avalanche_canyon_hill} gives an example of an
instance of~\eqref{equ:chain_extensions_posets}.
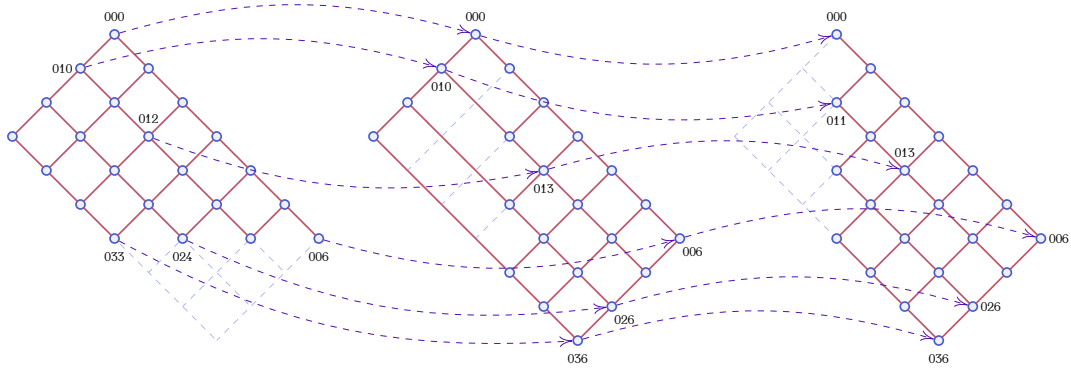
\begin{figure}[ht]
    \centering
    \scalebox{.8}{
    \begin{tikzpicture}[Centering,xscale=.8,yscale=.8]
        % Avalanche, m = 3, n = 3.
        \begin{scope}[xshift=0cm,yshift=0cm,rotate=-135]
        \draw[Grid](0,0)grid(3,6);
        \node[NodeGraph](A000)at(0,0){};
        \node[NodeGraph](A001)at(0,1){};
        \node[NodeGraph](A002)at(0,2){};
        \node[NodeGraph](A003)at(0,3){};
        \node[NodeGraph](A004)at(0,4){};
        \node[NodeGraph](A005)at(0,5){};
        \node[NodeGraph](A006)at(0,6){};
        \node[NodeGraph](A010)at(1,0){};
        \node[NodeGraph](A011)at(1,1){};
        \node[NodeGraph](A012)at(1,2){};
        \node[NodeGraph](A013)at(1,3){};
        \node[NodeGraph](A014)at(1,4){};
        \node[NodeGraph](A015)at(1,5){};
        \node[NodeGraph](A020)at(2,0){};
        \node[NodeGraph](A021)at(2,1){};
        \node[NodeGraph](A022)at(2,2){};
        \node[NodeGraph](A023)at(2,3){};
        \node[NodeGraph](A024)at(2,4){};
        \node[NodeGraph](A030)at(3,0){};
        \node[NodeGraph](A031)at(3,1){};
        \node[NodeGraph](A032)at(3,2){};
        \node[NodeGraph](A033)at(3,3){};
        \node[NodeLabelGraph,above of=A000]{$000$};
        \node[NodeLabelGraph,below of=A006]{$006$};
        \node[NodeLabelGraph,left of=A010]{$010$};
        \node[NodeLabelGraph,above of=A012]{$012$};
        \node[NodeLabelGraph,below of=A024]{$024$};
        \node[NodeLabelGraph,below of=A033]{$033$};
        \draw[EdgeGraph](A000)--(A001);
        \draw[EdgeGraph](A000)--(A010);
        \draw[EdgeGraph](A001)--(A002);
        \draw[EdgeGraph](A001)--(A011);
        \draw[EdgeGraph](A002)--(A003);
        \draw[EdgeGraph](A002)--(A012);
        \draw[EdgeGraph](A003)--(A004);
        \draw[EdgeGraph](A003)--(A013);
        \draw[EdgeGraph](A004)--(A005);
        \draw[EdgeGraph](A004)--(A014);
        \draw[EdgeGraph](A005)--(A006);
        \draw[EdgeGraph](A005)--(A015);
        \draw[EdgeGraph](A010)--(A011);
        \draw[EdgeGraph](A010)--(A020);
        \draw[EdgeGraph](A011)--(A012);
        \draw[EdgeGraph](A011)--(A021);
        \draw[EdgeGraph](A012)--(A013);
        \draw[EdgeGraph](A012)--(A022);
        \draw[EdgeGraph](A013)--(A014);
        \draw[EdgeGraph](A013)--(A023);
        \draw[EdgeGraph](A014)--(A015);
        \draw[EdgeGraph](A014)--(A024);
        \draw[EdgeGraph](A020)--(A021);
        \draw[EdgeGraph](A020)--(A030);
        \draw[EdgeGraph](A021)--(A022);
        \draw[EdgeGraph](A021)--(A031);
        \draw[EdgeGraph](A022)--(A023);
        \draw[EdgeGraph](A022)--(A032);
        \draw[EdgeGraph](A023)--(A024);
        \draw[EdgeGraph](A023)--(A033);
        \draw[EdgeGraph](A030)--(A031);
        \draw[EdgeGraph](A031)--(A032);
        \draw[EdgeGraph](A032)--(A033);
        \end{scope}
        %
        % Canyon, m = 3, n = 3.
        \begin{scope}[xshift=7.5cm,yshift=0cm,rotate=-135]
        \draw[Grid](0,0)grid(3,6);
        \node[NodeGraph](B000)at(0,0){};
        \node[NodeGraph](B001)at(0,1){};
        \node[NodeGraph](B002)at(0,2){};
        \node[NodeGraph](B003)at(0,3){};
        \node[NodeGraph](B004)at(0,4){};
        \node[NodeGraph](B005)at(0,5){};
        \node[NodeGraph](B006)at(0,6){};
        \node[NodeGraph](B010)at(1,0){};
        \node[NodeGraph](B012)at(1,2){};
        \node[NodeGraph](B013)at(1,3){};
        \node[NodeGraph](B014)at(1,4){};
        \node[NodeGraph](B015)at(1,5){};
        \node[NodeGraph](B016)at(1,6){};
        \node[NodeGraph](B020)at(2,0){};
        \node[NodeGraph](B023)at(2,3){};
        \node[NodeGraph](B024)at(2,4){};
        \node[NodeGraph](B025)at(2,5){};
        \node[NodeGraph](B026)at(2,6){};
        \node[NodeGraph](B030)at(3,0){};
        \node[NodeGraph](B034)at(3,4){};
        \node[NodeGraph](B035)at(3,5){};
        \node[NodeGraph](B036)at(3,6){};
        \node[NodeLabelGraph,above of=B000]{$000$};
        \node[NodeLabelGraph,below right of=B006]{$006$};
        \node[NodeLabelGraph,below of=B010]{$010$};
        \node[NodeLabelGraph,below of=B013]{$013$};
        \node[NodeLabelGraph,below right of=B026]{$026$};
        \node[NodeLabelGraph,below of=B036]{$036$};
        \draw[EdgeGraph](B000)--(B001);
        \draw[EdgeGraph](B000)--(B010);
        \draw[EdgeGraph](B001)--(B002);
        \draw[EdgeGraph](B002)--(B003);
        \draw[EdgeGraph](B002)--(B012);
        \draw[EdgeGraph](B003)--(B004);
        \draw[EdgeGraph](B003)--(B013);
        \draw[EdgeGraph](B004)--(B005);
        \draw[EdgeGraph](B004)--(B014);
        \draw[EdgeGraph](B005)--(B006);
        \draw[EdgeGraph](B005)--(B015);
        \draw[EdgeGraph](B006)--(B016);
        \draw[EdgeGraph](B010)--(B012);
        \draw[EdgeGraph](B010)--(B020);
        \draw[EdgeGraph](B012)--(B013);
        \draw[EdgeGraph](B013)--(B014);
        \draw[EdgeGraph](B013)--(B023);
        \draw[EdgeGraph](B014)--(B015);
        \draw[EdgeGraph](B014)--(B024);
        \draw[EdgeGraph](B015)--(B016);
        \draw[EdgeGraph](B015)--(B025);
        \draw[EdgeGraph](B016)--(B026);
        \draw[EdgeGraph](B020)--(B023);
        \draw[EdgeGraph](B020)--(B030);
        \draw[EdgeGraph](B023)--(B024);
        \draw[EdgeGraph](B024)--(B025);
        \draw[EdgeGraph](B024)--(B034);
        \draw[EdgeGraph](B025)--(B026);
        \draw[EdgeGraph](B025)--(B035);
        \draw[EdgeGraph](B026)--(B036);
        \draw[EdgeGraph](B030)--(B034);
        \draw[EdgeGraph](B034)--(B035);
        \draw[EdgeGraph](B035)--(B036);
        \end{scope}
        %
        % Hill, m = 3, n = 3.
        \begin{scope}[xshift=15cm,yshift=0cm,rotate=-135]
        \draw[Grid](0,0)grid(3,6);
        \node[NodeGraph](C000)at(0,0){};
        \node[NodeGraph](C001)at(0,1){};
        \node[NodeGraph](C002)at(0,2){};
        \node[NodeGraph](C003)at(0,3){};
        \node[NodeGraph](C004)at(0,4){};
        \node[NodeGraph](C005)at(0,5){};
        \node[NodeGraph](C006)at(0,6){};
        \node[NodeGraph](C011)at(1,1){};
        \node[NodeGraph](C012)at(1,2){};
        \node[NodeGraph](C013)at(1,3){};
        \node[NodeGraph](C014)at(1,4){};
        \node[NodeGraph](C015)at(1,5){};
        \node[NodeGraph](C016)at(1,6){};
        \node[NodeGraph](C022)at(2,2){};
        \node[NodeGraph](C023)at(2,3){};
        \node[NodeGraph](C024)at(2,4){};
        \node[NodeGraph](C025)at(2,5){};
        \node[NodeGraph](C026)at(2,6){};
        \node[NodeGraph](C033)at(3,3){};
        \node[NodeGraph](C034)at(3,4){};
        \node[NodeGraph](C035)at(3,5){};
        \node[NodeGraph](C036)at(3,6){};
        \node[NodeLabelGraph,above of=C000]{$000$};
        \node[NodeLabelGraph,right of=C006]{$006$};
        \node[NodeLabelGraph,below of=C011]{$011$};
        \node[NodeLabelGraph,above of=C013]{$013$};
        \node[NodeLabelGraph,right of=C026]{$026$};
        \node[NodeLabelGraph,below of=C036]{$036$};
        \draw[EdgeGraph](C000)--(C001);
        \draw[EdgeGraph](C001)--(C002);
        \draw[EdgeGraph](C001)--(C011);
        \draw[EdgeGraph](C002)--(C003);
        \draw[EdgeGraph](C002)--(C012);
        \draw[EdgeGraph](C003)--(C004);
        \draw[EdgeGraph](C003)--(C013);
        \draw[EdgeGraph](C004)--(C005);
        \draw[EdgeGraph](C004)--(C014);
        \draw[EdgeGraph](C005)--(C006);
        \draw[EdgeGraph](C005)--(C015);
        \draw[EdgeGraph](C006)--(C016);
        \draw[EdgeGraph](C011)--(C012);
        \draw[EdgeGraph](C012)--(C013);
        \draw[EdgeGraph](C012)--(C022);
        \draw[EdgeGraph](C013)--(C014);
        \draw[EdgeGraph](C013)--(C023);
        \draw[EdgeGraph](C014)--(C015);
        \draw[EdgeGraph](C014)--(C024);
        \draw[EdgeGraph](C015)--(C016);
        \draw[EdgeGraph](C015)--(C025);
        \draw[EdgeGraph](C016)--(C026);
        \draw[EdgeGraph](C022)--(C023);
        \draw[EdgeGraph](C023)--(C024);
        \draw[EdgeGraph](C023)--(C033);
        \draw[EdgeGraph](C024)--(C025);
        \draw[EdgeGraph](C024)--(C034);
        \draw[EdgeGraph](C025)--(C026);
        \draw[EdgeGraph](C025)--(C035);
        \draw[EdgeGraph](C026)--(C036);
        \draw[EdgeGraph](C033)--(C034);
        \draw[EdgeGraph](C034)--(C035);
        \draw[EdgeGraph](C035)--(C036);
        \end{scope}
        %
        % Arrows.
        \draw[MapGraph](A000)edge[bend left=16](B000);
        \draw[MapGraph](A010)edge[bend left=16](B010);
        \draw[MapGraph](A012)edge[bend right=16](B013);
        \draw[MapGraph](A033)edge[bend right=16](B036);
        \draw[MapGraph](A024)edge[bend right=16](B026);
        \draw[MapGraph](A006)edge[bend right=16](B006);
        \draw[MapGraph](B000)edge[bend right=16](C000);
        \draw[MapGraph](B010)edge[bend right=16](C011);
        \draw[MapGraph](B013)edge[bend left=16](C013);
        \draw[MapGraph](B036)edge[bend left=16](C036);
        \draw[MapGraph](B026)edge[bend left=16](C026);
        \draw[MapGraph](B006)edge[bend left=16](C006);
    \end{tikzpicture}}
    \caption{\footnotesize From the left to the right, here are the posets
    $\SetAvalanche_{\mathbf 3}(3)$, $\SetCanyon_{\mathbf 3}(3)$, and $\SetHill_{\mathbf
    3}(3)$. The poset on the right is an order extension of the one at middle, which is
    itself an order extension of the one at the left.}
    \label{fig:poset_extensions_avalanche_canyon_hill}
\end{figure}
\medbreak

%%%%%%%%%%%%%%%%%%%%%%%%%%%%%%%%%%%%%%%%%%%%%%%%%%%%%%%%%%%%%%%%%%%%%%%%%%%%%%%%%%%%%%%%%%%%
\subsubsection{Isomorphisms between subposets}
There is a link between the hill and the canyon posets, as stated by the following result.
\medbreak

\begin{Theorem} \label{thm:input_wings_canyons_hills_bijection}
    For any $m \geq 1$ and $n \geq 0$,
    \begin{equation}
        \begin{tikzpicture}[Centering,xscale=3.5,yscale=2.2,font=\small]
            \node(Hills)at(0,0){$\SetHill_{\mathbf m - 1}(n)$};
            \node(InputWingsCanyon)at(1,0){$\InputWings\Par{\SetCanyon_\MapM}(n)$};
            \draw[MapIsomorphism](Hills)--(InputWingsCanyon)node[midway,above]{$\varphi_3$};
        \end{tikzpicture}
    \end{equation}
    is a poset isomorphism, where $\varphi_3$ is the map defined in the statement of
    Theorem~\ref{thm:poset_morphisms_hill}.
\end{Theorem}
\begin{proof}
    This follows from the description of the input-wings of $\SetCanyon_\MapM(n)$ provided
    by Proposition~\ref{prop:wings_butterflies_canyon}.
\end{proof}
\medbreak

Figure~\ref{fig:isomorphism_input_wings_canyon_hills} gives an example of the poset
isomorphism described by the statement of
Theorem~\ref{thm:input_wings_canyons_hills_bijection}.
\begin{figure}[ht]
    \centering
    \scalebox{.8}{
    \begin{tikzpicture}[Centering,xscale=.7,yscale=.7,
        x={(0,-.5cm)}, y={(-1.0cm,-1.0cm)}, z={(1.0cm,-1.0cm)}]
        % Hill, m = 1, n = 4.
        \begin{scope}[xshift=0cm,yshift=-7cm]
        \DrawGridSpace{1}{2}{3}
        \node[MarkedNodeGraph](A0000)at(0,0,0){};
        \node[MarkedNodeGraph](A0001)at(0,0,1){};
        \node[MarkedNodeGraph](A0002)at(0,0,2){};
        \node[MarkedNodeGraph](A0003)at(0,0,3){};
        \node[MarkedNodeGraph](A0011)at(0,1,1){};
        \node[MarkedNodeGraph](A0012)at(0,1,2){};
        \node[MarkedNodeGraph](A0013)at(0,1,3){};
        \node[MarkedNodeGraph](A0022)at(0,2,2){};
        \node[MarkedNodeGraph](A0023)at(0,2,3){};
        \node[MarkedNodeGraph](A0111)at(1,1,1){};
        \node[MarkedNodeGraph](A0112)at(1,1,2){};
        \node[MarkedNodeGraph](A0113)at(1,1,3){};
        \node[MarkedNodeGraph](A0122)at(1,2,2){};
        \node[MarkedNodeGraph](A0123)at(1,2,3){};
        \node[NodeLabelGraph,above of=A0000]{$0000$};
        \node[NodeLabelGraph,right of=A0002]{$0002$};
        \node[NodeLabelGraph,right of=A0003]{$0003$};
        \node[NodeLabelGraph,above of=A0011]{$0011$};
        \node[NodeLabelGraph,above left of=A0022]{$0022$};
        \node[NodeLabelGraph,below of=A0123]{$0123$};
        \draw[MarkedEdgeGraph](A0000)--(A0001);
        \draw[MarkedEdgeGraph](A0001)--(A0002);
        \draw[MarkedEdgeGraph](A0001)--(A0011);
        \draw[MarkedEdgeGraph](A0002)--(A0003);
        \draw[MarkedEdgeGraph](A0002)--(A0012);
        \draw[MarkedEdgeGraph](A0003)--(A0013);
        \draw[MarkedEdgeGraph](A0011)--(A0012);
        \draw[MarkedEdgeGraph](A0011)--(A0111);
        \draw[MarkedEdgeGraph](A0012)--(A0013);
        \draw[MarkedEdgeGraph](A0012)--(A0022);
        \draw[MarkedEdgeGraph](A0012)--(A0112);
        \draw[MarkedEdgeGraph](A0013)--(A0023);
        \draw[MarkedEdgeGraph](A0013)--(A0113);
        \draw[MarkedEdgeGraph](A0022)--(A0023);
        \draw[MarkedEdgeGraph](A0022)--(A0122);
        \draw[MarkedEdgeGraph](A0023)--(A0123);
        \draw[MarkedEdgeGraph](A0111)--(A0112);
        \draw[MarkedEdgeGraph](A0112)--(A0113);
        \draw[MarkedEdgeGraph](A0112)--(A0122);
        \draw[MarkedEdgeGraph](A0113)--(A0123);
        \draw[MarkedEdgeGraph](A0122)--(A0123);
        \end{scope}
        %
        % Input-wings, canyon, m = 2, n = 4.
        \begin{scope}[xshift=12cm,yshift=-1.5cm]
        \DrawGridSpace{2}{4}{6}
        \node[NodeGraph](B0000)at(0,0,0){};
        \node[NodeGraph](B0001)at(0,0,1){};
        \node[NodeGraph](B0002)at(0,0,2){};
        \node[NodeGraph](B0003)at(0,0,3){};
        \node[NodeGraph](B0004)at(0,0,4){};
        \node[NodeGraph](B0005)at(0,0,5){};
        \node[NodeGraph](B0006)at(0,0,6){};
        \node[NodeGraph](B0010)at(0,1,0){};
        \node[NodeGraph](B0012)at(0,1,2){};
        \node[NodeGraph](B0013)at(0,1,3){};
        \node[NodeGraph](B0014)at(0,1,4){};
        \node[NodeGraph](B0015)at(0,1,5){};
        \node[NodeGraph](B0016)at(0,1,6){};
        \node[NodeGraph](B0020)at(0,2,0){};
        \node[NodeGraph](B0023)at(0,2,3){};
        \node[NodeGraph](B0024)at(0,2,4){};
        \node[NodeGraph](B0025)at(0,2,5){};
        \node[NodeGraph](B0026)at(0,2,6){};
        \node[NodeGraph](B0030)at(0,3,0){};
        \node[NodeGraph](B0034)at(0,3,4){};
        \node[NodeGraph](B0035)at(0,3,5){};
        \node[NodeGraph](B0036)at(0,3,6){};
        \node[NodeGraph](B0040)at(0,4,0){};
        \node[NodeGraph](B0045)at(0,4,5){};
        \node[NodeGraph](B0046)at(0,4,6){};
        \node[NodeGraph](B0100)at(1,0,0){};
        \node[NodeGraph](B0101)at(1,0,1){};
        \node[NodeGraph](B0103)at(1,0,3){};
        \node[NodeGraph](B0104)at(1,0,4){};
        \node[NodeGraph](B0105)at(1,0,5){};
        \node[NodeGraph](B0106)at(1,0,6){};
        \node[NodeGraph](B0120)at(1,2,0){};
        \node[MarkedNodeGraph](B0123)at(1,2,3){};
        \node[MarkedNodeGraph](B0124)at(1,2,4){};
        \node[MarkedNodeGraph](B0125)at(1,2,5){};
        \node[MarkedNodeGraph](B0126)at(1,2,6){};
        \node[NodeGraph](B0130)at(1,3,0){};
        \node[MarkedNodeGraph](B0134)at(1,3,4){};
        \node[MarkedNodeGraph](B0135)at(1,3,5){};
        \node[MarkedNodeGraph](B0136)at(1,3,6){};
        \node[NodeGraph](B0140)at(1,4,0){};
        \node[MarkedNodeGraph](B0145)at(1,4,5){};
        \node[MarkedNodeGraph](B0146)at(1,4,6){};
        \node[NodeGraph](B0200)at(2,0,0){};
        \node[NodeGraph](B0201)at(2,0,1){};
        \node[NodeGraph](B0204)at(2,0,4){};
        \node[NodeGraph](B0205)at(2,0,5){};
        \node[NodeGraph](B0206)at(2,0,6){};
        \node[NodeGraph](B0230)at(2,3,0){};
        \node[MarkedNodeGraph](B0234)at(2,3,4){};
        \node[MarkedNodeGraph](B0235)at(2,3,5){};
        \node[MarkedNodeGraph](B0236)at(2,3,6){};
        \node[NodeGraph](B0240)at(2,4,0){};
        \node[MarkedNodeGraph](B0245)at(2,4,5){};
        \node[MarkedNodeGraph](B0246)at(2,4,6){};
        \node[NodeLabelGraph,above of=B0000]{$0000$};
        \node[NodeLabelGraph,below of=B0123]{$0123$};
        \node[NodeLabelGraph,below of=B0125]{$0125$};
        \node[NodeLabelGraph,below of=B0126]{$0126$};
        \node[NodeLabelGraph,left of=B0134]{$0134$};
        \node[NodeLabelGraph,left of=B0145]{$0145$};
        \node[NodeLabelGraph,below of=B0246]{$0246$};
        \draw[EdgeGraph](B0000)--(B0001);
        \draw[EdgeGraph](B0000)--(B0010);
        \draw[EdgeGraph](B0000)--(B0100);
        \draw[EdgeGraph](B0001)--(B0002);
        \draw[EdgeGraph](B0001)--(B0101);
        \draw[EdgeGraph](B0002)--(B0003);
        \draw[EdgeGraph](B0002)--(B0012);
        \draw[EdgeGraph](B0003)--(B0004);
        \draw[EdgeGraph](B0003)--(B0013);
        \draw[EdgeGraph](B0003)--(B0103);
        \draw[EdgeGraph](B0004)--(B0005);
        \draw[EdgeGraph](B0004)--(B0014);
        \draw[EdgeGraph](B0004)--(B0104);
        \draw[EdgeGraph](B0005)--(B0006);
        \draw[EdgeGraph](B0005)--(B0015);
        \draw[EdgeGraph](B0005)--(B0105);
        \draw[EdgeGraph](B0006)--(B0016);
        \draw[EdgeGraph](B0006)--(B0106);
        \draw[EdgeGraph](B0010)--(B0012);
        \draw[EdgeGraph](B0010)--(B0020);
        \draw[EdgeGraph](B0012)--(B0013);
        \draw[EdgeGraph](B0013)--(B0014);
        \draw[EdgeGraph](B0013)--(B0023);
        \draw[EdgeGraph](B0014)--(B0015);
        \draw[EdgeGraph](B0014)--(B0024);
        \draw[EdgeGraph](B0015)--(B0016);
        \draw[EdgeGraph](B0015)--(B0025);
        \draw[EdgeGraph](B0016)--(B0026);
        \draw[EdgeGraph](B0020)--(B0023);
        \draw[EdgeGraph](B0020)--(B0030);
        \draw[EdgeGraph](B0020)--(B0120);
        \draw[EdgeGraph](B0023)--(B0024);
        \draw[EdgeGraph](B0023)--(B0123);
        \draw[EdgeGraph](B0024)--(B0025);
        \draw[EdgeGraph](B0024)--(B0034);
        \draw[EdgeGraph](B0024)--(B0124);
        \draw[EdgeGraph](B0025)--(B0026);
        \draw[EdgeGraph](B0025)--(B0035);
        \draw[EdgeGraph](B0025)--(B0125);
        \draw[EdgeGraph](B0026)--(B0036);
        \draw[EdgeGraph](B0026)--(B0126);
        \draw[EdgeGraph](B0030)--(B0034);
        \draw[EdgeGraph](B0030)--(B0040);
        \draw[EdgeGraph](B0030)--(B0130);
        \draw[EdgeGraph](B0034)--(B0035);
        \draw[EdgeGraph](B0034)--(B0134);
        \draw[EdgeGraph](B0035)--(B0036);
        \draw[EdgeGraph](B0035)--(B0045);
        \draw[EdgeGraph](B0035)--(B0135);
        \draw[EdgeGraph](B0036)--(B0046);
        \draw[EdgeGraph](B0036)--(B0136);
        \draw[EdgeGraph](B0040)--(B0045);
        \draw[EdgeGraph](B0040)--(B0140);
        \draw[EdgeGraph](B0045)--(B0046);
        \draw[EdgeGraph](B0045)--(B0145);
        \draw[EdgeGraph](B0046)--(B0146);
        \draw[EdgeGraph](B0100)--(B0101);
        \draw[EdgeGraph](B0100)--(B0120);
        \draw[EdgeGraph](B0100)--(B0200);
        \draw[EdgeGraph](B0101)--(B0103);
        \draw[EdgeGraph](B0101)--(B0201);
        \draw[EdgeGraph](B0103)--(B0104);
        \draw[EdgeGraph](B0103)--(B0123);
        \draw[EdgeGraph](B0104)--(B0105);
        \draw[EdgeGraph](B0104)--(B0124);
        \draw[EdgeGraph](B0104)--(B0204);
        \draw[EdgeGraph](B0105)--(B0106);
        \draw[EdgeGraph](B0105)--(B0125);
        \draw[EdgeGraph](B0105)--(B0205);
        \draw[EdgeGraph](B0106)--(B0126);
        \draw[EdgeGraph](B0106)--(B0206);
        \draw[EdgeGraph](B0120)--(B0123);
        \draw[EdgeGraph](B0120)--(B0130);
        \draw[MarkedEdgeGraph](B0123)--(B0124);
        \draw[MarkedEdgeGraph](B0124)--(B0125);
        \draw[MarkedEdgeGraph](B0124)--(B0134);
        \draw[MarkedEdgeGraph](B0125)--(B0126);
        \draw[MarkedEdgeGraph](B0125)--(B0135);
        \draw[MarkedEdgeGraph](B0126)--(B0136);
        \draw[EdgeGraph](B0130)--(B0134);
        \draw[EdgeGraph](B0130)--(B0140);
        \draw[EdgeGraph](B0130)--(B0230);
        \draw[MarkedEdgeGraph](B0134)--(B0135);
        \draw[MarkedEdgeGraph](B0134)--(B0234);
        \draw[MarkedEdgeGraph](B0135)--(B0136);
        \draw[MarkedEdgeGraph](B0135)--(B0145);
        \draw[MarkedEdgeGraph](B0135)--(B0235);
        \draw[MarkedEdgeGraph](B0136)--(B0146);
        \draw[MarkedEdgeGraph](B0136)--(B0236);
        \draw[EdgeGraph](B0140)--(B0145);
        \draw[EdgeGraph](B0140)--(B0240);
        \draw[MarkedEdgeGraph](B0145)--(B0146);
        \draw[MarkedEdgeGraph](B0145)--(B0245);
        \draw[MarkedEdgeGraph](B0146)--(B0246);
        \draw[EdgeGraph](B0200)--(B0201);
        \draw[EdgeGraph](B0200)--(B0230);
        \draw[EdgeGraph](B0201)--(B0204);
        \draw[EdgeGraph](B0204)--(B0205);
        \draw[EdgeGraph](B0204)--(B0234);
        \draw[EdgeGraph](B0205)--(B0206);
        \draw[EdgeGraph](B0205)--(B0235);
        \draw[EdgeGraph](B0206)--(B0236);
        \draw[EdgeGraph](B0230)--(B0234);
        \draw[EdgeGraph](B0230)--(B0240);
        \draw[MarkedEdgeGraph](B0234)--(B0235);
        \draw[MarkedEdgeGraph](B0235)--(B0236);
        \draw[MarkedEdgeGraph](B0235)--(B0245);
        \draw[MarkedEdgeGraph](B0236)--(B0246);
        \draw[EdgeGraph](B0240)--(B0245);
        \draw[MarkedEdgeGraph](B0245)--(B0246);
        \end{scope}
        %
        % Arrows.
        \draw[MapGraph](A0000)edge[bend left=16](B0123);
        \draw[MapGraph](A0011)edge[bend left=16](B0134);
        \draw[MapGraph](A0002)edge[bend right=16](B0125);
        \draw[MapGraph](A0003)edge[bend right=16](B0126);
        \draw[MapGraph](A0022)edge[bend right=16](B0145);
        \draw[MapGraph](A0123)edge[bend right=16](B0246);
    \end{tikzpicture}}
    \caption{\footnotesize The subposet of $\SetCanyon_\MapTwo(4)$ formed by its input-wings
    is isomorphic to $\SetHill_\MapOne(4)$.}
    \label{fig:isomorphism_input_wings_canyon_hills}
\end{figure}
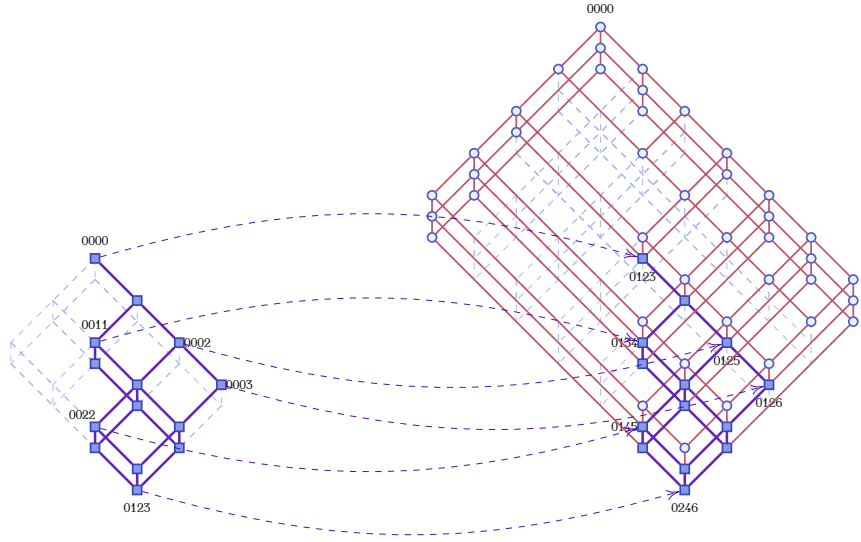
A consequence of Theorem~\ref{thm:input_wings_canyons_hills_bijection} is that, for any $m
\geq 2$ and $n \geq 0$, the image by~$\varphi_3^{-1}$ of
\begin{math}
    \SetCanyon_{\mathbf m - 1}(n) \cap \InputWings\Par{\SetCanyon_\MapM}(n)
\end{math}
is $\SetHill_{\mathbf m - 2}(n)$. Indeed, the set
\begin{math}
    \SetCanyon_{\mathbf m - 1}(n) \cap \InputWings\Par{\SetCanyon_\MapM}(n)
\end{math}
is nothing but the set $\InputWings\Par{\SetCanyon_{\mathbf m - 1}}(n)$.
\medbreak

To summarize the whole situation of the links between the three studied Fuss-Catalan posets
stated by Theorems~\ref{thm:poset_morphisms_avalanche}, \ref{thm:poset_morphisms_hill},
\ref{thm:poset_morphisms_canyon}, \ref{thm:morphism_canyon_hill},
\ref{thm:input_wings_canyons_hills_bijection}, and
Proposition~\ref{prop:elevation_map_poset_morphism}, one has for any $m \geq 1$ and $n \geq
0$ the diagram
\begin{equation} \begin{split}
    \begin{tikzpicture}[Centering,xscale=2.8,yscale=1.6,font=\footnotesize]
        \node(Avalanches)at(2,2){$\SetAvalanche_{\mathbf m - 1}(n)$};
        \node(InputWingsAvalanche)at(3,2){$\InputWings\Par{\SetAvalanche_\MapM}(n)$};
        \node(OutputWingsAvalanche)at(4,2){$\OutputWings\Par{\SetAvalanche_\MapM}(n)$};
        \node(ButterfliesAvalanche)at(5,2)
            {$\Butterflies\Par{\SetAvalanche_{\mathbf m + 1}}(n)$};
        \node(Canyons)at(2,1){$\SetCanyon_{\mathbf m - 1}(n)$};
        \node(OutputWingsCanyon)at(1,1){$\OutputWings\Par{\SetCanyon_\MapM}(n)$};
        \node(InputWingsCanyon)at(1,0){$\InputWings\Par{\SetCanyon_\MapM}(n)$};
        \node(ButterfliesCanyon)at(1,-1){$\Butterflies\Par{\SetCanyon_{\mathbf m + 1}}(n)$};
        \node(Hills)at(2,0){$\SetHill_{\mathbf m - 1}(n)$};
        \node(InputWingsHill)at(3,0){$\InputWings\Par{\SetHill_\MapM}(n)$};
        \node(OutputWingsHill)at(4,0){$\OutputWings\Par{\SetHill_\MapM}(n)$};
        \node(ButterfliesHill)at(3,-1){$\Butterflies\Par{\SetHill_{\mathbf m + 1}}(n)$};
        \draw[Map](Avalanches)--(Canyons)node[midway,right]
            {$\ElevationMap_{\SetCanyon_{\mathbf m - 1}}^{-1}$};
        \draw[Map](Canyons)--(Hills)node[midway,right]
            {$\ElevationMap_{\SetHill_{\mathbf m - 1}}^{-1} \circ
            \ElevationMap_{\SetCanyon_{\mathbf m - 1}}$};
        \draw[Map](OutputWingsCanyon)--(InputWingsCanyon)node[midway,left]{$\rho$};
        \draw[MapEmbedding](InputWingsAvalanche)--(OutputWingsAvalanche)node[midway,above]
            {$\varphi_2$};
        \draw[MapEmbedding](InputWingsHill)--(ButterfliesHill)node[midway,right]{$\Id$};
        \draw[MapIsomorphism](Avalanches)--(InputWingsAvalanche)node[midway,above]
            {$\varphi_1$};
        \draw[MapIsomorphism](OutputWingsAvalanche)--(ButterfliesAvalanche)
            node[midway,above]
            {$\varphi_1$};
        \draw[MapIsomorphism](Hills)--(InputWingsCanyon)node[midway,below]{$\varphi_3$};
        \draw[MapIsomorphism](InputWingsCanyon)--(ButterfliesCanyon)node[midway,left]
            {$\varphi_4$};
        \draw[MapIsomorphism](Hills)--(InputWingsHill)node[midway,below]{$\varphi_3$};
        \draw[MapIsomorphism](InputWingsHill)--(OutputWingsHill)node[midway,below]
            {$\varphi_2$};
    \end{tikzpicture}
\end{split} \end{equation}
of poset morphisms, embeddings, or isomorphisms.
\medbreak

%%%%%%%%%%%%%%%%%%%%%%%%%%%%%%%%%%%%%%%%%%%%%%%%%%%%%%%%%%%%%%%%%%%%%%%%%%%%%%%%%%%%%%%%%%%%
%%%%%%%%%%%%%%%%%%%%%%%%%%%%%%%%%%%%%%%%%%%%%%%%%%%%%%%%%%%%%%%%%%%%%%%%%%%%%%%%%%%%%%%%%%%%
%%%%%%%%%%%%%%%%%%%%%%%%%%%%%%%%%%%%%%%%%%%%%%%%%%%%%%%%%%%%%%%%%%%%%%%%%%%%%%%%%%%%%%%%%%%%
\section{Associative algebras of $\delta$-cliffs} \label{sec:cliff_algebras}
This part of the work is devoted to endow the sets of $\delta$-cliffs with algebraic
structures. We describe a graded associative algebra on $\delta$-cliffs motivated by a
connection with the $\delta$-cliff posets. Indeed, the product of two $\delta$-cliffs is a
sum of $\delta$-cliffs forming an interval of a $\delta$-cliff poset. This property is
shared by a lot of combinatorial and algebraic structures. For instance, the algebra
$\FQSym$ of permutations is related to the weak Bruhat order~\cite{DHT02,AS05}, the algebra
$\PBT$ of binary trees is related to the Tamari order~\cite{LR02,HNT05}, and the algebra
$\Sym$ of integer compositions is related to the hypercube~\cite{GKLLRT94}.
\medbreak

%%%%%%%%%%%%%%%%%%%%%%%%%%%%%%%%%%%%%%%%%%%%%%%%%%%%%%%%%%%%%%%%%%%%%%%%%%%%%%%%%%%%%%%%%%%%
%%%%%%%%%%%%%%%%%%%%%%%%%%%%%%%%%%%%%%%%%%%%%%%%%%%%%%%%%%%%%%%%%%%%%%%%%%%%%%%%%%%%%%%%%%%%
\subsection{Coalgebras and algebras}
We introduce here a cograded coalgebra structure on the linear span of all $\delta$-cliffs
and then, by considering the dual structure, we obtain a graded algebra. When $\delta$
satisfies some properties, this gives an associative algebra.
\medbreak

From now, $\K$ is the ground field (of characteristic zero) of all algebraic structures to
be defined next. For any graded vector space $\SpaceV$, we denote by
$\HilbertSeries_\SpaceV(t)$ the Hilbert series of~$\SpaceV$.
\medbreak

We use here the notions of valleys in range maps and of valley-free range maps, defined in
Section~\ref{subsubsec:first_definitions_cliffs}.
\medbreak

%%%%%%%%%%%%%%%%%%%%%%%%%%%%%%%%%%%%%%%%%%%%%%%%%%%%%%%%%%%%%%%%%%%%%%%%%%%%%%%%%%%%%%%%%%%%
\subsubsection{Coalgebras of $\delta$-cliffs} \label{subsubsec:cliff_coalgebras}
For any range map $\delta$, let $\SpaceCliff_\delta$ be the linear span of all
$\delta$-cliffs. This space is graded and decomposes as
\begin{math}
    \SpaceCliff_\delta = \bigoplus_{n \geq 0} \SpaceCliff_\delta(n),
\end{math}
where $\SpaceCliff_\delta(n)$, $n \geq 0$, is the linear span of all $\delta$-cliffs of size
$n$.  By definition, the set $\Bra{\BasisF_u : u \in \SetCliff_\delta}$ is a basis of
$\SpaceCliff_\delta$, and we shall refer to it as the \Def{fundamental basis} or as the
\Def{$\BasisF$-basis}.  Let also $c : \SpaceCliff_\delta \to \K$ be the linear map defined
by $c\Par{\BasisF_\epsilon} := 1$ and by $c\Par{\BasisF_u} := 0$ for any $u \in
\SetCliff_\delta \setminus \{\epsilon\}$.
\medbreak

For any $n \geq 0$, the \Def{$\delta$-reduction map} is the map
\begin{math}
    \Reduction_\delta : \N^n \to \SetCliff_\delta(n)
\end{math}
defined for any word $u \in \N^n$ and any $i \in [n]$ by
\begin{math}
    \Par{\Reduction_\delta(u)}_i := \min \Bra{u_i, \delta(i)}.
\end{math}
For instance, $\Reduction_\MapOne(212066) = 012045$ and $\Reduction_\MapTwo(212066) =
012066$.
\medbreak

Let
\begin{math}
    \Coproduct : \SpaceCliff_\delta
    \to \SpaceCliff_\delta \otimes \SpaceCliff_\delta
\end{math}
be the cobinary coproduct defined, for any $w \in \SetCliff_\delta$, by
\begin{equation}
    \Delta\Par{\BasisF_w} :=
    \sum_{\substack{
        u, v \in \N^* \\
        w = uv
    }}
    \BasisF_u \otimes \BasisF_{\Reduction_\delta(v)},
\end{equation}
where $\N^*$ denotes the set of all words on $\N$.  This coproduct is well-defined since any
prefix of a $\delta$-cliff is a $\delta$-cliff and the image of a word on $\N$ by the
$\delta$-reduction map is by definition a $\delta$-cliff.  For instance, for $\delta :=
1221013^\omega$, we have in $\SpaceCliff_\delta$,
\begin{equation}
    \Coproduct\Par{\BasisF_{1021}} =
    \BasisF_{\epsilon} \otimes \BasisF_{1021}
    + \BasisF_{1} \otimes \BasisF_{021}
    + \BasisF_{10} \otimes \BasisF_{11}
    + \BasisF_{102} \otimes \BasisF_{1}
    + \BasisF_{1021} \otimes \BasisF_{\epsilon},
\end{equation}
and
\begin{equation} \begin{split}
    \Coproduct\Par{\BasisF_{1211010}} & =
    \BasisF_{\epsilon} \otimes \BasisF_{1211010}
    + \BasisF_{1} \otimes \BasisF_{111000}
    + \BasisF_{12} \otimes \BasisF_{11010}
    + \BasisF_{121} \otimes \BasisF_{1010} \\
    & \quad + \BasisF_{1211} \otimes \BasisF_{010}
    + \BasisF_{12110} \otimes \BasisF_{10}
    + \BasisF_{121101} \otimes \BasisF_{0}
    + \BasisF_{1211010} \otimes \BasisF_{\epsilon}.
\end{split} \end{equation}
\medbreak

\begin{Theorem} \label{thm:cliff_coassociative_coalgebra}
    Let $\delta$ be a range map. The space $\SpaceCliff_\delta$ endowed with the coproduct
    $\Coproduct$ and the counit $c$ is a counital cograded coalgebra. Moreover, $\Coproduct$
    is coassociative if and only if $\delta$ is valley-free.
\end{Theorem}
\begin{proof}
    The first part of the statement is a direct consequence of the definition of
    $\Coproduct$.
    \smallbreak

    To establish the second part, let us compute the two ways to apply twice the coproduct
    $\Coproduct$ on a basis element of $\SpaceCliff_\delta$. For any $w \in
    \SetCliff_\delta$, we have
    \begin{equation} \label{equ:cliff_coassociative_coalgebra_1}
        (\Delta \otimes I) \Delta\Par{\BasisF_w}
        =
        \sum_{\substack{
            x, y, z \in \N^* \\
            w = x y z
        }}
        \BasisF_x \otimes \BasisF_{\Reduction_\delta(y)}
        \otimes \BasisF_{\Reduction_\delta(z)}
    \end{equation}
    and
    \begin{equation} \label{equ:cliff_coassociative_coalgebra_2}
        (I \otimes \Delta) \Delta\Par{\BasisF_w}
        %& =
        =
        \sum_{\substack{
            u, v \in \N^* \\
            w = u v
        }}
        \sum_{\substack{
            y', z' \in \N^* \\
            \Reduction_\delta(v) = y' z'
        }}
        \BasisF_u \otimes \BasisF_{y'} \otimes \BasisF_{\Reduction_\delta\Par{z'}}
        =
        \sum_{\substack{
            x, y, z \in \N^* \\
            w = x y z
        }}
        \BasisF_x \otimes \BasisF_{\Reduction_\delta(y)}
        \otimes \BasisF_{\Reduction_{\delta_{|y|}}(z)},
    \end{equation}
    where for any $k \geq 0$, $\delta_k$ is the range map satisfying $\delta_k(i) = \min
    \Bra{\delta(i), \delta(k + i)}$ for any $i \geq 1$. The second equality
    of~\eqref{equ:cliff_coassociative_coalgebra_2} comes from the two following facts.
    First, for any $i \in \Han{\left| y' \right|}$,
    \begin{math}
        y'_i = \Reduction_\delta(v)_i = \Reduction_\delta(y)_i
    \end{math}
    where $y$ is the factor $w_{|u| + 1} \dots w_{|u| + \left| y' \right|}$ of $w$.  Second,
    we have for any $j \in \Han{\left| z' \right|}$,
    \begin{math}
        z'_j
        = {\Reduction_\delta(v)}_{\left| y' \right| + j}
        = \min \Bra{v_{\left| y' \right| + j}, \delta\Par{\left| y' \right| + j}},
    \end{math}
    so that for any $i \in \Han{\left| z' \right|}$,
    \begin{math}
        {\Reduction_\delta\Par{z'}}_i
        = \min \Bra{z'_i, \delta(i)}
        = \min \Bra{v_{\left| y' \right| + i}, \delta\Par{\left| y' \right| + i}, \delta(i)}
        = \Reduction_{\delta_{\left| y' \right|}}(z)_i,
    \end{math}
    where $z$ is the suffix of length $\left| z' \right|$ of $w$.
    \smallbreak

    Let us now prove that~\eqref{equ:cliff_coassociative_coalgebra_1}
    and~\eqref{equ:cliff_coassociative_coalgebra_2} are different if and only if $\delta$
    has a valley. These two elements are different if and only if there exists a
    factorization $w = x y z$ with $x, y, z \in \N^*$ such that $\Reduction_\delta(z) \ne
    \Reduction_{\delta_{|y|}}(z)$. This is equivalent to the fact there exists an index $i
    \in [|z|]$ such that $\Reduction_\delta(z)_i \ne \Reduction_{\delta_{|y|}}(z)_i$.  Since
    $z$ is a suffix of $w$, there exists a $j \in [|x| + |y| + 1, |w|]$ such that $z = w_j
    w_{j + 1} \dots w_{|w|}$. Now, we have
    \begin{equation}
        {\Reduction_\delta(z)}_i
        = \min \Bra{w_{j + i - 1}, \delta(i)}
        \ne \min \Bra{w_{j + i - 1}, \delta(|y| + i), \delta(i)}
        = {\Reduction_{\delta_{|y|}}(z)}_i.
    \end{equation}
    To have this difference, we necessarily have $\delta(|y| + i) < z_i$ and $\delta(|y| +
    i) < \delta(i)$.  Now, since $w$ is in particular a $\delta$-cliff, we have $z_i = w_{j
    + i - 1} \leq \delta(j + i - 1)$. Therefore, we obtain
    \begin{math}
        \delta(i) > \delta(|y| + i) < \delta(j + i - 1).
    \end{math}
    Since $j \geq |y| + 1$, this leads to the fact that $\delta$ has a valley. This
    establishes that $\Coproduct$ is coassociative if and only $\delta$ is valley-free.
\end{proof}
\medbreak

%%%%%%%%%%%%%%%%%%%%%%%%%%%%%%%%%%%%%%%%%%%%%%%%%%%%%%%%%%%%%%%%%%%%%%%%%%%%%%%%%%%%%%%%%%%%
\subsubsection{Algebras of $\delta$-cliffs}
Let
\begin{math}
    \Product : \SpaceCliff_\delta \otimes \SpaceCliff_\delta \to \SpaceCliff_\delta
\end{math}
be the binary product defined as the dual of the coproduct $\Coproduct$ introduced in
Section~\ref{subsubsec:cliff_coalgebras}, where the graded dual space
${\SpaceCliff_\delta}^\ast$ is identified with $\SpaceCliff_\delta$.  By duality, this
product $\Product$ satisfies, for any $u, v \in \SetCliff_\delta$,
\begin{equation}
    \BasisF_u \Product \BasisF_v =
    \sum_{w \in \SetCliff_\delta}
    \Angle{\BasisF_u \otimes \BasisF_v, \Coproduct\Par{\BasisF_w}} \, \BasisF_w,
\end{equation}
where, for any $w \in \SetCliff_\delta$, $\Angle{\BasisF_u \otimes \BasisF_v,
\Coproduct\Par{\BasisF_w}}$ is the coefficient of $\BasisF_u \otimes \BasisF_v$ in
$\Coproduct\Par{\BasisF_w}$. Therefore,
\begin{equation} \label{equ:product_cliff_algebra}
    \BasisF_u \Product \BasisF_v =
    \sum_{\substack{
        v' \in \Reduction_\delta^{-1}(v) \\
        u v' \in \SetCliff_\delta
    }}
    \BasisF_{u v'},
\end{equation}
where $\Reduction_\delta^{-1}(v)$ is the fiber of $v$ under the map $\Reduction_\delta$. For
instance, in $\SpaceCliff_\MapOne$,
\begin{equation}
    \BasisF_{00} \Product \BasisF_{011}
    =
    \BasisF_{00011} + \BasisF_{00021} + \BasisF_{00031} + \BasisF_{00111} + \BasisF_{00121}
    + \BasisF_{00131} + \BasisF_{00211} + \BasisF_{00221} + \BasisF_{00231},
\end{equation}
in $\SpaceCliff_\MapTwo$,
\begin{equation}
    \BasisF_{00} \Product \BasisF_{011} =
    \BasisF_{00011} + \BasisF_{00111} + \BasisF_{00211} + \BasisF_{00311} + \BasisF_{00411},
\end{equation}
and in $\SpaceCliff_\delta$, where $\delta = 01312^\omega$, we have both
\begin{equation}
    \BasisF_{00} \Product \BasisF_{011} =
    \BasisF_{00011} + \BasisF_{00111} + \BasisF_{00211} + \BasisF_{00311}
\end{equation}
and
\begin{equation}
    \BasisF_{00} \Product \BasisF_{013} = 0.
\end{equation}
\medbreak

By Theorem~\ref{thm:cliff_coassociative_coalgebra}, the product $\Product$ admits the linear
map $\Unit : \K \to \SpaceCliff_\delta$ satisfying $\Unit(1) = \BasisF_\epsilon$ as unit,
and is graded.  Moreover, again by this last theorem, $\Product$ is associative if and only
if $\delta$ is valley-free.  For instance, for $\delta := 101^\omega$, $\delta$ has a valley
    and since
\begin{equation}
    \Par{\BasisF_0 \Product \BasisF_0} \Product \BasisF_0
    - \BasisF_0 \Product \Par{\BasisF_0 \Product \BasisF_0}
    = \BasisF_{000} - \Par{\BasisF_{000} + \BasisF_{001}}
    = - \BasisF_{001}
    \ne 0,
\end{equation}
the product $\Product$ of $\SpaceCliff_\delta$ is not associative.
\medbreak

We now establish a link between this product $\Product$ on the $\BasisF$-basis of
$\SpaceCliff_\delta$ and the posets $\SetCliff_\delta(n)$, $n \geq 0$, introduced and
studied in the previous sections. For this, let for any $n_1, n_2 \geq 0$ the two binary
operations
\begin{math}
    \Over, \, \Under : \SetCliff_\delta\Par{n_1} \times
    \SetCliff_\delta\Par{n_2} \to \N^{n_1 + n_2}
\end{math}
defined, for any $u, v \in \SetCliff_\delta$, by $u \Over v := u v$ and $u \Under v := u v'$
where $v'$ is the word on $\N$ of length $|v|$ satisfying, for any $i \in [|v|]$,
\begin{equation}
    v'_i =
    \begin{cases}
        \delta(|u| + i) & \mbox{if } v_i = \delta(i), \\
        v_i & \mbox{otherwise}.
    \end{cases}
\end{equation}
For instance, for $\delta = 112334^\omega$, $010 \Over 1021 = 010 1021$ and $010 \Under 1021
= 010 3041$.  For $\delta = 210^\omega$, $21 \Under 11 = 2110$. Observe that this last word
is not a $\delta$-cliff.
\medbreak

\begin{Lemma} \label{lem:over_under_well_definitions}
    Let $\delta$ be a range map and $u, v \in \SetCliff_\delta$. If the word $u \Over v$ is
    a $\delta$-cliff, then $u \Under v$ also is.
\end{Lemma}
\begin{proof}
    Assume that $w := u \Over v \in \SetCliff_\delta$. Hence, for any $i \in [|w|]$, $w_i
    \leq \delta(i)$. In particular, this implies that for any $i \in [|v|]$, $v_i = w_{|u| +
    i} \leq \delta(|u| + i)$. By definition of the operation $\Under$, the word $w' := u
    \Under v$ satisfies $w_{|u| + i} \in \Bra{v_i, \delta(|u| + i)}$.  Moreover, the fact
    that $u$ is a $\delta$-cliff implies that for any $i \in [|u|]$, $u_i = w'_i \leq
    \delta(i)$. Therefore, $w'$ is a $\delta$-cliff.
\end{proof}
\medbreak

\begin{Lemma} \label{lem:over_under_weakly_increasing_range_map}
    A range map $\delta$ is weakly increasing if and only if for any $u, v \in
    \SetCliff_\delta$, $u \Over v$ is a $\delta$-cliff.
\end{Lemma}
\begin{proof}
    Assume that $\delta$ is weakly increasing and let $w := u \Over v$ where $u, v \in
    \SetCliff_\delta$.  Hence, since $v$ is a $\delta$-cliff, for any $i \in [|v|]$, $w_{|u|
    + i} = v_i \leq \delta(i)$.  Since $\delta$ is weakly increasing, we have $\delta(i)
    \leq \delta(|u| + i)$. This implies that $w_{|u| + i} \leq \delta(|u| + i)$.  Moreover,
    the fact that $u$ is $\delta$-cliff implies that, for any $i \in [|u|]$, $w_i = u_i \leq
    \delta(i)$.  Therefore, $u \Over v$ is a $\delta$-cliff.
    \smallbreak

    Conversely, assume that all $w := u \Over v$ are $\delta$-cliffs for all $u, v \in
    \SetCliff_\delta$. In particular, this holds for $u := 0$ and $v :=
    \GreatestElement_\delta(n)$ for any $n \geq 0$. We have $\delta(i - 1) = v_{i - 1} = w_i
    \leq \delta(i)$ for any $i \in [2, n + 1]$. Therefore, $\delta$ is weakly increasing.
\end{proof}
\medbreak

Let $\CharacteristicCliff_\delta : \N^* \to \K$ be the map defined for any $u \in \N^*$ by
$\CharacteristicCliff_\delta(u) := \IndicatorFunction_{u \Over v \in \SetCliff_\delta}$.
\medbreak

\begin{Theorem} \label{thm:product_cliff_intervals}
    For any range map $\delta$, the product $\Product$ of $\SpaceCliff_\delta$ satisfies,
    for any $u, v \in \SetCliff_\delta$,
    \begin{equation}
        \BasisF_u \Product \BasisF_v =
        \CharacteristicCliff_\delta\Par{u \Over v}
        \sum_{w \in \Han{u \Over v, u \Under v}}
        \BasisF_w
    \end{equation}
    where $\Han{u \Over v, u \Under v}$ is an interval of the poset
    $\SetCliff_\delta\Par{|u| + |v|}$.
\end{Theorem}
\begin{proof}
    Assume first that $w := u \Over v \in \SetCliff_\delta$. By
    Lemma~\ref{lem:over_under_well_definitions}, $u \Under v \in \SetCliff_\delta$.
    By~\eqref{equ:product_cliff_algebra}, for any $w' \in \SetCliff_\delta$, $\BasisF_{w'}$
    appears in $\BasisF_u \Product \BasisF_v$ if and only if there is $v' \in
    \Reduction_\delta^{-1}(v)$ such that $u v' = w'$. This implies that
    $\Reduction_\delta\Par{v'} = v$ and, by definition of the $\delta$-reduction map, for
    any $i \in [|v|]$, $v'_i \geq v_i$.  Moreover, since $w'$ is a $\delta$-cliff, we have
    for any $i \in [|v|]$, $v'_i = w'_{|u| + i} \leq \delta(|u| + i)$. Therefore, for all
    $i \in [|v|]$, $v_i \leq v'_i \leq \delta(|u| + i)$. This is equivalent to the fact
    that $u \Over v \Leq w' \Leq u \Under v$ and leads to the expression of the statement of
    theorem.
    \smallbreak

    Assume finally that $w := u \Over v \notin \SetCliff_\delta$. Since $u$ and $v$ are
    $\delta$-cliffs, there exists an index $i \in [|v|]$ such that $w_{|u| + i} > \delta(|u|
    + i)$. Since $w_{|u| + i} = v_i$, this implies that $v_i > \delta(|u| + i)$. Observe
    that by definition of the $\delta$-reduction map, for all $v' \in
    \Reduction_\delta^{-1}(v)$ and $j \in [|v|]$, $v'_j \geq v_j$.  Therefore, no $u v'$ can
    be a $\delta$-cliff.  By inspecting Formula~\eqref{equ:product_cliff_algebra} for the
    product~$\Product$, we obtain that the sum is empty, so that $\BasisF_u \Product
    \BasisF_v = 0$.
\end{proof}
\medbreak

For instance, for $\delta := 01120^\omega$,
\begin{equation*}
    \BasisF_{01} \Product \BasisF_{010}
    = \BasisF_{01010} + \BasisF_{01020} + \BasisF_{01110} + \BasisF_{01120},
\end{equation*}
and, since $01 \Over 011 = 01 011 \notin \SetCliff_\delta$,
\begin{equation*}
    \BasisF_{01} \Product \BasisF_{011} = 0.
\end{equation*}
\medbreak

In particular when $\delta$ is weakly increasing,
Lemma~\ref{lem:over_under_weakly_increasing_range_map} and
Theorem~\ref{thm:product_cliff_intervals} state that any product of two elements of the
$\BasisF$-basis of $\SpaceCliff_\delta$ is a sum of elements of the $\BasisF$-basis ranging
in an interval of a $\delta$-cliff poset.
\medbreak

%%%%%%%%%%%%%%%%%%%%%%%%%%%%%%%%%%%%%%%%%%%%%%%%%%%%%%%%%%%%%%%%%%%%%%%%%%%%%%%%%%%%%%%%%%%%
%%%%%%%%%%%%%%%%%%%%%%%%%%%%%%%%%%%%%%%%%%%%%%%%%%%%%%%%%%%%%%%%%%%%%%%%%%%%%%%%%%%%%%%%%%%%
\subsection{Bases and algebraic study}
We construct alternative bases for $\SpaceCliff_\delta$ and establish several properties of
this structure w.r.t. properties of the range map $\delta$. The main results are summarized
in Table~\ref{tab:overview_properties_cliff_algebra}.
\begin{table}[ht]
    \centering
    \renewcommand{\arraystretch}{1.15}
    \small
    \begin{tabular}{|c|c|} \hline
        {\bf Properties of $\delta$} & {\bf Properties of $\SpaceCliff_\delta$}
            \\ \hline \hline
        \multirow{2}{*}{None} & Unital graded magmatic algebra \\
        & Products on the $\BasisF$-basis are intervals in $\delta$-cliff posets \\ \hline
        Valley-free & Associative algebra \\ \hline
        Valley-free and $1$-dominated & Finite presentation \\ \hline
        Weakly increasing & Free as unital associative algebra \\ \hline
    \end{tabular}
    \smallbreak

    \caption{\footnotesize Properties of the algebras $\SpaceCliff_\delta$ implied by
    properties of range maps~$\delta$.}
    \label{tab:overview_properties_cliff_algebra}
\end{table}
We also provide a classification of all associative algebras $\SpaceCliff_\delta$ in four
classes, depending on properties of the valley-free range map~$\delta$.
\medbreak

We use here the notion of $j$-dominated range maps, defined in
Section~\ref{subsubsec:first_definitions_cliffs}.
\medbreak

%%%%%%%%%%%%%%%%%%%%%%%%%%%%%%%%%%%%%%%%%%%%%%%%%%%%%%%%%%%%%%%%%%%%%%%%%%%%%%%%%%%%%%%%%%%%
\subsubsection{$\BasisG$-basis}
For any $u \in \SetCliff_\delta$, let
\begin{equation}
    \BasisG_u := \BasisF_{\Complementary_\delta(u)},
\end{equation}
where $\Complementary(u)$ is the complementary of $u$ as defined in
Section~\ref{subsubsec:first_definitions_cliff_posets}.  Due to the fact that
$\Complementary$ is an involution, the set $\Bra{\BasisG_u : u \in \SetCliff_\delta}$ is a
basis of~$\SpaceCliff_\delta$, called \Def{$\BasisG$-basis}.
\medbreak

We now consider that $\delta$ is a weakly increasing range map.  We need, to state the next
result, to introduce for any $n_1, n_2 \geq 0$ the two binary operations
\begin{math}
    \OverG, \UnderG :
    \SetCliff_\delta\Par{n_1} \times \SetCliff_\delta\Par{n_2} \to \N^{n_1 + n_2}
\end{math}
defined, for any $u, v \in \SetCliff_\delta$, by $u \OverG v := u v'$ where $v'$ is the word
on $\N$ of length $|v|$ satisfying, for any $i \in [|v|]$,
\begin{equation}
    v'_i =
    \begin{cases}
        \delta(|u| + i) - \delta(i) + v_i & \mbox{if } v_i \ne 0, \\
        0 & \mbox{otherwise},
    \end{cases}
\end{equation}
and by $u \UnderG v := u v'$ where $v'$ is the word on $\N$ of length $|v|$ satisfying, for
any $i \in [|v|]$,
\begin{equation}
    v'_i = \delta(|u| + i) - \delta(i) + v_i.
\end{equation}
For instance, for $\delta = 112334^\omega$, $010 \Over 1021 = 010 3042$ and $010 \Under 1021
= 010 3242$.  Observe that in the case where $\delta = \mathbf{m}$ for an $m \geq 0$, $u
\OverG v$ is the word obtained by concatenating $u$ and $v$ and by incrementing by $m|u|$
the letters coming from $v$ that are different from $0$, and $u \UnderG v$ is the word
obtained by concatenating $u$ and $v$ and by incrementing by $m|u|$ all the letters coming
from~$v$.
\medbreak

\begin{Lemma} \label{lem:over_under_alternative}
    For any weakly increasing range map $\delta$ and any $u, v \in \SetCliff_\delta$,
    \begin{math}
        u \OverG v =
        \Complementary_\delta\Par{\Complementary_\delta(u) \Under \Complementary_\delta(v)},
    \end{math}
    and
    \begin{math}
        u \UnderG v =
        \Complementary_\delta\Par{\Complementary_\delta(u) \Over \Complementary_\delta(v)}.
    \end{math}
\end{Lemma}
\begin{proof}
    These identities follow by straightforward computations based upon the definitions of
    the operations $\Over$, $\Under$, $\OverG$, $\UnderG$, and $\Complementary_\delta$.
\end{proof}
\medbreak

Notice that, by Lemmas~\ref{lem:over_under_alternative},
\ref{lem:over_under_well_definitions}, and~\ref{lem:over_under_weakly_increasing_range_map},
if $\delta$ is weakly increasing and $u$ and $v$ are $\delta$-cliffs, both $u \OverG v$ and
$u \UnderG v$ are $\delta$-cliffs.
\medbreak

\begin{Proposition} \label{prop:product_cliff_g_basis}
    For any weakly increasing range map $\delta$, the product $\Product$ of
    $\SpaceCliff_\delta$ satisfies, for any $u, v \in \SetCliff_\delta$,
    \begin{equation} \label{equ:product_cliff_g_basis}
        \BasisG_u \Product \BasisG_v =
        \sum_{w \in \Han{u \OverG v, u \UnderG v}} \BasisG_w.
    \end{equation}
    where $\Han{u \OverG v, u \UnderG v}$ is an interval of the poset
    $\SetCliff_\delta\Par{|u| + |v|}$.
\end{Proposition}
\begin{proof}
    By Lemma~\ref{lem:over_under_alternative}, we have
    \begin{equation}
        \BasisG_u \Product \BasisG_v
        =
        \BasisF_{\Complementary_\delta(u)} \Product \BasisF_{\Complementary_\delta(v)}
        =
        \sum_{
            w \in \Han{\Complementary_\delta(u) \Over \Complementary_\delta(v),
            \Complementary_\delta(u) \Under \Complementary_\delta(v)}
        }
        \BasisF_w
        =
        \sum_{
            w
            \in \Han{
                \Complementary_\delta\Par{
                    \Complementary_\delta(u) \Under \Complementary_\delta(v)},
                \Complementary_\delta\Par{
                    \Complementary_\delta(u) \Over \Complementary_\delta(v})}
        }
        \BasisG_w.
    \end{equation}
    Now, again by Lemma~\ref{lem:over_under_alternative}, we obtain the stated expression.
\end{proof}
\medbreak

For instance, in $\SpaceCliff_\MapOne$,
\begin{equation}
    \BasisG_{01} \Product \BasisG_{010}
    =
    \BasisG_{01030} + \BasisG_{01031} + \BasisG_{01032} + \BasisG_{01130} + \BasisG_{01131}
    + \BasisG_{01132} + \BasisG_{01230} + \BasisG_{01231} + \BasisG_{01232},
\end{equation}
and in $\SpaceCliff_\MapTwo$,
\begin{equation} \begin{split}
    \BasisG_{01} \Product \BasisG_{010}
    & =
    \BasisG_{01050} + \BasisG_{01051} + \BasisG_{01052}
    + \BasisG_{01053} + \BasisG_{01054} + \BasisG_{01150}
    + \BasisG_{01151} + \BasisG_{01152} + \BasisG_{01153} \\
    & \enspace + \BasisG_{01154} + \BasisG_{01250} + \BasisG_{01251}
    + \BasisG_{01252} + \BasisG_{01253} + \BasisG_{01254}
    + \BasisG_{01350} + \BasisG_{01351} + \BasisG_{01352} \\
    & \enspace + \BasisG_{01353}
    + \BasisG_{01354} + \BasisG_{01450} + \BasisG_{01451}
    + \BasisG_{01452} + \BasisG_{01453} + \BasisG_{01454}.
\end{split} \end{equation}
\medbreak

%%%%%%%%%%%%%%%%%%%%%%%%%%%%%%%%%%%%%%%%%%%%%%%%%%%%%%%%%%%%%%%%%%%%%%%%%%%%%%%%%%%%%%%%%%%%
\subsubsection{$\BasisE$ and $\BasisH$-bases}
\label{subsubsec:e_h_bases_cliff}
By mimicking the construction of bases of several combinatorial spaces by using a particular
partial order on their basis element (see for instance~\cite{DHT02,HNT05}), let for any $u
\in \SetCliff_\delta$,
\begin{multicols}{2}
\begin{equation} \label{equ:basis_e_cliff}
    \BasisE_u :=
    \sum_{\substack{v \in \SetCliff_\delta \\
        u \Leq v
    }}
    \BasisF_v,
\end{equation}

\begin{equation} \label{equ:basis_h_cliff}
    \BasisH_u :=
    \sum_{\substack{v \in \SetCliff_\delta \\
        v \Leq u
    }}
    \BasisF_v.
\end{equation}
\end{multicols}
\noindent By triangularity, the sets $\Bra{\BasisE_u : u \in \SetCliff_\delta}$ and
$\Bra{\BasisH_u : u \in \SetCliff_\delta}$ are bases of $\SpaceCliff_\delta$, called
respectively \Def{elementary basis} and \Def{homogeneous basis}, or respectively
\Def{$\BasisE$-basis} and \Def{$\BasisH$-basis}. For instance, for $\delta := 1021^\omega$,
\begin{equation}
    \BasisE_{10010} =
    \BasisF_{10010} + \BasisF_{10011} + \BasisF_{10110} + \BasisF_{10111} + \BasisF_{10210}
    + \BasisF_{10211},
\end{equation}
and
\begin{equation}
    \BasisH_{10010} = \BasisF_{10010} + \BasisF_{10000} + \BasisF_{00010} + \BasisF_{00000}.
\end{equation}
\medbreak

\begin{Proposition} \label{prop:product_cliff_e_basis}
    For any range map $\delta$, the product $\Product$ of $\SpaceCliff_\delta$ satisfies,
    for any $u, v \in \SetCliff_\delta$,
    \begin{equation}
        \BasisE_u \Product \BasisE_v =
        \CharacteristicCliff_\delta\Par{u \Over v} \; \BasisE_{u \Over v}
    \end{equation}
\end{Proposition}
\begin{proof}
    By~\eqref{equ:product_cliff_algebra}, we have
    \begin{equation} \label{equ:product_cliff_e_basis}
        \BasisE_u \Product \BasisE_v
        =
        \sum_{\substack{
            u', v' \in \SetCliff_\delta \\
            u \Leq u' \\
            v \Leq v'
        }}
        \sum_{\substack{
            v'' \in \Reduction_\delta^{-1}\Par{v'} \\
            u' v'' \in \SetCliff_\delta
        }}
        \BasisF_{u' v''}
        =
        \sum_{\substack{
            u' \in \SetCliff_\delta \\
            u \Leq u'
        }}
        \sum_{\substack{
            v'' \in \N^* \\
            v \Leq \Reduction_\delta\Par{v''} \\
            u' v'' \in \SetCliff_\delta
        }}
        \BasisF_{u' v''}
        =
        \sum_{\substack{
            u' \in \SetCliff_\delta \\
            u \Leq u'
        }}
        \enspace
        \sum_{\substack{
            v'' \in \N^{|v|} \\
            \forall i \in [|v|], v_i \leq v''_i \\
            u' v'' \in \SetCliff_\delta \\
        }}
        \BasisF_{u' v''}.
    \end{equation}
    The equality between the third and the last member of~\eqref{equ:product_cliff_e_basis}
    is a consequence of the fact that for any $v'' \in \N^*$, one has $v \Leq
    \Reduction_\delta\Par{v''}$ if and only if $v_i \leq v''_i$ for all $i \in [|v|]$. By
    definition of the $\BasisE$-basis provided by~\eqref{equ:basis_e_cliff}, the last member
    of~\eqref{equ:product_cliff_e_basis} is equal to the stated formula.
\end{proof}
\medbreak

\begin{Proposition} \label{prop:product_cliff_h_basis}
    For any range map $\delta$, the product $\Product$ of $\SpaceCliff_\delta$ satisfies,
    for any $u, v \in \SetCliff_\delta$,
    \begin{equation}
        \BasisH_u \Product \BasisH_v = \BasisH_{\Reduction_\delta\Par{u \Under v}}.
    \end{equation}
\end{Proposition}
\begin{proof}
    By~\eqref{equ:product_cliff_algebra}, we have
    \begin{equation} \label{equ:product_cliff_h_basis}
        \BasisH_u \Product \BasisH_v
        \sum_{\substack{
            u', v' \in \SetCliff_\delta \\
            u' \Leq u \\
            v' \Leq v
        }}
        \sum_{\substack{
            v'' \in \Reduction_\delta^{-1}\Par{v'} \\
            u' v'' \in \SetCliff_\delta
        }}
        \BasisF_{u' v''}
        \sum_{\substack{
            u' \in \SetCliff_\delta \\
            u' \Leq u
        }}
        \sum_{\substack{
            v'' \in \N^* \\
            \Reduction_\delta\Par{v''} \Leq v \\
            u' v'' \in \SetCliff_\delta
        }}
        \BasisF_{u' v''}
        \sum_{\substack{
            u' \in \SetCliff_\delta \\
            u' \Leq u
        }}
        \enspace
        \sum_{\substack{
            v'' \in \N^{|v|} \\
            \forall i \in [|v|],
                v_i < \delta(i) \Rightarrow v''_i \leq v_i \\
            u' v'' \in \SetCliff_\delta
        }}
        \BasisF_{u' v''}.
    \end{equation}
    The equality between the third and the last member of~\eqref{equ:product_cliff_h_basis}
    is a consequence of the fact that for any $v'' \in \N^*$, one has $
    \Reduction_\delta\Par{v''} \Leq v$ if and only if for all $i \in [|v|]$, $v_i <
    \delta(i)$ implies $v''_i \leq v_i$.  By definition of the $\BasisH$-basis provided
    by~\eqref{equ:basis_h_cliff}, and since $\BasisF_{\Reduction_\delta(u \Under v)}$ is the
    element with the greatest index appearing in the last member
    of~\eqref{equ:product_cliff_h_basis}, this expression is equal to the stated formula.
\end{proof}
\medbreak

%%%%%%%%%%%%%%%%%%%%%%%%%%%%%%%%%%%%%%%%%%%%%%%%%%%%%%%%%%%%%%%%%%%%%%%%%%%%%%%%%%%%%%%%%%%%
\subsubsection{Presentation by generators and relations}
\label{subsubsec:presentation_cliff_algebra}
A nonempty $\delta$-cliff $u$ is \Def{$\delta$-prime} if the relation $u = v \Over w$ with
$v, w \in \SetCliff_\delta$ implies $(v, w) \in \Bra{(u, \epsilon), (\epsilon, u)}$.  The
graded collection of all these elements is denoted by $\SetPrime_\delta$. For instance, for
$\delta := 021^\omega$, among others, the $\delta$-cliffs $0$, $01$, and $021111$ are
$\delta$-prime, and $0210 = 021 \Over 0$ and $011101 = 0111 \Over 01$ are not.
\medbreak

\begin{Proposition} \label{prop:cliff_algebra_minimal_generating_set}
    For any range map $\delta$, the set $\Bra{\BasisE_u : u \in \SetPrime_\delta}$ is a
    minimal generating family of the unital magmatic algebra $\Par{\SpaceCliff_\delta,
    \Product, \Unit}$.
\end{Proposition}
\begin{proof}
    Let us call $G$ the set of the elements of $\SpaceCliff_\delta$ appearing in the
    statement of the proposition.  We proceed by proving that any $\BasisE_u$, $u \in
    \SetCliff_\delta$, can be expressed as a product of elements of $G$ by induction on the
    size $n$ of $u$. Since for any $\lambda \in \K$, $\Unit(\lambda) = \lambda
    \BasisE_\epsilon$, $\BasisE_\epsilon$ is the unit of $\SpaceCliff_\delta$. Since
    moreover $\BasisE_\epsilon$ is generated by the empty product of elements of $G$, the
    property is true for $n = 0$.  Assume now that $u$ is a nonempty $\delta$-cliff
    satisfying $u = u^{(1)} \Over \cdots \Over u^{(k - 1)} \Over u^{(k)}$ where $k \geq 1$
    is maximal and the $u^{(i)}$, $i \in [k]$, are some nonempty $\delta$-cliffs.  If $k =
    1$, since $k$ is maximal, $u$ is $\delta$-prime and therefore $\BasisE_u$ belongs to
    $G$.  Otherwise, by setting $u' := u^{(1)} \Over \cdots \Over u^{(k - 1)}$, we have $u =
    u' \Over u^{(k)}$ where $u'$ and $u^{(k)}$ are both nonempty $\delta$-cliffs. Then, by
    Proposition~\ref{prop:product_cliff_e_basis},
    \begin{math}
        \BasisE_u
        = \BasisE_{u' \Over u^{(k)}}
        = \BasisE_{u'} \Product \BasisE_{u^{(k)}}.
    \end{math}
    Since the degrees of $u'$ and of $u^{(k)}$ are both smaller than $n$, by induction
    hypothesis, $\BasisE_{u'}$ and $\BasisE_{u^{(k)}}$ are generated by $G$.  Therefore,
    $\BasisE_u$ also is.
    \smallbreak

    It remains to prove that $G$ is minimal w.r.t. set inclusion. For this, let any $u \in
    \SetPrime_\delta$ and set $G' := G \setminus \Bra{\BasisE_u}$. Since by definition of
    $\delta$-prime elements, and due to the product rule on the $\BasisE$-basis provided by
    Proposition~\ref{prop:product_cliff_e_basis}, $\BasisE_u$ cannot be expressed as a
    product of elements of $G'$, $\BasisE_u$ is not generated by $G'$. Therefore, $G$ is
    minimal.
\end{proof}
\medbreak

\begin{Lemma} \label{lem:unique_factorization_prime_cliffs}
    Let $\delta$ be a range map. If $u$ is a nonempty $\delta$-cliff, then $u$ admits as
    suffix a unique $\delta$-prime $\delta$-cliff.
\end{Lemma}
\begin{proof}
    Assume by contradiction that there are two different suffixes $w$ and $w'$ of $u$ which
    are $\delta$-prime. Therefore, $|w| \ne |w'|$, and the shortest word among $w$ and $w'$
    is the suffix of the other.  Assume without loss of generality that $w'$ is shorter than
    $w$. This implies that there is a nonempty word $v \in \N^*$ such that $w = v w'$.  Now,
    since by hypothesis $w$ is a $\delta$-cliff, and since any prefix of a $\delta$-cliff is
    a $\delta$-cliff, we obtain that $v$ is a $\delta$-cliff. Therefore, $w = v \Over w'$
    where $v$ and $w'$ are both nonempty $\delta$-cliffs. This implies that $w$ is not
    $\delta$-prime, which is in contradiction with our assumptions.
\end{proof}
\medbreak

Let $\AlphabetVar_{\SetPrime_\delta}$ be the alphabet $\Bra{a_u : u \in \SetPrime_\delta}$.
We denote by $\K \Angle{\AlphabetVar_{\SetPrime_\delta}}$ the free associative algebra
generated by $\AlphabetVar_{\SetPrime_\delta}$. By definition, the elements of this algebra
are noncommutative polynomials on $\AlphabetVar_{\SetPrime_\delta}$. For any $u \in
\SetCliff_\delta$, we denote by $a^u$ the monomial $a_{u^{(1)}} \dots a_{u^{(k)}}$ where
$\Par{u^{(1)}, \dots, u^{(k)}}$, $k \geq 0$, is the unique sequence of $\delta$-prime
$\delta$-cliffs such that $u = u^{(1)} \Over \cdots \Over u^{(k)}$.  By
Lemma~\ref{lem:unique_factorization_prime_cliffs}, this definition is consistent since any
$\delta$-cliff admits exactly one factorization on $\delta$-prime $\delta$-cliffs.
\medbreak

\begin{Proposition} \label{prop:cliff_algebra_presentation}
    For any valley-free range map $\delta$, the unital associative algebra
    $\Par{\SpaceCliff_\delta, \Product, \Unit}$ is isomorphic to
    \begin{math}
        \K \Angle{\AlphabetVar_{\SetPrime_\delta}}/_{\RelationSpace_\delta}
    \end{math}
    where $\RelationSpace_\delta$ is the subspace of $\K
    \Angle{\AlphabetVar_{\SetPrime_\delta}}$ defined as the linear span of the set
    \begin{equation}
        \Bra{a_{u^{(1)}} \dots a_{u^{(k)}} : k \geq 1,
        \mbox{ for all } i \in [k], u^{(i)} \in \SetPrime_\delta,
        \mbox{ and } u^{(1)} \dots u^{(k)} \notin \SetCliff_\delta}.
    \end{equation}
\end{Proposition}
\begin{proof}
    Let us prove that $\RelationSpace_\delta$ is an ideal of $\K
    \Angle{\AlphabetVar_{\SetPrime_\delta}}$.  Let $u := u^{(1)} \dots u^{(k)}$ such that
    for all $i \in [k]$, $u^{(i)} \in \SetPrime_\delta$, and $u \notin \SetCliff_\delta$.
    Let also $v := v^{(1)} \dots v^{(\ell)}$ such that for all $i \in [\ell]$, $v^{(i)}
    \in \SetPrime_\delta$, and $v \in \SetCliff_\delta$.  Therefore, we have
    $a_{u^{(1)}} \dots a_{u^{(k)}} \in \RelationSpace_\delta$ and $a_{v^{(1)}} \dots
    a_{v^{(\ell)}} \in \K \Angle{\AlphabetVar_{\SetPrime_\delta}}$. Since $u \notin
    \SetCliff_\delta$, there is an index $j \in [|u|]$ such that $u_j > \delta(j)$.
    This implies that $(u v)_j = u_j > \delta(j)$, so that $u v \notin
    \SetCliff_\delta$.  Hence, $a_{u^{(1)}} \dots a_{u^{(k)}} a_{v^{(1)}} \dots
    a_{v^{(\ell)}} \in \RelationSpace_\delta$.  Now, assume by contradiction that $v u
    \in \SetCliff_\delta$.  Since $u \notin \SetCliff_\delta$, there is an index $j \in
    [|u|]$ such that $u_j > \delta(j)$. This letter $u_j$ appears at a certain position
    $j'$ of a factor $u^{(i)}$, $i \in [k]$, of $u$. Since $u^{(i)} \in
    \SetCliff_\delta$, we have $u_j = u^{(i)}_{j'} \leq \delta\Par{j'}$. Besides, since
    $v u \in \SetCliff_\delta$, we have $u_j = (v u)_{|v| + j} \leq \delta(|v| + j)$.
    Therefore, we obtain
    \begin{math}
        \delta\Par{j'} > \delta(j) < \delta\Par{|v| + j},
    \end{math}
    and since $j' < j < |v| + j$, this leads to the fact that $\delta$ has a valley, which
    is in contradiction with our hypothesis. Hence, $v u \notin \SetCliff_\delta$ so that
    $a_{v^{(1)}} \dots a_{v^{(\ell)}} a_{u^{(1)}} \dots a_{u^{(k)}} \in
    \RelationSpace_\delta$.  This shows that $\RelationSpace_\delta$ is an ideal of $\K
    \Angle{\AlphabetVar_{\SetPrime_\delta}}$.
    \smallbreak

    Let the linear map
    \begin{math}
        \phi :
        \SpaceCliff_\delta
        \to \K \Angle{\AlphabetVar_{\SetPrime_\delta}}/_{\RelationSpace_\delta}
    \end{math}
    satisfying $\phi\Par{\BasisE_u} = a^u$ for any $u \in \SetCliff_\delta$. Recall that
    Lemma~\ref{lem:unique_factorization_prime_cliffs} implies that $a^u$ is a well-defined
    monomial of
    \begin{math}
        \K \Angle{\AlphabetVar_{\SetPrime_\delta}}/_{\RelationSpace_\delta}.
    \end{math}
    Recall also that since $\delta$ is valley-free, by
    Theorem~\ref{thm:cliff_coassociative_coalgebra}, $\SpaceCliff_\delta$ is a unital
    associative algebra. Hence, it remains to prove that $\phi$ is a unital associative
    algebra isomorphism. First, $\phi$ is an isomorphism of spaces since it establishes a
    one-to-one correspondence between basis elements $\BasisE_u$, $u \in \SetCliff_\delta$,
    of $\SpaceCliff_\delta$ and basis elements $\phi\Par{\BasisE_u} = a^u$ of
    \begin{math}
        \K \Angle{\AlphabetVar_{\SetPrime_\delta}}/_{\RelationSpace_\delta}.
    \end{math}
    Let $u, v \in \SetCliff_\delta$.  By Proposition~\ref{prop:product_cliff_e_basis}, when
    $u \Over v \in
    \SetCliff_\delta$, we have
    \begin{equation}
        \phi\Par{\BasisE_u \Product \BasisE_v}
        = \phi\Par{\BasisE_{u \Over v}}
        = a^{u v}
        = a^u a^v
        = \phi\Par{\BasisE_u} \phi\Par{\BasisE_v}.
    \end{equation}
    Moreover, when $u \Over v \notin \SetCliff_\delta$, we have
    \begin{equation}
        \phi\Par{\BasisE_u \Product \BasisE_v}
        = \phi(0)
        = 0
        = a^u a^v
        = \phi\Par{\BasisE_u} \phi\Par{\BasisE_v}
    \end{equation}
    since $a^u a^v \in \RelationSpace_\delta$. Finally, we have $\phi\Par{\BasisE_\epsilon}
    = a^\epsilon = 1$.  This shows that $\phi$ is a unital associative algebra morphism.
\end{proof}
\medbreak

Let $\LeqSuffix$ be the partial order relation on the monomials of $\K
\Angle{\AlphabetVar_{\SetPrime_\delta}}$ wherein for any monomials $a_{u^{(1)}} \dots
a_{u^{(k)}}$ and $a_{v^{(1)}} \dots a_{v^{(\ell)}}$ of $\K
\Angle{\AlphabetVar_{\SetPrime_\delta}}$, one has
\begin{math}
    a_{u^{(1)}} \dots a_{u^{(k)}} 
    \LeqSuffix a_{v^{(1)}} \dots a_{v^{(\ell)}}
\end{math}
if the word $u^{(1)} \dots u^{(k)}$ is a suffix of $v^{(1)} \dots v^{(\ell)}$. Given a set
$M$ of monomials of $\K \Angle{\AlphabetVar_{\SetPrime_\delta}}$, we denote by
$\min_{\LeqSuffix} M$ the set of all minimal elements of $M$ w.r.t. the order
relation~$\LeqSuffix$.
\medbreak

\begin{Proposition} \label{prop:cliff_algebra_relations}
    For any valley-free range map $\delta$, the ideal $\RelationSpace_\delta$ of $\K
    \Angle{\AlphabetVar_{\SetPrime_\delta}}$ is minimally generated by the set
    \begin{equation}
        \min_{\LeqSuffix} \Bra{a^u a_v : u \in \SetCliff_\delta, v \in \SetPrime_\delta,
        \mbox{ and } u v \notin \SetCliff_\delta}.
    \end{equation}
\end{Proposition}
\begin{proof}
    Let us call $G$ the set of monomials of $\K \Angle{\AlphabetVar_{\SetPrime_\delta}}$
    appearing in the statement of the proposition.  First, since for any $a^u a_v \in G$, we
    have $u v \notin \SetCliff_\delta$, by
    Proposition~\ref{prop:cliff_algebra_presentation}, $a^u a_v \in
    \RelationSpace_\delta$.  For this reason, the ideal generated by $G$ is a subspace of
    $\RelationSpace_\delta$. Now, let $x := a_{u^{(1)}} \dots a_{u^{(k)}}$ be a basis
    element of $\RelationSpace_\delta$. Since $u^{(1)} \dots u^{(k)} \notin
    \SetCliff_\delta$, there is a smallest index $\ell \in [k - 1]$ such that $u' :=
    u^{(1)} \dots u^{(\ell)} \in \SetCliff_\delta$ and $u^{(1)} \dots u^{(\ell)} u^{(\ell +
    1)} \notin \SetCliff_\delta$. Then, by setting $v' := u^{(\ell + 1)}$, we have $x =
    a^{u'} a_{v'} u^{(\ell + 2)} \dots u^{(k)}$. Therefore, $x$ belongs to the ideal
    generated by $G$. This shows that $G$ generates $\RelationSpace_\delta$ as an ideal.  It
    remains to show that $G$ is minimal w.r.t. set inclusion. For this, assume that $G$ is
    nonempty. Let any $x := a^u a_v \in G$ and set $G' := G \setminus \{x\}$.  Since $x$ is
    a minimal element in $G$ w.r.t. the order relation $\LeqSuffix$, there is no $a^{u'} a_v
    \in G$ such that $a^{u'} a_v \LeqSuffix a^u a_v$. For this reason, $x$ cannot be
    expressed as a product $y x'$ where $x' \in G'$ and $y \in \K
    \Angle{\AlphabetVar_{\SetPrime_\delta}}$. Moreover, since $u$ is a $\delta$-cliff and
    all its prefixes still are, $x$ cannot by expressed as a product $x' y$ where
    $x' \in G'$ and $y \in \K \Angle{\AlphabetVar_{\SetPrime_\delta}}$. Therefore, $G'$ does
    not generates $\RelationSpace_\delta$. This shows the minimality of~$G$.
\end{proof}
\medbreak

\begin{Lemma} \label{lem:cliff_algebra_free}
    Let $\delta$ be a valley-free range map. The unital associative algebra
    $\Par{\SpaceCliff_\delta, \Product, \Unit}$ is free if and only if $\delta$ is weakly
    increasing.
\end{Lemma}
\begin{proof}
    We use here the fact that, by Proposition~\ref{prop:cliff_algebra_presentation},
    $\SpaceCliff_\delta$ is isomorphic as a unital associative algebra to the quotient
    \begin{math}
        \K \Angle{\AlphabetVar_{\SetPrime_\delta}}/_{\RelationSpace_\delta}
    \end{math}
    and the description of the generating familly of $\RelationSpace_\delta$ provided by
    Proposition~\ref{prop:cliff_algebra_relations}.
    \smallbreak

    When $\delta$ is weakly increasing, for all $u, v \in \SetCliff_\delta$, $u v \in
    \SetCliff_\delta$. Therefore, $\RelationSpace_\delta$ is the null space so that
    $\SpaceCliff_\delta$ is free as a unital associative algebra. Conversely, when $\delta$
    is not weakly increasing, there is an $i \geq 1$ such that $\delta(i) > \delta(i + 1)$.
    In this case, there are $u, v \in \SetCliff_\delta$ such that $u v \notin
    \SetCliff_\delta$.  By expressing $u$ and $v$ respectively as products of $\delta$-prime
    $\delta$-cliffs, this leads to the existence of a relation in $\RelationSpace_\delta$.
\end{proof}
\medbreak

\begin{Lemma} \label{lem:cliff_algebra_finite_type}
    Let $\delta$ be a valley-free range map. The unital associative algebra
    $\Par{\SpaceCliff_\delta, \Product, \Unit}$ admits a finite number of generators and a
    finite number of nontrivial relations between the generators if and only if $\delta$ is
    $1$-dominated.
\end{Lemma}
\begin{proof}
    We use here the fact that, by Proposition~\ref{prop:cliff_algebra_presentation},
    $\SpaceCliff_\delta$ is isomorphic as a unital associative algebra to the quotient
    \begin{math}
        \K \Angle{\AlphabetVar_{\SetPrime_\delta}}/_{\RelationSpace_\delta}
    \end{math}
    and the description of the generating familly of $\RelationSpace_\delta$ provided by
    Proposition~\ref{prop:cliff_algebra_relations}.
    \smallbreak

    Assume that $\delta$ is $1$-dominated. The $\delta$-cliffs $0$, $1$, \dots, $\delta(1)$
    are $\delta$-prime. Moreover, since $\delta$ is $1$-dominated, there is an $\ell \geq 0$
    such that any $\delta$-cliff $u$ of size $\ell$ or more decomposes as $u = v \Over w$
    such that $v \in \SetCliff_\delta(\ell)$ and $w$ is a $\delta$-cliff having only letters
    nongreater than $\delta(1)$. This implies that all $\delta$-prime $\delta$-cliffs have
    $\ell$ as maximal size. Therefore, $\SpaceCliff_\delta$ is finitely generated.
    Moreover, the finite number of nontrivial relations in $\SpaceCliff_\delta$ is the
    consequence of the finiteness of the generating set of $\SpaceCliff_\delta$ and the
    description of the relations of $\RelationSpace_\delta$.
    Indeed, there is a finite number of monomials $a^u a_v$ with $u \in \SetCliff_\delta$,
    $v \in \SetPrime_\delta$, and $u v \notin \SetCliff_\delta$ that are not suffixes of any
    other one satisfying the same description.  Conversely, assume that $\delta$ is not
    $1$-dominated. Thus, since $\delta$ is valley-free, there is an index $j \geq 1$ such
    that $\delta(1) = \dots = \delta(j)$ and for all $i \geq j + 1$, $\delta(i) >
    \delta(1)$. For any $k \geq j + 1$, set $u$ as the $\delta$-cliff of size $k$ defined by
    $u_i := \delta(i)$ for all $i \in [k]$. By
    Lemma~\ref{lem:unique_factorization_prime_cliffs}, there is a unique $\delta$-prime
    $\delta$-cliff $u'$ being a suffix of $u$.  Since $u'$ is in particular a
    $\delta$-cliff, one must have $u'_1 \leq \delta(1)$. Due to the previous description of
    $\delta$, we necessarily have $u = u'$. Therefore, $u$ is $\delta$-prime. This shows
    that there are infinitely many $\delta$-prime $\delta$-cliffs and thus, that
    $\SpaceCliff_\delta$ admits an infinite number of generators.
\end{proof}
\medbreak

The set of all valley-free range maps can be partitioned into the following four classes:
\begin{itemize}
    \item The class of \Def{type {\bf A}} range maps, containing all constant range maps;

    \item The class of \Def{type {\bf B}} range maps, containing all weakly increasing
    range maps having at least one ascent;

    \item The class of \Def{type {\bf C}} range maps, containing all $1$-dominated range
    maps having at least one descent;

    \item The class of \Def{type {\bf D}} range maps, containing all range maps that are not
    $1$-dominated and having at least one descent.
\end{itemize}
\medbreak

\begin{Theorem} \label{thm:classification_cliff_algebras}
    Let $\delta$ be a valley-free range map. Each unital associative algebra
    $\Par{\SpaceCliff_\delta, \Product, \Unit}$ admits the presentation
    \begin{math}
        \K \Angle{\AlphabetVar_{\SetPrime_\delta}}/_{\RelationSpace_\delta}
    \end{math}
    which fits into one of the following four classes:
    \begin{enumerate}[label={\it (\roman*)}]
        \item \label{item:classification_cliff_algebras_A}
        If $\delta$ is of type {\bf A}, then $\AlphabetVar_{\SetPrime_\delta}$ is
        finite and $\RelationSpace_\delta$ is the zero space;

        \item \label{item:classification_cliff_algebras_B}
        If $\delta$ is of type {\bf B}, then $\AlphabetVar_{\SetPrime_\delta}$ is
        infinite and $\RelationSpace_\delta$ is the zero space;

        \item \label{item:classification_cliff_algebras_C}
        If $\delta$ is of type {\bf C}, then $\AlphabetVar_{\SetPrime_\delta}$ is
        finite and $\RelationSpace_\delta$ is finitely generated and nonzero;

        \item \label{item:classification_cliff_algebras_D}
        If $\delta$ is of type {\bf D}, then  $\AlphabetVar_{\SetPrime_\delta}$ is
        infinite and $\RelationSpace_\delta$ is infinitely generated.
    \end{enumerate}
\end{Theorem}
\begin{proof}
    This is a consequence of the presentation by generators and relations of
    $\SpaceCliff_\delta$ provided by Propositions~\ref{prop:cliff_algebra_presentation}
    and~\ref{prop:cliff_algebra_relations}, and of the properties of the generating sets and
    relations spaces of $\SpaceCliff_\delta$ raised by Lemmas~\ref{lem:cliff_algebra_free}
    and~\ref{lem:cliff_algebra_finite_type}.
\end{proof}
\medbreak

%%%%%%%%%%%%%%%%%%%%%%%%%%%%%%%%%%%%%%%%%%%%%%%%%%%%%%%%%%%%%%%%%%%%%%%%%%%%%%%%%%%%%%%%%%%%
\subsubsection{Examples}
We provide here some examples of unital associative algebras $\SpaceCliff_\delta$ for
particular range maps $\delta$ and describe their structure thanks to the classification
provided by Theorem~\ref{thm:classification_cliff_algebras}.
\medbreak

%%%%%%%%%%%%%%%%%%%%%%%%%%%%%%%%%%%%%%%%%%%%%%%%%%%%%%%%%%%%%%%%%%%%%%%%%%%%%%%%%%%%%%%%%%%%
\paragraph{Type {\bf A}}
Let $\delta$ by a range map of type {\bf A}. Thus, there is a value $c \in \N$ such that
$\delta(i) = c$ for all $i \in \N$. Thus, $\SpaceCliff_\delta$ is the free unital
associative algebra generated by $a_0$, $a_1$, \dots, $a_c$.
\medbreak

%%%%%%%%%%%%%%%%%%%%%%%%%%%%%%%%%%%%%%%%%%%%%%%%%%%%%%%%%%%%%%%%%%%%%%%%%%%%%%%%%%%%%%%%%%%%
\paragraph{Type {\bf B}}
For any $m \geq 1$, $\MapM$ is of type {\bf B}. Each $\SpaceCliff_\MapM$ is free as a unital
algebra and its minimal generating sets are infinite.
\medbreak

\begin{Proposition} \label{prop:series_generators_m_cliff_algebra}
    For any $m \geq 0$, the generating series of the minimal generating set of
    $\SpaceCliff_\MapM$ satisfies
    \begin{equation}
        \GeneratingSeries_{\SetPrime_\MapM}(t)
        = 1 - \frac{1}{\HilbertSeries_{\SpaceCliff_\MapM}(t)}.
    \end{equation}
\end{Proposition}
\begin{proof}
    Since $\SpaceCliff_\MapM$ is a free unital associative algebra, its Hilbert series and
    the generating series of its minimal generating set satisfies the relation
    \begin{math}
        \HilbertSeries_{\SpaceCliff_\MapM}(t)
        =
        \Par{1 - \GeneratingSeries_{\SetPrime_\MapM}(t)}^{-1}.
    \end{math}
    This leads to the stated expression for $\GeneratingSeries_{\SetPrime_\MapM}(t)$.
\end{proof}
\medbreak

The first generators of $\SpaceCliff_\MapOne$ are
\begin{multline}
    a_{0}, \quad
    a_{01}, \quad
    a_{002}, a_{011}, a_{012}, \\
    a_{0003}, a_{0013}, a_{0021}, a_{0022}, a_{0023}, a_{0102}, a_{0103}, a_{0111},
    a_{0112}, a_{0113}, a_{0121}, a_{0122}, a_{0123},
\end{multline}
and the first generators of $\SpaceCliff_\MapTwo$ are
\begin{equation}
    a_{0}, \quad
    a_{01}, a_{02}, \quad
    a_{003}, a_{004}, a_{011}, a_{012}, a_{013}, a_{014}, a_{021}, a_{022}, a_{023},
    a_{024}.
\end{equation}
\medbreak

A consequence of the freeness of $\SpaceCliff_\MapOne$ is that $\SpaceCliff_\MapOne$ is
isomorphic as a unital associative algebra to $\FQSym$~\cite{MR95,DHT02}, an associative
algebra on the linear span of all permutations. This follows from the fact that $\FQSym$ is
also free as a unital associative algebra and that its Hilbert series is the same as the one
of~$\SpaceCliff_\MapOne$. Moreover, in~\cite{NT14}, the authors construct some associative
algebras $^m \FQSym$ as generalizations of $\FQSym$ whose bases are indexed by objects being
generalizations of permutations. The algebras $\SpaceCliff_\MapM$, $m \geq 0$, can therefore
be seen as other generalizations of $\FQSym$, not isomorphic to $^m \FQSym$ when~$m \geq 2$.
\medbreak

%%%%%%%%%%%%%%%%%%%%%%%%%%%%%%%%%%%%%%%%%%%%%%%%%%%%%%%%%%%%%%%%%%%%%%%%%%%%%%%%%%%%%%%%%%%%
\paragraph{Type {\bf C}}
Let the range map $\delta := 010^\omega$ of type {\bf C}. The unital associative algebra
$\SpaceCliff_\delta$ admits the presentation
\begin{math}
    \SpaceCliff_\delta \simeq \K \Angle{a_0, a_{01}} /_{\RelationSpace_\delta}
\end{math}
where $\RelationSpace_\delta$ is minimally generated by the elements
\begin{equation}
    a_0 a_{01},  a_{01} a_{01}.
\end{equation}
\medbreak

Let the range map $\delta := 0110^\omega$ of type {\bf C}. The unital associative algebra
$\SpaceCliff_\delta$ admits the presentation
\begin{math}
    \SpaceCliff_\delta \simeq \K \Angle{a_0, a_{01}, a_{011}} /_{\RelationSpace_\delta}
\end{math}
where $\RelationSpace_\delta$ is minimally generated by the elements
\begin{equation}
    a_0 a_0 a_{01}, a_{01} a_{01}, a_{01} a_0 a_{01}, a_{011} a_{01}, a_{011} a_0 a_{01},
    a_0 a_{011}, a_{01} a_{011}, a_{011} a_{011}.
\end{equation}
\medbreak

Let the range map $\delta := 210^\omega$ of type {\bf C}. The unital associative algebra
$\SpaceCliff_\delta$ admits the presentation
\begin{math}
    \SpaceCliff_\delta \simeq \K \Angle{a_0, a_1, a_2} /_{\RelationSpace_\delta}
\end{math}
where $\RelationSpace_\delta$ is minimally generated
\begin{equation}
    a_0 a_0 a_1, a_0 a_1 a_1, a_1 a_0 a_1, a_1 a_1 a_1, a_2 a_0 a_1, a_2 a_1 a_1, a_0 a_2,
    a_1 a_2, a_2 a_2.
\end{equation}
\medbreak

%%%%%%%%%%%%%%%%%%%%%%%%%%%%%%%%%%%%%%%%%%%%%%%%%%%%%%%%%%%%%%%%%%%%%%%%%%%%%%%%%%%%%%%%%%%%
\paragraph{Type {\bf D}}
Let the range map $\delta := 021^\omega$ of type {\bf D}. The unital associative algebra
$\SpaceCliff_\delta$ admits the presentation
\begin{math}
    \SpaceCliff_\delta
    \simeq
    \K \Angle{a_0, a_{01}, a_{02}, a_{011}, a_{021}, a_{0111}, a_{0211}, a_{01111}
    , a_{02111}, \dots} /_{\RelationSpace_\delta}
\end{math}
where $\RelationSpace_\delta$ is generated by the relations
\begin{equation}
    a_0 a_{02}, a_{01} a_{02}, a_{02} a_{02}, a_{011} a_{02}, a_{021} a_{02}, a_0 a_{021},
    a_{01} a_{021}, a_{02} a_{021}, a_0 a_{0211}, \dots.
\end{equation}
\medbreak

Let the range map $\delta := 1232^\omega$ of type {\bf D}. The unital associative algebra
$\SpaceCliff_\delta$ admits the presentation
\begin{math}
    \SpaceCliff_\delta
    \simeq
    \K \Angle{a_{0}, a_{1}, a_{02}, a_{12}, a_{003}, a_{013}, a_{022}, a_{023}, a_{103},
    a_{113}, a_{122}, a_{123}, \dots} /_{\RelationSpace_\delta}
\end{math}
where $\RelationSpace_\delta$ is generated by the relations
\begin{equation}
    a_{0} a_{003}, a_{1} a_{003}, a_{02} a_{003}, a_{12} a_{003}, a_{0} a_{013},
    a_{1} a_{013}, a_{02} a_{013}, a_{12} a_{013}, \dots.
\end{equation}
\medbreak

%%%%%%%%%%%%%%%%%%%%%%%%%%%%%%%%%%%%%%%%%%%%%%%%%%%%%%%%%%%%%%%%%%%%%%%%%%%%%%%%%%%%%%%%%%%%
%%%%%%%%%%%%%%%%%%%%%%%%%%%%%%%%%%%%%%%%%%%%%%%%%%%%%%%%%%%%%%%%%%%%%%%%%%%%%%%%%%%%%%%%%%%%
\subsection{Quotient algebras}
This last section of this work provides an answer to the problem set out in the
introduction.  This question concerns the possibility of constructing a hierarchy of
substructures of $\SpaceCliff_\delta$ similar to that of $\FQSym$. For this, we consider
quotients of $\SpaceCliff_\delta$ obtained by considering a graded subset $\SubFamilly$ of
$\SetCliff_\delta$ and by equating the basis elements $\BasisF_u$ with $0$ whenever $u
\notin \SubFamilly$.  As we shall see, this is possible only under some combinatorial
conditions on $\SubFamilly$.  We describe the products of these quotient algebras and give a
sufficient condition for the fact that it can be expressed by interval of the poset
$\SubFamilly(n)$ for a certain $n \geq 0$. We end this part by studying the quotients of
$\SpaceCliff_\MapM$ obtained from $\MapM$-hills and $\MapM$-canyons.
\medbreak

%%%%%%%%%%%%%%%%%%%%%%%%%%%%%%%%%%%%%%%%%%%%%%%%%%%%%%%%%%%%%%%%%%%%%%%%%%%%%%%%%%%%%%%%%%%%
\subsubsection{Quotient space}
Let $\delta$ be a range map.  Given a graded subset $\SubFamilly$ of $\SetCliff_\delta$, let
$\SpaceCliff_\SubFamilly$ be the quotient space of $\SpaceCliff_\delta$ defined by
\begin{math}
    \SpaceCliff_\SubFamilly := \SpaceCliff_\delta /_{\SpaceV_\SubFamilly}
\end{math}
such that $\SpaceV_\SubFamilly$ is the linear span of the set
\begin{math}
    \Bra{\BasisF_u : u \in \SetCliff_\delta \setminus \SubFamilly}.
\end{math}
By definition, the set
\begin{math}
    \Bra{\BasisF_u : u \in \SubFamilly}
\end{math}
is a basis of $\SpaceCliff_\SubFamilly$.
\medbreak

Let us introduce here an important combinatorial condition for the sequel on $\SubFamilly$.
We say that$\SubFamilly$ is \Def{closed by suffix reduction} if for any $u \in \SubFamilly$,
for all suffixes $u'$ of $u$, $\Reduction_\delta \Par{u'} \in \SubFamilly$.
\medbreak

\begin{Proposition} \label{prop:quotient_cliff}
    Let $\delta$ be a valley-free range map and $\SubFamilly$ be a graded subset of
    $\SetCliff_\delta$. If $\SubFamilly$ is closed by prefix and is closed by suffix
    reduction, then $\SpaceCliff_\SubFamilly$ is a quotient algebra of the unital
    associative algebra~$\Par{\SpaceCliff_\delta, \Product, \Unit}$.
\end{Proposition}
\begin{proof}
    Notice first that, since $\delta$ is valley-free, $\SpaceCliff_\delta$ is by
    Theorem~\ref{thm:cliff_coassociative_coalgebra} a well-defined unital associative
    algebra.  We have to prove that $\SpaceV_\SubFamilly$ is an associative algebra ideal of
    $\SpaceCliff_\SubFamilly$. For this, let $\BasisF_u \in \SpaceV_\SubFamilly$ and
    $\BasisF_v \in \SpaceCliff_\SubFamilly$. Let us look at
    Expression~\eqref{equ:product_cliff_algebra} for computing the product of
    $\SpaceCliff_\delta$.  Assume that there is a cliff $uv' \in \SubFamilly$ such that
    $\BasisF_{uv'}$ appears in $\BasisF_u \Product \BasisF_v$.  Then, since $\SubFamilly$ is
    closed by prefix, $u \in \SubFamilly$, which contradicts our hypothesis. For this
    reason, $\BasisF_u \Product \BasisF_v$ belongs to $\SpaceV_\SubFamilly$.  Moreover, let
    $\BasisF_u \in \SpaceCliff_\SubFamilly$ and $\BasisF_v \in \SpaceV_\SubFamilly$. Assume
    that there is a cliff $uv' \in \SubFamilly$ such that $\BasisF_{uv'}$ appears in
    $\BasisF_u \Product \BasisF_v$.  Then, since $\SubFamilly$ is closed by suffix
    reduction, one has $\Reduction_\delta\Par{v'} \in \SubFamilly$.
    By~\eqref{equ:product_cliff_algebra}, $\Reduction_\delta\Par{v'} = v$, leading to the
    fact that $v \in \SubFamilly$ holds, and which contradicts our hypothesis. Therefore,
    $\BasisF_u \Product \BasisF_v$ belongs to $\SpaceV_\SubFamilly$. This establishes the
    statement of the proposition.
\end{proof}
\medbreak

Notice that the graded subset $\SetAvalanche_\delta$ is not closed by suffix reduction. For
instance, even if $00112$ is an $\MapOne$-avalanche, the $\MapOne$-reduction of its suffix
$112$ is $012$, which is not an $\MapOne$-avalanche.
\medbreak

Let us denote by
\begin{math}
    \theta_\SubFamilly : \SpaceCliff_\delta \to \SpaceCliff_\SubFamilly
\end{math}
the canonical projection map. By definition, this map satisfies, for any $u \in
\SetCliff_\delta$,
\begin{math}
    \theta_\SubFamilly\Par{\BasisF_u} = \IndicatorFunction_{u \in \SubFamilly} \; \BasisF_u.
\end{math}
\medbreak

%%%%%%%%%%%%%%%%%%%%%%%%%%%%%%%%%%%%%%%%%%%%%%%%%%%%%%%%%%%%%%%%%%%%%%%%%%%%%%%%%%%%%%%%%%%%
\subsubsection{Product}
We show here that under some conditions of $\SubFamilly$, the product in
$\SpaceCliff_\SubFamilly$ can be described by using the poset structure of~$\SubFamilly$.
More precisely, we say that $\SpaceCliff_\SubFamilly$ has the \Def{interval condition} if
the support of any product $\BasisF_u \Product \BasisF_v$, $u, v \in \SubFamilly$, is empty
or is an interval of a poset $\SubFamilly(n)$, $n \geq 0$.
\medbreak

\begin{Lemma} \label{lem:quotient_cliff_bounds_f_basis_sublattice}
    Let $\delta$ be a range map and $\SubFamilly$ be a graded subset of $\SetCliff_\delta$
    such that for any $n \geq 0$, $\SubFamilly(n)$ is a meet (resp.\ join) semi-sublattice
    of $\SetCliff_\delta(n)$. For any $u, v \in \SubFamilly$, if $u \Over v$ is a
    $\delta$-cliff, then the set
    \begin{math}
        \Han{u \Over v, u \Under v} \cap \SubFamilly
    \end{math}
    admits at most one minimal (resp.\ maximal) element.
\end{Lemma}
\begin{proof}
    Assume that $\SubFamilly(n)$ is a meet semi-sublattice of $\SetCliff_\delta(n)$ and that
    $u \Over v \in \SetCliff_\delta$.  By Lemma~\ref{lem:over_under_well_definitions}, $u
    \Under v \in \SetCliff_\delta$ so that $I := \Han{u \Over v, u \Under v}$ is a
    well-defined interval of $\SetCliff_\delta(n)$.  Assume that there exist two
    $\delta$-cliffs $w$ and $w'$ belonging to $I \cap \SubFamilly$. Since $\SubFamilly(n)$
    is a meet semi-sublattice of $\SetCliff_\delta(n)$, by setting $w'' := w \Meet w'$, one
    has $w'' \in \SubFamilly$. Since $u \Over v$ is a lower bound of both $w$ and $w'$, we
    necessarily have $ u \Over v \Leq w''$ and $w'' \in I$.  This shows that when $I \cap
    \SubFamilly$ is nonempty, this set admits exactly one minimal element.  The proof is
    analogous for the respective part of the statement of the proposition.
\end{proof}
\medbreak

When for any $n \geq 0$, $\SubFamilly(n)$ is a lattice, we denote by $\Meet_\SubFamilly$
(resp.\ $\JJoin_\SubFamilly$) its meet (resp.\ join) operation. In this case, $\SubFamilly$
is \Def{meet-stable} (resp.  \Def{join-stable}) if, for any $n \geq 0$ and any $u, v \in
\SubFamilly(n)$, the relation $u_i = v_i$ for an $i \in [n]$ implies that the $i$-th letter
of $u \Meet_\SubFamilly v$ (resp.\ $u \JJoin_\SubFamilly v$) is equal to~$u_i$.
\medbreak

\begin{Lemma} \label{lem:quotient_cliff_bounds_f_basis_lattice}
    Let $\delta$ be a range map and $\SubFamilly$ be a closed by prefix, maximally
    extendable, and join-stable graded subset of $\SetCliff_\delta$.  For any $u, v \in
    \SubFamilly$ such  $u \Over v$ is a $\delta$-cliff, the set
    \begin{math}
        \Han{u \Over v, u \Under v} \cap \SubFamilly
    \end{math}
    admits at most one maximal element.
\end{Lemma}
\begin{proof}
    Assume that $u \Over v \in \SetCliff_\delta$. By
    Lemma~\ref{lem:over_under_well_definitions}, $u \Under v \in \SetCliff_\delta$ so that
    $I := \Han{u \Over v, u \Under v}$ is a well-defined interval of $\delta$-cliff poset.
    Assume that there exist two $\delta$-cliffs $w$ and $w'$ belonging to $I \cap
    \SubFamilly$. It follows from the hypotheses on $\SubFamilly$ of the statement that, by
    Theorem~\ref{thm:decr_incr_meet_join}, the operation $\JJoin_{\SubFamilly}$ is the join
    operation of the posets $\SubFamilly(n)$, $n \geq 0$ (see
    Section~\ref{subsubsec:incrementation_decrementation_maps}). First, since $w \Leq u
    \Under v$ and $w' \Leq u \Under v$, we have $w \JJoin w' \Leq u \Under v$. Moreover, by
    definition of the $\JJoin_{\SubFamilly}$ operation, $w'' := w \JJoin_{\SubFamilly} w'$
    is obtained by incrementing by some values some letters of $w \JJoin w'$.  Now, observe
    that due to the definitions of the operations $\Over$ and $\Under$, $w$ and $w'$ write
    respectively as $w = u r$ and $w' = u r'$ where $r$ and $r'$ are some words on $\N$.
    Moreover, if there is an index $i \in [|r|]$ such that $r_i \ne r'_i$, then $v_i =
    \delta(i)$ and $(u \Under v)_{|u| + i} = \delta(|u| + i)$.  This, the definition of the
    $\JJoin_{\SubFamilly}$ operation, and the fact that $\SubFamilly$ is join-stable imply
    that $w'' \Leq u \Under v$. Therefore, $w'' \in I \cap \SubFamilly$.  This shows that
    when $I \cap \SubFamilly$ is nonempty, this set admits exactly one maximal element.
\end{proof}
\medbreak

\begin{Theorem} \label{thm:quotient_cliff_bounds_f_basis}
    Let $\delta$ be a valley-free range map and $\SubFamilly$ be a graded subset of
    $\SetCliff_\delta$ closed by prefix and by suffix reduction. If at least one the
    following conditions is satisfied:
    \begin{enumerate}[label={\it (\roman*)}]
        \item \label{item:quotient_cliff_bounds_f_basis_1}
        for any $n \geq 0$, all posets $\SubFamilly(n)$ are sublattices of
        $\SetCliff_\delta(n)$;
        \item \label{item:quotient_cliff_bounds_f_basis_2} for any $n \geq 0$, all posets
        $\SubFamilly(n)$ are meet semi-sublattices of $\SetCliff_\delta(n)$, maximally
        extendable, and join-stable;
    \end{enumerate}
    then $\SpaceCliff_\SubFamilly$ has the interval condition.
\end{Theorem}
\begin{proof}
    First, by Proposition~\ref{prop:quotient_cliff}, $\SpaceCliff_\SubFamilly$ is a
    well-defined unital associative algebra quotient of $\SpaceCliff_\delta$.  Now, the
    product $\BasisF_u \Product \BasisF_v$ in $\SpaceCliff_\SubFamilly$ can be computed as
    the image by $\theta_\SubFamilly$ of the product of the same inputs in
    $\SpaceCliff_\delta$. By Theorem~\ref{thm:product_cliff_intervals}, this product is
    equal to zero or its support $I$ is an interval of a $\delta$-cliff poset. By
    construction of $\SpaceCliff_\SubFamilly$, the support of the product $\BasisF_u
    \Product \BasisF_v$ in $\SpaceCliff_\SubFamilly$ is equal to $I' := I \cap \SubFamilly$.
    If~\ref{item:quotient_cliff_bounds_f_basis_1} holds, then by
    Lemma~\ref{lem:quotient_cliff_bounds_f_basis_sublattice}, $I'$ admits both a minimal and
    a maximal element. If~\ref{item:quotient_cliff_bounds_f_basis_2} holds, then by
    Lemma~\ref{lem:quotient_cliff_bounds_f_basis_sublattice}, $I'$ admits a minimal element,
    and by Lemma~\ref{lem:quotient_cliff_bounds_f_basis_lattice}, $S'$ admits a maximal
    element. In both cases, $I'$ is an interval of a poset~$\SubFamilly(n)$, $n \geq 0$.
\end{proof}
\medbreak

%%%%%%%%%%%%%%%%%%%%%%%%%%%%%%%%%%%%%%%%%%%%%%%%%%%%%%%%%%%%%%%%%%%%%%%%%%%%%%%%%%%%%%%%%%%%
\subsubsection{Examples: two Fuss-Catalan associative algebras}
We define and study the associative algebras related to the $\MapM$-hill posets and to the
$\MapM$-canyon posets.
\medbreak

%%%%%%%%%%%%%%%%%%%%%%%%%%%%%%%%%%%%%%%%%%%%%%%%%%%%%%%%%%%%%%%%%%%%%%%%%%%%%%%%%%%%%%%%%%%%
\paragraph{Hill associative algebras.}
For any $m \geq 0$, let $\SpaceHill_m$ be the quotient $\SpaceCliff_{\SetHill_\MapM}$. This
quotient is well-defined due to the fact that $\SetHill_\MapM$ satisfies the conditions of
Proposition~\ref{prop:quotient_cliff}. Moreover, by
Proposition~\ref{prop:properties_hill_objects} and
Point~\ref{item:quotient_cliff_bounds_f_basis_1} of
Theorem~\ref{thm:quotient_cliff_bounds_f_basis}, $\SpaceHill_m$ has the interval condition.
For instance, one has in $\SpaceHill_1$,
\begin{subequations}
\begin{equation}
    \BasisF_{01} \Product \BasisF_{01}
    =
    \BasisF_{0111} + \BasisF_{0112} + \BasisF_{0113} + \BasisF_{0122} + \BasisF_{0123},
\end{equation}
\begin{equation}
    \BasisF_{01} \Product \BasisF_{00}
    =
    0,
\end{equation}
\begin{equation}
    \BasisF_{001} \Product \BasisF_{0122}
    =
    \BasisF_{0011122} + \BasisF_{0011222} + \BasisF_{0012222}.
\end{equation}
\end{subequations}
In $\SpaceHill_2$, one has
\begin{subequations}
\begin{equation}
    \BasisF_{02} \Product \BasisF_{023}
    =
    \BasisF_{02223} + \BasisF_{02233} + \BasisF_{02333},
\end{equation}
\begin{equation}
    \BasisF_{011} \Product \BasisF_{01}
    =
    \BasisF_{01111},
\end{equation}
\begin{equation}
    \BasisF_{0015} \Product \BasisF_{014} = 0.
\end{equation}
\end{subequations}
\medbreak

By computer exploration, minimal generating families of $\SpaceHill_1$ and $\SpaceHill_2$,
respectively up to degree $5$ and $4$, are
\begin{multline}
    \BasisF_{0},
    \quad
    \BasisF_{00},
    \quad
    \BasisF_{001}, \BasisF_{011},
    \quad
    \BasisF_{0002}, \BasisF_{0011}, \BasisF_{0012}, \BasisF_{0022}, \BasisF_{0112},
    \BasisF_{0122},
    \\
    \BasisF_{00003}, \BasisF_{00013}, \BasisF_{00023}, \BasisF_{00033}, \BasisF_{00112},
    \BasisF_{00113}, \BasisF_{00122}, \BasisF_{00123}, \BasisF_{00133}, \BasisF_{00222},
    \\
    \BasisF_{00223}, \BasisF_{00233}, \BasisF_{01113}, \BasisF_{01122}, \BasisF_{01123},
    \BasisF_{01133}, \BasisF_{01223}, \BasisF_{01233},
\end{multline}
and
\begin{multline}
    \BasisF_{0},
    \quad
    \BasisF_{00}, \BasisF_{01},
    \quad
    \BasisF_{001}, \BasisF_{002}, \BasisF_{003}, \BasisF_{012}, \BasisF_{013},
    \BasisF_{022}, \BasisF_{023},
    \\
    \BasisF_{0004}, \BasisF_{0005}, \BasisF_{0012}, \BasisF_{0013}, \BasisF_{0014},
    \BasisF_{0015}, \BasisF_{0022}, \BasisF_{0023}, \BasisF_{0024}, \BasisF_{0025},
    \BasisF_{0033}, \BasisF_{0034}, \BasisF_{0035},
    \\
    \BasisF_{0044}, \BasisF_{0045}, \BasisF_{0114}, \BasisF_{0115}, \BasisF_{0122},
    \BasisF_{0123}, \BasisF_{0124}, \BasisF_{0125}, \BasisF_{0133}, \BasisF_{0134},
    \BasisF_{0135}, \BasisF_{0144}, \BasisF_{0145},
    \\
    \BasisF_{0223}, \BasisF_{0224}, \BasisF_{0225}, \BasisF_{0234}, \BasisF_{0235},
    \BasisF_{0244}, \BasisF_{0245}.
\end{multline}
Moreover, the sequences for the numbers of generators of $\SpaceHill_1$ and $\SpaceHill_2$,
degree by degree begin respectively by
\begin{equation}
    0, 1, 1, 2, 6, 18, 59, 196, 669,
\end{equation}
and
\begin{equation}
    0, 1, 2, 7, 33, 168, 900, 4980.
\end{equation}
We can observe that for any $m \geq 1$, $\SpaceHill_m$ is not free as unital associative
algebra. Indeed, the quasi-inverse of the respective generating series of these elements is
not the Hilbert series of $\SpaceHill_m$, which is expected when this algebra is free.
\medbreak

%%%%%%%%%%%%%%%%%%%%%%%%%%%%%%%%%%%%%%%%%%%%%%%%%%%%%%%%%%%%%%%%%%%%%%%%%%%%%%%%%%%%%%%%%%%%
\paragraph{Canyon associative algebras.}
For any $m \geq 0$, let $\SpaceCanyon_m$ be the quotient $\SpaceCliff_{\SetCanyon_{\mathbf
m}}$. This quotient is well-defined due to the fact that $\SetCanyon_{\mathbf m}$ satisfies
the conditions of Proposition~\ref{prop:quotient_cliff}. Moreover, by
Proposition~\ref{prop:properties_canyon_objects}, the fact that for any $m \geq 0$ and $n
\geq 0$, $\SetCanyon_{\mathbf m}(n)$ is join-stable, and by
Point~\ref{item:quotient_cliff_bounds_f_basis_2} of
Theorem~\ref{thm:quotient_cliff_bounds_f_basis}, $\SpaceCanyon_m$ has the interval
condition. For instance, one has in $\SpaceCanyon_1$,
\begin{subequations}
\begin{equation}
    \BasisF_{0} \Product \BasisF_{01}
    =
    \BasisF_{001} + \BasisF_{002} + \BasisF_{012},
\end{equation}
\begin{equation}
    \BasisF_{0} \Product \BasisF_{002}
    =
    \BasisF_{0002} + \BasisF_{0003} + \BasisF_{0103},
\end{equation}
\begin{equation} \begin{split}
    \BasisF_{0012} \Product \BasisF_{0103}
    & =
    \BasisF_{00120103} + \BasisF_{00120106} + \BasisF_{00120107} + \BasisF_{00120406}
    + \BasisF_{00120407}
    \\
    & \enspace + \BasisF_{00120507} + \BasisF_{00123406} + \BasisF_{00123407}
    + \BasisF_{00123507} + \BasisF_{00124507}.
\end{split} \end{equation}
\end{subequations}
In $\SpaceCanyon_2$, one has
\begin{subequations}
\begin{equation}
    \BasisF_{01} \Product \BasisF_{0014}
    =
    0,
\end{equation}
\begin{equation}
    \BasisF_{01} \Product \BasisF_{0013}
    =
    \BasisF_{010013}.
\end{equation}
\begin{equation} \begin{split}
    \BasisF_{020} \Product \BasisF_{02}
    & =
    \BasisF_{02002} + \BasisF_{02005} + \BasisF_{02006} + \BasisF_{02007} + \BasisF_{02008}
    + \BasisF_{02012} + \BasisF_{02015} + \BasisF_{02016} + \BasisF_{02017} + \BasisF_{02018}
    \\
    & \enspace + \BasisF_{02045} + \BasisF_{02046} + \BasisF_{02047} + \BasisF_{02048}
    + \BasisF_{02056} + \BasisF_{02057} + \BasisF_{02058} + \BasisF_{02067}
    + \BasisF_{02068}.
\end{split} \end{equation}
\end{subequations}
\medbreak

By computer exploration, minimal generating families of $\SpaceCanyon_1$ and
$\SpaceCanyon_2$, respectively up to respectively up to degree $5$ and $4$, are
\begin{multline}
    \BasisF_{0},
    \quad
    \BasisF_{00},
    \quad
    \BasisF_{000}, \BasisF_{001},
    \quad
    \BasisF_{0000}, \BasisF_{0001}, \BasisF_{0002}, \BasisF_{0010}, \BasisF_{0012},
    \\
    \BasisF_{00000}, \BasisF_{00001}, \BasisF_{00002}, \BasisF_{00003}, \BasisF_{00010},
    \BasisF_{00012}, \BasisF_{00013}, \BasisF_{00020}, \BasisF_{00023}, \BasisF_{00100},
    \\
    \BasisF_{00101}, \BasisF_{00103}, \BasisF_{00120}, \BasisF_{00123},
\end{multline}
and
\begin{multline}
    \BasisF_{0},
    \quad
    \BasisF_{00}, \BasisF_{01},
    \quad
    \BasisF_{000}, \BasisF_{002}, \BasisF_{003}, \BasisF_{010}, \BasisF_{012},
    \BasisF_{013}, \BasisF_{023},
    \\
    \BasisF_{0000}, \BasisF_{0003}, \BasisF_{0004}, \BasisF_{0005}, \BasisF_{0014},
    \BasisF_{0015}, \BasisF_{0020}, \BasisF_{0023}, \BasisF_{0024}, \BasisF_{0025},
    \BasisF_{0030}, \BasisF_{0034}, \BasisF_{0035}, \BasisF_{0045}, \BasisF_{0100},
    \\
    \BasisF_{0104}, \BasisF_{0105}, \BasisF_{0120}, \BasisF_{0124}, \BasisF_{0125},
    \BasisF_{0130}, \BasisF_{0134}, \BasisF_{0135}, \BasisF_{0145}, \BasisF_{0204},
    \BasisF_{0205}, \BasisF_{0230}, \BasisF_{0234}, \BasisF_{0235}, \BasisF_{0245}.
\end{multline}
The associative algebra $\SpaceCanyon_1$ is the Loday-Ronco algebra~\cite{LR98},
also known as $\PBT$~\cite{HNT05}. It is known that this associative algebra is free and
that the dimension of its generators are a shifted version of Catalan numbers:
\begin{equation}
    0, 1, 1, 2, 5, 14, 42, 132, 429.
\end{equation}
The sequence for the numbers of generators of $\SpaceCanyon_2$ degree by degree begins by
\begin{equation}
    0, 1, 2, 7, 30, 149, 788, 4332.
\end{equation}
We can observe that for any $m \geq 2$, $\SpaceCanyon_m$ is not free as unital associative
algebra. It follows, from the same argument as the previous section, that $\SpaceCanyon_m$
is not free.
\medbreak

%%%%%%%%%%%%%%%%%%%%%%%%%%%%%%%%%%%%%%%%%%%%%%%%%%%%%%%%%%%%%%%%%%%%%%%%%%%%%%%%%%%%%%%%%%%%
%%%%%%%%%%%%%%%%%%%%%%%%%%%%%%%%%%%%%%%%%%%%%%%%%%%%%%%%%%%%%%%%%%%%%%%%%%%%%%%%%%%%%%%%%%%%
%%%%%%%%%%%%%%%%%%%%%%%%%%%%%%%%%%%%%%%%%%%%%%%%%%%%%%%%%%%%%%%%%%%%%%%%%%%%%%%%%%%%%%%%%%%%
\section*{Conclusion and open questions}
This work presents three new families of posets on Fuss-Catalan objects and associative
algebras on their linear spans. All this are based upon $\delta$-cliffs, a combinatorial
family of words of integers satisfying some conditions. Some general properties about
subposets of the posets of $\delta$-cliffs have been presented, as well as general
properties about quotients of the associative algebras defined on the linear span of
$\delta$-cliffs.
\medbreak

Here is a list of open questions raised by this research:
\begin{enumerate}[fullwidth]
    \item {\bf (Generalization of the weak Bruhat order)} --- The first open question
    concerns the alternative order relation on $\delta$-cliffs introduced in
    Section~\ref{subsubsec:links_weak_bruhat_order}. This consists in considering
    Conjecture~\ref{con:bruhat_order_increasing_trees} and in proving that the posets
    $\Par{\SetCliff_\delta(n), \Leq'}$ are semi-distributive lattices, or at least lattices.
    \medbreak

    \item {\bf (Coproducts and Hopf bialgebras)} --- As explained above, the associative
    algebras $\SpaceCliff_\MapOne$ and $\SpaceCanyon_\MapOne$ are already known algebraic
    structures which are in fact Hopf bialgebras. They are endowed with a coproduct
    satisfying some compatibility relations with the product. The question here consists in
    endowing $\SpaceCliff_\delta$ with a coproduct where $\delta$ is a unimodal range map.
    We can ask also for a general definition of such a coproduct for the quotients
    $\SpaceCliff_\SubFamilly$ of $\SpaceCliff_\delta$ for some subfamilies $\SubFamilly$ of
    $\delta$-cliffs.
    \medbreak

    \item {\bf (Other subposets and quotient algebras)} --- There are other subfamilies of
    $\delta$-cliffs than $\delta$-hills and $\delta$-canyons which seem to lead to
    interesting posets and associative algebras. Among these, there are
    \Def{$\delta$-dunes}, which are $\delta$-cliffs $u$ such that $\Brr{u_i - u_{i + 1}}
    \leq \Brr{\delta(i) - \delta(i + 1)}$ for all $i \in [|u| - 1]$. For $\delta = \MapOne$,
    we obtain a family in one-to-one correspondence with directed animals, which are
    enumerated by Sequence~\OEIS{A005773} of~\cite{Slo}. For $\delta = \MapTwo$, we obtain a
    family enumerated by Sequence~\OEIS{A180898} of~\cite{Slo}. This axis consists in
    studying in particular the posets and the associative algebras of dunes.
\end{enumerate}
\medbreak

%%%%%%%%%%%%%%%%%%%%%%%%%%%%%%%%%%%%%%%%%%%%%%%%%%%%%%%%%%%%%%%%%%%%%%%%%%%%%%%%%%%%%%%%%%%%
%%%%%%%%%%%%%%%%%%%%%%%%%%%%%%%%%%%%%%%%%%%%%%%%%%%%%%%%%%%%%%%%%%%%%%%%%%%%%%%%%%%%%%%%%%%%
%%%%%%%%%%%%%%%%%%%%%%%%%%%%%%%%%%%%%%%%%%%%%%%%%%%%%%%%%%%%%%%%%%%%%%%%%%%%%%%%%%%%%%%%%%%%
\bibliographystyle{alpha}
\bibliography{Bibliography}

\end{document}